\documentclass[reqno,11pt,letterpaper]{amsart}

\usepackage{amsmath}
\usepackage{amssymb}
\usepackage{amsthm}
\usepackage{mathrsfs}
\usepackage{accents}
\usepackage{calc}
\usepackage{upgreek}
\usepackage{slashed}
\usepackage{xifthen}
\usepackage{graphicx}
\usepackage{longtable}

\usepackage{xcolor}
\definecolor{winered}{rgb}{0.6,0,0}
\definecolor{lessblue}{rgb}{0,0,0.7}

\usepackage[pdftex,colorlinks=true,linkcolor=winered,citecolor=lessblue,urlcolor=lessblue,breaklinks=true,bookmarksopen=true]{hyperref}

\makeatletter
\newcommand{\myitem}[3]{\item[#2]\def\@currentlabel{#3}\label{#1}}
\makeatother

\usepackage{titletoc}

\makeatletter

\def\@tocline#1#2#3#4#5#6#7{
\begingroup
  \par
    \parindent\z@ \leftskip#3 \relax \advance\leftskip\@tempdima\relax
                  \rightskip\@pnumwidth plus 4em \parfillskip-\@pnumwidth
    \ifcase #1 
       \vskip 0.6em \hskip 0em 
       \or
       \or \hskip 0em 
       \or \hskip 1em 
    \fi%
    #6
    \nobreak\relax{\leavevmode\leaders\hbox{\,.}\hfill}
    \hbox to\@pnumwidth {\@tocpagenum{#7}}
  \par
\endgroup
}

 \def\l@section{\@tocline{0}{0pt}{0pc}{}{}}

\renewcommand{\tocsection}[3]{%
  \indentlabel{\@ifnotempty{#2}{ 
    \ignorespaces\bfseries{#2. #3}}}
  \indentlabel{\@ifempty{#2}{\ignorespaces\bfseries{#3}}{}} 
    \vspace{1.5pt}}

\renewcommand{\tocsubsection}[3]{%
  \indentlabel{\@ifnotempty{#2}{
    \ignorespaces#2. #3}}
  \indentlabel{\@ifempty{#2}{\ignorespaces #3}{}}
    \vspace{1.5pt}}

\renewcommand{\tocsubsubsection}[3]{%
  \indentlabel{\@ifnotempty{#2}{
    \ignorespaces#2. #3}}
  \indentlabel{\@ifempty{#2}{\ignorespaces #3}{}}
    \vspace{1.5pt}}

\makeatother

\makeatletter
\def\@nomenstarted{0}
\newlength{\@nomenoldtabcolsep}

\newcommand{\nomenstart}
  {%
    \def\@nomenstarted{1}%
    \setlength{\@nomenoldtabcolsep}{\tabcolsep}%
    \setlength{\tabcolsep}{3.5pt}%
    \begin{longtable}{p{0.11\textwidth} p{0.86\textwidth}}
  }

\newcommand{\nomenitem}[2]{%
    \ifcase\@nomenstarted%
      \or 
      \or \\ 
    \fi%
    #1\,{\leavevmode\leaders\hbox{\,.}\hfill} & #2%
    \def\@nomenstarted{2}%
  }%
\newcommand{\nomenend}
  {\\%
      \end{longtable}%
      \setlength{\tabcolsep}{\@nomenoldtabcolsep}%
      \def\@nomenstarted{0}%
  }
\makeatother

\makeatletter
\newcommand{\bigish}{\bBigg@{0}}
\newcommand{\vast}{\bBigg@{4}}
\newcommand{\Vast}{\bBigg@{5}}
\newcommand{\VAST}[1]{\bBigg@{#1}}
\makeatother

\allowdisplaybreaks

\numberwithin{equation}{section}
\numberwithin{figure}{section}
\newtheorem{thm}{Theorem}[section]

\newtheorem{prop}[thm]{Proposition}
\newtheorem{lemma}[thm]{Lemma}
\newtheorem{cor}[thm]{Corollary}
\newtheorem*{thm*}{Theorem}
\newtheorem*{prop*}{Proposition}
\newtheorem*{cor*}{Corollary}
\newtheorem*{conj*}{Conjecture}

\theoremstyle{definition}
\newtheorem{definition}[thm]{Definition}

\theoremstyle{remark}
\newtheorem{rmk}[thm]{Remark}
\newtheorem{example}[thm]{Example}

\makeatletter
\newcommand{\fakephantomsection}{%
  \Hy@MakeCurrentHref{\@currenvir.\the\Hy@linkcounter}
  \Hy@raisedlink{\hyper@anchorstart{\@currentHref}\hyper@anchorend}%
  \Hy@GlobalStepCount\Hy@linkcounter%
}
\makeatother

\newcommand{\mc}{\mathcal}
\newcommand{\cA}{\mc A}
\newcommand{\cB}{\mc B}
\newcommand{\cC}{\mc C}

\newcommand{\cF}{\mc F}

\newcommand{\cH}{\mc H}
\newcommand{\cI}{\mc I}

\newcommand{\cL}{\mc L}
\newcommand{\cM}{\mc M}
\newcommand{\cN}{\mc N}
\newcommand{\cO}{\mc O}
\newcommand{\cP}{\mc P}

\newcommand{\cR}{\mc R}

\newcommand{\cU}{\mc U}
\newcommand{\cV}{\mc V}
\newcommand{\cW}{\mc W}

\newcommand{\ms}{\mathscr}

\newcommand{\sC}{\ms C}

\newcommand{\scri}{\ms I}

\newcommand{\C}{\mathbb{C}}
\newcommand{\N}{\mathbb{N}}
\newcommand{\R}{\mathbb{R}}
\newcommand{\Z}{\mathbb{Z}}

\newcommand{\Sph}{\mathbb{S}}

\newcommand{\sfH}{\mathsf{H}}

\newcommand{\fm}{\mathfrak{m}}

\newcommand{\ft}{\mathfrak{t}}

\newcommand{\slg}{\slashed{g}{}}

\newcommand{\slomega}{\slashed{\omega}{}}

\newcommand{\Err}{{\mathrm{Err}}{}}

\newcommand{\End}{\operatorname{End}}

\newcommand{\Hom}{\operatorname{Hom}}
\renewcommand{\Re}{\operatorname{Re}}
\renewcommand{\Im}{\operatorname{Im}}

\newcommand{\supp}{\operatorname{supp}}
\newcommand{\sgn}{\operatorname{sgn}}

\newcommand{\tr}{\operatorname{tr}}

\newcommand{\dv}{\operatorname{div}}

\newcommand{\rank}{\operatorname{rank}}

\newcommand{\diag}{\operatorname{diag}}

\newcommand{\eps}{\epsilon}

\newcommand{\hra}{\hookrightarrow}
\newcommand{\la}{\langle}

\newcommand{\ol}{\overline}
\newcommand{\pa}{\partial}
\newcommand{\dd}{{\mathrm d}}
\newcommand{\ra}{\rangle}
\newcommand{\spec}{\operatorname{spec}}

\newcommand{\ul}[1]{\underline{#1}{}}

\newcommand{\wh}{\widehat}
\newcommand{\wt}{\widetilde}
\newcommand{\xra}{\xrightarrow}
\newcommand{\ubar}[1]{\underaccent{\bar}#1}
\newcommand{\pfstep}[1]{$\bullet$\ \underline{\textit{#1}}}
\newcommand{\pfsubstep}[2]{{\bf#1}\ \textit{#2}}

\newcommand{\bop}{{\mathrm{b}}}

\newcommand{\scop}{{\mathrm{sc}}}

\newcommand{\ebop}{{\mathrm{e,b}}}
\newcommand{\eop}{{\mathrm{e}}}

\newcommand{\semi}{\hbar}

\newcommand{\ff}{\mathrm{ff}}

\newcommand{\lb}{{\mathrm{lb}}}
\newcommand{\rb}{{\mathrm{rb}}}

\newcommand{\cp}{{\mathrm{c}}}

\newcommand{\Diff}{\mathrm{Diff}}

\DeclareMathOperator{\Op}{Op}

\newcommand{\Vb}{\cV_\bop}
\newcommand{\Vz}{\cV_0}
\newcommand{\Ve}{\cV_\eop}
\newcommand{\Diffb}{\Diff_\bop}

\newcommand{\Veb}{\cV_\ebop}

\newcommand{\Diffeb}{\Diff_\ebop}
\newcommand{\Psieb}{\Psi_\ebop}

\newcommand{\Vsc}{\cV_\scop}

\newcommand{\WF}{\mathrm{WF}}
\newcommand{\Ell}{\mathrm{Ell}}
\newcommand{\Char}{\mathrm{Char}}

\newcommand{\Omegab}{{}^{\bop}\Omega}
\newcommand{\Omegaeb}{{}^{\ebop}\Omega}
\newcommand{\Omegasc}{{}^{\scop}\Omega}

\newcommand{\Tb}{{}^{\bop}T}

\newcommand{\Tsc}{{}^\scop T}

\newcommand{\Teb}{{}^{\ebop}T}

\newcommand{\Seb}{{}^{\ebop}S}

\newcommand{\WFeb}{\WF_{\ebop}}
\newcommand{\Elleb}{\mathrm{Ell}_\ebop}

\newcommand{\half}{{\tfrac{1}{2}}}

\newcommand{\sigmaeb}{{}^\ebop\upsigma}

\newcommand{\loc}{{\mathrm{loc}}}
\newcommand{\CI}{\cC^\infty}

\newcommand{\CIc}{\cC^\infty_\cp}

\newcommand{\Hb}{H_{\bop}}

\newcommand{\Hext}{\bar H}

\newcommand{\Hsupp}{\dot H}

\newcommand{\Heb}{H_{\ebop}}

\newcommand{\Hsc}{H_{\scop}}

\newcommand{\Ric}{\mathrm{Ric}}

\newcommand{\bhm}{\fm}

\newcommand{\openbigpmatrix}[1]
  {%
    \def\@bigpmatrixsize{#1}%
    \addtolength{\arraycolsep}{-#1}%
    \begin{pmatrix}%
  }
\newcommand{\closebigpmatrix}
  {%
    \end{pmatrix}%
    \addtolength{\arraycolsep}{\@bigpmatrixsize}%
  }

\newlength{\enummargin}

\newcommand{\usref}[1]{{\upshape\ref{#1}}}

\DeclareGraphicsExtensions{.mps}

\begin{document}

\title[Microlocal analysis near null infinity]{Microlocal analysis near null infinity in asymptotically flat spacetimes}

\date{\today}

\subjclass[2010]{Primary 35L05, Secondary 58J47, 35B40}

\author{Peter Hintz}
\address{Department of Mathematics, ETH Z\"urich, R\"amistrasse 101, 8092 Z\"urich, Switzerland}
\email{peter.hintz@math.ethz.ch}

\author{Andr\'as Vasy}
\address{Department of Mathematics, Stanford University, CA 94305-2125, USA}
\email{andras@math.stanford.edu}

\begin{abstract}
  We present a novel approach to the analysis of regularity and decay for solutions of wave equations in a neighborhood of null infinity in asymptotically flat spacetimes of any dimension. The classes of metrics and wave type operators we consider near null infinity include those arising in nonlinear stability problems for Einstein's field equations in $1+3$ dimensions. In a neighborhood of null infinity, in an appropriate compactification of the spacetime to a manifold with corners, the wave operators are of edge type at null infinity and totally characteristic at spacelike and future timelike infinity. On a corresponding scale of Sobolev spaces, we demonstrate how microlocal regularity propagates across or into null infinity via a sequence of radial sets. As an application, inspired by work of the second author with Baskin and Wunsch, we prove regularity and decay estimates for forward solutions of wave type equations on asymptotically flat spacetimes which are asymptotically homogeneous with respect to scaling in the forward timelike cone and have an appropriate structure at null infinity. These estimates are new even for the wave operator on Minkowski space.

  The results obtained here are also used as black boxes in a global theory of wave type equations on asymptotically flat and asymptotically stationary spacetimes developed by the first author.
\end{abstract}

\maketitle

\setlength{\parskip}{0.00in}
\tableofcontents
\setlength{\parskip}{0.05in}

\section{Introduction}
\label{SI}

In this paper, we introduce a new point of view for the analysis of linear waves on asymptotically flat spacetimes near null infinity. The main novelty is the fully microlocal nature of our approach (apart from a simple energy estimate). This allows one to combine the estimates proved here in the usual modular microlocal fashion with microlocal estimates far from null infinity; a first implementation of this is given in \cite{HintzNonstat}.

Our regularity theory appears to be new even on Minkowski space. Let us work on $\R^{1+n}=\R_t\times\R_x^n$ (with $n\geq 1$) equipped with the Minkowski metric $g_0=-\dd t^2+\sum_{j=1}^n(\dd x^j)^2$ and volume density $|\dd g_0|=|\dd t\,\dd x^1\,\cdots\,\dd x^n|$, and introduce polar coordinates $x=r\omega$, $r\geq 0$, $\omega\in\Sph^{n-1}$, on $\R^n$. In the exterior domain
\[
  \Omega = \{ (t,x)\in\R^{1+n} \colon 0\leq t\leq r-1 \},
\]
and for weights $\alpha_0,\alpha_{\!\scri}\in\R$, we set\footnote{The factor of $2$ in the second order is introduced so that $\alpha_{\!\scri}$ corresponds to the amount of $r$-decay at null infinity ($|t-r|\lesssim 1$, $r\to\infty$) as measured by powers of $r^{-1}$. The normalization of the orders of the Sobolev space on the other hand corresponds to the geometric singular analysis structure, i.e.\ the edge-b-structure which is explained further below.}
\begin{equation}
\label{EqIHeb0}
\begin{split}
  &\Heb^{0,(\alpha_0,2\alpha_{\!\scri})}(\Omega) = \rho_0^{\alpha_0}x_{\!\scri}^{2\alpha_{\!\scri}} L^2(\Omega,|\dd g_0|) = \bigl\{ \rho_0^{\alpha_0}x_{\!\scri}^{2\alpha_{\!\scri}}u \colon u\in L^2(\Omega,|\dd g_0|) \bigr\}, \\
  &\qquad \rho_0:=\frac{1}{r-t},\quad
  x_{\!\scri}:=\sqrt{\frac{r-t}{r}}\,.
\end{split}
\end{equation}
Define the vector fields (which we will refer to as \emph{edge-b-vector fields})\footnote{One can replace $V_0$ here by $2 V_0-\frac{t+r}{r}V_1=(t-r)(\pa_t-\pa_r)$, which is a weighted derivative along incoming light cones.}
\begin{equation}
\label{EqIVfs}
\begin{alignedat}{3}
  V_0&=-\rho_0\pa_{\rho_0}&&=t\pa_t+r\pa_r &&\quad \text{(scaling)}, \\
  V_1&=-\half x_{\!\scri}\pa_{x_{\!\scri}}&&=r(\pa_t+\pa_r) &&\quad \text{(weighted derivative along outgoing light cones)}, \\
  V_a&=x_{\!\scri}\Omega_a&&\ \ (a=2,\ldots,N) &&\quad\text{(where the $\Omega_a$ span $\cV(\Sph^{n-1})$ over $\CI(\Sph^{n-1})$)}.
\end{alignedat}
\end{equation}
We then define \emph{weighted edge-b-Sobolev spaces} for $s\in\N_0$ by
\[
  \Heb^{s,(\alpha_0,2\alpha_{\!\scri})}(\Omega) = \{ u\in\Heb^{0,(\alpha_0,2\alpha_{\!\scri})}(\Omega) \colon V^\beta u\in\Heb^{0,(\alpha_0,2\alpha_{\!\scri})}(\Omega)\ \forall\,\beta\in\N_0^N,\ |\beta|\leq s \}.
\]

\begin{thm}[Edge-b-regularity and decay of waves in the exterior domain]
\label{ThmI}
  Suppose $\alpha_{\!\scri}<\min(-\half,\alpha_0+\half)$, and let $s\in\N$. Let $f\in\Heb^{s-1,(\alpha_0+2,2\alpha_{\!\scri}+2)}(\Omega)$. Then the forward solution $u$ of $\Box_{g_0}u=f$ (that is, the unique solution of $(-D_t^2+\sum_{j=1}^n D_{x^j}^2)u=f$ with $u|_{t<0}=0$) satisfies $u\in\Heb^{s,(\alpha_0,2\alpha_{\!\scri})}(\Omega)$.
\end{thm}

As we will discuss in~\S\ref{SsIeb}, a key property of the operator $\Box_{g_0}$ used in Theorem~\ref{ThmI} is the fact that, up to an overall weight $\rho_0^2 x_{\!\scri}^2$, it is to leading order a nondegenerate (Lorentzian signature) quadratic form in the edge-b-vector fields~\eqref{EqIVfs}, uniformly as $r\to\infty$ in $\Omega$. We also remark that these vector fields are related to a parabolic scaling near null hypersurfaces, the null hypersurface of interest here being null infinity in an appropriate compactification of $\Omega$; see Remark~\ref{RmkIParabolic}.

The restriction $\alpha_{\!\scri}<-\half$ in Theorem~\ref{ThmI} is sharp, as it precisely guarantees that the edge-b-Sobolev space for $u$ permits the well-known pointwise $r^{-\frac{n-1}{2}}$ decay along light cones $t-r=\rm{const}.$ towards null infinity $\scri^+$. The restriction $\alpha_{\!\scri}<\alpha_0+\half$ is likewise necessary since a forcing term $f$ which is large (i.e.\ $\alpha_0<-1$, roughly corresponding to less than pointwise $r^{-\frac{n-1}{2}}$ decay) near spacelike infinity $I^+$ (i.e.\ for $t/r\in[0,1)$ with $r$ large) produces a wave $u$ which is large at $I^0$ and $\scri^+$. Roughly speaking, $u$ has two orders of decay (as measured by powers of $r^{-1}$) less than $f$ at $I^0\setminus\scri^+$, and one order of decay less than $f$ at $\scri^+\setminus I^0$.

Theorem~\ref{ThmI} remains valid on a large class of generalizations of Minkowski space, including those arising in $n+1=3+1$ dimensions as solutions of the Einstein field equations in the context of the nonlinear stability of Minkowski space, or more generally in the context of the existence of a piece of null infinity for asymptotically flat data sets \cite{ChristodoulouKlainermanStability,KlainermanNicoloEvolution,LindbladRodnianskiGlobalStability,HintzVasyMink4}. (Whether the spacetimes constructed in \cite{BieriZipserStability,BieriChruscielADMBondi} under minimal assumptions on the initial data lie in this class is not clear at this point.) In fact, the companion paper \cite{HintzMink4Gauge} revisits this latter problem from the edge-b-perspective. We discuss such geometric generalizations in~\S\ref{SsIeb}; see also Theorem~\ref{ThmSA}.

Our proof of Theorem~\ref{ThmI} proceeds via an energy estimate in the case $s=1$ (using a weighted linear combination of $V_0,V_1$ as the vector field multiplier)---see Theorem~\ref{ThmMEFw}---and follows for general $s$ (including real orders above a certain negative threshold) from results on the microlocal propagation of edge-b-regularity; see~\S\ref{SsIeb}. If the forcing term $f$ remains in the stated space upon application of up to $k\in\N_0$ vector fields $V_0,V_1,\Omega_a$ (we shall refer to this as $k$ orders of \emph{b-regularity}\footnote{This is the general terminology for regularity under repeated application of smooth vector fields, on a manifold with corners, which are tangent to all boundary hypersurfaces. The relevant manifold in the present setting is depicted in Figure~\ref{FigIIBlowup} below.}), then the solution $u$ has the same extra regularity. Discarding the edge-b-regularity information on $u$, and taking the $\Omega_a$ to be generators of rotations (which are thus symmetries of the spacetime), this is essentially the original starting point of Klainerman's vector field method \cite{KlainermanUniformDecay} (see also \cite{DafermosRodnianskiRp,MoschidisRp} for recent developments). A minor novelty of our approach is that we can prove b-regularity using arbitrary spherical vector fields; correspondingly, the underlying metric or operator under study need not have any (asymptotic) spherical symmetry at $\scri^+$.

We already remark here that there is a significant difference between edge-b- and b-regularity (besides b-regularity being stronger as far as regularity in the spherical directions is concerned): roughly speaking, edge-b-regularity can be tracked \emph{microlocally} using by now essentially off-the-shelf microlocal techniques, namely, symbolic positive commutator arguments (see~\S\ref{SM}). On the other hand, a satisfactory microlocal framework for b-regularity at $\scri^+$ remains elusive; we explain the structural reason in Remark~\ref{RmkINob}.

In light of the success of vector field methods, we shall attempt to provide further justification for our insistence on developing a microlocal approach. In order to do so, we turn from the exterior region $\Omega$ to the forward causal cone. (There, we can define edge-b-Sobolev spaces in an analogous manner, now using weights in $\rho_+:=\frac{1}{t-r}$ and $x_{\!\scri}:=\sqrt{\frac{t-r}{r}}$ and testing regularity using $\rho_+\pa_{\rho_+}$, $x_{\!\scri}\pa_{x_{\!\scri}}$, and $x_{\!\scri}\Omega_a$.) Suppose that we are given a metric on $g$ on $\R^{1+n}$ which is (approximately) equal to $g_0$ in a neighborhood $r/t\in(1-\eps,1+\eps)$, $r\gg 1$, of $\scri^+$. Crucially note then that one \emph{cannot} solve the wave equation $\Box_g u=f$ \emph{locally} near $\scri^+$ (or even just in regions such as $t-r\geq 1$, $r/t\in(1-\eps,1]$), since the behavior of $u$ in such a region is global in character: it depends on the spacetime geometry (or more generally on the coefficients of the wave type operator under consideration), and on the forcing $f$, in a full neighborhood of future timelike infinity $I^+$---which includes regions far from $\scri^+$. A concrete example to keep in mind is the case that $g$ is a Schwarzschild or Kerr metric, or a perturbation thereof. Thus, if even \emph{just one} part of the analysis of $\Box_g$ near future timelike infinity (i.e.\ in a region where $r/t<1-\eps$ and $t\gg 1$) uses microlocal tools, it is desirable to have a microlocal perspective in a full neighborhood of $I^+$, including near (a future affine complete part of) $\scri^+$. More specifically, microlocal regularity results near critical or invariant sets of the null-geodesic flow (lifted to the cotangent bundle), such as radial points or normally hyperbolic trapping, take the form: \emph{if regularity is known on the stable manifold of the critical/invariant set, then it holds at the set itself and thus also on its unstable manifold}. (In the noncompact setting of interest here, `regularity' entails uniform ($L^2$-)integrability of $u$ and its derivatives.) A microlocal perspective near $\scri^+$ is exactly what allows for a clean separation of those null-geodesics, lifted to phase space, which are incoming (i.e.\ which are on their way to a region far away from $\scri^+$ where they may encounter, say, a black hole and normally hyperbolic trapping), and those which are outgoing (i.e.\ which tend towards $\scri^+$).

With this in mind, we briefly describe the microlocal propagation of edge-b-regularity through $\scri^+$ in~\S\ref{SsIeb}. In order to illustrate the way in which the microlocal analysis near $\scri^+$ fits into global analysis in the forward cone, we prove an extension of Theorem~\ref{ThmI} which proves the membership of waves in weighted edge-b-spaces also near $I^+$ under suitable assumptions on the global geometry away from $\scri^+$ (which are satisfied by the Minkowski metric); see Theorem~\ref{ThmSA}. The assumptions on the spacetime away from $\scri^+$ are inspired by work of the second author with Baskin and Wunsch \cite{BaskinVasyWunschRadMink}, the relationship to which we discuss in detail in \S\ref{SsII}. In this introduction, we content ourselves with the Minkowski setting of Theorem~\ref{ThmSA} (see also Example~\ref{ExSAMink}), and with extra b-regularity:

\begin{thm}[Global edge-b-regularity of waves]
\label{ThmIGlobal}
  Let $\Omega=\{t\geq 0\}$. Suppose $\alpha_++\frac12<\alpha_{\!\scri}<\min(-\half,\alpha_0+\half)$, and let $s\in\N$. Let $f\in\Heb^{s-1,(\alpha_0+2,2\alpha_{\!\scri}+2,\alpha_++2)}(\Omega)$. Then the forward solution $u$ of the wave equation $\Box_{g_0}u=f$ on Minkowski space satisfies $u\in\Heb^{s,(\alpha_0,2\alpha_{\!\scri},\alpha_+)}(\Omega)$. If $f$ enjoys additional $k$ orders of b-regularity,\footnote{That is, $V_1\cdots V_j f\in\Heb^{s-1,(\alpha_0+2,2\alpha_{\!\scri}+2,\alpha_++2)}(\Omega)$ for all $j=0,\ldots,k$, where the $V_i$ are b-vector fields on $\tilde M$: in $r>1$ and $0\leq t\leq 2 r$, a basis of these is given by $\la t-r\ra(\pa_t-\pa_r)$, $r(\pa_t+\pa_r)$, $\Omega_a$; and in $r=|x|\leq\frac{t}{2}$, one can take $\la t\ra\pa_t$ and $\la t\ra\pa_x$.} then $u$ enjoys additional $k$ degrees of b-regularity as well.
\end{thm}

The best pointwise decay that follows directly from this result via Sobolev embedding is the bound $|u|\lesssim \la t\ra^{-\frac{n-1}{2}+\eps}$ for all $\eps>0$ (provided $f$ has appropriate decay itself); see Remark~\ref{RmkSADecay}.

A more elaborate setting in which the results of the present paper play a crucial role is described in \cite{HintzNonstat}: the spacetimes considered there have asymptotically stationary regions, in which case the analysis far from $\scri^+$ requires yet different microlocal tools which are developed in \cite{Hintz3b,HintzNonstat}.

\subsection{Prior work on microlocal analysis near null infinity}
\label{SsII}

The radial compactification of $\R^{n+1}$ is the smooth manifold
\begin{equation}
\label{EqIRadComp}
  \ol{\R^{n+1}} = \Bigl( \R^{n+1} \sqcup \bigl([0,\infty)\times\Sph^n\bigr) \Bigr) / \sim,\qquad \R^{n+1}\setminus\{0\}\ni z=R\varpi \sim (R^{-1},\varpi)\in (0,\infty)\times\Sph^n,
\end{equation}
with boundary `at infinity' given by $\pa\ol{\R^{n+1}}=\{\tilde\varrho=0\}\cong\Sph^n$, $\tilde\varrho=R^{-1}$. (Here, $R^2=|z|^2$ is defined using any fixed positive definite quadratic form on $\R^{n+1}$, such as the Euclidean one.) All future null-geodesics on Minkowski space $(\R^{n+1},g_0)$ limit to the same codimension $1$ submanifold $Y\subset\pa\ol{\R^{n+1}}$; in local coordinates
\[
  \varrho=(t+r)^{-1},\quad
  v=\frac{t-r}{t+r},\quad
  \omega\in\Sph^{n-1},
\]
on $\ol{\R^{n+1}}$ near $t/r\in(0,\infty)$, $t>0$, this submanifold is given by $\{\varrho=v=0\}$. See Figure~\ref{FigIIRad}.

\begin{figure}[!ht]
\centering
\includegraphics{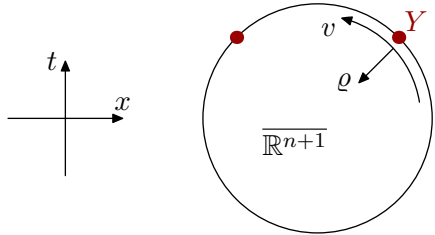}
\caption{The radial compactification of $\R^{n+1}$, and the light cone at infinity $Y\cong\Sph^{n-1}$, in the case $n=1$.}
\label{FigIIRad}
\end{figure}

In \cite[\S5]{VasyMicroKerrdS} and \cite{BaskinVasyWunschRadMink}, the analysis of the wave operator $\Box_{g_0}$ on Minkowski space $(\R^{n+1},g_0)$ focuses on its homogeneity of degree $-2$ under spacetime dilations $(t,x)\mapsto(\lambda t,\lambda x)$, $\lambda>0$. More specifically, the rescaling $\tilde\varrho^2\Box_{g_0}$ is invariant under dilations in $\tilde\varrho$ in the collar neighborhood $[0,1)_{\tilde\varrho}\times\Sph^n$ of $\pa\ol{\R^{n+1}}$. In fact, it is a hyperbolic totally characteristic operator, or \emph{b-differential operator} in the terminology of \cite{MelroseAPS}, which in the local coordinates $\varrho,v,\omega$ from before means that it is constructed from the b-vector fields $\varrho\pa_\varrho$, $\pa_v$, $\pa_\omega$ (spherical vector fields)---or more invariantly (and globally) smooth vector fields tangent to the boundary of $\ol{\R^{n+1}}$---and smooth functions on $\ol{\R^{n+1}}$. (See Example~\ref{ExOpMink} for the explicit expression.) One can more generally consider metrics $g$ on $\R^{n+1}$ (and the corresponding rescaled wave operators $\tilde\varrho^2\Box_g$) which have an \emph{asymptotic} homogeneity under dilations and an appropriate structure near $Y$; these are called \emph{Lorentzian scattering metrics} in~\cite{BaskinVasyWunschRadMink}.

Working with the Minkowski case for definiteness, the operator $\tilde\varrho^2\Box_{g_0}$ is analyzed in \cite{BaskinVasyWunschRadMink} (see also \cite[\S5]{VasyMicroKerrdS} and \cite{VasyMinkDSHypRelation}) using microlocal techniques in the b-cotangent bundle, which is an extension of $T^*\R^{n+1}$ to a bundle $\Tb^*\ol{\R^{n+1}}\to\ol{\R^{n+1}}$ with smooth frame given by the 1-forms dual to the aforementioned b-vector fields. In the characteristic set of $\tilde\varrho^2\Box_{g_0}$, the Hamiltonian vector field of the principal symbol (i.e.\ the generator of the lifted null-geodesic flow) has a sink over $Y$, corresponding to the outgoing null-geodesics which tend to $Y$. Null-bicharacteristics over $\pa\ol{\R^{n+1}}$ may instead also cross $Y$ first\footnote{That is, they are incoming null-geodesics lifted to phase space, or more precisely limits of families of such geodesics. A concrete example, projected to the base $\ol{\R^{n+1}}$, is the limit of $(0,1)\ni s\mapsto(t,r,\omega)=(T s,T-T s,\omega_0)$ as $T\nearrow\infty$.} and only at a later time tend to $Y$. Importantly, one can track microlocal b-regularity\footnote{This is b-regularity on $\ol{\R^{1+n}}$, and thus differs from b-regularity on the manifold $\tilde M$ in Figure~\ref{FigIIBlowup} near null infinity. In the coordinates $\varrho,v,\omega$ from above, b-regularity on $\ol{\R^{1+n}}$ tests for regularity using $\varrho\pa_\varrho$, $\pa_v$, $\pa_\omega$, whereas b-regularity on $\tilde M$ uses $\varrho\pa_\varrho$, $\varrho\pa_v$, $v\pa_v$, $\pa_\omega$; note that the $v$-derivatives here come with a prefactor that vanishes at null infinity.} in the b-phase space over a full neighborhood of $Y$, as demonstrated in~\cite[\S4]{BaskinVasyWunschRadMink}; the amount of b-regularity at $Y$ is necessarily limited (the issue being limited regularity upon differentiation along the weighted incoming vector field $\pa_v\sim r(\pa_t-\pa_r)$). One can furthermore obtain additional (integer amounts of) \emph{module regularity} at $Y$ (using techniques going back to \cite{HassellMelroseVasySymbolicOrderZero}).\footnote{This module regularity is in fact equivalent to the b-regularity on $\tilde M$ mentioned earlier.} In~\cite{VasyWrochnaPropagators,HintzVasySemilinear}, such a regularity theory on $\ol{\R^{n+1}}$ was used to describe asymptotic data for Feynman propagators and to solve semilinear wave equations.

Equipped with this b- and module regularity, the second author with Baskin and Wunsch \cite{BaskinVasyWunschRadMink} (see \cite{BaskinVasyWunschRadMink2} for a more general class of metrics) obtains a full asymptotic expansion (i.e.\ the polyhomogeneity) of solutions $u$ of the wave equation on a resolution of $\ol{\R^{n+1}}$ defined by blowing up $Y$ \cite{MelroseDiffOnMwc}. Recall that passage to the blow-up
\[
  \tilde M := [ \ol{\R^{n+1}}; Y ]
\]
of $\ol{\R^{n+1}}$ at $Y$ amounts to the introduction of polar coordinates around $Y$; see Figure~\ref{FigIIBlowup} for an illustration. Local coordinates near the interior of the front face\footnote{We use tildes here for consistency with the notation used in~\ref{SsIeb} and in the main part of the paper.} $\tilde\scri^+$ are $\varrho=(t+r)^{-1}$, $\frac{v}{\varrho}=t-r$, and $\omega\in\Sph^{n-1}$; at $\varrho=0$, this is the usual parameterization of null infinity. The asymptotic expansion of $u$ in particular captures its radiation field, which is (a derivative of) the restriction to $\tilde\scri^+$ of the rescaling $r^{\frac{n-1}{2}}u$. Coordinates near the past boundary $\tilde\scri^+\cap\tilde I^0$ of null infinity are
\begin{equation}
\label{EqIICoord}
  \rho_0=\frac{1}{r-t}\geq 0,\qquad
  \rho_{\!\scri}=\frac{r-t}{r},\qquad
  \omega;
\end{equation}
thus, level sets of $\rho_0$ are outgoing null cones, whereas level sets $\rho_{\!\scri}=c\in(0,1)$ are spacelike hypersurfaces with boundary at infinity contained in $\tilde I^0$. See~\cite[\S1.1.1]{HintzVasyMink4} for an extensive discussion.

\begin{figure}[!ht]
\centering
\includegraphics{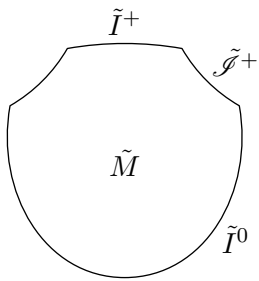}
\caption{The blow-up $\tilde M$ of $\ol{\R^{n+1}}$ at $Y$, and labels for the boundary hypersurfaces.}
\label{FigIIBlowup}
\end{figure}

Now, spacetime metrics $g$ arising from the solution of \emph{quasilinear} wave equations on Minkowski space---a key example being the solutions of the Einstein vacuum equations $\Ric(g)=0$ with initial data close to those of Minkowski space \cite{ChristodoulouKlainermanStability,LindbladRodnianskiGlobalStability,HintzVasyMink4}---typically are approximately dilation-invariant on spacetime \emph{except} near null infinity; there, the quasilinear waves or perturbed metrics $g$ are instead regular only when described on the resolution $\tilde M$ since the radiation field couples back into the metric. Put differently, rescaled wave operators $\tilde\varrho^2\Box_g$ associated with such $(\R^{n+1},g)$ (arising e.g.\ via linearization of the quasilinear equation) have highly \emph{singular} coefficients if one regards them as b-differential operators on $\ol{\R^{n+1}}$. Such operators are not b-microlocal at $Y$, and indeed they typically create many extra singularities at $Y$ (in the sense of b-wave front sets), rendering a precise microlocal regularity theory very delicate, if not impossible.\footnote{This is analogous to how, for example, the solution of a wave equation on $\R^{1+n}$, even if it is initially smooth, typically develops singularities at places where the coefficients of the wave operator are singular.} The point of the present paper is thus to describe a point of view which, unlike the b-setting on $\ol{\R^{n+1}}$, \emph{is} microlocal near null infinity.

\begin{rmk}[Geometric singular analysis on $\tilde M$]
\label{RmkIGeoSing}
  The analysis of the Einstein equations by Wang \cite{WangThesis} and the authors \cite{HintzVasyMink4,HintzMink4Gauge}, while inspired by the perspective of geometric singular analysis (in particular \cite{MelroseAPS,MazzeoEdge}) and taking place on $\tilde M$, is not microlocal at $\scri^+$, but rather fully relies on energy estimates and adaptations of the vector field method.
\end{rmk}

\subsection{The edge-b-perspective near null infinity}
\label{SsIeb}

The geometric setup for Theorem~\ref{ThmI} involves the smooth manifold $M$ with corners which is obtained from $\tilde M=[\ol{\R^{n+1}};Y]$ via performing a \emph{square root blow-up} of its front face $\tilde\scri^+$; that is, $M=\tilde M$ as sets, but a defining function of $\scri^+=\tilde\scri^+$ is now given by the square root of a defining function of $\tilde\scri^+\subset\tilde M$. In the local coordinates~\eqref{EqIICoord} on $\tilde M$, this amounts to regarding $\rho_0=\frac{1}{r-t}\geq 0$, $x_{\!\scri}=\sqrt{\rho_{\!\scri}}=\sqrt{\frac{r-t}{r}}\geq 0$, $\omega\in\Sph^{n-1}$ as smooth local coordinates on $M$; cf.\ \eqref{EqIHeb0}.

Following Melrose \cite{MelroseAPS} and Mazzeo \cite{MazzeoEdge}, we then write $\Veb(M)$ for the Lie algebra of smooth vector fields on $M$ which are tangent to all boundary hypersurfaces, and which at $\scri^+$ are in addition tangent to the fibers of the (blow-down) map $\scri^+\to Y$ given in local coordinates by $(\rho_0,\omega)\mapsto\omega$. This Lie algebra is spanned, over $\CI(M)$, by the vector fields in~\eqref{EqIVfs}. The corresponding classes of (pseudo)differential operators and Sobolev spaces are discussed in~\S\ref{SEB}.

An explicit calculation shows that the wave operator on Minkowski space satisfies
\begin{equation}
\label{EqIebOp}
  2\rho_0^{-2}x_{\!\scri}^{-2}\Box_{g_0} \equiv \bigl(x_{\!\scri}\pa_{x_{\!\scri}} - (n-1)\bigr)\bigl(x_{\!\scri}\pa_{x_{\!\scri}}-2\rho_0\pa_{\rho_0}) + 2 x_{\!\scri}^2\Delta_{g_{\Sph^{n-1}}},
\end{equation}
where $\Delta_{g_{\Sph^{n-1}}}$ is the non-negative Laplacian on the standard $(n-1)$-sphere, and where we write `$\equiv$' for equality modulo the space $x_{\!\scri}\Diffeb^2(M)$ of linear combinations of up to twofold products of the vector fields~\eqref{EqIVfs}, with coefficients vanishing at $x_{\!\scri}=0$. Intimately related to this is the fact that the rescaled Minkowski metric $\rho_0^2 x_{\!\scri}^2 g_0$ is a Lorentzian edge-b-metric, i.e.\ a nondegenerate Lorentzian signature quadratic form in the 1-forms $\frac{\dd\rho_0}{\rho_0}$, $\frac{\dd x_{\!\scri}}{x_{\!\scri}}$, $\frac{\dd\omega}{x_{\!\scri}}$ dual to the vector fields~\eqref{EqIVfs}. (See~\eqref{EqOGeoMetI0g0} for the computation of the dual metric.)

A systematic microlocal analysis of the operator~\eqref{EqIebOp} is performed in~\S\ref{SM}. The null-bicharacteristic flow of its (edge-b-)principal symbol---which is the same as the lifted null-geodesic flow for the metric $\rho_0^2 x_{\!\scri}^2 g_0$---has a rather intricate structure involving three radial sets having different source/sink/saddle point structures. (There is a fourth radial set over the future boundary $\scri^+\cap I^+$.) See Figure~\ref{FigIMF}. Crucially, the linearization of the Hamiltonian vector field at each of these radial sets is nondegenerate in the normal directions, which allows for a proof of microlocal radial point propagation estimates through (or into) them by means of standard positive commutator methods, see \cite[\S9]{MelroseEuclideanSpectralTheory}, \cite[\S2]{VasyMicroKerrdS}, \cite[Appendix~E.4]{DyatlovZworskiBook}. Combined with a simple energy estimate on edge-b-spaces, which appeared already in~\cite[\S4.1]{HintzVasyMink4} (albeit in less generality, and without identifying the underlying singular geometric structure), we can then prove Theorem~\ref{ThmI} and its generalizations in~\S\ref{SE}.

\begin{figure}[!ht]
\centering
\includegraphics{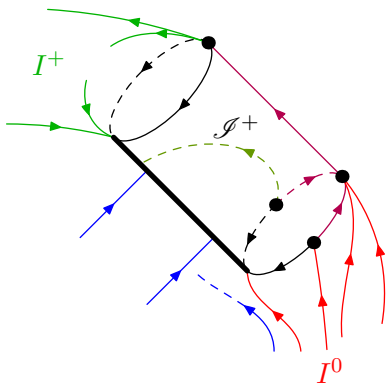}
\caption{Structure of the null-bicharacteristic flow near null infinity in $2+1$ spacetime dimensions. The cross sections of the cylinder are cross sections of the future light cones inside of each fiber of the eb-phase space over $\scri^+$. The thick black sets are the radial sets (the two antipodal points over $\scri^+\cap I^0$ forming a connected radial set in higher dimensions). See~\S\ref{SsMF} and Figure~\ref{FigMF} for details.}
\label{FigIMF}
\end{figure}

The class of operators (or metrics) to which our analysis applies is the natural generalization of $\rho_0^{-2}x_{\!\scri}^{-2}\Box_{g_0}$ within the class of edge-b-operators, in that one can allow for the operators (or metrics) to have additional lower order terms (in the sense of decay at $x_{\!\scri}=0$). We more generally consider operators acting on sections of vector bundles, as long as their principal part is that of a scalar wave operator still. Furthermore, one can allow for the presence of subprincipal terms at $x_{\!\scri}=0$; a particularly important example for applications to nonlinear stability problems (see for example \cite[Proposition~3.29 and \S3.6]{HintzMink4Gauge}) is the replacement of $n-1$ in~\eqref{EqIebOp} to another constant (or bundle endomorphism), this constant (or the eigenvalues) governing the decay rate(s) towards $\scri^+=\{x_{\!\scri}=0\}$. (Note that the constant $n-1$ is exactly the exponent in the $x_{\!\scri}^{n-1}\sim r^{-\frac{n-1}{2}}$ (for $|t-r|\lesssim 1$) decay rate towards null infinity of waves on Minkowski space.) Our general setup encompasses a large class of asymptotically Minkowskian spacetimes, allowing in particular for the presence of long range mass terms (as present in Schwarzschild or Kerr metrics) and radiation field type decay for certain metric coefficients in $3+1$ dimensions compatible with nonlinear stability problems; see also Example~\ref{ExOGeoAsy}.

\begin{rmk}[Parabolic scaling]
\label{RmkIParabolic}
  Via the Friedlander coordinate change \cite{FriedlanderRadiation}, or via the conformal embedding of Minkowski space into the Einstein cylinder $(\R_s\times\Sph^n,-\dd s^2+g_{\Sph^n})$, the Minkowski metric can be conformally rescaled to a metric on $\tilde M$ which is smooth and nondegenerate down to $(\tilde\scri^+)^\circ$, with $(\tilde\scri^+)^\circ$ becoming a null hypersurface. The edge perspective at null infinity of Minkowski space is then directly related to the following general setup: consider a null hypersurface $\cI$ in an $(n+1)$-dimensional Lorentzian manifold. Near a point $p\in\cI$, one can then find local coordinates $u,v,y$ so that $u,v$ are null, the hypersurface $\cI$ is given by $u=0$, and at $p$ the spacetime metric is $g_p=\dd u\,\dd v+\dd y^2$. Thus, $g_p$ is homogeneous of degree $2$ under the $\cI$-preserving parabolic scaling $(u,v,y)\mapsto(\lambda^2 u,v,\lambda y)$. The rescaling $u^{-1}g_p=\frac{\dd u}{u}\,\dd v+u^{-1}\dd y^2$ is scaling-invariant, and passing to $x:=\sqrt u$ in $u>0$ gives the metric $x^{-2}g_p=2\frac{\dd x}{x}\dd v+\frac{\dd y^2}{x^2}$ which is invariant under the homogeneous scaling $(x,v,y)\mapsto(\lambda x,v,\lambda y)$. Thus, in $u>0$, the rescaling $x^{-2}g_p$ is an edge metric (i.e.\ a non-degenerate Lorentzian signature expression in $\frac{\dd x}{x}$, $\dd v$, $\frac{\dd y}{x}$) on the manifold $[0,1)_x\times\R_v\times\R^{n-1}_y$, with the boundary $\cI=x^{-1}(0)$ fibered via $(v,y)\mapsto y$.
\end{rmk}

\begin{rmk}[Conormal coefficients at null infinity]
\label{RmkIdHUV}
  The conformal perspective on Minkowski space mentioned in the Remark~\ref{RmkIParabolic} largely breaks down when applied to perturbations of Minkowski space in the context of the stability problem in $3+1$ dimensions. (See \cite{ChristodoulouNoPeeling,DafermosChristodoulouExpose} and \cite[Remark~8.12]{HintzVasyMink4} and the references therein; but see also \cite{FriedrichStability} for a more restrictive setting in which a conformal approach does succeed.) However, the conformally rescaled metric typically does have some conormal regularity down to $\tilde\scri^+$. It is thus conceivable that one can adapt the methods used in \cite{DeHoopUhlmannVasyDiffraction} (see also~\cite{GannotWunschPotential}) for the diffraction of singularities by mildly singular timelike boundaries to the lightlike case.
\end{rmk}

\begin{rmk}[Klein--Gordon equation]
\label{RmkIKG}
  The structure near null infinity of the Klein--Gordon operator $\Box_{g_0}+m^2$, with $m\in\R\setminus\{0\}$, is altogether different; for example, unlike in~\eqref{EqIebOp} one cannot factor out $\rho_0^{-2}x_{\!\scri}^{-2}$ from $\Box_{g_0}+m^2$. Sussman \cite{SussmanKG} thus develops a different microlocal framework for studying regularity and decay of solutions of the Klein--Gordon equation near null infinity (and also globally) on Minkowski space, and generalizations thereof; this framework is based on the Lie algebra of `double edge, scattering' vector fields, which have an additional order of vanishing at $I^0$, $\scri^+$, and $I^+$ compared to the edge-b-vector fields used in the present paper.
\end{rmk}

The edge-b-microlocal analysis of $\Box_{g_0}$ and its generalizations is supplemented in~\S\ref{Sb} with propagation estimates on edge-b-spaces which carry an additional integer amount of b-regularity on $\tilde M$ (or equivalently on $M$). This extra b-regularity is captured via testing with vector fields, much as in the aforementioned Klainerman vector field method. 

\begin{rmk}[b-perspective]
\label{RmkINob}
  No matter how one changes the smooth structure of $\tilde M$, the operator $\Box_{g_0}$ is not a nondegenerate (weighted) b-differential operator. (For example, the spherical Laplacian term in~\eqref{EqIebOp} is of lower order in the sense of decay at $x_{\!\scri}=0$ than the first term.) This means that a b-microlocal analysis of the operator $\Box_{g_0}$ and its perturbations, if possible at all, is delicate due to the degenerate nature of the operator as a b-differential operator.
\end{rmk}

Besides the microlocal edge-b-regularity theory and the solvability theory away from future timelike infinity $I^+$, we also show the invertibility of the \emph{edge normal operator} of $\Box_{g_0}$ at $\scri^+$, which is an invariant (cf.\ Remark~\ref{RmkIParabolic}) model at each fiber of $\scri^+$; see~\S\ref{SNe}. The inversion of normal operators is the main ingredient in the development of Fredholm theory for elliptic operators \cite{MazzeoMelroseHyp,MelroseAPS,MazzeoEdge,MazzeoMelroseFibred} and also for nonelliptic operators \cite[\S2]{HintzVasySemilinear}, \cite{GellRedmanHaberVasyFeynman}.

Edge-b-operators such as (a weighted version of) $\Box_{g_0}$ have a dilation-invariant normal operator also at the boundary hypersurface $I^+\subset M$. In~\S\ref{SNp}, we prove microlocal estimates for this normal operator which only use the structure of the underlying wave type operator near $\scri^+$. We indicate how these edge normal operator inverses and normal operator estimates, together with our microlocal regularity theory, can be put to use in the \emph{global} analysis of a wave equation on a class of asymptotically Minkowski spaces in (the proof of) Theorem~\ref{ThmSA}. (A significantly more elaborate setting is discussed in \cite{HintzNonstat}.)

\subsection{Outline of the paper}
\label{SsIO}

Edge-b-vector fields, (pseudo)differential operators, Sobolev spaces, and related concepts are introduced in~\S\ref{SEB}. In~\S\ref{SO}, we describe the Minkowski metric and its generalizations, called \emph{admissible metrics} in this paper, from the edge-b-perspective, and also introduce the class of wave type operators that our methods can handle. In~\S\ref{SM} then, the microlocal heart of the paper, we analyze the admissible operators of~\S\ref{SO} from an edge-b-microlocal point of view. Estimates on spaces capturing higher order b-regularity are proved in~\S\ref{Sb}. The solvability of wave equations away from the future boundary of $\scri^+$ (with Theorem~\ref{ThmI} being a special case) is proved using energy estimates and microlocal propagation results in~\S\ref{SE}.

Estimates for the normal operators associated with the edge-b-nature of admissible wave operators are proved in~\S\S\ref{SNe}--\ref{SNp}. An application to the global existence, regularity, and decay of waves on a class of asymptotically Minkowskian spacetimes is given in~\S\ref{SA}.

\subsection*{Acknowledgments}

Part of this research was conducted during the time P.H.\ held Clay and Sloan Research Fellowships. P.H.\ and A.V.\ gratefully acknowledge support from the U.S.\ National Science Foundation under Grants No.\ DMS-1955614 (P.H.) as well as DMS-1664683 and DMS-1953987 (A.V.). This material is furthermore based upon work supported by the NSF under Grant No.\ DMS-1440140 while the authors were in residence at the Mathematical Sciences Research Institute in Berkeley, California, during the Fall 2019 semester.

\section{Edge-b-geometry and analysis}
\label{SEB}

We begin by recalling, in~\S\ref{SsEBS}, the Lie algebra of edge-b-vector fields on a manifold with corners and a fibered boundary hypersurface. The corresponding spaces of differential operators are defined in~\S\ref{SsEBS} as well, and their normal operators are described in~\S\ref{SsEBO}. The algebra of edge-b-pseudodifferential operators is recalled in \S\ref{SsEBP}. We will use these operator algebras for the (microlocal) analysis near null infinity. Sobolev spaces are discussed in~\S\ref{SsEBH}, and various notions required for the study of certain normal operators of edge-b-operators are defined in~\S\S\ref{SsEBI}--\ref{SsEB0}. The material in this section is largely based on \cite{MelroseAPS,MazzeoEdge,MelroseVasyWunschDiffraction}.

\subsection{Vector fields, differential operators, bundles}
\label{SsEBS}

Let $M$ denote an $n$-di\-men\-sion\-al manifold with corners. Denote by $H_1,\ldots,H_N\subset\pa M$, $N\in\N$, the collection of its boundary hypersurfaces which we require to be embedded submanifolds of $M$. Recall then:

\begin{definition}[Lie algebras of vector fields]
\label{DefEBSbsc}
  Denote by $\cV(M)=\CI(M;T M)$ the Lie algebra of smooth vector fields on $M$.
  \begin{enumerate}
  \item The space $\Vb(M)\subset\cV(M)$ of \emph{b-vector fields} \cite{MelroseTransformation,MelroseAPS} consists of all $V\in\cV(M)$ which are tangent to $\pa M$, i.e.\ to $H_j$ for each $j=1,\ldots,N$.
  \item If $N=1$, then the space $\Vz(M)\subset\cV(M)$ of \emph{0-vector fields} \cite{MazzeoMelroseHyp} is defined as $\Vz(M)=\{\rho V\colon V\in\cV(M)\}$, where $\rho\in\CI(M)$ is a boundary defining function (that is, $\rho\geq 0$ vanishes only at $\pa M$ and has nonvanishing differential there).
  \item If $N=1$, then we define the space $\Vsc(M)\subset\Vb(M)$ of \emph{scattering vector fields} \cite{MelroseEuclideanSpectralTheory} as $\Vsc(M)=\{\rho V\colon V\in\Vb(M)\}$, where $\rho$ is a boundary defining function.
  \item If $N=1$ and $\pa M=H_1$ is the total space of a smooth fibration $Z-\pa M\to Y$, then the space $\Ve(M)\subset\Vb(M)$ of \emph{edge vector fields} \cite{MazzeoEdge} consists of all b-vector fields which at $\pa M$ are tangent to the fibers of $\pa M$.
  \item If $N\geq 2$ and some (but not all) $H_j$ are total spaces of fibrations, the space $\Veb(M)\subset\Vb(M)$ of \emph{edge-b-vector fields} consists of all b-vector fields which at each fibered boundary hypersurface are tangent to the fibers; see also \cite{MelroseVasyWunschDiffraction,AlbinGellRedmanDirac,Hintz3b}.
  \end{enumerate}
\end{definition}

In local coordinates $x\in[0,\infty)^k$ and $y\in\R^{n-k}$ near a point inside a codimension $k$ corner, with $\pa M$ locally given by $x=0$, b-vector fields are linear combinations with smooth coefficients of
\begin{equation}
\label{EqEBSb}
  x^i\pa_{x^i}\ (i=1,\ldots,k),\qquad \pa_{y^j}\ (j=1,\ldots,n-k).
\end{equation}
When $N=1$ and $k=1$, this frame becomes $x\pa_x,\pa_{y^j}$; the space of scattering vector fields is then spanned by
\begin{equation}
\label{EqEBSsc}
  x^2\pa_x,\qquad x\pa_{y^j}\ (j=1,\ldots,n-1).
\end{equation}
Correspondingly, there are smooth vector bundles
\[
  \Tb M\to M,\qquad \Tsc M\to M,
\]
and bundle maps $\Tb M\to T M$, $\Tsc M\to T M$, which over the interior $M^\circ$ are isomorphisms, so that $\Vb(M)=\CI(M;\Tb M)$ and $\Vsc(M)=\CI(M;\Tsc M)$; a smooth basis of the fibers of $\Tb M$ and $\Tsc M$ is given, in the respective settings, by~\eqref{EqEBSb} and \eqref{EqEBSsc}. The dual bundles are denoted $\Tb^*M$ and $\Tsc^*M$, and the corresponding density bundles by $\Omegab M$ and $\Omegasc M$.

As an important special case, consider the radial compactification $\ol{\R^n}$ of $\R^n$, defined in~\eqref{EqIRadComp}. Denoting by $x^1,\ldots,x^n$ standard coordinates on $\R^n$, then in the closure of the region where, say, $x^1$ is relatively large, meaning $x^1\geq c|x^j|$ for $j=2,\ldots,n$ and some $c>0$, we can use $\rho=\frac{1}{x^1}$, $y^j=\frac{x^j}{x^1}$ as smooth local coordinates on $\ol{\R^n}$; since $\pa_{x^1}=-\rho^2\pa_\rho-\rho\sum_{j=2}^n y^j\pa_{y^j}$, $\pa_{x^j}=\rho\pa_{y^j}$ one finds that translation invariant vector fields on $\R^n$ are special case of scattering vector fields on $\ol{\R^n}$, and indeed $\Vsc(\ol{\R^n})$ is spanned over $\CI(\ol{\R^n})$ by such translation-invariant vector fields. Similarly then, the coordinate differentials $\dd x^1,\ldots,\dd x^n$ extend by continuity from $\R^n$ to give a basis of $\Tsc^*\ol{\R^n}$, and a basis of $\Omegasc\ol{\R^n}$ is given by the Euclidean volume density $|\dd x^1\ldots\dd x^n|$.

For the sake of notational simplicity, we discuss the edge-b-setting only in the special case of interest in this paper. Thus, we assume that $H_2\subset M$ is the total space of a fibration $Z-H_2\xra{\phi} Y$, where $Z\cong[-1,1]$ is a closed interval, and $Y$ is a compact $(n-2)$-dimensional manifold without boundary. Moreover, we assume that $H_j\cap H_2=\emptyset$ except for $j=1,3$, and $H_j$ for $j\neq 2$ is not fibered.\footnote{Equivalently, $H_j$, $j\neq 2$, is equipped with the trivial fibration whose base is a singleton set. But we prefer to speak of b-behavior at $H_j$, $j\neq 2$, rather than of an extreme type of edge behavior.} See Figure~\ref{FigEBS}. A local coordinate description is as follows:
\begin{enumerate}
\item near the interior $H_2^\circ$ of $H_2$, we can choose local coordinates $x\in[0,\infty)$, $y\in\R^{n-2}$, $z\in\R$ on $M$, with $y,z$ local coordinates on $Y,Z$, so that $H_2$ is locally given by $x^{-1}(0)$ and the fibration of $H_2$ takes the form $(y,z)\mapsto y$. Edge-b-vector fields on $M$ are then smooth linear combinations of
  \begin{subequations}
  \begin{equation}
  \label{EqEBSe1}
    x\pa_x,\qquad x\pa_{y^j}\ (j=1,\ldots,n-2),\qquad \pa_z;
  \end{equation}
\item near $H_1\cap H_2$ (and analogously near $H_3\cap H_2$), we can choose local coordinates $x,z\in[0,\infty)$ and $y\in\R^{n-2}$ with the same properties as above, and so that in addition $H_1$ is locally given by $z=0$. Since b-vector fields are spanned by $x\pa_x$, $\pa_{y^j}$, $z\pa_z$, the space of edge-b-vector fields is now spanned by
  \begin{equation}
  \label{EqEBSe2}
    x\pa_x,\qquad x\pa_{y^j}\ (j=1,\ldots,n-2),\qquad z\pa_z.
  \end{equation}
  \end{subequations}
\end{enumerate}
Again, we conclude that $\Veb(M)$ is equal to the space of sections of a smooth vector bundle $\Teb M\to M$, with local frames near $H_2$ given by~\eqref{EqEBSe1}--\eqref{EqEBSe2}. The \emph{edge-b-cotangent bundle} $\Teb^*M$ will play a key role in the present paper as the phase space for the microlocal analysis of edge-b-differential operators.

\begin{figure}[!ht]
\centering
\includegraphics{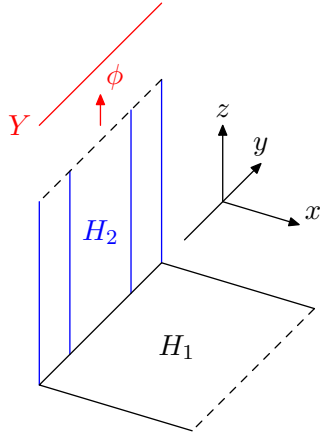}
\caption{A neighborhood of $H_1\cap H_2$ inside $M$, and the local coordinates $x,y,z$ used in~\eqref{EqEBSe2}. The fibration $\phi$ of $H_2$ and the base $Y$ are indicated in red, and the fibers of $H_2$ in blue.}
\label{FigEBS}
\end{figure}

The spaces of vector fields in Definition~\ref{DefEBSbsc} are Lie algebras; for
\[
  \bullet=\bop,\quad 0,\quad\scop,\quad\eop,\quad \ebop.
\]
The space of locally finite linear combinations of up to $m$-fold compositions of elements of $\cV_\bullet(M)$ is denoted $\Diff_\bullet^m(M)$; we put $\Diff_\bullet(M)=\bigoplus_{m\in\N_0}\Diff_\bullet^m(M)$. Given a weight $\alpha\in\R$, or in the case that $M$ has $N\geq 2$ boundary hypersurfaces a vector $\alpha=(\alpha_1,\ldots,\alpha_N)\in\R^N$ of weights, we put
\[
  \Diff_\bullet^{m,\alpha}(M) := \{ \rho^{-\alpha}P \colon P \in \Diff_\bullet^m(M) \},
\]
where $\rho$ is a boundary defining function when $N=1$, and a collection $\rho=(\rho_1,\ldots,\rho_N)$ where $\rho_j$ is a defining function of $H_j$ when $N\geq 2$; in the latter case, we use the notation $\rho^{-\alpha}:=\prod_{j=1}^N \rho_j^{-\alpha_j}$. Since for any b-vector field $V\in\Vb(M)$ one has $\rho^\alpha V(\rho^{-\alpha})\in\CI(M)$ (cf.~\eqref{EqEBSb}), it follows that one compose elements of $\bigoplus_{m,\alpha}\Diff_\bullet^{m,\alpha}(M)$, and the orders of a composition are the sums of the orders of the individual factors.

More generally, given a weight (vector) $\alpha$, we can consider the space of $L^\infty$-conormal functions on $M$,
\[
  \cA^\alpha(M) := \{ u \in \rho^\alpha L^\infty(M^\circ) \colon P u \in \rho^\alpha L^\infty(M^\circ)\ \forall\,P\in\Diffb(M) \}.
\]
One can then define the space
\[
  \cA^{-\alpha}\Diff_\bullet^m(M)
\]
of $\bullet$-differential operators with coefficients in $\cA^{-\alpha}(M)$ to consist of locally finite linear combinations of operators of the form $w P$ where $w\in\cA^{-\alpha}(M)$ and $P\in\Diff_\bullet(M)$. Since $\rho^{-\alpha}\in\cA^{-\alpha}(M)$, we have $\cA^{-\alpha}\Diff_\bullet^m(M)\supset\Diff_\bullet^{m,\alpha}(M)$. The space $\bigoplus_{\alpha,m}\cA^{-\alpha}\Diff_\bullet^m(M)$ is an algebra, and the orders are additive under composition.

We shall encounter further variants of the spaces $\cA^\alpha(M)$: at some hypersurfaces $H_{i_1}$, $\ldots$, $H_{i_k}$, we may require classical conormality, i.e.\ smoothness upon multiplication by $\rho_i^{-\alpha_i}$. Ordering indices so that $i_j=j$, $j=1,\ldots,k$, we thus introduce the notation
\[
  \cA^{((\alpha_1,0),(\alpha_2,0),\ldots,(\alpha_k,0),\alpha_{k+1},\ldots,\alpha_N)}(M) := \Biggl(\prod_{j=1}^N \rho_j^{\alpha_j}\Biggr)\cA^{((0,0),\ldots,(0,0),0,\ldots,0)}(M)
\]
for spaces of mixed conormal and classical conormal functions; here $\cA^{((0,0),\ldots,(0,0),0,\ldots,0)}(M)\subset\cA^0(M)$ consists of all $u\in\cA^0(M)$ so that $V_1\cdots V_J u\in L^\infty(M)$ for all $J\in\N$ and $V_j\in\cV(M)$ which are tangent to $H_{k+1},\ldots,H_N$ (but not necessarily to $H_1,\ldots,H_k$). Thus, elements of $\cA^{((0,0),\ldots,(0,0),0,\ldots,0)}(M)$ are smooth down to $H_j$ for $j=1,\ldots,k$, and bounded conormal at $H_j$ for $j=k+1,\ldots,N$. One can then also consider spaces of differential operators with coefficients in these mixed conormal spaces.

Finally, if $E,F\to M$ are two smooth vector bundles, one can consider the space
\[
  \Diff_\bullet^m(M;E,F) = \Diff_\bullet^m(M) \otimes_{\CI(M)} \CI(M;\Hom(E,F))
\]
of differential operators mapping sections of $E$ into sections of $F$; in local trivializations of $E$ and $F$, these are simply $(\rank F)\times(\rank E)$ matrices of elements of $\Diff_\bullet^m(M)$. Versions of these spaces with weights or conormal coefficients are defined similarly.

\subsection{Symbols and normal operators of edge-b-differential operators}
\label{SsEBO}

Let $M$ denote a compact $n$-di\-men\-sion\-al manifold with corners, with embedded hypersurfaces $H_1,\ldots,H_N$, $N\geq 3$, and with $H_2$ the total space of a fibration $Z-H_2\to Y$ where $Z\cong[-1,1]$, and $Y$ is compact without boundary, and with $H_j\cap H_2=\emptyset$ unless $j=1,3$. This is the setting discussed around~\eqref{EqEBSe1}--\eqref{EqEBSe2} and illustrated in Figure~\ref{FigEBS}. Let us denote defining functions of $H_j$ by $\rho_j\in\CI(M)$. Note that the boundary hypersurfaces of the edge-b-cotangent bundle $\Teb^*M$ are $\Teb_{H_j}^*M$, and the fibration $\phi\colon H_2\to Y$ induces a fibration $\Teb_{H_2}^*M\to Y$ (with fibers given by the restriction of $\Teb^*M$ to the fibers of $\phi$) via composition with the projection map $\Teb_{H_2}^*M\to H_2$; thus, there is a natural notion of edge-b-vector fields on $\Teb^*M$.

Let us work in the local coordinates~\eqref{EqEBSe2}, so $x,z$ are defining functions of $H_2,H_1$, respectively, and $y$ denotes local coordinates on $Y$. A local frame of $\Teb^*M$ is thus
\[
  \frac{\dd x}{x},\qquad \frac{\dd y^j}{x}\ (j=1,\ldots,n-2),\qquad \frac{\dd z}{z}.
\]
Writing the canonical 1-form on $T^*M^\circ$ in the form
\begin{equation}
\label{EqEBOCoord}
  \xi\frac{\dd x}{x}+\sum_{j=1}^{n-2}\eta_j\frac{\dd y^j}{x} + \zeta\frac{\dd z}{z}
\end{equation}
defines smooth fiber-linear coordinates on $T^*M^\circ$ which extend by continuity to fiber-linear coordinates on $\Teb^*M$ over the local coordinate patch. Edge-b-vector fields on $\Teb^*M$ are linear combinations, with $\CI(\Teb^*M)$ coefficients, of the vector fields~\eqref{EqEBSe2} and $\pa_\xi,\pa_{\eta_j},\pa_\zeta$. We first elucidate the symplectic structure of $T^*M^\circ$ from the edge-b-perspective:

\begin{lemma}[Hamiltonian vector field on $\Teb^*M$]
\label{LemmaEBOHam}
  Let $p\in\CI(\Teb^*M)$. Then the Hamiltonian vector field $H_p$ satisfies $H_p\in\Veb(\Teb^*M)$. The map $p\mapsto H_p$ is a first order edge-b-differential operator,
  \begin{equation}
  \label{EqEBOHamOp}
    H_{(-)} \in \Diffeb^1\bigl(\Teb^*M;\ul\C,\Teb(\Teb^*M)\bigr),
  \end{equation}
  where $\ul\C$ denotes the trivial bundle $M\times\C$. In the local coordinates~\eqref{EqEBOCoord}, we have
  \begin{equation}
  \label{EqEBOHam}
  \begin{split}
    H_p &= (\pa_\xi p)\Bigl(x\pa_x+\sum_{j=1}^{n-2}\eta_j\pa_{\eta_j}\Bigr) + \sum_{j=1}^{n-2} (\pa_{\eta_j}p) x\pa_{y^j} + (\pa_\zeta p)z\pa_z \\
      &\qquad - \Bigl(\Bigl(x\pa_x+\sum_{j=1}^{n-2}\eta_j\pa_{\eta_j}\Bigr)p\Bigr)\pa_\xi - \sum_{j=1}^{n-2} (x\pa_{y^j}p)\pa_{\eta_j} - (z\pa_z p)\pa_\zeta.
  \end{split}
  \end{equation}
\end{lemma}
\begin{proof}
  Taking the exterior derivative of~\eqref{EqEBOCoord}, the symplectic form $\omega$ on $T^*M^\circ$ is
  \[
    \omega = \Bigl(\dd\xi+\sum_{j=1}^{n-2}\frac{\eta_j\,\dd y^j}{x}\Bigr)\wedge\frac{\dd x}{x} + \sum_{j=1}^{n-2} \dd\eta_j\wedge\frac{\dd y^j}{x} + \dd\zeta\wedge\frac{\dd z}{z}.
  \]
  The definition $\omega(-,H_p)=\dd p=(x\pa_x p)\frac{\dd x}{x}+\sum (x\pa_{y^j}p)\frac{\dd y^j}{x}+(z\pa_z p)\frac{\dd z}{z}+(\pa_\xi p)\dd\xi+(\pa_\eta p)\dd\eta+(\pa_\zeta p)\dd\zeta$ of $H_p$ then gives the expression~\eqref{EqEBOHam}. This expression also implies~\eqref{EqEBOHamOp}; more invariantly, \eqref{EqEBOHamOp} follows from the facts that $\omega\in\CI(\Teb^*M;\Lambda^2\,\Teb^*(\Teb^*M))$ is a nondegenerate edge-b-2-form and the exterior derivative $\dd$ satisfies $\dd\in\Diffeb^1(M;\ul\C,\Teb^*M)$.
\end{proof}

We pause to describe $H_p$ for homogeneous $p$; for this purpose, we denote by $\ol{\Teb^*}M$ the fiber-radial compactification of $\Teb^*M$. (Since the $GL(n)$-action on $\R^n$ extends to an action on $\ol{\R^n}$ by diffeomorphisms, $\ol{\Teb^*}M$ carries the structure of a smooth closed ball bundle over $M$.) The only one of its boundary hypersurfaces which we regard as the total space of a fibration is $\ol{\Teb^*_{H_2}}M$; the boundary $\Seb^*M$ at fiber infinity is not fibered, i.e.\ edge-b-vector fields on $\ol{\Teb^*}M$ are merely required to be tangent to $\Seb^*M$.

\begin{cor}[Rescaled Hamiltonian vector field]
\label{CorEBOHamResc}
  Let $\rho_\infty\in\CI(\ol{\Teb^*}M)$ be a defining function of $\Seb^*M$, and suppose $p\in\CI(\Teb^*M\setminus o)$ is homogeneous of degree $s$ with respect to (positive) dilations in the fibers. Then $\rho_\infty^{s-1}H_p\in\Veb(\ol{\Teb^*}M\setminus o)$.
\end{cor}
\begin{proof}
  The conclusion is independent of the choice of $\rho_\infty$. We shall work in local coordinates in which~\eqref{EqEBOHam} is valid, and in a region where $\xi$ is relatively large, i.e.\ $\xi>c(|\eta|+|\zeta|)$ for some $c>0$. (Regions where some $\eta_j$ or $\zeta$ are relatively large are analyzed similarly.) There, we shall take
  \[
    \rho_\infty = \xi^{-1},\quad \hat\eta=\frac{\eta}{\xi},\quad \hat\zeta=\frac{\zeta}{\xi}.
  \]
  Writing $p=\rho_\infty^{-s}p_0$ where $p_0$ is a smooth function of $(x,y,z,\hat\eta,\hat\zeta)$, we then compute
  \begin{align*}
    \rho_\infty^{s-1}H_p &= -\bigl(((\rho_\infty\pa_{\rho_\infty}-s)+\hat\eta\pa_{\hat\eta}+\hat\zeta\pa_{\hat\zeta})p_0\bigr)(x\pa_x+\hat\eta\pa_{\hat\eta}) + (\pa_{\hat\eta}p_0)\cdot x\pa_y + (\pa_{\hat\zeta}p_0)z\pa_z \\
      &\qquad + \bigl((x\pa_x+\hat\eta\pa_{\hat\eta})p_0\bigr)(\rho_\infty\pa_{\rho_\infty}+\hat\eta\pa_{\hat\eta}+\hat\zeta\pa_{\hat\zeta}) - (x\pa_y p_0)\cdot\pa_{\hat\eta} - (z\pa_z p_0)\pa_{\hat\zeta}.
  \end{align*}
  Since $\Veb(\ol{\Teb^*}M)$ is locally spanned by the vector fields $x\pa_x$, $\pa_y$, $z\pa_z$, $\rho_\infty\pa_{\rho_\infty}$, $\pa_{\hat\eta}$, and $\pa_{\hat\zeta}$, this proves the Corollary.
\end{proof}

Let now $P\in\Diffeb^m(M)$. In the above local coordinates, we have
\begin{equation}
\label{EqEBOOp}
  P = \sum_{j+k+|\alpha|\leq m} a_{j k\alpha}(x,y,z) (x D_x)^j (x D_y)^\alpha (z D_z)^k,\qquad a_{j k\alpha}\in\CI(M).
\end{equation}
It has a well-defined (independent of the choice of local coordinates $x,y,z$) principal symbol
\[
  \sigmaeb^m(P) := \sum_{j+k+|\alpha|=m} a_{j k\alpha}(x,y,z) \xi^j\eta^\alpha\zeta^k \in P^{[m]}(\Teb^*M),
\]
where $P^{[m]}(\Teb^*M)$ denotes the space of smooth functions on $\Teb^*M$ which are homogeneous polynomials of degree $m$ in the fibers. We have a short exact sequence
\begin{equation}
\label{EqEBOSeqSymb}
  0 \to \Diffeb^{m-1}(M) \hra \Diffeb^m(M) \xra{\sigmaeb^m} P^{[m]}(\Teb^*M) \to 0.
\end{equation}
Given $A_j\in\Diffeb^{m_j}(M)$ with principal symbols $a_j=\sigmaeb^{m_j}(A_j)$, we have the usual properties
\[
  \sigmaeb^{m_1+m_2}(A_1 A_2)=a_1 a_2,\qquad
  \sigmaeb^{m_1+m_2-1}(i[A_1,A_2])=H_{a_1}a_2.
\]

Note that, in general, the symbol of a commutator is nonzero at $\pa M$. Thus, just like the b- and edge-algebras, the algebra of edge-b-differential operators is commutative to leading order only in the differential order sense, but not in the sense of decay at any of the boundary hypersurfaces of $\pa M$. We thus proceed to recall the definitions of the normal operators of $P\in\Diffeb^m(M)$. We shall only discuss the normal operators at $H_1$ and $H_2$; the normal operator at $H_3$ is defined exactly as the one at $H_1$.

The normal operator at $H_1$ is defined by freezing coefficients at $H_1$; in terms of the local coordinates used in~\eqref{EqEBOOp}, this means setting
\begin{equation}
\label{EqEBONormH1}
  N_{H_1}(P) := \sum_{j+k+|\alpha|\leq m} a_{j k\alpha}(x,y,0)(x D_x)^j(x D_y)^\alpha(z D_z)^k.
\end{equation}
This can be defined invariantly as an element $N_{H_1}(P)\in\Diff_{\bop,I}^m({}^+N H_1)$, where ${}^+N H_1$ is the (nonstrictly) inward pointing normal bundle of $H_1$, and the subscript `$I$' restricts to operators which are invariant under the dilation action in the fibers of ${}^+N H_1$. Indeed, $N_{H_1}(P)$ can be defined for general b-differential operators $P\in\Diffb^m(M)$ as the multiplicative extension of the map $\Vb(M)\to\Vb(M)/z\Vb(M)=\CI(H_1;\Tb_{H_1}M)\cong\cV_{\bop,I}({}^+N H_1)$ \cite[\S4.15]{MelroseAPS}. In order to sharpen the description of $N_{H_1}(P)$ for edge-b-differential operators, note that the restriction of the fibration $H_2\to Y$ to $H_1\cap H_2$ is a diffeomorphism $H_1\cap H_2\xra{\cong}Y$, and hence the fibers of $H_1\cap H_2$ are points. Correspondingly, the boundary hypersurface ${}^+N_{H_1\cap H_2}H_1$ of ${}^+N H_1$ is the total space of a fibration given by the base projection ${}^+N_{H_1\cap H_2}H_1\to H_1\cap H_2\cong Y$. We can thus consider the spaces
\[
  \cV_{\ebop,I}({}^+N H_1),\qquad \Diff_{\ebop,I}^m({}^+N H_1)
\]
of dilation-invariant edge-b-vector fields and edge-b-operators. The local coordinate expression~\eqref{EqEBONormH1} shows directly:

\begin{definition}[b-normal operator at $H_1$]
\label{DefEBONormH1}
  The map $N_{H_1}\colon\Diffeb(M)\to\Diff_{\ebop,I}({}^+N H_1)$ is an algebra homomorphism and for each $m\in\N_0$ fits into the short exact sequence
  \begin{equation}
  \label{EqEBONormH1Seq}
    0 \to \rho_1\Diffeb^m(M) \hra \Diffeb^m(M) \xra{N_{H_1}} \Diff_{\ebop,I}^m({}^+N H_1) \to 0,
  \end{equation}
  where we recall that $\rho_1\in\CI(M)$ is a defining function of $H_1$.
\end{definition}

Upon fixing a trivialization
\begin{equation}
\label{EqEBOTrivNH1}
  {}^+N H_1 \cong H_1 \times [0,\infty)_{\rho_1},
\end{equation}
we can consider the action of $N_{H_1}(P)$ on functions of the form $\rho_1^{i\zeta}u$, $u\in\CI(H_1)$,
\[
  \wh{N_{H_1}}(P,\zeta)u := \Bigl(\rho_1^{-i\zeta}N_{H_1}(P)\bigl(\rho_1^{i\zeta}u\bigr)\Bigr)\Big|_{\rho_1=0}.
\]
In the local coordinates~\eqref{EqEBONormH1}, and with $\rho_1=z$, we have
\[
  \wh{N_{H_1}}(P,\zeta) = \sum_{j+k+|\alpha|\leq m} a_{j k\alpha}(x,y,0)\zeta^k(x D_x)^j(x D_y)^\alpha \in \Diff_0^m(H_1),
\]
which is thus a 0-differential operator \cite{MazzeoMelroseHyp} (i.e.\ an edge differential operator with respect to the fibration of the boundary given by the identity map $\pa  H_1\to\pa H_1$).

\begin{definition}[Mellin-transformed normal operator family]
\label{DefEBONormH1M}
  Fix a trivialization~\eqref{EqEBOTrivNH1}. For $P\in\Diffeb^m(M)$, the family $\wh{N_{H_1}}(P,\zeta)$, $\zeta\in\C$, of 0-differential operators on $H_1$ is the \emph{Mellin-transformed normal operator family} of $P$ at $H_1$.
\end{definition}

We now turn to the edge normal operator at $H_2$, which is a family of operators on the fiber $Z$, parameterized by the base $Y$ of the fibration $\phi\colon H_2\to Y$. It is defined by freezing coefficients at a fiber $\phi^{-1}(y_0)$, $y_0\in Y$; here, we are following \cite[Equation~(2.17)]{MazzeoEdge}. In local coordinates as in~\eqref{EqEBOOp}, this means setting
\begin{equation}
\label{EqEBONormH2y0}
  {}^\eop N_{H_2,y_0}(P) := \sum_{j+k+|\alpha|\leq m} a_{j k\alpha}(0,y_0,z)(x D_x)^j(x D_y)^\alpha(z D_z)^k.
\end{equation}
Here $(x,y)$ is now allowed to range over all of $[0,\infty)\times\R^{n-2}$; note that ${}^\eop N_{H_2,y_0}(P)$ is invariant under dilations in $(x,y)$ and translations in $y$. To define ${}^\eop N_{H_2,y_0}(P)$ invariantly, consider the (nonstrictly) inward pointing normal bundle
\[
  {}^+N\phi^{-1}(y_0)={}^+T_{\phi^{-1}(y_0)}M/T\phi^{-1}(y_0).
\]
Note that ${}^+N\phi^{-1}(y_0)$ is equipped with a natural $\R_+$-action given by dilations in the fibers; and moreover the image of $T_{\phi^{-1}(y_0)}H_2\subset{}^+T_{\phi^{-1}(y_0)}M$ in the quotient ${}^+N\phi^{-1}(y_0)$ acts by translations. Lastly, we can lift $\phi\colon H_2\to Y$ to a fibration $T H_2\to T Y$ and restrict to the tangent bundle over $\phi^{-1}(y_0)$, giving a fibration $T_{\phi^{-1}(y_0)}H_2\to T_{y_0}Y$; since $T\phi^{-1}(y_0)$ lies in the kernel of this fibration, this descends to a fibration ${}^+N\phi^{-1}(y_0)\to T_{y_0}Y$.

Given this edge-b-structure on ${}^+N\phi^{-1}(y_0)$, there is a natural bundle isomorphism
\begin{equation}
\label{EqEBONormH2eIso}
  \Teb_p M \cong \Teb_{(p,0)}({}^+N\phi^{-1}(y_0)),
\end{equation}
where $(p,0)$ denotes the unique point in the zero section of ${}^+N_p\phi^{-1}(y_0)$; in local coordinates in which ${}^+T_{\phi^{-1}(y_0)}M=\{ (z,\dot x\pa_x+\dot y\pa_y+\dot z\pa_z) \colon (y_0,z)\in\phi^{-1}(y_0),\ \dot x\geq 0\}$ and ${}^+N\phi^{-1}(y_0)=\{(z,\dot x\pa_x+\dot y\pa_y)\colon (y_0,z)\in\phi^{-1}(y_0),\ \dot x\geq 0\}$, this isomorphism takes $x\pa_x$, $x\pa_y$, $z\pa_z$ to $\dot x\pa_{\dot x}$, $\dot x\pa_{\dot y}$, $z\pa_z$.

We may furthermore consider the space $\cV_{\ebop,I}({}^+N\phi^{-1}(y_0))$ of edge-b-vector fields on ${}^+N\phi^{-1}(y_0)$ (with `b'-behavior at the normal bundle ${}^+N_{\pa\phi^{-1}(y_0)}\phi^{-1}(y_0)$ over the boundary of the fiber, which in our setting has two connected components) which are invariant under both the dilation and translation actions; we then have an isomorphism
\begin{equation}
\label{EqEBOInvariant}
  \CI(\phi^{-1}(y_0);\Teb_{\phi^{-1}(y_0)}M) \cong \cV_{\ebop,I}({}^+N\phi^{-1}(y_0))
\end{equation}
given by applying~\eqref{EqEBONormH2eIso} and passing to the unique invariant extension. The multiplicative extension of this map gives invariant meaning to
\[
  {}^\eop N_{H_2,y_0}(P) \in \Diff_{\ebop,I}^m({}^+N\phi^{-1}(y_0)),
\]
where the subscript `$I$' restricts to dilation- and translation-invariant operators.

\begin{definition}[Edge normal operator at $H_2$]
\label{DefEBONormH2}
  For $y_0\in Y$, the map ${}^\eop N_{H_2,y_0}\colon\Diffeb(M)\to\Diff_{\ebop,I}({}^+N\phi^{-1}(y_0))$ is an algebra homomorphism and for each $m\in\N_0$ fits into the short exact sequence
  \begin{equation}
  \label{EqEBONormH2Seq}
    0 \to \cI_{\phi^{-1}(y_0)}\Diffeb^m(M) \hra \Diffeb^m(M) \xra{{}^\eop N_{H_2,y_0}} \Diff_{\ebop,I}^m({}^+N\phi^{-1}(y_0)) \to 0,
  \end{equation}
  where $\cI_{\phi^{-1}(y_0)}\subset\CI(M)$ is the ideal of smooth functions vanishing along $\phi^{-1}(y_0)$. The collection ${}^\eop N_{H_2}=({}^\eop N_{H_2,y_0})_{y_0\in Y}$ fits into the short exact sequence
  \[
    0 \to \rho_2\Diffeb^m(M) \hra \Diffeb^m(M) \xra{{}^\eop N_{H_2}} \Diff_{\ebop,I}^m({}^+N\phi) \to 0,
  \]
  where ${}^+N\phi=\bigsqcup_{y\in Y} {}^+N_p(\phi^{-1}(y))$ is the total space of a fibration ${}^+N\phi\to T Y$, and the space $\Diff_{\ebop,I}^m({}^+N\phi)$ consists of differential operators (with smooth coefficients) which are tangent to each fiber ${}^+N\phi^{-1}(y_0)$, $y_0\in Y$, and whose restriction to ${}^+N\phi^{-1}(y_0)$ is an element of $\Diff_{\ebop,I}^m({}^+N\phi^{-1}(y_0))$.\footnote{Thus, $\Diff_{\ebop,I}^m({}^+N\phi)$ consists of collections of invariant edge-b-operators, indexed by $y_0\in Y$, which depend \emph{smoothly} on the parameter $y_0$.}
\end{definition}

A less precise normal operator can be defined by first regarding $P$ simply as a b-differential operator, $P\in\Diffb^m(M)$, and then considering the b-normal operator
\begin{equation}
\label{EqEBONormH2b}
  {}^\bop N_{H_2}(P) \in \Diff_{\bop,I}^m({}^+N H_2).
\end{equation}
In local coordinates as in~\eqref{EqEBOOp}, this means
\[
  {}^\bop N_{H_2}(P) := \sum_{j+k\leq m} a_{j k 0}(0,y,z)(x D_x)^j(z D_z)^k.
\]
We remark that ${}^\bop N_{H_2}(P)$ is uniquely determined by ${}^\eop N_{H_2}(P)$. (The conjugation of ${}^\bop N_{H_2}(P)$ by the Mellin transform in $x$ is the indicial family in the terminology of \cite[Definition~(2.18)]{MazzeoEdge}.) It is often the case that the properties of the simpler model ${}^\bop N_{H_2}(P)$ determine the asymptotic behavior of solutions of $P$ near $H_2$; see \cite[\S1.1.1]{HintzVasyMink4} or \cite[\S3.6]{HintzMink4Gauge}.

The extension of the above definitions to operators acting on sections of vector bundles, $P\in\Diffeb^m(M;E,F)$, only requires notational changes: for $P\in\Diffeb^m(M;E,F)$, we have
\begin{align*}
  \sigmaeb^m(P) &\in P^{[m]}(\Teb^*M;\pi^*\Hom(E,F)), \\
  N_{H_1}(P) &\in \Diff_{\ebop,I}^m({}^+N H_1;\pi^*E|_{H_1},\pi^*F|_{H_1}), \\
  {}^\eop N_{H_2,y_0} &\in \Diff_{\ebop,I}^m({}^+N\phi^{-1}(y_0);\pi^*E|_{\phi^{-1}(y_0)},\pi^*F|_{\phi^{-1}(y_0)}), \qquad y_0\in Y,
\end{align*}
where in each line $\pi$ denotes the relevant base projection, and we have corresponding short exact sequences mirroring \eqref{EqEBOSeqSymb}, \eqref{EqEBONormH1Seq}, \eqref{EqEBONormH2Seq}.

\subsection{Edge-b-pseudodifferential operators}
\label{SsEBP}

We now restrict attention to the case of compact $M$ for simplicity. We denote by $S^s(\Teb^*M)$ the space of symbols of order $s\in\R$. In the local coordinates~\eqref{EqEBOCoord}, this means that $a\in S^s(\Teb^*M)$ if and only if
\begin{equation}
\label{EqEBPSymbEst}
  |\pa_x^j\pa_y^\alpha\pa_z^k\pa_\xi^p\pa_\eta^\beta\pa_\zeta^q a(x,y,z,\xi,\eta,\zeta)| \leq C_{j\alpha k p\beta q}(1+|\xi|+|\eta|+|\zeta|)^{s-(p+|\beta|+q)}
\end{equation}
for all $j,k,p,q\in\N_0$ and $\alpha,\beta\in\N_0^{n-2}$. The condition~\eqref{EqEBPSymbEst} can be phrased invariantly using the radial compactification $\ol{\Teb^*}M$: let $\rho_\infty\in\CI(\ol{\Teb^*}M)$ denote a defining function of $\Seb^*M$; then $a\in S^s(\Teb^*M)$ if and only if $P a\in\rho_\infty^{-s}L^\infty(\Teb^*M)$ for all $P\in\Diff(\ol{\Teb^*}M)$ which are linear combinations of up to $m$-fold compositions of vector fields on $\ol{\Teb^*}M$ which are tangent to $\Seb^*M$. We also need to consider weighted symbols. Recall the notation $\rho=(\rho_1,\ldots,\rho_N)$, with $\rho_j$ a defining function of $H_j$; for $\alpha\in\R^N$, we then put
\[
  S^{s,\alpha}(\Teb^*M) := \rho^{-\alpha}S^s(\Teb^*M).
\]
More generally still, we can consider symbols with conormal regularity in the base; thus, we define the space
\[
  \cA^{-\alpha}S^s(\Teb^*M)
\]
to consist of all smooth functions $a$ on $T^*M^\circ$ so that $P a\in\rho^{-\alpha}\rho_\infty^{-s}L^\infty(T^*M^\circ)$ for all $P\in\Diffb(\ol{\Teb^*}M)$ (with b-behavior at \emph{all} boundary hypersurfaces of $\ol{\Teb^*}M$, not just at $\Seb^*M$). This contains the space of finite linear combinations of functions on $T^*M^\circ$ of the form $w a$ where $w\in\cA^{-\alpha}(M)$ (pulled back to $\Teb^*M$ along the base projection) and $a\in S^s(\Teb^*M)$.

Our main tool for the study of edge-b-differential operators will be the algebra
\[
  \Psieb(M) = \bigcup_{s\in\R} \Psieb^s(M)
\]
of edge-b-pseudodifferential operators, which we describe geometrically following \cite[Appendix~B]{MelroseVasyWunschDiffraction}, followed further below by a description in terms of quantization maps. Recall that $H_2$ is equipped with a fibration $\phi\colon H_2\to Y$. By \cite[Lemma~B.1]{MelroseVasyWunschDiffraction}, the fiber diagonal $(H_2)_\phi^2:=H_2\times_\phi H_2$ is a p-submanifold of $M^2$; it is moreover transversal to, or disjoint from, the remaining boundary diagonals $H_j^2$, $j\neq 2$. Collect the latter in the set $\cB:=\{H_j^2\colon j\neq 2\}$. The edge-b-double space is then the iterated blow-up
\begin{equation}
\label{EqEBPDouble}
  M^2_\ebop := [M^2; (H_2)_\phi^2; \cB].
\end{equation}
Let $\upbeta\colon M^2_\ebop\to M^2$ denote the blow-down map, let $\pi_{L/R}\colon M^2\to M$ denote the projection onto the left/right factor, and denote by $\pi_{\ebop,L/R}=\pi_{L/R}\circ\upbeta$ the stretched projection. Denote by $\diag_\ebop\subset M^2_\ebop$ the edge-b-diagonal, defined as the lift of the diagonal in $M^2$. Then the space $\Psieb^s(M)$ is defined on the level of Schwartz kernels by
\[
  \Psieb^s(M) = \bigl\{ \kappa\in I^s(M^2_\ebop,\diag_\ebop;(\pi_{\ebop,R})^*(\Omegaeb M)) \colon \kappa\equiv 0\ \text{at}\ \pa M^2_\ebop\setminus\ff(\upbeta) \bigr\},
\]
where $I^s$ denotes the space of conormal distributions \cite{HormanderFIO1} (with smooth coefficients down to $\pa\diag_\ebop$), `$\equiv 0$' denotes equality in Taylor series (i.e.\ infinite order of vanishing), and $\ff(\upbeta)$ denotes the union of all boundary hypersurfaces produced by the blow-ups in the definition~\eqref{EqEBPDouble} of $M^2_\ebop$.

Operators between vector bundles $E,F\to M$ arise by tensoring the bundle in which the conormal distributions take values with $\upbeta^*\Hom(\pi_R^*E,\pi_L^*F)$. Schwartz kernels of weighted edge-b-ps.d.o.s are defined by
\[
  \Psieb^{s,\alpha}(M) = \rho^{-\alpha}\Psieb^s(M) := \{ (\pi_{\ebop,L}^*\rho^{-\alpha})\kappa \colon \kappa\in\Psieb^s(M) \}.
\]
More generally, one can consider spaces of ps.d.o.s with conormal coefficients,
\[
  \cA^{-\alpha}\Psieb^s(M),
\]
by requiring their Schwartz kernels to be conormal distributions whose symbols are conormal (rather than smooth) down to $\diag_\ebop$, with weight $-\alpha_j$ at the lift of $(H_j)^2$ ($j\neq 2$) or $(H_2)^2_\phi$ ($j=2$).

Turning to the explicit description, $\Psieb^s(M)$ can be defined in terms of quantizations of symbols $a\in S^s(\Teb^*M)$. Concretely, in the local coordinates~\eqref{EqEBOCoord}, and for $u$ with support in the local coordinate patch, a typical element of $\Psieb^s(M)$ is $A=\Op_\ebop(a)$, acting on $u$ via
\begin{align*}
  A u(x,y,z) &= (2\pi)^{-n}\int \exp\Bigl(i\bigl((x-x')\xi+(y-y')\cdot\eta+(z-z')\zeta\bigr)\Bigr) \\
    &\hspace{6em} \times \chi\Bigl(\frac{x-x'}{x}\Bigr)\chi\Bigl(\frac{|y-y'|}{x}\Bigr)\chi\Bigl(\frac{z-z'}{z}\Bigr) \\
    &\hspace{6em} \times a(x,y,z;x\xi,x\eta,z\zeta)u(x',y',z')\,\dd\xi\,\dd\eta\,\dd\zeta\,\dd x'\,\dd y'\,\dd z';
\end{align*}
here $\chi\in\CIc((-1,1))$ is identically $1$ near $0$. In general, one defines $\Op_\ebop(a)$ using a partition of unity. Thus,
\[
  \Op_\ebop\colon S^s(\Teb^*M)\to\Psieb^s(M),
\]
and this map is surjective modulo the space $\Psieb^{-\infty}(M)$ of residual operators. (The space $\Psieb^{-\infty}(M)$ consists of all smooth right edge-b-densities on $M^2_\ebop$ which vanish to infinite order at $\pa M^2_\ebop\setminus\ff(\upbeta)$.) In fact, any finite number of compositions of quantizations is again a quantization upon enlarging the size of the cutoffs; the space $\Psieb^{-\infty}(M)$ is merely a simple device to capture all (symbolically trivial) off-diagonal terms. Spaces of edge-b-ps.d.o.s with weighted or conormal coefficients, and possibly acting on vector bundles, can be defined similarly.

The principal symbol of edge-b-ps.d.o.s fits into the short exact sequence
\[
  0 \to \Psieb^{s-1}(M) \hra \Psieb^s(M) \xra{\sigmaeb^s} (S^s/S^{s-1})(\Teb^*M) \to 0;
\]
it is defined by $\sigmaeb^s(\Op_\ebop(a))=[a]$ and $\sigmaeb^s|_{\Psieb^{-\infty}(M)}=0$. The principal symbol map is multiplicative; moreover, for $A_j\in\Psieb^{s_j}(M)$ with principal symbol $a_j$, we have
\[
  \sigmaeb^{s_1+s_2-1}(i[A_1,A_2])=H_{a_1}a_2.
\]
As usual, for $A\in\Psieb^s(M)$, we denote by\footnote{The elliptic set of $A$ depends on the order of the space of ps.d.o.s of which one regards $A$ as an element of; since this order is always clear from the context, we shall write $\Elleb(A)$ simply instead of the more cumbersome $\Elleb^s(A)$.}
\[
  \Elleb(A) \subset \Teb^*M\setminus o,\qquad
  \text{resp.}\qquad \WFeb'(A)\subset\Teb^*M\setminus o
\]
the elliptic set, resp.\ operator wave front set, consisting of points $\varpi\in\Teb^*M\setminus o$ so that $\sigmaeb^s(A)$ is invertible in a conic neighborhood of $\varpi$, resp.\ the full symbol of $A$ (in any local coordinate chart) is not of order $-\infty$ (at fiber infinity) in any conic neighborhood of $\varpi$. Identifying these sets with their boundaries at fiber infinity inside $\ol{\Teb^*}M$, we shall often regard
\[
  \Elleb(A),\ \WFeb'(A) \subset \Seb^*M.
\]
(Carefully note that $\WFeb'(A)=\emptyset$ implies $A\in\Psieb^{-\infty}(M)$, which is thus trivial in the differential order sense, but not in the sense of decay at $\pa M$.) We furthermore write $\Char_\ebop(A)=\Seb^*\setminus\Elleb(A)$ for the characteristic set.

\subsection{Sobolev spaces}
\label{SsEBH}

We continue assuming that $M$ is compact. Fixing a weight vector $w\in\R^N$, we shall work with a weighted edge-b-density
\[
  \mu \in \rho^w\CI(M;\Omegaeb M)
\]
which is positive, i.e.\ $\rho^{-w}\mu>0$ as an edge-b-density. In the local coordinates~\eqref{EqEBSe2}, this means
\[
  \mu = a(x,y,z)z^{w_1}x^{w_2}\bigl|\tfrac{\dd x}{x}\tfrac{\dd y}{x^{n-2}}\tfrac{\dd z}{z}\bigr|,\qquad 0<a\in\CI.
\]
We then set
\[
  \Heb^0(M,\mu) := L^2(M,\mu).
\]
Whenever the density $\mu$ is clear from the context, we shall omit it from the notation. For integer $s\in\N_0$, we define $\Heb^s(M)$ to consist of all $u\in\Heb^0(M)$ so that $P u\in\Heb^0(M)$ for all $P\in\Diffeb^s(M)$. For $s\in\R$, we can define $\Heb^s(M)$ via duality and interpolation. Alternatively, for $s\geq 0$, we have $\Heb^s(M)=\{u\in\Heb^0(M)\colon A u\in\Heb^0(M)\}$ where $A\in\Psieb^s(M)$ is any fixed elliptic operator; this is thus a Hilbert space with squared norm
\[
  \|u\|_{\Heb^s(M)}^2 = \|u\|_{\Heb^0(M)}^2 + \|A u\|_{\Heb^0(M)}^2.
\]
By elliptic regularity, any two choices of $A$ give equivalent norms. For $s<0$ we have $\Heb^s(M)=\{u_1+A u_2\colon u_1,u_2\in\Heb^0(M)\}$ where $A\in\Psieb^{-s}(M)$ is elliptic; since this is isomorphic to $(\Heb^{-s}(M))^*$ via the $L^2(M,\mu)$-pairing, it is a Hilbert space as well. Weighted versions of these spaces are defined by
\[
  \Heb^{s,\alpha}(M) := \rho^\alpha\Heb^s(M) = \{\rho^\alpha u\colon u\in\Heb^s(M) \}.
\]
These are Hilbert spaces as well, and $(\Heb^{s,\alpha}(M))^*=\Heb^{-s,-\alpha}(M)$ with respect to $L^2(M)$. When $\Omega\subset M$ is an open set, we denote spaces of extendible and supported distributions by
\[
  \bar H_\ebop^{s,\alpha}(\Omega) := \{ u|_\Omega \colon u\in\Heb^{s,\alpha}(M) \}, \qquad
  \dot H_\ebop^{s,\alpha}(\ol\Omega) := \{ u \in \Heb^{s,\alpha}(M) \colon \supp u\subset\ol\Omega \},
\]
following the notation of \cite[Appendix~B]{HormanderAnalysisPDE3}.

We also need spaces of sections of vector bundles $E\to M$: for $s=\alpha=0$, and fixing any positive definite fiber metric on $E$, we set $\Heb^{0,0}(M;E)=\Heb^0(M;E)=L^2(M,\mu;E)$, any two choices of fiber metrics giving the same space (up to equivalence of norms) since $M$ is compact. For general $s,\alpha$, the space $\Heb^{s,\alpha}(M;E)$ is then defined as above.

Edge-b-ps.d.o.s $A\in\Psieb^0(M)$ act continuously on $\Heb^0(M)$, see \cite[Appendix~B]{MelroseVasyWunschDiffraction}. Using the algebra properties of $\Psieb(M)$, this can be shown to imply that every $A\in\Psieb^{m,\beta}(M)$ defines a bounded linear map $A\colon\Heb^{s,\alpha}(M)\to\Heb^{s-m,\alpha-\beta}(M)$.

We define
\[
  \Heb^{-\infty,\alpha}(M) := \bigcup_{s\in\R} \Heb^{s,\alpha}(M),\qquad
  \Heb^{\infty,\alpha}(M) := \bigcap_{s\in\R} \Heb^{s,\alpha}(M).
\]
Given a distribution $u\in\Heb^{-\infty,\alpha}(M)$, we can define its wave front set in the usual manner: for $s\in\R$, we define
\[
  \WFeb^{s,\alpha}(u) \subset \Seb^*M
\]
as the complement of the set of $\varpi\in\Seb^*M$ for which there exists an operator $A\in\Psieb^s(M)$ which is elliptic at $\varpi$ and so that $A u\in\Heb^{0,\alpha}(M)=\rho^\alpha\Heb^0(M)$. We stress that we need to assume a priori that $u$ lies in an edge-b-Sobolev space with weight $\alpha$ in order to guarantee that $A'u\in\Heb^{0,\alpha}(M)$ as well for all $A'\in\Psieb^s(M)$ with $\WFeb'(A')\subset\Elleb(A)$. In particular,
\[
  u\in\Heb^{-\infty,\alpha}(M),\ \WFeb^{s,\alpha}(u)=\emptyset\ \implies\ u\in\Heb^{\infty,\alpha}(M),
\]
i.e.\ the wave front set controls edge-b-regularity, but not decay.

We shall also need to use Sobolev spaces associated with other Lie algebras, in particular b-Sobolev spaces $\Hb^{s,\alpha}(M)$ on manifolds with boundary or corners, scattering Sobolev spaces $\Hsc^{s,\alpha}(M)$ on manifolds with boundary, and 0-Sobolev spaces $H_0^{s,\alpha}(M)$ on manifolds with boundary; these spaces are defined in complete analogy with the edge-b-spaces above. Detailed discussions of these spaces (as well as of the corresponding calculi of pseudodifferential operators) can be found in \cite{MazzeoEdge,MelroseAPS,MelroseEuclideanSpectralTheory,MazzeoMelroseFibred,Hintz3b}.

\subsection{Invariant edge-b-operators and edge-b-Sobolev spaces}
\label{SsEBI}

We now turn to notions related to the analysis of edge normal operators of edge-b-differential operators. Fix $y_0\in Y$, and identify $Z=\phi^{-1}(y_0)$ with the zero section of
\[
  \cN:={}^+N\phi^{-1}(y_0) = {}^+T_{\phi^{-1}(y_0)}M / T\phi^{-1}(y_0).
\]
Recall on $\cN$ the space $\cV_{\ebop,I}(\cN)$ of edge-b-vector fields which are invariant with respect to the dilation and translation actions introduced in the paragraph preceding~\eqref{EqEBOInvariant}. In terms of coordinates $x\geq 0$, $y\in\R^{n-2}$ on the fibers of
\[
  \pi \colon \cN\to \phi^{-1}(y_0)=Z,
\]
so $\cN=Z\times[0,\infty)_x\times\R^{n-2}_y$, the space $\cV_{\ebop,I}(\cN)$ is thus spanned over $\CI(Z)$ by $x\pa_x$, $x\pa_y$, $\Vb(Z)$, with the dilation action given by $(x,y,z)\mapsto(\lambda x,\lambda y,z)$, $\lambda\in\R_+$, and the translation action given by $(x,y,z)\mapsto(x,y+y',z)$, $y'\in\R^{n-2}$. Here, we identify $\CI(Z)\subset\CI(\cN)$ via pullback. (Thus, $\CI(Z)$ is the space of dilation- and translation-invariant smooth functions.) We also recall the notation $\Diff_{\ebop,I}^m(\cN)$ for up to $m$-fold compositions of elements of $\cV_{\ebop,I}(\cN)$.

Following Mazzeo \cite[\S5]{MazzeoEdge}, one may analyze elements of $\Diff_{\ebop,I}^m(\cN)$ (see~\eqref{EqEBONormH2y0} for a local coordinate expression) by passing to the Fourier transform in $y$, with dual momentum denoted $\eta$, and then changing variables to $(\hat x,\hat\eta,z)$ where $\hat x=x|\eta|$ and $\hat\eta=\eta/|\eta|$, thus obtaining a smooth family (in $(y_0,\hat\eta)\in S^*Y$) of operators of Bessel type on $[0,\infty)_{\hat x}\times Z$ (or indeed of weighted b-scattering type on $[0,\infty]_{\hat x}\times Z$, with scattering behavior at the boundary hypersurface $\{\infty\}\times Z$). In this paper, we instead work directly with invariant edge-b-operators, as in the wave equation setting of interest here this makes the proofs of their mapping and regularity properties straightforward modifications of the proofs for the original edge-b-operator. We thus proceed to define spaces and ps.d.o.s for the analysis of invariant edge-b-operators on $\cN$.

The principal symbol $\sigmaeb^m(P)\in P^{[m]}(\Teb^*\cN)$ of $P\in\Diff_{\ebop,I}^m(\cN)$ is invariant under the lifts of the dilation and translation actions; in local coordinates on $Z$ and using fiber-linear coordinates $\xi,\eta_j,\zeta$ as in~\eqref{EqEBOCoord}, this simply means that $\sigmaeb^m(P)$ is independent of $x,y$. Thus, the principal symbol short exact sequence reads
\[
  0 \to \Diff_{\ebop,I}^{m-1}(\cN) \hra \Diff_{\ebop,I}^m(\cN) \xra{{}^{\ebop,I}\sigma^m} P_I^{[m]}(\Teb^*\cN) \to 0,
\]
where the subscript `$I$' restricts to invariant symbols. (Equivalently, restriction to $Z$ gives an isomorphism $P_I^{[m]}(\Teb^*\cN)\cong P^{[m]}(\Teb^*_Z\cN)$ with the inverse given by invariant extension.)

Note moreover that since $Z\cong[-1,1]$, the manifold $\cN$ has three boundary hypersurfaces which we denote
\[
  \cH_j := {}^+N_{\phi^{-1}(y_0)\cap H_j}\phi^{-1}(y_0),\quad j=1,3, \qquad
  \cH_2 := T_{\phi^{-1}(y_0)}H_2 / T\phi^{-1}(y_0).
\]
Thus, $\cH_1,\cH_3$ are the fibers of $\cN\to\phi^{-1}(y_0)$ over the two points of $\pa Z$, while $\cH_2$ is the fibered boundary of $\cN$ (given in the above coordinates by $x=0$). Any two defining functions $\rho_2,\rho_2'$ of $\cH_2$ which are invariant are related via $\rho_2'=a\rho_2$ where $0<a\in\CI(Z)$. Invariant defining functions of $\cH_1,\cH_3$ are pullbacks of defining functions of the two boundary points of $Z$.

\begin{definition}[Invariant edge-b-Sobolev spaces]
\label{DefEBIH}
  Denote by $\rho_1,\rho_2,\rho_3\in\CI(\cN)$ invariant defining functions of $\cH_1,\cH_2,\cH_3$, and let $\alpha=(\alpha_1,\alpha_2,\alpha_3)\in\R^3$. Fix an invariant edge-b-density $0<\mu_0\in\CI(\cN;\Omegaeb\cN)$, a weight $w\in\R^3$, and fix $\mu=\rho^w\mu_0$. Then $H_{\ebop,I}^0(\cN,\mu):=L^2(\cN,\mu)$. We drop $\mu$ from the notation from now on. For $s\in\N$, we set $H_{\ebop,I}^s(\cN)=\{u\in L^2(\cN)\colon P u\in L^2(\cN)\ \forall\,P\in\Diff_{\ebop,I}^k(\cN)\}$,\footnote{Simply put, the space $H_{\ebop,I}^s(\cN)$ consists of all elements $u=u(x,y,z)\in L^2(\cN)$ (where $x\geq 0$, $y\in\R^{n-2}$, $z\in Z$) so that its up to $s$-fold derivatives along $x\pa_x$, $x\pa_y$, and elements of $\cV(Z)$ remain in $L^2(\cN)$.} and we define $H_{\ebop,I}^s(\cN)$ for general $s\in\R$ via interpolation and duality. Finally, we set
  \[
    H_{\ebop,I}^{s,\alpha}(\cN) := \rho^\alpha H_{\ebop,I}^s(\cN).
  \]
  Given a vector bundle $E\to M$, invariant weighted edge-b-Sobolev spaces of sections of the bundle $\pi^*E|_{\phi^{-1}(y_0)}\to\cN$ are defined analogously.\footnote{This requires the choice of a connection on $E|_{\phi^{-1}(y_0)}$, but any two choices give the same space. Note that over any fiber $\pi^{-1}(z_0)$, $z_0\in\phi^{-1}(y_0)$, of $\cN$, the bundle $\pi^*E|_{\phi^{-1}(y_0)}$ is canonically trivial, and hence sections of $\pi^*E|_{\phi^{-1}(y_0)}\to\cN$ can be differentiated along the generators of the translation and dilation actions without any further choices.}
\end{definition}

In compact subsets of $\cN$, these are standard weighted edge-b-Sobolev spaces. The only reason for including the subscript `$I$' in the notation for these spaces is that $\cN$ is noncompact; invariant edge-b-Sobolev spaces provide a means to measure edge-b-regularity and square integrability uniformly on $\cN$. Given an open subset $\cU\subset Z$, the preimage $\Omega:=\pi^{-1}(\cU)\subset\cN$ is invariant; we shall then consider spaces of supported and extendible distributions on $\Omega$,
\[
  \bar H_{\ebop,I}^{s,\alpha}(\Omega) := \{ u|_\Omega \colon u\in H_{\ebop,I}^{s,\alpha}(\cN) \},\qquad
  \dot H_{\ebop,I}^{s,\alpha}(\ol\Omega) := \{ u\in H_{\ebop,I}^{s,\alpha}(\cN) \colon \supp u\subset\ol\Omega\},
\]
equipped as usual with the quotient, resp.\ subspace topology.

We shall also use invariant edge-b-pseudodifferential operators, i.e.\ edge-b-ps.d.o.s on $\cN$ which are invariant under the translation and dilation actions. From a Schwartz kernel perspective, they are thus uniquely determined by the restriction of their Schwartz kernels to the edge front face (the lift of $(\cH_2)^2_\phi$ to $\cN_\ebop^2=[\cN^2;(\cH_2)^2_\phi;\cH_1^2,\cH_3^2]$). Explicitly, consider a symbol $a\in S^s(\Teb^*_Z\cN)$, or equivalently an invariant symbol $a\in S_I^s(\Teb^*\cN)$ (the subspace of $S^s(\Teb^*\cN)$ consisting of invariant elements); in local coordinates $z\geq 0$ on $Z$, we thus have $a=a(z,\xi,\eta,\zeta)$, and the quantization of $a$ as an invariant edge-b-ps.d.o.\ is
\begin{align*}
  (\Op_\ebop(a)u)(x,y,z) &= (2\pi)^{-n}\int \exp\Bigl(i\Bigl(\frac{x-x'}{x}\xi+\frac{y-y'}{x}\cdot\eta+(z-z')\zeta\Bigr)\Bigr) \\
    &\hspace{6em} \times \chi\Bigl(\frac{x-x'}{x}\Bigr)\chi\Bigl(\frac{|y-y'|}{x}\Bigr)\chi\Bigl(\frac{z-z'}{z}\Bigr) \\
    &\hspace{6em} \times a(z;\xi,\eta,z\zeta)u(x',y',z')\,\dd\xi\,\dd\eta\,\dd\zeta\,\frac{\dd x'}{x}\,\frac{\dd y'}{x^{n-2}}\,\dd z'.
\end{align*}
Denoting the space of invariant edge-b-ps.d.o.s by $\Psi_{\ebop,I}^s(\cN)$, we correspondingly have a principal symbol short exact sequence
\[
  0 \to \Psi_{\ebop,I}^{s-1}(\cN) \hra \Psi_{\ebop,I}^s(\cN) \xra{{}^{\ebop,I}\sigma^s} (S_I^s/S_I^{s-1})(\Teb^*\cN) \to 0.
\]
One can also consider classes of weighted operators
\[
  \Psi_{\ebop,I}^{s,\alpha}(\cN) = \rho^{-\alpha}\Psi_{\ebop,I}^s(\cN),
\]
whose Schwartz kernels are homogeneous of degree $-\alpha_2$ with respect to the dilation action, and invariant under the translation action. (More generally, one can allow for conormal behavior at $\cH_1,\cH_3$.) Their principal symbols are elements of $(S_I^{s,\alpha}/S_I^{s-1,\alpha})(\Teb^*\cN)$, where $S_I^{s,\alpha}(\Teb^*\cN)$ is the subspace of $S^{s,\alpha}(\Teb^*\cN)$ consisting of symbols which are homogeneous of degree $-\alpha_2$ with respect to the dilation action, and invariant under the translation action.

One can then characterize the spaces $H_{\ebop,I}^{s,\alpha}(\cN)$ via testing with invariant edge-b-ps.d.o.s, analogously to the standard edge-b-setting in~\S\ref{SsEBH}. Moreover, an element $A\in\Psi_{\ebop,I}^{m,\beta}(\cN)$ defines a bounded linear map $H_{\ebop,I}^{s,\alpha}(\cN)\to H_{\ebop,I}^{s-m,\alpha-\beta}(\cN)$. We define the invariant wave front set
\begin{equation}
\label{EqEBIWF}
  \WF_{\ebop,I}^{s,\alpha}(u) \subset \Seb^*_Z\cN
\end{equation}
(which one can equivalently regard as an invariant subset of $\Seb^*\cN$) for $u\in H_{\ebop,I}^{-\infty,\alpha}(\cN)$ in the usual manner by testing with elliptic invariant edge-b-ps.d.o.s.

\subsection{Semiclassical 0-analysis}
\label{SsEB0}

The following material will be used only in~\S\ref{SNp}.

\begin{definition}[Semiclassical 0-differential operators, Sobolev spaces]
\label{DefNpHDiff}
  Let $X$ denote an $n$-dimensional manifold with boundary. Then $\Diff_{0,\semi}^m(X)$ denotes the space of $h$-dependent differential operators $P$ on $X^\circ$, $h\in(0,1)$, which in local coordinates $[0,\infty)_x\times\R^{n-1}_y$ on $X$ are of the form
  \[
    \sum_{j+|\alpha|\leq m} a_{j\alpha}(h,x,y) (h x D_x)^j (h x D_y)^\alpha,\qquad a_{j\alpha}\in\CI([0,1)_h\times[0,\infty)\times\R^{n-1}).
  \]
  If $X$ is compact and equipped with a positive weighted b-density (omitted from the notation), and if $\rho\in\CI(X)$ denotes a boundary defining function, then we define semiclassical Sobolev spaces $H_{0,h}^{s,\alpha}(X)=\rho^\alpha H_{0,h}^s(X)$ to be equal to $H_0^{s,\alpha}(X)$ as sets, but with norm for $s\in\N_0$ given by
  \[
    \|u\|_{H_{0,h}^{s,\alpha}}^2 := \sum_j \| P_j u \|_{L^2}^2,
  \]
  where $\{P_j\}\subset\Diff_{0,\semi}^s(X)$ is a finite spanning set of $\Diff_{0,h}^s(X)$ over $\CI([0,1)\times X)$. For real $s$, the norm on $H_{0,h}^{s,\alpha}(X)$ is defined via duality and interpolation.
\end{definition}

For $u$ supported in a coordinate patch, and for $s\in\N_0$, an equivalent norm on $H_{0,h}^{s,\alpha}(X)$ is given by
\[
  \|u\|_{H_{0,h}^{s,\alpha}}^2 := \sum_{j+|\alpha|\leq s} \|x^{-\alpha}(h x D_x)^j(h x D_y)^\alpha u\|_{L^2}^2.
\]
That is, every 0-derivative is weighted by a factor of $h$. There is an associated algebra of \emph{semiclassical 0-pseudodifferential operators} which we describe in~\S\ref{SssEB0Psdo} below.

There is a close relationship between 0-analysis on $X$ and edge-b-analysis on
\[
  M=[0,\infty)_{\rho_+}\times X,
\]
where the boundary hypersurface $[0,\infty)\times\pa X$ of $M$ is fibered by $[0,\infty)\times\pa X\to\pa X$.

\begin{lemma}[The Mellin transform, edge-b-Sobolev spaces on $M$ and 0-Sobolev spaces on $X$]
\label{LemmaNpHMellin}
  Fix a positive b-density $\mu_\bop$ on $X$ and the b-density $|\frac{\dd\rho_+}{\rho_+}|\mu_\bop$ on $M$. Let $s,\gamma_{\!\scri},\gamma_+\in\R$. Then the Mellin transform in $\rho_+$, defined by
  \begin{equation}
  \label{EqNpHMellin}
    (\cM u)(\lambda,x_{\!\scri},y) = \int_0^\infty \rho_+^{-i\lambda}u(\rho_+,x_{\!\scri},y)\,\frac{\dd\rho_+}{\rho_+},
  \end{equation}
  is an isomorphism
  \[
    \cM \colon \Heb^{s,(2\gamma_{\!\scri},\gamma_+)}(M) \xra{\cong} L^2\bigl(\{\Im\lambda=-\gamma_+\}; \la\lambda\ra^{-s}H_{0,\la\lambda\ra^{-1}}^{s,2\gamma_{\!\scri}}(X) \bigr),
  \]
  where $\Heb^{s,(2\gamma_{\!\scri},\gamma_+)}(M)=x_{\!\scri}^{2\gamma_{\!\scri}}\rho_+^{\gamma_+}\Heb^s(M)$, with $x_{\!\scri}\in\CI(X)$ denoting a boundary defining function of $X$. 
\end{lemma}
\begin{proof}
  It suffices to consider the case $\gamma_{\!\scri}=\gamma_+=0$. Via interpolation and duality, it suffices to consider the case $s\in\N_0$. For $s=0$, the claim is then a re-statement of Plancherel's theorem. For $s=1$, observe that, in local coordinates $(x,y)$ on $X$, the Mellin transforms intertwines the operators $1,x D_x,x D_y,\rho_+ D_{\rho_+}$ with $1,x D_x,x D_y,\lambda$. Therefore, writing $\hat u(\lambda;x,y)=(\cM u)(\lambda;x,y)$, we have
  \begin{align*}
    &\|u\|_{L^2(M)}^2 + \|x D_x u\|_{L^2(M)}^2 + \|x D_y u\|_{L^2(M)}^2 + \|\rho_+ D_{\rho_+}u\|_{L^2(M)}^2 \\
    &\quad = (1+|\lambda|^2)\Bigl(\|\hat u(\lambda)\|_{L^2(\R_\lambda;L^2(X))}^2 + \|\la\lambda\ra^{-1}x D_x\hat u(\lambda)\|_{L^2(\R_\lambda;L^2(X))}^2 \\
    &\quad\hspace{16em} + \|\la\lambda\ra^{-1}x D_y\hat u(\lambda)\|_{L^2(\R_\lambda;L^2(X))}^2\Bigr).
  \end{align*}
  This proves the Lemma for $s=1$; the case of general $s\in\N$ is similar. See also \cite[\S3.1]{VasyMicroKerrdS}.
\end{proof}

There is a close relationship also on the phase space level. This is discussed in \cite[\S3.3.4]{HintzThesis} in the simpler setting where $X$ is a closed manifold. We describe this relationship from the perspective of differential operators:

\begin{lemma}[Phase space relationship]
\label{LemmaNpHOp}
  Let $P\in\Diffeb^m(M)$.\footnote{One can also consider $P$ acting on sections of a bundle $E\to M$, in which case $P_{h,\hat\lambda}$ acts on sections of $E|_X$. We leave the necessary notational changes to the reader.} Define the semiclassical rescaling of its Mellin-transformed normal operator family as $P_{h,\hat\lambda}:=h^m\wh{N_X}(P,h^{-1}\hat\lambda)$, where we identify $X$ with the boundary hypersurface $\{0\}\times X$ of $M$. Then $P_{h,\hat\lambda}\in\Diff_{0,\semi}^m(X)$, and its principal symbol $p_{\hat\lambda}=\sigma_{0,\semi}^m(P_{h,\hat\lambda})\in P^m({}^0 T^*X)$ is related to $p=\sigmaeb^m(P)$ by
  \[
    p_{\hat\lambda}(\varpi) = p\Bigl(\hat\lambda\frac{\dd\rho_+}{\rho_+}+\varpi\Bigr),\qquad \varpi\in{}^0 T^*X \subset \Teb^*_X M.
  \]
  Here, the inclusion ${}^0 T^*X\subset\Teb^*_X M$ is the adjoint of the map $\Teb_X M\to{}^0 T X$ induced by restriction of vector fields to $X$.
\end{lemma}
\begin{proof}
  In local coordinates and by linearity, it suffices to consider
  \begin{equation}
  \label{EqNpHOpTerm}
    P = a(\rho_+,x,y)(\rho_+ D_{\rho_+})^j (x D_x)^k(x D_y)^\alpha
  \end{equation}
  where $j+k+|\alpha|\leq m$. Then $\wh{N_X}(P,\lambda)$ is obtained by restricting $a$ to $\rho_+=0$ and formally replacing $\rho_+ D_{\rho_+}$ by $\lambda$; hence
  \begin{equation}
  \label{EqNpHOpTerm2}
    P_{h,\hat\lambda} = h^{m-(j+k+|\alpha|)}a(0,x,y)\hat\lambda^j(h x D_x)^k(h x D_y)^\alpha.
  \end{equation}
  Thus, unless $j+k+|\alpha|=m$, the respective principal symbols of $P$ and $P_{h,\hat\lambda}$ vanish, whereas if $j+k+|\alpha|=m$, we have, at a point $(x,y)$ on $X$ and at the corresponding point $(0,x,y)\in\{0\}\times X\subset M$,
  \[
    \sigma_{0,\semi}^m(P_{h,\hat\lambda})\Bigl(\xi\frac{\dd x}{x}+\eta\frac{\dd y}{x}\Bigr) = a(0,x,y)\hat\lambda^j\xi^k\eta^\alpha = \sigmaeb^m(P)\Bigl(\hat\lambda\frac{\dd\rho_+}{\rho_+}+\xi\frac{\dd x}{x}+\eta\frac{\dd y}{x}\Bigr).\qedhere
  \]
\end{proof}

The case of main interest is when $\hat\lambda=\pm 1+\cO(h)$, which arises when $h=\la\lambda\ra^{-1}$ and $\hat\lambda=\frac{\lambda}{\la\lambda\ra}$ and $\lambda$ is restricted to a line of constant imaginary part. Since in the notation of Lemma~\ref{LemmaNpHOp}, the Hamiltonian vector field $H_p$ is tangent to the level sets of $\rho_+\pa_{\rho_+}$ inside $\Teb^*_X M$ (cf.\ Lemma~\ref{LemmaEBOHam} where $z$ and $\zeta$ play the roles of $\rho_+$ and $\rho_+\pa_{\rho_+}(\cdot)$), the restriction of $H_p$ to $\Teb^*_X M\supset\pm\frac{\dd\rho_+}{\rho_+}+{}^0 T^*X\cong{}^0 T^*X$ is equal to the Hamiltonian vector field $H_{p_{\pm 1}}$.

In particular, $H_{p_{\pm 1}}$ has a critical point at $\varpi\in{}^0 T^*X$ if $H_p$ has a critical point at $\pm\frac{\dd\rho_+}{\rho_+}+\varpi$, and the linearizations of both vector fields have a simple relationship (namely, one drops the $\rho_+\pa_{\rho_+}$ term of the latter). Microlocal estimates at such critical points require control of subleading terms; we record:

\begin{lemma}[Relationship of subprincipal terms]
\label{LemmaNpHOpSub}
  Let $P,P_{h,\hat\lambda}$ be as in Lemma~\usref{LemmaNpHOp}. Fix positive b-densities on $X,M$ as in Lemma~\usref{LemmaNpHMellin}. Suppose that $P$ has a real scalar principal symbol, and let $\hat\lambda=\pm 1+\cO(h)$. Then for $\varpi\in {}^0 T^*X \subset \Teb^*_X M$, we have
  \begin{equation}
  \label{EqNpHOpSub}
    \sigma_{0,\semi}^{m-1}\biggl(\frac{P_{h,\hat\lambda}-P_{h,\hat\lambda}^*}{2 i h}\biggr)(\varpi) = \biggl(\sigmaeb^{m-1}\Bigl(\frac{P-P^*}{2 i}\Bigr) + \frac{\Im\hat\lambda}{h}\rho_+^{-1}H_p\rho_+\biggr)\bigg|_{\pm\tfrac{\dd\rho_+}{\rho_+}+\varpi}.
  \end{equation}
\end{lemma}
\begin{proof}
  It suffices to consider $P$ of the form~\eqref{EqNpHOpTerm}. When $j+k+|\alpha|\leq m-2$, both sides of~\eqref{EqNpHOpSub} vanish. When $j+k+|\alpha|=m-1$, the principal symbol of $(2 i)^{-1}(P-P^*)$ at a point $\hat\lambda\frac{\dd\rho_+}{\rho_+}+\varpi$, with $\varpi=\xi\frac{\dd x}{x}+\eta\frac{\dd y}{x}$, over $\rho_+=0$ is $\Im(a(0,x,y)\hat\lambda^j)\xi^k\eta^\alpha$, which matches the principal symbol of $(2 i h)^{-1}(P_{h,\hat\lambda}-P_{h,\hat\lambda}^*)$ at $\varpi$ in view of~\eqref{EqNpHOpTerm2}.

  When $j+k+|\alpha|=m$, then $a$ is real-valued by assumption, and we compute
  \begin{align*}
    \frac{P - P^*}{2 i} &\equiv (\rho_+ D_{\rho_+})^j P_1 \bmod \rho_+\Diff_0^{m-1}(M), \\
    &\quad P_1 = \frac{1}{2 i}\bigl(a(0,x,y)(x D_x)^k(x D_y)^\alpha-((x D_x)^k(x D_y)^\alpha)^*a(0,x,y)\bigr) \in \Diff_0^{m-1}(M).
  \end{align*}
  On the other hand, writing $\hat\lambda=\pm 1+h\kappa$ (thus with $\kappa=\cO(1)$ as $h\to 0$) and noting that $\hat\lambda^j=(\pm 1)^j+(\pm 1)^{j-1}j h\kappa+\cO(h^2)$, we find
  \begin{align*}
    \frac{P_{h,\hat\lambda}-P_{h,\hat\lambda}^*}{2 i h} &\equiv (\pm 1)^{j-1} j(\Im\kappa)a(0,x,y)(h x D_x)^k(h x D_y)^\alpha \\
      &\quad\hspace{5em} + h^{k+|\alpha|-1}(\pm 1)^j P_1 \bmod h\Diff_{0,\semi}^{m-1}(M).
  \end{align*}
  It then remains to observe that the Hamiltonian vector field of $p=a\xi^k\eta^\alpha\zeta^j$ satisfies $\rho_+^{-1}H_p\rho_+=j\zeta^{-1}p$, which at $\zeta=\pm 1$ evaluates to $(\pm 1)^{j-1}j a\xi^k\eta^\alpha$.
\end{proof}

\subsubsection{Semiclassical 0-pseudodifferential operators}
\label{SssEB0Psdo}

Aspects of semiclassical 0-ps.d.o.s were described in \cite[\S3]{MelroseSaBarretoVasyResolvent}; we shall only need the small semiclassical 0-calculus here. We first describe the salient properties and basic applications of this ps.d.o.\ algebra in geometric terms before giving a local coordinate description and proving the composition property. We work with compact $X$ for simplicity. The space
\[
  \Psi_{0,\semi}^{s,\alpha}(X) = \rho^{-\alpha}\Psi_{0,\semi}^s(X)
\]
can then be defined in a geometric manner as follows: Schwartz kernels of elements of $\Psi_{0,\semi}^s(X)$ are distributions on the semiclassical 0-double space $X^2_{0,\semi}:=[[0,1)_h\times X^2;[0,1)\times\diag_{\pa X};\{0\}\times\diag_X]$ (with $\diag_{\pa X}\subset(\pa X)^2$ and $\diag_X\subset X^2$ denoting the diagonals) which are conormal distributions (with values in the right semiclassical 0-density bundle) of order $s-\frac14$ to the lift $\diag_{0,\semi}$ of $[0,1)\times\diag_X$ (in particular, their restriction to $h=h_0>0$ is a distribution on $X^2_0=[X^2;\diag_{\pa X}]$ which is conormal of order $s$ at the lift of $\diag_X$) and which vanish to infinite order at all boundary hypersurfaces of $X^2_{0,\semi}$ which are disjoint from $\diag_{0,\semi}$.

The principal symbol short exact sequence is
\[
  0 \to h\Psi_{0,\semi}^{s-1,\alpha}(X) \hra \Psi_{0,\semi}^{s,\alpha}(X) \to (S^{s,\alpha}/h S^{s-1,\alpha})([0,1)\times{}^0 T^*X) \to 0.
\]
Correspondingly, the elliptic set of an operator $A=(A_h)_{h\in(0,1)}\in\Psi_{0,\semi}^{s,\alpha}(X)$ is the open subset $\Ell_{0,\semi}^{s,\alpha}(A)\cup{}^+\Ell_{0,\semi}^{s,\alpha}(A)$ of $(\{0\}\times\ol{{}^0 T^*}X)\cup([0,1]\times{}^0 S^*X)$, where $\Ell_{0,\semi}^{s,\alpha}(A)\subset\{0\}\times\ol{{}^0 T^*X}$ and ${}^+\Ell_{0,\semi}^{s,\alpha}(A)\subset[0,1]\times{}^0 S^*X$; the two elliptic sets agree in $\{0\}\times{}^0 S^*X$. We shall only be concerned with the elliptic set $\Ell_{0,\semi}^{s,\alpha}(A)$ over $h=0$, which we shall thus simply regard as a subset
\[
  \Ell_{0,\semi}^{s,\alpha}(A) \subset \ol{{}^0 T^*}X;
\]
it captures the ellipticity of $A$ to leading order at $h=0$ and thus arises from the simplified principal symbol map $\sigma_{0,\semi}^{s,\alpha}$,
\[
  0 \to h\Psi_{0,\semi}^{s,\alpha}(X) \hra \Psi_{0,\semi}^{s,\alpha}(X) \xra{\sigma_{0,\semi}^{s,\alpha}} S^{s,\alpha}({}^0 T^*X) \to 0.
\]
Note that if $\varpi\in\Ell_{0,\semi}^{s,\alpha}(A)$, then also $\{h,\varpi\}\in{}^+\Ell_0^{s,\alpha}(A)$ for small $h>0$. Thus, the elliptic parametrix construction gives, for elliptic $A$, an element $B\in\Psi_{0,\semi}^{-s,-\alpha}(X)$ so that $A B-I,B A-I\in h^\infty\Psi_{0,\semi}^{-\infty,\alpha}(X)$ upon restricting $h$ to $(0,h_0)$ for sufficiently small $h_0>0$. Restricted to the space $\Diff_{0,\semi}^s(X)$ of \emph{differential} operators, $\sigma_{0,\semi}^s=\sigma_{0,\semi}^{s,0}$ maps onto the space
\[
  P^s({}^0 T^*X) \subset S^s({}^0 T^*X)
\]
of symbols which are (not necessarily homogeneous) polynomials of degree $s$ in the fibers.

We also have an operator wave front set $\WF_{0,\semi}'(A)\subset\ol{{}^0 T^*}X$ for $A\in h^{-N}\Psi_{0,\semi}^{s,\alpha}(X)$ which is the complement of all points in $\ol{{}^0 T^*}X$ near which the full symbol of $A$ vanishes to infinite order; thus, $\WF_{0,\semi}'(A)=\emptyset$ if and only if $A\in h^\infty\Psi_{0,\semi}^{-\infty,\alpha}(X)$.

Semiclassical 0-ps.d.o.s act on weighted semiclassical 0-Sobolev spaces in the expected manner. We also define the corresponding wave front set:

\begin{definition}[Semiclassical 0-Sobolev wave front set]
\label{DefNpHWF}
  Let $s,\alpha$, and suppose that $u=(u_h)_{h\in(0,1)}\in h^{-N}H_{0,h}^{-N,\alpha}(X)$ for some $N$. Then
  \[
    \WF_{0,\semi}^{s,\alpha}(u) \subset \ol{{}^0 T^*}X
  \]
  is the complement of the set of all $\varpi\in\ol{{}^0 T^*X}$ for which there exists $A\in\Psi_{0,\semi}^{s,\alpha}(X)$, elliptic at $\varpi$, so that $A u$ is uniformly bounded (as $h\to 0$) in $L^2(X)$.
\end{definition}

Thus, if $u\in h^{-N}H_{0,h}^{-N,\alpha}(X)$ satisfies $\WF_{0,\semi}^{s,\alpha}(u)=\emptyset$, then $u\in H_{0,h}^{s,\alpha}(X)$ upon restricting to $h\in(0,h_0)$ with $h_0>0$ sufficiently small. The part $\WF_{0,\semi}^{s,\alpha}(u)\cap{}^0 T^*X$ of the semiclassical 0-wave front set of $u$ at finite 0-momenta is independent of $s$.

For the commutator of $A_j\in\Psi_{0,\semi}^{s_j,\alpha_j}(X)$, $j=1,2$, with $a_j=\sigma_{0,\semi}^{s_j,\alpha_j}(A_j)$, we have
\[
  \tfrac{i}{h}[A_1,A_2]\in\Psi_{0,\semi}^{s_1+s_2-1,\alpha_1+\alpha_2}(X),\quad
  \sigma_{0,\semi}^{s_1+s_2-1,\alpha_1+\alpha_2}\bigl(\tfrac{i}{h}[A_1,A_2]\bigr) = H_{a_1}a_2.
\]
Here, in terms of local coordinates on ${}^0 T^*X$ defined by writing the canonical 1-form as
\[
  \xi\frac{\dd x}{x} + \sum_{j=1}^{n-1}\eta_j\frac{\dd y^j}{x},
\]
the Hamiltonian vector field of $p=p(x,y,\xi,\eta)$ takes the form
\begin{align*}
  H_a &= (\pa_\xi p)\Bigl(x\pa_x+\sum_{j=1}^{n-1}\eta_j\pa_{\eta_j}\Bigr) + \sum_{j=1}^{n-1}(\pa_{\eta_j}p)x\pa_{y^j} \\
    &\qquad - \Bigl(\Bigl(x\pa_x+\sum_{j=1}^{n-1}\eta_j\pa_{\eta_j}\Bigr)p\Bigr)\pa_\xi - \sum_{j=1}^{n-1}(x\pa_{y^j}p)\pa_{\eta_j}.
\end{align*}

We now proceed to prove the above statements about compositions and principal symbols. With $X$ denoting a manifold with embedded boundary, the 0-double space is the blow-up of $X^2$ at the boundary $\pa\diag_X$ of the diagonal. If $x,y$ are local coordinates on the manifold $X$, with $x$ a boundary defining function and $y\in\R^{n-1}$ denoting local coordinates on $\pa X$, let us write $x,y$ also for their pullbacks to $X^2$ from the left factor; and we write $x',y'$ for their pullbacks from the right factor. This blow-up is then that of $x=x'=0$, $y=y'$, and in the interior of the front face, called the \emph{0-front face}, one can use
\[
  x,\quad y,\quad \frac{x-x'}{x},\quad \frac{y-y'}{x}
\]
as coordinates, with the lifted diagonal being given by
\[
  \Bigl\{\frac{x-x'}{x}=0,\ \frac{y-y'}{x}=0\Bigr\}.
\]
Note that near the boundary of the lifted diagonal one can indeed always assume that the coordinates from the left and right factors on $\partial X$ are identical, i.e.\ one is on the same coordinate chart in both factors, since one is in a neighborhood of the diagonal of $\partial X\times\partial X$. Equivalently, we could use coordinates
\[
  x',\quad y',\quad \frac{x-x'}{x'},\quad \frac{y-y'}{x'}.
\]

For semiclassical families one considers $X^2\times[0,1)_h$, and blows up the boundary of the diagonal still, so in the interior of the front face (still called the 0-front face) one can simply add $h$ to the list coordinates, i.e.\ they are
\[
  x,\quad y,\quad \frac{x-x'}{x},\quad \frac{y-y'}{x},\quad h.
\]
The semiclassical 0-double space the blows up the lifted diagonal at $h=0$, i.e.
\[
  \Bigl\{\frac{x-x'}{x}=0,\ \frac{y-y'}{x}=0,\ h=0\Bigr\}.
\]
The resulting front face is the \emph{semiclassical front face}, and local coordinates in the interior are
\begin{equation}
\label{Eq0hCoord}
  x,\quad y,\quad X=\frac{x-x'}{h x},\quad Y=\frac{y-y'}{h x},\quad h.
\end{equation}
While these are valid only in the interior, in a neighborhood of the semiclassical front face, smooth functions on the semiclassical double space which vanish to infinite order at the lift of the $h=0$ face are equivalently smooth functions of the above coordinates which are {\em Schwartz} in $X,Y$. Indeed $\langle (X,Y)\rangle^{-1}$ is a defining function of the lift of this face in this region, since e.g.\ where $\frac{x-x'}{x}$ is large relative to $\frac{y-y'}{x}$ and $h$, local coordinates are
\[
  x,\quad y,\quad \frac{x-x'}{x},\quad \hat Y=\frac{y-y'}{x-x'},\quad \frac{h x}{x-x'}=X^{-1},
\]
with the latter defining the lift of the $h=0$ face. Thus,
\[
  \la (X,Y)\ra^{-1}=(1+X^2+Y^2)^{-1/2}=X^{-1}(1+X^{-2}+\hat Y^2)^{-1/2}
\]
shows that $X^{-1}$ is equivalent to $\la (X,Y)\ra^{-1}$ in this region. Instead of~\eqref{Eq0hCoord}, one can equivalently use
\[
  x',\quad y',\quad \frac{x-x'}{x'h},\quad \frac{y-y'}{x'h},\quad h;
\]
infinite order vanishing at the lift of $h^{-1}(0)$ now corresponds to Schwartz decay in the third and fourth variables.

Semiclassical 0-pseudodifferential operators $\Psi_{0,\semi}^{s,\alpha,k}=h^{-k}\Psi_{0,\semi}^{s,\alpha,0}=h^{-k}\Psi_{0,\semi}^{s,\alpha}$, with the orders being the differential, decay and semiclassical orders, are defined on the semiclassical 0-double space with Schwartz kernels demanded to be conormal to the lifted diagonal (corresponding to the differential order $s$), and have rapid decay at all faces but the 0-face (corresponding to the decay order $\alpha$) and the semiclassical face (corresponding to the semiclassical order $k$) where they are conormal. In fact, regularity in $h$ is not necessary here, and will not be imposed in what follows. However, our arguments below only need the full control of the differential order and semiclassical behavior, i.e.\ we work modulo $\Psi_{0,\semi}^{-\infty,\alpha,-\infty}$. Due to this, we may impose that the Schwartz kernels are compactly supported in a neighborhood of the semiclassical front face in the sense that the support does not intersect any of the boundary hypersurfaces except the necessary ones, i.e.\ the 0-face, the semiclassical face and the lift of $h=0$, rapidly vanishing at the latter. Correspondingly, elements of $\Psi_{0,\semi}^{-\infty,\alpha,k}$ {\em with the stated support condition} are of the form
\[
  (A v)(x,y)=h^{-n}\int K(x,y,X,Y,h) v(x',y')\frac{\dd x'\,\dd y'}{(x')^n}
\]
where $K$ is a conormal function (i.e.\ a symbol) of order $\alpha,k$, with rapid decay in $(X,Y)$, i.e.\ it satisfies estimates
\begin{equation}
\label{eq:kernel-form}
  |(x D_x)^\eps D_y^\beta D_X^\gamma D_Y^\delta K(x,y,X,Y,h)|\leq C_{\eps\beta\gamma\delta N} h^{-k}x^{-\alpha}\la(X,Y)\ra^{-N},
\end{equation}
with support in
\[
  \Bigl|\frac{x-x'}{x}\Bigr|<c,\quad \Bigl|\frac{y-y'}{x}\Bigr|<c,
\]
where $c$ can be taken small (and in any case should be $<1/2$, so $\frac{x'}{x}$ is bounded and bounded away from $0$, i.e.\ $x,x'$ are comparable). Note that this can be rewritten using
\begin{equation}
\label{Eq0hCoord2}
  x'=x-h x X,\quad y'=y-h x Y,\quad \frac{x'}{x}=1-h X
\end{equation}
as
\begin{equation}
\label{eq:kernel-resolved-action}
  (A v)(x,y)=\int K(x,y,X,Y,h) v(x-h x X,y-h x Y)(1-h X)^{-n}\,\dd X\,\dd Y,
\end{equation}
with $K$ supported in
\begin{equation}
\label{eq:semicl-kernel-big-support}
  |X|<\frac{c}{h},\quad |Y|<\frac{c}{h}.
\end{equation}

We remark that in order to obtain arbitrary elements of $\Psi_{0,\semi}^{-\infty,\alpha,k}$ from these, we just need to add elements of $\Psi_{0,\semi}^{-\infty,\alpha,-\infty}$, which in view of the infinite order vanishing at the semiclassical front face are in fact simply $\cO(h^\infty)$ families of elements of $\Psi_0^{-\infty,\alpha}$, and hence have simple properties on semiclassical 0-Sobolev spaces.

Now, for general $s$, semiclassical 0-pseudodifferential operators $\Psi_{0,\semi}^{s,\alpha,k}$ can be considered as the sum of $\Psi_{0,\semi}^{-\infty,\alpha,k}$ and elements of $\Psi_{0,\semi}^{s,\alpha,k}$ whose support only intersects the faces that intersect the lifted diagonal, i.e.\ the semiclassical and 0-faces, and thus have compact support in $X,Y$.  Now (standard) symbols $a$ of compact support in $x,y,X,Y$ and of order $s,\alpha,k$ satisfy estimates
\begin{equation}
\label{eq:full-symbol-estimates}
  |(x D_x)^\eps D_y^\beta D_X^\gamma D_Y^\delta D_\xi^\mu D_\eta^\nu a(x,y,X,Y,h,\xi,\eta)|\leq C_{\eps\beta\gamma\delta\mu\nu} h^{-k}x^{-\alpha}\la(\xi,\eta)\ra^{s-|\mu|-|\nu|},
\end{equation}
so $s$ is the usual symbolic, $\alpha$ is the growth at the boundary, and $k$ the semiclassical order, with the sign convention that the space grows as $s,\alpha,k$ grow.  Hence, using the definition of conormal distributions, locally, the operators with Schwartz kernel supported near the diagonal acting on test functions $v$ are of the form
\[
  (A v)(x,y)=(2\pi h)^{-n}\int e^{i(X\xi+Y\eta)} a(x,y,X,Y,h,\xi,\eta) v(x',y')\,\dd\xi\,\dd\eta\,\frac{\dd x'\,\dd y'}{(x')^n},
\]
interpreted as an oscillatory integral (i.e.\ initially for symbols of sufficiently negative order, then extended by continuity in appropriate seminorms and density). The support may be taken to lie in $|X|,|Y|<C$, $C>0$, $x<x_0$, and one may even take $C,x_0$ small (though $C>0$ small is not important here, while $x_0$ small is implicit already from the region to which we are localizing in this whole discussion), where the powers of $x'$ and $h$ enter as the normalization so that a symbol of order $s,\alpha,k$ defines a pseudodifferential operator of the same order, as we shall momentarily see. Note that by~\eqref{Eq0hCoord2}, on the support of $a$, $\frac{x'}{x}$ differs from $1$ by an $\cO(h)$ term, and similarly $y$ and $y'$ differ by an $\cO(h x)$ term. The integral can be rewritten as
\begin{equation}
\begin{aligned}
\label{eq:joint-quantization-formula}
    (Av)(x,y)=(2\pi)^{-n}\int e^{i(X\xi+Y\eta)} &a(x,y,X,Y,h,\xi,\eta) \\
    &\times v(x-h x X,y-h x Y)(1-h X)^{-n}\,\dd\xi\,\dd\eta\,\dd X\,\dd Y.
\end{aligned}
\end{equation}
Note that the singular powers of $x'$ and $h$ have disappeared, so an operator of order $(s,0,0)$ is indeed a `standard' order\footnote{the shift by $-\frac14$ being the conventional shift due to the presence of the parametric variable $h$} $s-\frac14$ conormal distribution on the semiclassical 0-double space, explaining the order convention above; in general there is a weight $\alpha$ at the zero face and $k$ at the semiclassical face.

We can combine the two classes we discussed into a single quantization formula.\footnote{While our quantization formula is not completely global since we localize to the region $|X|<c/h$, $|Y|<c/h$, it could be made global by a slight twist. In order to do this, one should work with $\bar X=\frac{x-x'}{(x x')^{1/2}h}$, $\bar Y=\frac{y-y'}{(x x')^{1/2}h}$, in place of $X,Y$, in which case no cutoff $\chi$ is needed below in \eqref{eq:semicl-zero-quantization}. However, the formulae become more cumbersome, and as we do not need the global results, we stick with our choices for the simplicity of presentation.} Namely, if in \eqref{eq:full-symbol-estimates} we allow $a$ to be supported in $|X|<c/h$, $|Y|<c/h$, as in \eqref{eq:semicl-kernel-big-support} for the kernel $K$, without changing the estimate \eqref{eq:full-symbol-estimates} itself (so it is simply a uniform estimate in $X,Y$, subject to this support condition, though the precise estimate is not too important, and even polynomial growth in $(X,Y)$ could be allowed for the Schwartzness reason explained below), then the joint quantization formula \eqref{eq:joint-quantization-formula} and the kernel action \eqref{eq:kernel-resolved-action} show that the distributional kernel is in fact simply the inverse Fourier transform of $a$ in the last two variables, evaluated at $(X,Y)$. Calling the inverse Fourier variables $(\tilde X,\tilde Y)$ for a moment for clarity, given the symbolic estimates for $a$ this inverse Fourier transform is a conormal distribution to $\tilde X=0$, $\tilde Y=0$, with Schwartz behavior at infinity in $(\tilde X,\tilde Y)$, thus in $|(\tilde X,\tilde Y)|>1$, say, is $C^\infty$ and satisfies estimates of the form
\begin{equation}
\label{eq:full-symbol-estimates-IFT}
  |(x D_x)^\eps D_y^\beta D_X^\gamma D_Y^\delta D_{\tilde X}^\mu D_{\tilde Y}^\nu (\cF_{\xi,\eta}^{-1}a)(x,y,X,Y,h,\tilde X,\tilde Y)|\leq C_{\eps\beta\gamma\delta\mu\nu,N} h^{-k}x^{-\alpha}\la(\tilde X,\tilde Y)\ra^{-N}.
\end{equation}
Hence under pullback by the map
\[
  (x,y,X,Y,h)\mapsto (x,y,X,Y,h,\tilde X=X,\tilde Y=Y)
\]
away from $X=0,Y=0$, i.e.\ the lifted diagonal, the result is that the Schwartz kernel is $C^\infty$ away from the diagonal with estimates
\begin{equation}
\label{eq:full-symbol-estimates-IFT-pullback}
  |(x D_x)^\eps D_y^\beta D_X^\gamma D_Y^\delta(\cF_{\xi,\eta}^{-1}a)(x,y,X,Y,h,X,Y)|\leq C_{\eps\beta\gamma\delta,N}h^{-k}x^{-\alpha}\la(X,Y)\ra^{-N},
\end{equation}
which is exactly of the form \eqref{eq:kernel-form}. On the other hand, in $|(\tilde X,\tilde Y)|<2$, hence after pullback in $|(X,Y)|<2$, this is a conormal distribution (with $\tilde X D_{\tilde X}$, $\tilde X D_{\tilde Y}$, etc., i.e.\ vector fields tangent to $\tilde X=0$, $\tilde Y=0$, preserving regularity after the inverse Fourier transform, hence $X D_X$, $X D_Y$, etc., after the pull back), and hence exactly of the oscillatory integral form discussed above. Conversely, any kernel of the form \eqref{eq:kernel-form} can be divided up by a partition of unity to one supported in $|(X,Y)|>1$, and one in $|(X,Y)|<2$. Write the latter piece as $\chi(h(X,Y))\chi_2(X,Y)K$ (with the first factor identically $1$ on the support of $K$); then the Fourier transform of $\chi_2(X,Y)K$ in $X,Y$ gives rise to an amplitude, call it
\[
  a_2=\chi(h(X,Y))\cF_{(X,Y)\to(\xi,\eta)}(\chi_2(X,Y)K),
\]
which satisfies both the support conditions and the estimates \eqref{eq:full-symbol-estimates}.

We can mostly eliminate the $X,Y$ dependence of $a$ in the following sense: if we fix $\phi\in\CIc(\R^n)$ which is identically $1$ near $0$, we can require
\[
  a(x,y,X,Y,h,\xi,\eta)=\phi(X,Y)a_0(x,y,h,\xi,\eta)+a_N(x,y,X,Y,h,\xi,\eta).
\]
where $a_N$ has $N$ lower order in the differential order sense. Indeed, this follows by simply Taylor expanding the original $a$ in $X,Y$ around $0$, and using that $X^\eps Y^\beta$, after integration by parts, converts into $D_\xi^\eps D_\eta^\beta$. In particular, to leading order, i.e.\ modulo $S^{s-1,\alpha,k}$,
\[
  a_0(x,y,h,\xi,\eta)=a(x,y,0,0,h,\xi,\eta),
\]
where $a$ on the right hand side is the original $a$. An asymptotic summation argument then allows one to replace $a_N$ by $a_\infty$ which is order $-\infty$ in the differential order sense. One can then replace $\phi(X,Y)$ in front of $a_0$ by $\chi(h(X,Y))$, $\chi$ identically $1$ near $0$, for the above arguments show that the difference is of order $-\infty$ in the differential order sense and thus can be absorbed into $a_\infty$. Finally, the above arguments using the Fourier transform also show that $a_\infty$ gives rise to an operator of order $-\infty$ which locally on $\supp\chi(h(X,Y))$ can be written as the oscillatory integral for a symbol in $S^{-\infty,\alpha,k}$ which in fact is independent of $X,Y$. Thus, modulo $\Psi_{0,\semi}^{-\infty,\alpha,-\infty}$, semiclassical zero pseudodifferential operators are of the form
\begin{equation}
\begin{aligned}
\label{eq:semicl-zero-quantization}
  (A v)(x,y)&=(q_{L,0,\semi}(a)v)(x,y)\\
    &=(2\pi h)^{-n}\int e^{i(X\xi+Y\eta)} \chi(h(X,Y))a(x,y,h,\xi,\eta) v(x',y')\,\dd\xi\,\dd\eta\,\frac{\dd x'\,\dd y'}{(x')^n},
\end{aligned}
\end{equation}
where $\chi$ is as above, i.e.\ smooth, compactly supported and identically $1$ near $0$.

We next consider the composition of two such operators. For the second operator, we use the second systems of local coordinates mentioned previously. Since the product Schwartz kernel needs to be evaluated on $X^3$, we use $x'$ as the output and $x''$ as the input variables of the second operator, so
\begin{align*}
  B u(x',y')=(2\pi h)^{-n}\int e^{i(X'\xi'+Y'\eta')} &b(x'',y'',X',Y',h,\xi',\eta') \\
    & u(x'',y'')\,\dd\xi'\,\dd\eta'\,\frac{\dd x''\,\dd y''}{(x'')^n},
\end{align*}
where
\[
  X'=\frac{x'-x''}{h x''},\quad Y'=\frac{y'-y''}{h x''}.
\]
Thus, $\frac{x'}{x''}=1+h X'$ differs from $1$ by $\cO(c)$.

The Schwartz kernel of $AB$ then is of the form
\begin{align*}
  A B u(x,y)=(2\pi h)^{-2 n}\int &e^{i(X\xi+Y\eta+X'\xi'+Y'\eta')}\\
    & \times a(x,y,X,Y,h,\xi,\eta) b(x'',y'',X',Y',h,\xi',\eta')\\
    & \times u(x'',y'') \,\dd\xi\,\dd\eta\,\frac{\dd x'\,\dd y'}{(x')^n}\,\dd\xi'\,\dd\eta'\,\frac{\dd x''\,\dd y''}{(x'')^n};
\end{align*}
our goal is to show that this indeed is the Schwartz kernel of a pseudodifferential operator of the desired type.  First, let us express the integrals in terms of the expected coordinates
\[
  x,\quad y,\quad X''=\frac{x-x''}{h x},\quad Y''=\frac{y-y''}{h x},\quad h;
\]
note that
\[
  x''=x-h x X'',\quad y''=y-h x Y'',\quad \frac{x''}{x}=1-h X''.
\]
We will keep the variables of integration $X,Y,\xi,\eta,\xi',\eta'$, and express $X',Y'$ in terms of $X,Y,X'',Y''$. Namely,
\[
  \frac{x'}{x''}=\frac{x'}{x}\cdot\frac{x}{x''}=\frac{1-h X}{1-h X''},
\]
so
\[
  X'=h^{-1}\Big(\frac{x'}{x''}-1\Big)=\frac{X''-X}{1-h X''},\qquad
  Y'=\frac{x}{x''}(Y''-Y)=\frac{Y''-Y}{1-h X''}.
\]
Moreover, the density $\frac{\dd x'\,\dd y'}{(x')^n}$ becomes $(1-h X)^{-n}h^n\,\dd X\,\dd Y$ as above.  Hence, our integral takes the form (suppressing the input test function, and writing out the Schwartz kernel)
\begin{align*}
  K_{AB}(x,y,X'',Y'') & \\
    =(2\pi)^{-2n} h^{-n}\int
    &\, \exp\Bigl[i\Bigl(X\xi+Y\eta+\frac{X''-X}{1-h X''}\xi'+\frac{Y''-Y}{1-h X''}\eta'\Bigr)\Bigr] \\
    & \times a(x,y,X,Y,h,\xi,\eta)\\
    & \times b\Bigl(x-h x X'', y-h x Y'', \frac{X''-X}{1-h X''}, \frac{Y''-Y}{1-h X''},h,\xi',\eta'\Bigr) \\
    & \times (1-h X)^{-n}\,\dd\xi\,\dd\eta\,\dd X\,\dd Y\,\dd\xi'\,\dd\eta'.
\end{align*}
Note that on the support of $a b$, $\frac{x'}{x}$, $\frac{x'}{x''}$ are both $1+\cO(c)$, and hence their ratio $\frac{x}{x''}$ is also $1+\cO(c)$, so $h X''$ is indeed $\cO(c)$, and thus $hY''$ is also $\cO(c)$.  This is {\em almost} a slightly non-standard parameterization of a conormal distribution (in the region where $X''$, $Y''$ are bounded), except that $a b$ is not a symbol (though it has the desired support properties), rather a product type symbol. This is, however, easily remedied by noting that the phase is non-stationary where the symbol type behavior fails, i.e.\ where $(\xi',\eta')$ is small relative to $(\xi,\eta)$, or vice versa. To see this, note that the derivative of the phase in $(X,Y)$ is
\[
  \Bigl(\xi-\frac{\xi'}{1-h X''}\Bigr)\,\dd X+\Bigl(\eta-\frac{\eta'}{1-h X''}\Bigr)\,\dd Y,
\]
which, if $h X''$ is small (to which region we can focus as already noted), is non-zero under either of these two scenarios. Thus, integration by parts in $(X,Y)$ allows us to rewrite the integral, when the integrand is localized to this a priori troublesome region, as one with an amplitude that is in fact rapidly decaying, and thus for which one easily sees that the composite operator behaves as desired. Thus, modulo such smoothing operators, the integrand can be assumed to be localized away from this region, hence $ab$ can be regarded as a symbol, and then the integral is a usual parameterization of a Lagrangian distribution, albeit in the slightly broader interpretation that we are localizing to $h X'',hY''$ being $\cO(c)$, and not to the standard interpretation where $X'',Y''$ are considered bounded.

One can then rewrite this oscillatory integral in the usual (though again non-standard as we are working in a larger region as above) conormal parameterization formulation by using the standard stationary phase lemma (rescaling the symbolic variables to formally get into its compactly supported form, which can be done as we may assume based on the above discussion that $|(\xi,\eta)|\sim|(\xi',\eta')|$), noting that the critical points of the phase in $(X,Y,\xi',\eta')$ are $\xi=\frac{\xi'}{1-h X''}$, $\eta=\frac{\eta'}{1-h X''}$, $X''=X$, $Y''=Y$, while the determinant of the Jacobian is $(1-h X'')^{-2n}$. This shows that
\[
  K_{A B}(x,y,X'',Y'')=(2\pi h)^{-n}\int e^{i(X''\xi+Y''\eta)}c(x,y,X'',Y'',h,\xi,\eta)\,\dd\xi\,\dd\eta,
\]
where $c$ is a symbol given by
\[
  a(x,y,X'',Y'',h,\xi,\eta)b(x, y, 0, 0,h,\xi,\eta)
\]
modulo terms with extra $h(X'',Y'')$ vanishing, i.e.\ keeping in mind that extra $(X'',Y'')$ can be used to lower the symbolic order, at the operator level modulo $\Psi_{0,\semi}^{s-1,\alpha,k-1}$.

Recall also that $a$ can be arranged to be independent of $X'',Y''$ up to an overall compactly supported factor $\chi(h(X'',Y''))$, and hence if $A=q_{L,0,\semi}(a)$, $B=q_{L,0,\semi}(b)$ then $AB=q_{L,0,\semi}(c)$, where $c$ is given by
\[
  a(x,y,h,\xi,\eta)b(x, y,h,\xi,\eta)
\]
modulo $S^{s-1,\alpha,k-1}$, and thus at the operator level modulo $\Psi_{0,\semi}^{s-1,\alpha,k-1}$.

\section{Geometric and analytic setting}
\label{SO}

The goal of this section is to define the class of wave type operators which we shall study in this paper (see Definition~\ref{DefOp}). We discuss the underlying geometry in~\S\ref{SsOGeo} before turning to the differential operators of interest in~\S\ref{SsOp}.

\subsection{Admissible metrics}
\label{SsOGeo}

Before describing the general setup in~\S\ref{SssOGeo2}, we consider as a model case the Minkowski metric.

\subsubsection{The Minkowski metric near null infinity.}
\label{SssOGeoMink}

In polar coordinates, the Minkowski metric on $\R^{1+n}$ takes the form
\begin{equation}
\label{EqOMink}
  g_0 = -\dd t^2+\dd r^2 + r^2\slg,
\end{equation}
where $t\in\R$, $r>0$, and $\slg$ is the standard metric on $\Sph^{n-1}$. Since this can equivalently be written as $g_0=-\dd t^2+\sum_{j=1}^n(\dd x^j)^2$ where $(t,x^1,\ldots,x^n)$ are standard coordinates on $\R^{1+n}$, we find that $g_0$ is a Lorentzian scattering metric, by which we mean
\[
  g_0 \in \CI(\ol{\R^{n+1}};S^2\,\Tsc^*\ol{\R^{n+1}}),\qquad
  g_0^{-1}\in\CI(\ol{\R^{n+1}};S^2\,\Tsc\ol{\R^{n+1}}).
\]

We are interested in the structure of $g_0$ near the large end of the future light cone; thus, we introduce
\begin{equation}
\label{EqOGeoMinkRhoV}
  \varrho := \frac{1}{t+r},\qquad v:=\frac{t-r}{t+r}.
\end{equation}
Inserting these into~\eqref{EqOMink} gives
\begin{equation}
\label{EqOMinkBdy}
  g_0 = -v\frac{\dd\varrho^2}{\varrho^4} + \frac{\dd\varrho}{\varrho^2}\otimes_s \frac{\dd v}{\varrho} + \Bigl(\frac{1-v}{2\varrho}\Bigr)^2\slg.
\end{equation}
The closure of every future light cone in $\R^{n+1}$ intersects $\pa\ol{\R^{n+1}}$ in the $(n-1)$-dimensional set
\[
  Y:=\{\varrho=v=0\}\subset\ol{\R^{n+1}}.
\]
Thus, to resolve different light cones also at infinity, we blow up $Y$ and obtain the manifold with corners
\[
  \tilde M = [\ol{\R^{n+1}}; Y].
\]
This is defined as $(\ol{\R^{n+1}}\setminus Y)\sqcup S N^+Y$, with $S N^+Y=N^+Y/\R_+$ denoting the (non-strictly) inward pointing spherical normal bundle, and equipped with the minimal smooth structure in which polar coordinates around $Y$ are smooth. We denote the front face (i.e.\ $S N^+Y$) of $\tilde M$ by $\tilde\scri^+$, the `north cap'---the closure of $\{\varrho=0,v>0\}$---by $\tilde I^+$, and the closure of $\{\varrho=0,v<0\}$ by $\tilde I^0$. (The latter terminology is imprecise and should only be taken seriously in $t>-r$, as we are not resolving past null infinity here. See also \cite[\S3]{BaskinVasyWunschRadMink} and \cite[\S2]{HintzVasyMink4}.) The restriction of the blow-down map $\tilde M\to\ol{\R^{n+1}}$ (the identity on $\ol{\R^{n+1}}\setminus Y$, and the base projection on $S N^+Y$) restricts to a fibration $\tilde\scri^+\to Y$, with fibers diffeomorphic to a closed interval. See Figure~\ref{FigOGeo}.

Near the interior of $\tilde\scri^+$, local coordinates are given by
\[
  \varrho=\frac{1}{t+r},\qquad t_*:=\frac{v}{\varrho}=t-r,\qquad\omega\in\Sph^{n-1},
\]
in which the metric takes the double null form
\[
  g_0 = \frac{\dd\varrho}{\varrho^2}\otimes_s\dd t_* + \Bigl(\frac14+\cO(\varrho)\Bigr)\varrho^{-2}\slg.
\]
Freezing coefficients at a single fiber of $\tilde\scri^+$, let us consider the model metric
\[
  \frac{\dd\varrho}{\varrho^2}\otimes_s\dd t_* + \frac{\dd y^2}{\varrho^2}
\]
where $y\in\R^{n-1}$. This is homogeneous of degree $-2$ under pullback by the parabolic scaling $(0,\infty)\ni\lambda\colon(\varrho,y)\mapsto(\lambda^2\varrho,\lambda y)$ (see also Remark~\ref{RmkIParabolic}); it is thus natural to change the homogeneity of the metric by multiplying it by $\varrho$, and subsequently introducing $x_{\!\scri}:=\varrho^{1/2}$. This produces the metric $2\frac{\dd x_{\!\scri}}{x_{\!\scri}}\otimes_s\dd t_*+\frac{\dd y^2}{x_{\!\scri}^2}$, which is an edge metric on $\R_{t_*}\times[0,1)_{x_{\!\scri}}\times\R_y^{n-1}$. To capture this invariantly, one defines the manifold
\begin{equation}
\label{EqOMinkM}
  M = [\tilde M;\tilde\scri^+,\half]
\end{equation}
as the \emph{square root blow-up} of $\tilde M$ at $\tilde\scri^+$, defined to be $\tilde M$ as a set, but with the square root of a defining function of $\tilde\scri^+$ adjoined to the smooth structure; denoting the lifts of $\tilde I^0,\tilde\scri^+,\tilde I^+$ by $I^0,\scri^+,I^+$, null infinity $\scri^+$ is the total space of a fibration with base $\Sph^{n-1}$ and closed intervals as fibers as before, and $\varrho g_0$ is a nondegenerate Lorentzian edge metric on $M\setminus(I^0\cup I^+)$.

\begin{figure}[!ht]
\centering
\includegraphics{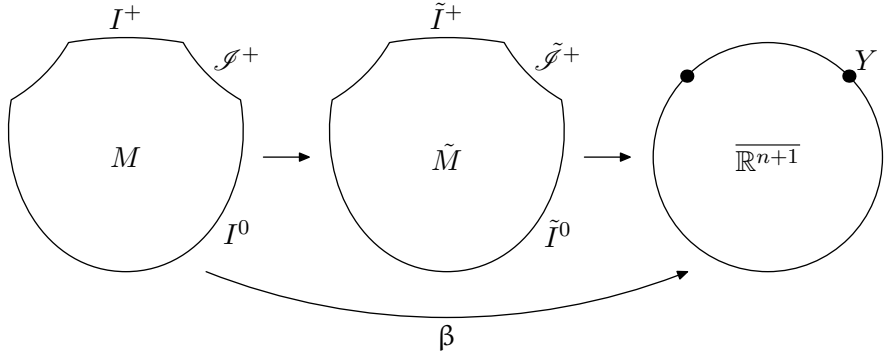}
\caption{The resolution $\tilde M$ of the radial compactification $\ol{\R^{n+1}}$ at the light cone $Y$ at future infinity, and the boundary hypersurfaces of $\tilde M$. The manifold $M$ is the square root blow-up of $\tilde M$ at $\tilde\scri^+$, and $\upbeta$ is the total blow-down map $M\to\ol{\R^{n+1}}$.}
\label{FigOGeo}
\end{figure}

Away from $\scri^+\subset M$, i.e.\ away from $Y\subset\pa\ol{\R^{n+1}}$ on the other hand, the degree $-2$ homogeneity of $g_0$ in~\eqref{EqOMinkBdy} under scaling in $\varrho$ suggests working with $\varrho^2 g_0$, which is a nondegenerate Lorentzian b-metric on $\ol{\R^{n+1}}\setminus Y$.\footnote{While the nondegeneracy and homogeneity are valid on all of $\ol{\R^{n+1}}$, we ignore this fact here, as we are ultimately interested in metrics and operators which have smooth or conormal coefficients only on $M$, but not on $\ol{\R^{n+1}}$---working on the resolved space $M$ with its edge-b-structure will be the appropriate framework for this task.} Globally on $M$, the rescaling $\varrho^2 x_{\!\scri}^{-2} g_0$, where $x_{\!\scri}\in\CI(M)$ is a defining function of $\tilde\scri^+$, is a smooth nondegenerate Lorentzian edge-b-metric on $M$, as we proceed to demonstrate.

To study $g_0$ near the transition regions $I^0\cap\scri^+$ and $\scri^+\cap I^+$, we first pass to the coordinates $t_*=t-r$, $r$, $\omega\in\Sph^{n-1}$, in which the Minkowski metric and its dual metric take the form
\[
  g_0 = -\dd t_*^2 - 2\dd t_*\,\dd r + r^2\slg,\qquad
  g_0^{-1} = -2\pa_{t_*}\otimes_s\pa_r + \pa_r^2 + r^{-2}\slg^{-1}.
\]
We then introduce:

\begin{definition}[Open sets containing parts of $\scri^+$; local coordinates]
\label{DefOGeoU0p}
  For $T\in\R$, we define
  \[
    \cU_0(T):=\Bigl\{r>1,\ \ t_*<T-1\Bigr\}, \qquad
    \cU_+(T):=\Bigl\{r>1,\ \ t_*>T+1\Bigr\}\subset M.
  \]
  On $\cU_0(T)$, we define
  \begin{equation}
  \label{EqOGeoCoordI0}
    \rho_0 = \frac{1}{T-t_*},\qquad
    x_{\!\scri} = \sqrt{\frac{T-t_*}{r}},\qquad
    \rho_+ = 1,
  \end{equation}
  while on $\cU_+(T)$ we put
  \begin{equation}
  \label{EqOGeoCoordIp}
    \rho_0 = 1,\qquad
    x_{\!\scri} = \sqrt{\frac{t_*-T}{r}},\qquad
    \rho_+ = \frac{1}{t_*-T}\,.
  \end{equation}
\end{definition}

Thus, $[0,1)_{\rho_0}\times[0,1)_{x_{\!\scri}}\times\Sph^{n-1}\subset\cU_0(T)$ and $[0,1)_{x_{\!\scri}}\times[0,1)_{\rho_+}\times\Sph^{n-1}\subset\cU_+(T)$ are two coordinate charts on $M$. The union $\cU_0(T)\cup\cU_+(T')$ contains a neighborhood of $\scri^+$ provided $T'+1<T-1$. The functions $\rho_0,x_{\!\scri}$, and $\rho_+$ defined in the two charts do not agree on the overlap of the two charts (and which definition we use will always be made explicit); they are local defining functions of $I^0$, $\scri^+$, and $I^+$, respectively. Note also that $\varrho$ is a smooth positive multiple of $\rho_0 x_{\!\scri}^2\rho_+=r^{-1}$ on either coordinate chart. Now, on $\cU_0(T)$, we have
\begin{subequations}
\begin{equation}
\label{EqOGeoI0VF}
  \pa_{t_*} = \rho_0\Bigl(\rho_0\pa_{\rho_0}-\frac12 x_{\!\scri}\pa_{x_{\!\scri}}\Bigr),\qquad
  \pa_r = -\frac12\rho_0 x_{\!\scri}^3\pa_{x_{\!\scri}},
\end{equation}
and therefore
\begin{equation}
\label{EqOGeoMetI0g0}
  \rho_0^{-2}x_{\!\scri}^{-2}g_0^{-1} \equiv -\frac12 x_{\!\scri}\pa_{x_{\!\scri}}\otimes_s(x_{\!\scri}\pa_{x_{\!\scri}}-2\rho_0\pa_{\rho_0}) + x_{\!\scri}^2\slg^{-1} \bmod x_{\!\scri}\CI(\cU_0(T);S^2\,\Teb M),
\end{equation}
\end{subequations}
with $\Teb M$ spanned by $x_{\!\scri}\pa_{x_{\!\scri}}$, $\rho_0\pa_{\rho_0}$, $x_{\!\scri}\,T\Sph^{n-1}$. On $\cU_+(T)$, we similarly compute
\begin{subequations}
\begin{equation}
\label{EqOGeoIpVF}
  \pa_{t_*} = -\rho_+\Bigl(\rho_+\pa_{\rho_+}-\frac12 x_{\!\scri}\pa_{x_{\!\scri}}\Bigr),\qquad
  \pa_r = -\frac12\rho_+ x_{\!\scri}^3\pa_{x_{\!\scri}},
\end{equation}
and therefore
\begin{equation}
\label{EqOGeoMetIpg0}
\begin{split}
  x_{\!\scri}^{-2}\rho_+^{-2}g_0^{-1} &\equiv \frac12 x_{\!\scri}\pa_{x_{\!\scri}}\otimes_s(x_{\!\scri}\pa_{x_{\!\scri}}-2\rho_+\pa_{\rho_+}) + x_{\!\scri}^2\slg^{-1} \bmod x_{\!\scri}\CI(\cU_+(T);S^2\,\Teb M),
\end{split}
\end{equation}
\end{subequations}
with $\Teb M$ now spanned by $x_{\!\scri}\pa_{x_{\!\scri}}$, $\rho_+\pa_{\rho_+}$, $x_{\!\scri}\,T\Sph^{n-1}$.

\subsubsection{General setup}
\label{SssOGeo2}

We proceed to generalize the above example of the Minkowski metric. Thus, suppose $M_0$ is an $(n+1)$-dimensional manifold with boundary $\pa M_0\neq\emptyset$, and suppose $Y\subset\pa M_0$ is a compact and boundaryless embedded $1$-codimensional submanifold whose normal bundle (inside $\pa M_0$) is orientable; thus, there exists a collar neighborhood $[0,\eps)_\varrho\times(-\eps,\eps)_v\times Y$ of $Y$ in $M_0$. Recalling~\eqref{EqOMinkBdy}, we make the following definition. We define $\tilde M=[M_0;Y]$ and $M=[\tilde M;\tilde\scri^+,\half]$ (with $\tilde\scri^+$ the lift of $Y$) as in~\eqref{EqOMinkM}. We call $I^0$ (spacelike infinity), resp.\ $I^+\subset M$ (future timelike infinity) the boundary hypersurface of $M$ where $v<0$, resp.\ $v>0$ (and we assume that $I^0\neq I^+$), and $\scri^+\subset M$ (null infinity) denotes the lift of $\tilde\scri^+$.

\begin{definition}[Admissible metrics]
\label{DefOGeoAdm}
  Let $\ell_0,\ell_+\in(0,1]$ and $\ell_{\!\scri}\in(0,\half]$. A Lorentzian metric $g\in\CI(M_0^\circ;S^2 T^*M_0^\circ)$ is called an \emph{$(\ell_0,2\ell_{\!\scri},\ell_+)$-admissible metric} (or simply \emph{admissible metric}) if
  \begin{equation}
  \label{EqOGeoAdmDef}
  \begin{split}
    &\rho_0^2 x_{\!\scri}^2\rho_+^2\biggl[g - \Bigl(-v\frac{\dd\varrho^2}{\varrho^4}+\frac{\dd\varrho}{\varrho^2}\otimes_s\frac{\dd v}{\varrho}+\frac{k(y,\dd y)}{4\varrho^2}\Bigr)\biggr] \\
    &\hspace{10em} \in (x_{\!\scri}\CI+\cA^{(\ell_0,2\ell_{\!\scri},\ell_+)})(M;S^2\,\Teb^*M),
  \end{split}
  \end{equation}
  where $k$ is a Riemannian metric on $Y$, and $\rho_0,x_{\!\scri}$, and $\rho_+$ are defining functions of $I^0$, $\scri^+$, and $I^+$, respectively.
\end{definition}

\begin{rmk}[Weaker error term: geometry]
\label{RmkOGeoAdmWeaker}
  The remainder of this section as well as the analysis of the null-bicharacteristic flow of admissible metrics in the eb-phase space in~\S\ref{SsMF} go through with only notational modifications if we relax~\eqref{EqOGeoAdmDef} to membership in the space $\cA^{\bigish(((0,0),\ell_0),\,2\ell_{\!\scri},\,((0,0),\ell_+)\bigish)}(M;S^2\,\Teb^*M)$. Admissible metrics arising in applications (e.g.\ in nonlinear stability problems) typically have the stronger form~\eqref{EqOGeoAdmDef}, which in particular entails the smoothness (rather than mere conormality), as an eb-metric, of the restriction of $\rho_0^2 x_{\!\scri}^2\rho_+^2 g$ to $I^0$, resp.\ $I^+$ down to $I^0\cap\scri^+$, resp.\ $I^+\cap\scri^+$. See also Remark~\ref{RmkOpWeak}.
\end{rmk}

The inclusion of the normalization factor $2$ in the weight at $\scri^+$ is merely a matter of convention, motivated by the square root blow-up (so $x_{\!\scri}^{2\ell_{\!\scri}}=\rho_{\!\scri}^{\ell_{\!\scri}}$ where $\rho_{\!\scri}=x_{\!\scri}^2$ is a defining function of $\tilde\scri^+\subset\tilde M$). The factor $\frac14$ is inserted so as to make $k=\slg$ in the case of the Minkowski metric (cf.\ \eqref{EqOMinkBdy}). Since the leading order term in~\eqref{EqOGeoAdmDef} is the Minkowski metric from~\eqref{EqOMinkBdy} (possibly with a different metric on $Y\cong\Sph^{n-1}$), we see from~\eqref{EqOGeoMetI0g0} and \eqref{EqOGeoMetIpg0} that Definition~\ref{DefOGeoAdm} can be stated in a number of equivalent ways.
\begin{enumerate}
\item (Dual metric.) The dual metric $g^{-1}$ is of the form
  \begin{equation}
  \label{EqOGeoAdmDual}
  \begin{split}
    &\rho_0^{-2}x_{\!\scri}^{-2}\rho_+^{-2}\Bigl[g^{-1}-\varrho^2\Bigl(4(\varrho\pa_\varrho+v\pa_v)\otimes_s\pa_v+4 k^{-1}(y,\pa_y)\Bigr)\Bigr] \\
    &\hspace{10em} \in (x_{\!\scri}\CI+\cA^{(\ell_0,2\ell_{\!\scri},\ell_+)})(M;S^2\,\Teb M).
  \end{split}
  \end{equation}
\item (Near $I^0\cap\scri^+$, resp.\ $\scri^+\cap I^+$.) On $\cU_0(T)$, resp.\ $\cU_+(T)$ (see Definition~\ref{DefOGeoU0p}) and in the coordinates\footnote{Here, $\rho_0,x_{\!\scri}$, resp.\ $x_{\!\scri},\rho_+$ are defined using $t_*=t-r$, with $t$ and $r$ in turn determined by $\varrho$ and $v$ via~\eqref{EqOGeoMinkRhoV}; that is, $t=\frac{1+v}{2\varrho}$ and $r=\frac{1-v}{2\varrho}$.}~\eqref{EqOGeoCoordI0}, resp.\ \eqref{EqOGeoCoordIp},\footnote{In terms of the metric $g$, this is equivalent to $\rho_0^2 x_{\!\scri}^2 g \equiv 2\frac{\dd\rho_0}{\rho_0}\otimes_s\bigl(\frac{\dd\rho_0}{\rho_0}+2\frac{\dd x_{\!\scri}}{x_{\!\scri}}\bigr) + k\bigl(y,\frac{\dd y}{x_{\!\scri}}\bigr)$, resp.\ $x_{\!\scri}^2\rho_+^2 g \equiv -2\frac{\dd\rho_+}{\rho_+}\otimes_s\bigl(\frac{\dd\rho_+}{\rho_+}+2\frac{\dd x_{\!\scri}}{x_{\!\scri}}\bigr) + k\bigl(y,\frac{\dd y}{x_{\!\scri}}\bigr)$ modulo sections of $S^2\,\Teb^*M$ over $\cU_0(T)$, resp.\ $\cU_+(T)$ of class $x_{\!\scri}\CI+\cA^{(\ell_0,2\ell_{\!\scri},\ell_+)}$.}
  \begin{align}
  \label{EqOGeoMetI0}
  \begin{split}
    &\rho_0^{-2}x_{\!\scri}^{-2}g^{-1} \equiv -\frac12 x_{\!\scri}\pa_{x_{\!\scri}}\otimes_s(x_{\!\scri}\pa_{x_{\!\scri}}-2\rho_0\pa_{\rho_0}) + k^{-1}(y,x_{\!\scri}\pa_y) \\
    &\hspace{10em} \bmod (x_{\!\scri}\CI+\cA^{(\ell_0,2\ell_{\!\scri})})(\cU_0(T);S^2\,\Teb M),
  \end{split} \\
  \label{EqOGeoMetIp}
  \begin{split}
    &x_{\!\scri}^{-2}\rho_+^{-2}g^{-1} \equiv \frac12 x_{\!\scri}\pa_{x_{\!\scri}}\otimes_s(x_{\!\scri}\pa_{x_{\!\scri}}-2\rho_+\pa_{\rho_+}) + k^{-1}(y,x_{\!\scri}\pa_y) \\
    &\hspace{10em} \bmod (x_{\!\scri}\CI+\cA^{(2\ell_{\!\scri},\ell_+)})(\cU_+(T);S^2\,\Teb M).
  \end{split}
  \end{align}
  (We drop the weight at $I^+$, resp.\ $I^0$ from the notation, since $\cU_0(T)\cap I^+=\emptyset$, resp.\ $\cU_+(T)\cap I^0=\emptyset$.)
\end{enumerate}

\begin{example}[Double null formulation]
\label{ExOGeoAsy}
  Examples of admissible metrics which arise in the context of the nonlinear stability of asymptotically flat spacetimes as solutions of the Einstein vacuum equations are given in \cite[\S\S3.2--3.3]{HintzMink4Gauge}. Concretely, writing the Minkowski metric~\eqref{EqOMink} (which is $(1,1,1)$-admissible) in double null form
  \[
    g_0 = -\dd x^0\,\dd x^1 + r^2\slg,\qquad x^0=t+r=\frac{1}{\varrho},\quad x^1=t-r=\frac{v}{\varrho},
  \]
  we compute, in terms of $x_{\!\scri}=\sqrt{v}$, $\rho_+=\frac{\varrho}{v}$ (thus $x^0=\frac{1}{\rho_+ x_{\!\scri}^2}$, $x^1=\rho_+^{-1}$) on $\cU_+(0)$ (which are smooth positive multiples of the coordinates $x_{\!\scri},\rho_+$ used in Definition~\ref{DefOGeoU0p}),
  \begin{alignat*}{2}
    \dd x^0 &= -\frac{1}{\rho_+ x_{\!\scri}^2}\Bigl(\frac{\dd\rho_+}{\rho_+} + 2\frac{\dd x_{\!\scri}}{x_{\!\scri}}\Bigr) &&\in \rho_+^{-1}x_{\!\scri}^{-2}\CI(M;\Teb^*M), \\
    \dd x^1 &= -\rho_+^{-1}\frac{\dd\rho_+}{\rho_+} &&\in \rho_+^{-1}\CI(M;\Teb^*M), \\
    r\,\dd\omega &= x_{\!\scri} r\frac{\slomega}{x_{\!\scri}} &&\in \rho_+^{-1}x_{\!\scri}^{-1}\CI(M;\Teb^*M)
  \end{alignat*}
  for $\dd\omega\in\CI(\Sph^{n-1};T^*\Sph^{n-1})$. (We omit the analogous computation on $\cU_0(2)$.) Thus, a metric $g$ is $(\ell_0,2\ell_{\!\scri},\ell_+)$-admissible if (and only if, provided one fixes the Riemannian metric $k$ on $Y$ to be $k=\slg$) it is the sum of $g_0$ and linear combinations of
  \begin{equation}
  \label{EqOGeoAsy}
  \begin{alignedat}{3}
    &x_{\!\scri}^2(\dd x^0)^2, &\qquad& \dd x^0\otimes_s\dd x^1, &\qquad& x_{\!\scri}\dd x^0\otimes_s r\,\dd\omega, \\
    &x_{\!\scri}^{-2}(\dd x^1)^2, &\qquad& x_{\!\scri}^{-1}\dd x^1\otimes_s r\,\dd\omega, &\qquad& r\,\dd\omega\otimes_s r\,\dd\omega,
  \end{alignedat}
  \end{equation}
  where $\dd\omega$ may change between any two occurrences, with coefficients in $x_{\!\scri}\CI+\cA^{(\ell_0,2\ell_{\!\scri},\ell_+)}$. The class of metrics $g$ in particular includes the Schwarzschild metric $-(1-\frac{2\bhm}{r})\dd x^0\,\dd x^1+r^2\slg$ (and its perturbations of this class).
\end{example}

We fix a time orientation on $(M,g)$ near $\scri^+$ by declaring, at $\scri^+$, the causal (for $g_\ebop:=\rho_0^2 x_{\!\scri}^2\rho_+^2 g$) edge-b tangent vectors
\begin{equation}
\label{EqOGeoTimelike}
\begin{alignedat}{3}
  &\rho_0\pa_{\rho_0}-\frac12 x_{\!\scri}\pa_{x_{\!\scri}},&\quad& {-}x_{\!\scri}\pa_{x_{\!\scri}} &\quad&\text{(in the coordinates~\eqref{EqOGeoCoordI0})}, \\
  &{-}\rho_+\pa_{\rho_+}+\frac12 x_{\!\scri}\pa_{x_{\!\scri}},&\quad& {-}x_{\!\scri}\pa_{x_{\!\scri}} &\quad&\text{(in the coordinates~\eqref{EqOGeoCoordIp})}
\end{alignedat}
\end{equation}
to be future causal. Using~\eqref{EqOGeoI0VF} and \eqref{EqOGeoIpVF}, this matches the standard time orientation in the Minkowski case, where $\pa_{t_*}$ and $\pa_r$ (i.e.\ in $t,r$ coordinates $\pa_t$ and $\pa_t+\pa_r$) are future timelike and null, respectively.

\begin{lemma}[Hypersurfaces in $(M,g)$]
\label{LemmaOGeoHyp}
  Let $g$ be an $(\ell_0,2\ell_{\!\scri},\ell_+)$-admissible metric on $M$.
  \begin{enumerate}
  \item\label{ItOGeoHypI0} Let $T\in\R$, and define $\rho_0,x_{\!\scri}$ on $\cU_0(T)$ by~\eqref{EqOGeoCoordI0}. Let $\bar\rho_0\in(0,1)$. Then there exist $\delta_0>0$ and $C>0$ so that for all $\delta\in(0,\delta_0]$, the hypersurface $\{x_{\!\scri}=\delta,\ \rho_0<\bar\rho_0\}$ is spacelike and the hypersurface $\{x_{\!\scri}<2\delta,\ \rho_0=\rho_{0,0}+C x_{\!\scri}^{2\ell_{\!\scri}}\}$ is spacelike for all $\rho_{0,0}\in(0,1-2 C(2\delta)^{2\ell_{\!\scri}})$.
  \item\label{ItOGeoHypIp} Let $T\in\R$, and define $x_{\!\scri},\rho_+$ on $\cU_+(T)$ by~\eqref{EqOGeoCoordIp}. Then there exist $\delta_0>0$, $\bar\rho_+\in(0,1)$, and $C>0$ so that for all $\delta\in(0,\delta_0]$, the hypersurface $\{x_{\!\scri}=\delta,\ \rho_+<\bar\rho_+\}$ is timelike and the hypersurface $\{x_{\!\scri}<2\delta,\ \rho_+=\rho_{+,0}-C x_{\!\scri}^{2\ell_{\!\scri}}\}$ is spacelike for all $\rho_{+,0}<1$.
  \end{enumerate}
\end{lemma}

In the last statement of part~\eqref{ItOGeoHypI0}, the range of $\rho_{0,0}$ is chosen so that the stated hypersurface is contained in the local coordinate chart $\cU_0(T)$. See Figure~\ref{FigOGeoHyp}.

\begin{figure}[!ht]
\centering
\includegraphics{FigOGeoHyp}
\caption{Illustration of Lemma~\ref{LemmaOGeoHyp}\eqref{ItOGeoHypI0}. \textit{On the left:} the perspective of $M$. \textit{On the right:} the perspective of $(1+1)$-dimensional Minkowski space with coordinates $(t,r)$.}
\label{FigOGeoHyp}
\end{figure}

\begin{figure}[!ht]
\centering
\includegraphics{FigOGeoHyp2}
\caption{Illustration of Lemma~\ref{LemmaOGeoHyp}\eqref{ItOGeoHypIp}. \textit{On the left:} the perspective of $M$. \textit{On the right:} the perspective of $(1+1)$-dimensional Minkowski space with coordinates $(t,r)$.}
\label{FigOGeoHyp2}
\end{figure}

\begin{proof}[Proof of Lemma~\usref{LemmaOGeoHyp}]
  We only prove part~\eqref{ItOGeoHypI0}; the proof of part~\eqref{ItOGeoHypIp} is analogous. Let $g_\ebop^{-1}:=\rho_0^{-2}x_{\!\scri}^{-2}g^{-1}$. Using the expression~\eqref{EqOGeoMetI0}, we compute
  \[
    g_\ebop^{-1}\Bigl(\frac{\dd x_{\!\scri}}{x_{\!\scri}},\frac{\dd x_{\!\scri}}{x_{\!\scri}}\Bigr) \equiv -\frac12 \bmod x_{\!\scri}\CI+\rho_0^{\ell_0}x_{\!\scri}^{2\ell_{\!\scri}}\cA^0,
  \]
  which is thus negative for $0\leq\rho_0<\bar\rho_0$ when $0<x_{\!\scri}<\delta_0$ for sufficiently small $\delta_0$ (depending on $\bar\rho_0$); thus, $\dd x_{\!\scri}$ is timelike there.

  We next compute the squared norm of $\dd(\rho_0-C x_{\!\scri}^{2\ell_{\!\scri}})=\rho_0\frac{\dd\rho_0}{\rho_0}-\tilde C x_{\!\scri}^{2\ell_{\!\scri}}\frac{\dd x_{\!\scri}}{x_{\!\scri}}$, where $\tilde C=2\ell_{\!\scri} C$, with respect to $g_\ebop^{-1}$ using the expression~\eqref{EqOGeoMetI0} to be $\half$ times
  \begin{align*}
    &x_{\!\scri}^{2\ell_{\!\scri}}\Bigl(-\tilde C^2 x_{\!\scri}^{2\ell_{\!\scri}} - 2\tilde C\rho_0 + \Err\Bigr), \\
    &\qquad |\Err|\leq C(g)(x_{\!\scri}^{1-2\ell_{\!\scri}} + \rho_0^{\ell_0})\bigl(\rho_0^2+\tilde C\rho_0 x_{\!\scri}^{2\ell_{\!\scri}}+\tilde C^2 x_{\!\scri}^{4\ell_{\!\scri}}\bigr) \leq C'(g)(\rho_0^2+\tilde C^2 x_{\!\scri}^{4\ell_{\!\scri}}),
  \end{align*}
  where the constants $C(g),C'(g)$ only depend on $g$; here we use that $\ell_{\!\scri}\leq\half$. Restricting $\rho_0$ to any compact subinterval of $[0,1)$, note then that $\tilde C\rho_0$ dominates, for sufficiently large $\tilde C$, the error term $C'(g)\rho_0^2$, and $\tilde C^2 x_{\!\scri}^{2\ell_{\!\scri}}$ then dominates, for $x_{\!\scri}<\delta_0$ and sufficiently small $\delta_0>0$, the error term $C'(g)\tilde C^2 x_{\!\scri}^{4\ell_{\!\scri}}\leq(C'(g)(2\delta_0)^{2\ell_{\!\scri}})\tilde C^2 x_{\!\scri}^{2\ell_{\!\scri}}$. This proves the Lemma.
\end{proof}

\subsection{Admissible operators}
\label{SsOp}

The core of our analysis will be local near $\scri^+\subset M$. Thus, we work on the $(n+1)$-dimensional manifold $M_0=[0,\eps)_\varrho\times(-\eps,\eps)_v\times Y$, put $\tilde M=[M_0;Y]$, and set $M=[\tilde M;\tilde\scri^+,\half]$ as before, where $\tilde\scri^+\subset\tilde M$ denotes the front face; we denote the closures of $\{\varrho=0,v<0\}$, resp.\ $\{\varrho=0,v>0\}$ by $I^0$, resp.\ $I^+\subset M$, and the lift of $\tilde\scri^+$ by $\scri^+\subset M$. Denote by
\[
  \upbeta\colon M\to M_0
\]
the blow-down map, and by $\rho_0$, $x_{\!\scri}$, $\rho_+$ defining functions of $I^0,\scri^+,I^+$. We fix an $(\ell_0,2\ell_{\!\scri},\ell_+)$-admissible metric $g$ on $M$ (see Definition~\ref{DefOGeoAdm}); its dual metric is thus of the form~\eqref{EqOGeoAdmDual}, or equivalently~\eqref{EqOGeoMetI0}--\eqref{EqOGeoMetIp}.

\begin{definition}[Admissible operators]
\label{DefOp}
  Let $E\to M_0$ denote a smooth vector bundle. Then an operator $P\in\Diff^2(M^\circ;E)$ is called \emph{$g$-admissible} (or simply \emph{admissible}) if the following conditions are satisfied:
  \begin{enumerate}
  \item\label{ItOpSymb} $P$ is principally scalar, with principal symbol equal to the dual metric function $G\colon T^*M^\circ\ni\zeta\mapsto g^{-1}(\zeta,\zeta)$;
  \item\label{ItOpStruct} $P$ can be written as a sum $P=P_0+\tilde P$ where
  \begin{equation}
  \label{EqOpStruct}
  \begin{split}
    P_0&\in\rho_0^2 x_{\!\scri}^2\rho_+^2(\CI+\cA^{(\ell_0,(0,0),\ell_+)})\Diffeb^2(M;\upbeta^*E), \\
    \tilde P&\in\rho_0^2 x_{\!\scri}^2\rho_+^2\cA^{(\ell_0,2\ell_{\!\scri},\ell_+)}\Diffeb^2(M;\upbeta^*E);
  \end{split}
  \end{equation}
  \item\label{ItOpNorm} there exist
  \[
    p_0\in\rho_0\rho_+(\CI+\cA^{(\ell_0,\ell_+)})(\scri^+;\upbeta^*\End(E)),\qquad
    p_1\in(\CI+\cA^{(\ell_0,\ell_+)})(\scri^+;\upbeta^*\End(E)),
  \]
  so that in the coordinates~\eqref{EqOGeoCoordI0} (with $T=T_0$), resp.\ \eqref{EqOGeoCoordIp} (with $T=T_+<T_0$), and setting $p_0^0=\rho_0^{-1}p_0\in(\CI+\cA^{\ell_0})(\scri^+\setminus I^+;\upbeta^*\End(E))$ and $p_0^+\in(\CI+\cA^{\ell_+})(\scri^+\setminus I^0;\upbeta^*\End(E))$, the edge normal operator of $2\rho_0^{-2}x_{\!\scri}^{-2}P_0$, resp.\ $2 x_{\!\scri}^{-2}\rho_+^{-2}P_0$ at the fiber $\scri^+_{y_0}$ of $\upbeta|_{\scri^+}\colon\scri^+\to Y$ over $y_0\in Y$ is given by\footnote{The zeroth order terms are consistent: multiplying the first, resp.\ second line in~\eqref{EqOpNorm} by $\rho_0$, resp.\ $\rho_+$, this follows from the definition of $p_0^0$ and $p_0^+$ and the fact that $\rho_0^{-1}x_{\!\scri}^{-2}=r^{-1}=x_{\!\scri}^{-2}\rho_+^{-1}$ in the two coordinate systems.}
    \begin{equation}
    \label{EqOpNorm}
    \begin{split}
      -&\Bigl(x_{\!\scri} D_{x_{\!\scri}}-2 i^{-1}\Bigl(\frac{n-1}{2}+p_1\Bigr)\Bigr)(x_{\!\scri} D_{x_{\!\scri}}-2\rho_0 D_{\rho_0}) + 2 k^{i j}(y_0)(x_{\!\scri} D_{y^i})(x_{\!\scri} D_{y^j}) + p_0^0, \\
      \text{resp.}\ \  &\Bigl(x_{\!\scri} D_{x_{\!\scri}}-2 i^{-1}\Bigl(\frac{n-1}{2}+p_1\Bigr)\Bigr)(x_{\!\scri} D_{x_{\!\scri}}-2\rho_+ D_{\rho_+}) + 2 k^{i j}(y_0)(x_{\!\scri} D_{y^i})(x_{\!\scri} D_{y^j}) + p_0^+.
    \end{split}
    \end{equation}
  \end{enumerate}
\end{definition}

Note that the operators~\eqref{EqOpNorm} are invariant edge-b-operators on $[0,1)_{\rho_0}\times[0,\infty)_{x_{\!\scri}}\times\R^{n-1}_y$, resp.\ $[0,1)_{\rho_+}\times[0,\infty)_{x_{\!\scri}}\times\R^{n-1}_y$ which act on sections of the fixed vector space $E_{y_0}$. Thus, the derivatives (of sections of $\upbeta^*E|_{\scri^+_{y_0}}$) appearing in~\eqref{EqOpNorm} are well-defined. The inclusion of the term $\frac{n-1}{2}$ is a convenient normalization of $p_1$; see Example~\ref{ExOpMink} below.

By Definition~\ref{DefOGeoAdm}, we have $G\in\rho_0^2 x_{\!\scri}^2\rho_+^2(S^2(\Teb^*M)+\cA^{(\ell_0,2\ell_{\!\scri},\ell_+)}S^2(\Teb^*M))$; condition~\eqref{ItOpStruct} thus demands the corresponding membership of $P$ as an edge-b-differential operator, and allows for an additional (necessarily subprincipal) term which is of leading order at $\scri^+$ in the sense of decay. Condition~\eqref{ItOpSymb} determines the leading order part of $P$ in the differential sense, and~\eqref{ItOpNorm} determines the leading order part in the sense of decay at $\scri^+$. We do not place any restrictions on the structure of $P$ at $I^0$ or $I^+$ (beyond~\eqref{EqOpStruct}).

\begin{rmk}[Normal operator]
\label{RmkOpNorm}
  One can check that an equivalent phrasing of condition~\eqref{ItOpNorm} is the requirement
  \begin{equation}
  \label{EqOpNorm2}
  \begin{split}
    &P - \Bigl[4\varrho\Bigl(\varrho D_\varrho+v D_v-i^{-1}\Bigl(\frac{n-1}{2}+p_1\Bigr)\Bigr)\varrho D_v + 4\varrho^2 \Delta_k + p_0\Bigr] \\
    &\hspace{8em} \in \rho_0^2 x_{\!\scri}^2\rho_+^2(x_{\!\scri}\CI+\cA^{(\ell_0,2\ell_{\!\scri},\ell_+)})\Diffeb^2(M;E),
  \end{split}
  \end{equation}
  where $\varrho=\frac{1}{t+r}$ and $v=\frac{t-r}{t+r}$ as in~\eqref{EqOGeoAdmDual}; indeed, passing to the coordinates~\eqref{EqOGeoCoordI0} or \eqref{EqOGeoCoordIp}, one computes
  \[
    \varrho\pa_\varrho+v\pa_v \equiv \frac12 x_{\!\scri}\pa_{x_{\!\scri}},\qquad
    \varrho\pa_v = \rho_0\Bigl(\rho_0\pa_{\rho_0}-\frac12 x_{\!\scri}\pa_{x_{\!\scri}}\Bigr) = -\rho_+\Bigl(\rho_+\pa_{\rho_+}-\frac12 x_{\!\scri}\pa_{x_{\!\scri}}\Bigr)
  \]
  modulo the space $x_{\!\scri}\Veb(M)$ of edge-b-vector fields which vanish at $\scri^+$; moreover, $\varrho\equiv\half\rho_0 x_{\!\scri}^2\equiv\half x_{\!\scri}^2\rho_+$ modulo $\rho_0 x_{\!\scri}^3\rho_+\CI$, and finally $p_0=\half\varrho^{-1}\rho_0^2 x_{\!\scri}^2 p_0^0=\rho_0 p_0^0$, resp.\ $p_0=\half\varrho^{-1}x_{\!\scri}^2\rho_+^2 p_0^+=\rho_+ p_0^+$ in the coordinates~\eqref{EqOGeoCoordI0}, resp.\ \eqref{EqOGeoCoordIp} indeed in the notation of Definition~\ref{DefOp}. In equation~\eqref{EqOpNorm2}, $\Delta_k\in\Diff^2(Y;E|_Y)$ is principally scalar with principal part $k^{i j}(y)D_{y^i}D_{y^j}$. (Modifications of $\Delta_k$ by subprincipal terms, i.e.\ elements of $\Diff^1(Y;E|_Y)$, contribute terms of class $\rho_0^2 x_{\!\scri}^2\rho_+^2\cdot x_{\!\scri}\Diffeb^1(M;E)$, which are thus error terms.) The principal part of the operator in parentheses here is prescribed by condition~\eqref{ItOpSymb} (see~\eqref{EqOGeoAdmDual}), so the new pieces of information from condition~\eqref{ItOpNorm} are the bundle endomorphisms $p_0,p_1$.
\end{rmk}

\begin{rmk}[Weaker error term: operator]
\label{RmkOpWeak}
  Mirroring Remark~\ref{RmkOGeoAdmWeaker}, a natural and slightly more permissible definition would only require
  \begin{equation}
  \label{EqOpWeak}
    P \in \rho_0^2 x_{\!\scri}^2\rho_+^2\cA^{\bigish(((0,0),\ell_0),\,2\ell_{\!\scri},\,((0,0),\ell_+)\bigish)}(M;\upbeta^*E)
  \end{equation}
  instead of~\eqref{EqOpStruct}. The analysis in~\S\S\ref{SsME}, \ref{Sb}, \ref{SE}, and \ref{SNe} goes through in this generality with only notational modifications. In~\S\ref{SNp} however, we appeal to an elliptic parametrix construction in the 0-calculus which, in the existing literature, is stated only for smooth coefficient operators; extending this to the case of coefficients which are smooth plus decaying conormal would suffice to treat the case~\eqref{EqOpWeak}.
\end{rmk}

\begin{example}[Minkowski wave operator]
\label{ExOpMink}
  The wave operator on Minkowski space $(\R^4,g_0)$ is an example of a $g_0$-admissible operator. Indeed, in the coordinates~\eqref{EqOGeoMinkRhoV} used in~\eqref{EqOMinkBdy}, one computes
  \begin{align*}
    \Box_{g_0} &= 4\varrho(\varrho D_\varrho+v D_v)\varrho D_v + \frac{2 i(n-1)}{1-v}\varrho\bigl(\varrho^2 D_\varrho+\varrho(1+v)D_v\bigr) + \frac{4\varrho^2}{(1-v)^2}\Delta_\slg \\
      &\equiv 4\varrho\Bigl(\varrho D_\varrho+v D_v-i^{-1}\frac{n-1}{2}\Bigr)\varrho D_v + 4\varrho^2\Delta_\slg \bmod \rho_0^2 x_{\!\scri}^2\rho_+^2\cdot x_{\!\scri}\Diffeb^2(M),
  \end{align*}
  which is thus indeed of the form~\eqref{EqOpNorm2} for $p_0=p_1=0$ and $k=\slg$.
\end{example}

\begin{example}[Wave operator of an admissible metric]
\label{ExOpBoxAdm}
  Generalizing the previous example, one can show that if $g$ is an admissible metric, then the scalar wave operator $\Box_g$ is an admissible operator. Indeed, the Koszul formula, together with the fact that conjugation by any weight $\rho_0^{\alpha_0}x_{\!\scri}^{2\alpha_{\!\scri}}\rho_+^{\alpha_+}$ preserves the space $\Veb(M)$, implies that the Levi-Civita connection of $g$ satisfies $\nabla^g\in(\CI+\cA^{(\ell_0,2\ell_{\!\scri},\ell_+)})\Diffeb^1(M;\Teb M,\Teb^*M\otimes\Teb M)$, similarly for the connection acting on sections of tensor products of $\Teb^*M$ and $\Teb M$. This implies $\Box_g\in\rho_0^2 x_{\!\scri}^2\rho_+^2\Diffeb^2(M)$, and we have~\eqref{EqOpNorm2} with $p_0=p_1=0$ since the normal operators of $\Box_g$ at $\scri^+$ are equal to those of $\Box_{g_0}$ where $g_0$ is of the Minkowskian form~\eqref{EqOMinkBdy} with $\slg$ replaced by $k$.
\end{example}

\begin{example}[Linearization of the gauge-fixed Einstein operator]
\label{ExOpMinkEin}
  Consider the case $n+1=\dim M=4$; then in the coordinates $\rho_0,x_{\!\scri}$, the edge normal operator of $\rho_{\!\scri}\rho^{-3}P\rho=\rho_0^{-2}x_{\!\scri}^{-2} (\rho_0 x_{\!\scri}^2)^{-1}P \rho_0 x_{\!\scri}^2$ is given by the conjugation of the first operator in~\eqref{EqOpNorm} by the weight $(\rho_0 x_{\!\scri}^2)^{-1}$, and thus (modulo the term $k^{i j}(y_0)(x_{\!\scri} D_{y^i})(x_{\!\scri} D_{y^j})$) equal to $-(x_{\!\scri} D_{x_{\!\scri}}-2 i^{-1}p_1)(x_{\!\scri} D_{x_{\!\scri}}-2\rho_0 D_{\rho_0})=-4(\rho_{\!\scri}\pa_{\rho_{\!\scri}}-p_1)(\rho_0\pa_{\rho_0}-\rho_{\!\scri}\pa_{\rho_{\!\scri}})$. Thus, the linearized gauge-fixed Einstein operator considered in \cite[Proposition~3.29]{HintzMink4Gauge} is an admissible operator, acting on sections of the bundle $S^2\wt T^*M$ in the notation of the reference.
\end{example}

\section{Microlocal edge-b-regularity theory near \texorpdfstring{$\scri^+$}{future null infinity}}
\label{SM}

We now fix an $(\ell_0,2\ell_{\!\scri},\ell_+)$-admissible metric $g$ (see Definition~\ref{DefOGeoAdm}) and a $g$-admissible operator $P$ on the $(n+1)$-dimensional manifold $M$ (see Definition~\ref{DefOp}). We shall assume that $n\geq 2$ for simplicity of notation; the case $n=1$ can be treated with the same methods upon making some straightforward simplifications and modifications (related to the fact that the `$0$-sphere' $\Sph^0=\{-1,1\}$ is disconnected); we leave these to the interested reader.

We denote the edge-b-principal symbol of $P$ by $G$. Upon fixing arbitrary choices of boundary defining functions $\rho_0,x_{\!\scri},\rho_+$, it satisfies
\[
  G_\ebop := \rho_0^{-2}x_{\!\scri}^{-2}\rho_+^{-2}G \in S^2(\Teb^*M) + \cA^{(\ell_0,2\ell_{\!\scri},\ell_+)}S^2(\Teb^*M).
\]
The characteristic set $G_\ebop^{-1}(0)\subset\Teb^*M\setminus o$ of $P$ is conic in the fibers; we shall work near fiber infinity and thus define
\[
  \Sigma \subset \ol{\Teb^*}M\setminus o
\]
as the closure of $G_\ebop^{-1}(0)$ inside $\ol{\Teb^*}M\setminus o$. With respect to the time orientation of $M$ near $\scri^+$ introduced in~\eqref{EqOGeoTimelike}, we split $\Sigma$ into its two connected components
\[
  \Sigma = \Sigma^+ \cup \Sigma^-,
\]
with $\Sigma^+$ containing the future null covectors. In~\S\ref{SsMF}, we analyze the dynamical structure of the flow of the Hamiltonian vector field $H_{G_\ebop}$ inside $\Sigma$. This information is used in~\S\ref{SsME} to track the microlocal regularity of solutions of $P u=f$ near $\Teb^*_{\scri^+}M$.

Note that if $0<f\in\CI(\ol{\Teb^*}M)$, then the restriction of the Hamiltonian vector field
\[
  H_{f G_\ebop} = f H_{G_\ebop} + G_\ebop H_f,
\]
to $\Sigma$ is a positive rescaling of $H_{G_\ebop}$. Therefore the properties of the $H_{G_\ebop}$-flow inside $\Sigma$ of interest for the microlocal analysis of $P$, i.e.\ critical sets and the signs and ratios of eigenvalues of the linearization at critical sets, are independent of the choice of boundary defining functions. We can thus make convenient choices in the calculations below; by an abuse of notation, we shall always call them $\rho_0,x_{\!\scri},\rho_+$ (as already done in the previous section) in order to not overburden the notation. Moreover, since $H_{G_\ebop}$ is homogeneous of degree $1$ with respect to dilations in the fibers of $\Teb^*M$, we shall work with the rescaled vector field
\begin{equation}
\label{EqMHamResc}
  \sfH := \rho_\infty H_{G_\ebop} \in \Veb(\ol{\Teb^*}M\setminus o)
\end{equation}
where $\rho_\infty\in\CI(\ol{\Teb^*}M)$ is a defining function of $\Seb^*M$; see Corollary~\ref{CorEBOHamResc}. Again, we may change $\rho_\infty$ via multiplication with a positive function without altering the relevant properties of $\sfH$.

\subsection{Structure of the null-bicharacteristic flow near \texorpdfstring{$\scri^+$}{null infinity}}
\label{SsMF}

We first work near $I^0\cap\scri^+$, or indeed in a region $\cU_0=\cU_0(T)$ for any fixed $T\in\R$ (see Definition~\ref{DefOGeoU0p}), and use the coordinates~\eqref{EqOGeoCoordI0}. Writing the canonical 1-form on the cotangent bundle as
\begin{equation}
\label{EqMFCoordI0}
  \xi\frac{\dd x_{\!\scri}}{x_{\!\scri}} + \sum_{j=1}^{n-1} \eta_j\frac{\dd y^j}{x_{\!\scri}} + \zeta\frac{\dd\rho_0}{\rho_0},
\end{equation}
we can read off from~\eqref{EqOpNorm} (using the summation convention, and dropping the weight at $I^+$ from the notation)
\begin{subequations}
\begin{equation}
\label{EqMFSymbI0}
\begin{split}
  G_\ebop^0 &:= \rho_0^{-2}x_{\!\scri}^{-2}G = -\frac12\xi(\xi-2\zeta) + k^{i j}(y)\eta_i\eta_j + \tilde G_\ebop^0, \\
  &\qquad \tilde G_\ebop^0 \in x_{\!\scri} S^2(\Teb_{\cU_0}^*M)+\cA^{(\ell_0,2\ell_{\!\scri})}S^2(\Teb_{\cU_0}^*M).
\end{split}
\end{equation}
Using Lemma~\ref{LemmaEBOHam} (where $x,z$ are equal to $x_{\!\scri},\rho_0$ in present notation), we compute
\begin{align}
\label{EqMFHamI0}
  H_{G_\ebop^0} &= (\zeta-\xi)(x_{\!\scri}\pa_{x_{\!\scri}}+\eta\pa_\eta) + \xi\rho_0\pa_{\rho_0} - 2|\eta|_{k^{-1}}^2\pa_\xi + \tilde H, \\
  &\qquad \rho_\infty\tilde H\in (x_{\!\scri}\CI+\cA^{(\ell_0,2\ell_{\!\scri})})\Vb(\ol{\Teb^*_{\cU_0}}M). \nonumber
\end{align}
\end{subequations}
Note carefully that we regard the error $\rho_\infty\tilde H$ as a b-vector field (rather than an edge-b-vector field); therefore, the terms involving derivatives along $x_{\!\scri}\pa_y$ can be regarded as error terms. At $x_{\!\scri}=0$, the two components of $\Sigma$ are given by
\[
  \Sigma^\pm \cap \Teb^*_{\scri^+}M = \{ -(\xi-\zeta)^2+|\zeta|^2+2|\eta|_{k^{-1}}^2=0 \colon {\pm}(\xi-\zeta) > 0 \}.
\]
(Indeed, the coefficient of $x_{\!\scri}\pa_{x_{\!\scri}}$ of $H_{G_\ebop^0}$ is then negative on $\Sigma^+$, cf.\ \eqref{EqOGeoTimelike}.)

We restrict attention to $\Sigma^+$, and proceed to determine the radial sets of $H_{G_\ebop^0}$, i.e.\ the places where $H_{G_\ebop^0}$ is a multiple of the generator of fiber dilations. At $\rho_0=0$ and $0<x_{\!\scri}<\delta_0$ (with $\delta_0>0$ given by Lemma~\ref{LemmaOGeoHyp}\eqref{ItOGeoHypI0}), we have $H_{G_\ebop^0}x_{\!\scri}<0$, and hence $H_{G_\ebop^0}$ is not radial there. Working at $x_{\!\scri}=0$, we ask when $H_{G_\ebop^0}=c(\xi\pa_\xi+\eta\pa_\eta+\zeta\pa_\zeta)$ for some $c\in\R$. If $c=0$, then $\eta=0$ by inspection of the $\pa_\xi$-coefficient in~\eqref{EqMFHamI0}, and hence $\xi=0$ or $\xi=2\zeta$ on $\Sigma^+$. At $\eta=0$, $\xi=0$, the vector field $H_{G_\ebop^0}$ is indeed radial. At $\eta=0$, $\xi=2\zeta$ on the other hand, we need $\xi\neq 0$ as we are working away from the zero section of $\Teb^*M$, and then necessarily $\rho_0=0$, for otherwise the $\pa_{\rho_0}$ coefficient of $H_{G_\ebop^0}$ would be nonzero; this identifies a second radial set. Next, if $c\neq 0$ and still $x_{\!\scri}=0$, then necessarily $\zeta=0$ since $H_{G_\ebop^0}$ has a vanishing $\pa_\zeta$-coefficient; thus, $\xi^2=2|\eta|_{k^{-1}}^2$ and $c=-\xi$ (by inspection of the $\pa_\xi$-coefficient), and then the $\eta\pa_\eta$-coefficient of $H_{G_\ebop^0}$ is $(\zeta-\xi)=-\xi=c$; the requirement $H_{G_\ebop^0}\rho_0=0$ forces $\rho_0=0$, and we have thus identified a third radial set.

Turning to a neighborhood of $\scri^+\cap I^+$, or indeed a region $\cU_+=\cU_+(T)$ (see Definition~\ref{DefOGeoU0p}) for any fixed $T\in\R$, we use the coordinates~\eqref{EqOGeoCoordIp} and write the canonical 1-form (re-purposing the above notation) as
\begin{equation}
\label{EqMFCoordIp}
  \xi\frac{\dd x_{\!\scri}}{x_{\!\scri}}+\sum_{j=1}^{n-1}\eta_j\frac{\dd y^j}{x_{\!\scri}} + \zeta\frac{\dd\rho_+}{\rho_+}.
\end{equation}
From~\eqref{EqOpNorm} and Lemma~\ref{LemmaEBOHam}, we find (dropping the weight at $I^0$ from the notation)
\begin{subequations}
\begin{align}
\label{EqMFSymbIp}
  G_\ebop^+ &:= x_{\!\scri}^{-2}\rho_+^{-2}G = \frac12\xi(\xi-2\zeta) + k^{i j}(y)\eta_i\eta_j + \tilde G_\ebop^+, \\
\label{EqMFHamIp}
  \begin{split}
  H_{G_\ebop^+} &= (\xi-\zeta)(x_{\!\scri}\pa_{x_{\!\scri}}+\eta\pa_\eta) - \xi\rho_+\pa_{\rho_+} - 2|\eta|_{k^{-1}}^2\pa_\xi + 2 x_{\!\scri} k^{i j}\eta_i\pa_{y^j} + \tilde H,
  \end{split}
\end{align}
\end{subequations}
where the error terms satisfy
\begin{align*}
  \tilde G_\ebop^+ &\in x_{\!\scri} S^2(\Teb_{\cU_+}^*M) + \cA^{(2\ell_{\!\scri},\ell_+)}S^2(\Teb^*_{\cU_+}M), \\
  \rho_\infty\tilde H &\in (x_{\!\scri}\CI+\cA^{(2\ell_{\!\scri},\ell_+)})\Veb(\ol{\Teb^*_{\cU_+}}M).
\end{align*}
We record $H_{G_\ebop^+}$ as an edge-b-vector field here in order to keep nondegenerate track of its $\pa_y$-component. We moreover have
\begin{equation}
\label{EqMFCharIp}
  \Sigma^\pm \cap \Teb^*_{\scri^+}M = \{ -\zeta^2 + (\xi-\zeta)^2 + 2|\eta|_{k^{-1}}^2 = 0 \colon {\pm}\zeta > 0 \}
\end{equation}
by comparison of the $x_{\!\scri}\pa_{x_{\!\scri}}$-coefficient of $H_{G_\ebop^+}$ with~\eqref{EqOGeoTimelike}.

We proceed to locate the radial sets inside $\Sigma^+$ in these coordinates. At $\rho_+=0$ and $0<x_{\!\scri}<\delta_0$, let us work in $\xi^2+|\eta|_{k^{-1}}^2+\zeta^2=1$. The vanishing of $x_{\!\scri}^{-1}H_{G_\ebop^+}y^j=2 k^{i j}\eta_i+\cO(x_{\!\scri})$ forces $|\eta|=\cO(x_{\!\scri})$, and therefore $x_{\!\scri}^{-1}H_{G_\ebop^+}x_{\!\scri}=(\xi-\zeta)+\cO(x_{\!\scri})$ cannot vanish; indeed, if it did, then $\xi-\zeta=\cO(x_{\!\scri})$ and thus $\zeta=\cO(x_{\!\scri})$ on $\Sigma$ by~\eqref{EqMFCharIp}, giving the contradiction $1=\xi^2+|\eta|^2+\zeta^2=\cO(x_{\!\scri})$ when $\delta_0>0$ is sufficiently small. Thus, in a small neighborhood of $\scri^+$, the radial sets all lie over $\scri^+$. Within $\scri^+$ then, suppose $H_{G_\ebop^+}=c(\xi\pa_\xi+\eta\pa_\eta+\zeta\pa_\zeta)$. If $c=0$, then $\eta=0$, so $\xi=0$ or $\xi=2\zeta$; in the former case, we are at a radial set indeed, whereas in the latter case we must have $\xi,\zeta\neq 0$ to stay clear of the zero section, but then the vanishing of $H_{G_\ebop^+}\rho_+$ forces $\rho_+=0$; and a radial set is indeed located here. If on the other hand $c\neq 0$, then since $H_{G_\ebop^+}\zeta=0$ we must have $c\zeta=0$ and thus $\zeta=0$, which cannot happen inside $\Sigma$ by~\eqref{EqMFCharIp}.

We summarize the above computations as follows:

\begin{lemma}[Radial sets of $H_{G_\ebop}$]
\label{LemmaMFCrit}
  In a sufficiently small neighborhood $\cU\subset M$ of $\scri^+$, a complete list of the radial sets of $H_{G_\ebop}$ inside $\Sigma^\pm\cap(\Teb^*_\cU M\setminus o)$ is as follows:
  \begin{align}
  \label{EqMFCritInm}
    \cR^\pm_{\rm in,-} &:= \{ (\rho_0,x_{\!\scri},y;\zeta,\xi,\eta) \colon \rho_0=x_{\!\scri}=0,\ \eta=0,\ \xi=2\zeta,\ \pm\xi>0 \}, \\
  \label{EqMFCritCross}
    \cR^\pm_{\rm c} &:= \{ (\rho_0,x_{\!\scri},y;\zeta,\xi,\eta) \colon \rho_0=x_{\!\scri}=0,\ \zeta=0,\ \xi=\pm\sqrt 2|\eta|_{k^{-1}} \}, \\
  \label{EqMFCritOut}
    \begin{split}
    \cR^\pm_{\rm out} &:= \{ (\rho_0,x_{\!\scri},y;\zeta,\xi,\eta) \colon x_{\!\scri}=0,\ \xi=\eta=0,\ \pm\zeta<0 \}\ \ (\text{in the coordinates~\eqref{EqMFCoordI0}}) \\
      &= \{ (\rho_+,x_{\!\scri},y;\zeta,\xi,\eta) \colon x_{\!\scri}=0,\ \xi=\eta=0,\ {\pm}\zeta>0 \}\ \ (\text{in the coordinates~\eqref{EqMFCoordIp}}),
    \end{split} \\
  \label{EqMFCritInp}
    \cR^\pm_{\rm in,+} &:= \{ (\rho_+,x_{\!\scri},y;\zeta,\xi,\eta) \colon \rho_+=x_{\!\scri}=0,\ \eta=0,\ \xi=2\zeta,\ \pm\xi>0 \}.
  \end{align}
\end{lemma}

Note now that $H_{G_\ebop}$ is radial at a point $\varpi\in\Teb^*M\setminus o$ if and only if the restriction of the rescaled vector field $\sfH=\rho_\infty H_{G_\ebop}$ to $\Seb^*M$ vanishes, as a smooth vector field, at $[\varpi]\in\Seb^*M$ (i.e.\ the boundary at fiber infinity of the ray $\R_+\varpi$). Explicitly, near $\cR_{\rm in,-}^+$ and $\cR_{\rm out}^+$, we use the coordinates~\eqref{EqMFCoordI0} and, near fiber infinity,
\begin{subequations}
\begin{equation}
\label{EqMFRinmCoord}
  \rho_\infty = |\zeta|^{-1},\qquad
  \hat\eta=\frac{\eta}{\zeta},\qquad
  \hat\xi=\frac{\xi}{\zeta}.
\end{equation}
Thus, $\rho_\infty=0$ at fiber infinity, and $\rho_\infty=\zeta^{-1}$ near, and $(\hat\eta,\hat\xi)=(0,2)$ at, $\pa\cR_{\rm in,-}^+$; and $\rho_\infty=-\zeta^{-1}$ near, and $(\hat\eta,\hat\xi)=(0,0)$ at, $\pa\cR_{\rm out}^+$. For a submanifold $S\subset\Seb^*M$, denote by $\cI_S\subset\CI(\ol{\Teb^*}M)$ the ideal of functions vanishing at $S$. Inserting the coordinates~\eqref{EqMFRinmCoord} into~\eqref{EqMFHamI0}, we find that
\begin{align}
\label{EqMFLinInm}
  \rho_\infty H_{G_\ebop^0} &\equiv 2\rho_0\pa_{\rho_0} - x_{\!\scri}\pa_{x_{\!\scri}} - \hat\eta\pa_{\hat\eta} \bmod (\cI_{\pa\cR_{\rm in,-}^+}+\cA^{(\ell_0,2\ell_{\!\scri})})\Vb(\ol{\Teb^*_{\cU_0}}M), \\
\label{EqMFLinOut1}
  \rho_\infty H_{G_\ebop^0} &\equiv -x_{\!\scri}\pa_{x_{\!\scri}}-\hat\eta\pa_{\hat\eta} \bmod (\cI_{\pa\cR_{\rm out}^+}+\cA^{(\ell_0,2\ell_{\!\scri})})\Vb(\ol{\Teb^*_{\cU_0}}M).
\end{align}
\end{subequations}
Near $\cR_{\rm c}^+$, we need to work with different coordinates; we choose
\begin{subequations}
\begin{equation}
\label{EqMFRcCoord}
  \rho_\infty = \xi^{-1},\qquad
  \hat\eta=\frac{\eta}{\xi},\qquad
  \hat\zeta=\frac{\zeta}{\xi},
\end{equation}
and compute (using that $\hat\zeta=0$ and $|\hat\eta|_{k^{-1}}=\frac{1}{\sqrt 2}$ at $\pa\cR_{\rm c}^+$)
\begin{equation}
\label{EqMFLinC}
  \rho_\infty H_{G_\ebop^0} \equiv \rho_\infty\pa_{\rho_\infty} + \rho_0\pa_{\rho_0} - x_{\!\scri}\pa_{x_{\!\scri}} + \hat\zeta\pa_{\hat\zeta} \bmod (\cI_{\pa\cR_{\rm c}^+}+\cA^{(\ell_0,2\ell_{\!\scri})})\Vb(\ol{\Teb^*_{\cU_0}}M).
\end{equation}
\end{subequations}

Finally, near $\cR_{\rm in,+}^+$ and $\cR_{\rm out}^+$ and in the coordinates~\eqref{EqMFCoordIp} and
\begin{subequations}
\begin{equation}
\label{EqMFRinpCoord}
  \rho_\infty = |\zeta|^{-1},\qquad
  \hat\eta=\frac{\eta}{\zeta},\qquad
  \hat\xi=\frac{\xi}{\zeta},
\end{equation}
we have $\rho_\infty=\zeta^{-1}$ near, and $(\hat\eta,\hat\xi)=(0,2)$ at, $\pa\cR_{\rm in,+}^+$, and $\rho_\infty=\zeta^{-1}$ near, and $(\hat\eta,\hat\xi)=(0,0)$ at, $\pa\cR_{\rm out}^+$. Inserting this into~\eqref{EqMFHamIp} gives
\begin{align}
\label{EqMFLinInp}
  \rho_\infty H_{G_\ebop^+} &\equiv x_{\!\scri}\pa_{x_{\!\scri}} - 2\rho_+\pa_{\rho_+} + \hat\eta\pa_{\hat\eta} \bmod (\cI_{\pa\cR_{\rm in,+}^+}+\cA^{(2\ell_{\!\scri},\ell_+)})\Vb(\ol{\Teb^*_{\cU_+}}M), \\
\label{EqMFLinOut2}
  \rho_\infty H_{G_\ebop^+} &\equiv -x_{\!\scri}\pa_{x_{\!\scri}} - \hat\eta\pa_{\hat\eta} \bmod (\cI_{\pa\cR_{\rm out,+}}+\cA^{(2\ell_{\!\scri},\ell_+)})\Vb(\ol{\Teb^*_{\cU_+}}M).
\end{align}
\end{subequations}

From the expressions~\eqref{EqMFLinInm}, \eqref{EqMFLinOut1}, \eqref{EqMFLinC}, \eqref{EqMFLinInp}, and \eqref{EqMFLinOut2}, one can read off the dynamical nature (source, sink, saddle point) of the respective radial sets; see Figure~\ref{FigMF}. We determine the structure of the flow between these radial sets by computing the flows of $H_{G_\ebop^0}$ and $H_{G_\ebop^+}$ explicitly inside the future characteristic set $\Sigma^+$. Consider the time $s$ flow of
\[
  H^0:=H_{G_\ebop^0}|_{\Teb^*_{\scri^+}M}=(\zeta-\xi)\eta\pa_\eta+\xi\rho_0\pa_{\rho_0}-2|\eta|_{k^{-1}}^2\pa_\xi
\]
inside $(G_\ebop^0)^{-1}(0)=\{-\half\xi(\xi-2\zeta)+|\eta|_{k^{-1}}^2=0\}$. Dropping the coordinates $x_{\!\scri},y$ (which are constants of integration) from the notation and only keeping $(\rho_0,\zeta,\xi,\eta)$, we have for $\zeta=0$ (and thus $\xi=\sqrt{2}|\eta|_{k^{-1}}>0$)
\[
  e^{s H^0}(\rho_0,0,\xi,\eta) = \Bigl(\rho_0(1+s\xi), 0, \frac{\xi}{1+s\xi}, \frac{\eta}{1+s\xi}, \Bigr),
\]
with domain of definition given by $s\in(-\frac{1}{\xi},\frac{\rho_0^{-1}-1}{\xi})$ when $\rho_0>0$ (the upper bound ensuring that $\rho_0<1$, i.e.\ we stay in the coordinate patch); when $s\searrow-\frac{1}{\xi}$, this tends to a point on $\pa\cR_{\rm c}^+$. In the case $\rho_0=0$, the flow exists for all $s>-\frac{1}{\xi}$ and merely dilates the fiber variables.

When $\zeta\neq 0$ on the other hand, then $\xi\geq\max(0,2\zeta)$ on $\Sigma^+$ and
\[
  e^{s H^0}(\rho_0,\zeta,\xi,\eta) = \biggl( \rho_0\Bigl(1+\frac{\xi}{2\zeta}(e^{2\zeta s}-1)\Bigr), \zeta, \frac{2\xi\zeta}{\xi-(\xi-2\zeta)e^{-2\zeta s}}, \frac{2\eta\zeta e^{-\zeta s}}{\xi-(\xi-2\zeta)e^{-2\zeta s}} \biggr).
\]
When $\xi=2\zeta$ (and thus $\eta=0$ on $\Sigma^+$), this tends to $\cR_{\rm in,-}^+$ as $s\searrow-\infty$. When $\xi=0$ (and thus $\eta=0$, $\zeta<0$ on $\Sigma^+$), i.e.\ on $\cR_{\rm out}^+$, this is stationary. In the remaining case that $\xi\neq 0,2\zeta$ and $\zeta\neq 0$ (and thus $\eta\neq 0$), one has $e^{s H^0}(\rho_0,\zeta,\xi,\eta)\to\pa\cR_{\rm c}^+$ as $s\searrow\frac{1}{2\zeta}\log(\frac{\xi-2\zeta}{\xi})$ (the infimum of the maximal interval of existence of the flow); the flow leaves the local coordinate chart into $\rho_0\geq 1$ when $s$ exceeds a positive threshold. For $\rho_0=0$ finally, the flow is defined for all $s\geq 0$, and for $\zeta>0$, resp.\ $\zeta<0$ tends to $\cR_{\rm in,-}^+$, resp.\ $\cR_{\rm out}^+$ as $s\nearrow\infty$.

We now turn to the time $s$ flow of
\[
  H^+:=H_{G_\ebop^+}|_{\Teb^*_{\scri^+}M}=(\xi-\zeta)\eta\pa_\eta-\xi\rho_+\pa_{\rho_+}-2|\eta|_{k^{-1}}^2\pa_\xi
\]
inside $\Sigma^+$ over $\scri^+$, where $\half\xi(\xi-2\zeta)+|\eta|_{k^{-1}}^2=0$. Thus, $\zeta>0$ and therefore $0\leq\xi\leq 2\zeta$ in $\Sigma^+$, and we find (dropping $x_{\!\scri}\equiv 0$ and the constant $y$ from the notation)
\[
  e^{s H^+}(\rho_+,\zeta,\xi,\eta) = \biggl( \rho_+\Bigl(1+\frac{\xi}{2\zeta}(e^{-2\zeta s}-1)\Bigr), \zeta, \frac{2\xi\zeta}{\xi+(2\zeta-\xi)e^{2\zeta s}}, \frac{2\eta\zeta e^{\zeta s}}{\xi+(2\zeta-\xi)e^{2\zeta s}} \biggr).
\]
When $\xi=2\zeta$ and thus $\eta=0$, the fiber coordinates remain fixed and $\rho_+\to 0$ as $s\nearrow\infty$, thus this tends to $\cR_{\rm in,+}^+$. When $\xi\neq 2\zeta$, this tends to the point $(\rho_+(1-\frac{\xi}{2\zeta}),0,0,\zeta)\in\cR_{\rm out}^+$ as $s\nearrow\infty$.

The fact that the set $\xi=2\zeta$, $\eta=0$ (in either coordinate system) is preserved under the flow motivates the first half of the following definition:

\begin{definition}[Stable/unstable manifolds]
\label{DefMFStable}
  We define the following submanifolds of $\Sigma^\pm\cap(\Teb^*_{\scri^+}M\setminus o)$:
  \begin{align}
  \label{EqMFStable}
    \begin{split}
      \cW^\pm_{\rm in} &:= \{ (\rho_0,x_{\!\scri},y;\zeta,\xi,\eta) \colon x_{\!\scri}=0,\ \eta=0,\ \xi=2\zeta,\ \pm\xi>0 \}\ \ (\text{in the coordinates~\eqref{EqMFCoordI0}}) \\
        &= \{ (\rho_+,x_{\!\scri},y;\zeta,\xi,\eta) \colon x_{\!\scri}=0,\ \eta=0,\ \xi=2\zeta,\ \pm\xi>0 \}\ \ (\text{in the coordinates~\eqref{EqMFCoordIp}}),
    \end{split} \\
    \cW^\pm_{\rm c} &:= \{ (\rho_0,x_{\!\scri},y;\zeta,\xi,\eta) \colon \rho_0=x_{\!\scri}=0,\ \pm\zeta\geq 0 \} \cap \Sigma^\pm.
  \end{align}
\end{definition}

Thus, inside $\Sigma^+$, the set $\cW_{\rm in}^+$ is the unstable, resp.\ stable manifold of $\cR_{\rm in,-}^+$, resp.\ $\cR_{\rm in,+}^+$, while $\cW_{\rm c}^+$ is the intersection of the unstable manifold of $\cR_{\rm c}^+$ and the stable manifold of $\cR_{\rm in,-}^+$. Inside $\Sigma^-$, the analogous statements hold with `stable' and `unstable' switched. See Figure~\ref{FigMF}.

\begin{figure}[!ht]
\centering
\includegraphics{FigMF}
\caption{Illustration of the null-bicharacteristic flow of the rescaled Hamiltonian vector field $\sfH$, see~\eqref{EqMHamResc}, at fiber infinity over a fiber $\scri^+_y$, $y\in Y$, of $\scri^+$. We identify conic subsets of $\Teb^*M\setminus o$ with their boundaries at fiber infinity. We draw $\Sigma^+\cap\Seb^*_{\scri^+_y}M$ as a cylinder with base $\scri^+_y$ (a closed interval) and cross section $\Sigma^+\cap\Seb^*_p M\cong\Sph^{n-1}$, drawn here for $n=2$. The direction of the flow is indicated by arrows, and the boundaries at fiber infinity of the radial sets defined in Lemma~\ref{LemmaMFCrit} are drawn in bold black. We also show the boundary at fiber infinity of the two stable/unstable manifolds from Definition~\ref{DefMFStable}. Moreover, we sketch null-bicharacteristics in $(I^0)^\circ$ in red, null-bicharacteristics in $(I^+)^\circ$ in green, and null-bicharacteristics lying over $M^\circ$ in blue.}
\label{FigMF}
\end{figure}

\subsection{Microlocal estimates}
\label{SsME}

We continue using the notation introduced at the beginning of~\S\ref{SM}. The structure of the linearizations of the Hamilton flow at the radial sets determines the structure of the corresponding microlocal propagation results, even at a quantitative level as far as the relative sizes of weights and differential orders are concerned. In addition, suitable subprincipal symbols of the operator $P$ affect the precise threshold conditions entering in the radial point estimates at $\cR_{\rm c}$ and $\cR_{\rm out}$; there, the endomorphism $p_1$ from Definition~\ref{DefOp}\eqref{ItOpNorm} enters (see also Remark~\ref{RmkMEThr}). We thus define:

\begin{definition}[Minimum of $p_1$]
\label{DefMEMin}
  Given $p_1\in(\CI+\cA^{(\ell_0,\ell_+)})(\scri^+;\upbeta^*\End(E))$, we set
  \[
    \ubar p_1 := \inf_{p\in\scri^+} \{\Re\mu \colon \mu\in\spec p_1(p)\}.
  \]
  Here, $\spec p_1(p)\subset\C$ is the spectrum (in the sense of linear algebra) of the linear map $p_1(p)\colon E_y\to E_y$, $y=\upbeta(p)$. We define $\ubar p_{1,+}$ by the same expression except for taking the infimum only over $p\in\scri^+\cap I^+$.
\end{definition}

We rewrite this in a manner more amenable for use in estimates:

\begin{lemma}[Characterization of $\ubar p_1$]
\label{LemmaMEMin}
  Let $p_1\in(\CI+\cA^{(\ell_0,\ell_+)})(\scri^+;\upbeta^*\End(E))$. For a positive definite fiber inner product $h_E$ on $E|_Y$, set
  \[
    \ubar p_1(h_E) := \inf_{p\in\scri^+} \spec\Bigl(\frac12\bigl(p_1(p)+p_1(p)^*\bigr)\Bigr),
  \]
  where $p_1(p)^*\in\End((\upbeta^*E)_p)=\End(E_y)$, $y=\upbeta(p)$, is the adjoint of $p_1(p)$ with respect to $h_E(y)$. Then $\ubar p_1=\sup_{h_E}\ubar p_1(h_E)$.
\end{lemma}
\begin{proof}
  For any $h_E$, one has $\ubar p_1(h_E)\leq\ubar p_1$, and hence the statement is that for any $\eps>0$, one can find an inner product $h_E$ so that $\ubar p_1-\eps<\ubar p_1(h_E)\leq\ubar p_1$. The Lemma is a general linear algebra result, which is a minor variant of \cite[Proposition~B.1 and Lemma~B.2]{HintzNonstat}: we need to use that due to $\ell_0,\ell_+>0$, we have $(\CI+\cA^{(\ell_0,\ell_+)})(\scri^+)\subset\cC^0(\scri^+)$, and thus the compactness argument of the proof of \cite[Proposition~B.1]{HintzNonstat} still applies.
\end{proof}

We shall identify conic subsets of $\Teb^*M\setminus o$ with their boundaries at fiber infinity. We shall moreover drop the vector bundle $\upbeta^*E\to M$, in which all distributions below are valued, from the notation. Finally, in order to state our estimates, we need to fix an integration density on $M$, which we take to be a positive element of
\begin{equation}
\label{EqMEDensity}
  \rho_0^{-n-1}x_{\!\scri}^{-2 n}\rho_+^{-n-1}\CI(M;\Omegab M).
\end{equation}
The motivation is that the lift of the metric density $|\dd g|$ on $M^\circ$ is given by $|\dd g|=\rho_0^{-n-1}x_{\!\scri}^{-n-1}\rho_+^{-n-1}|\dd(\rho_0^2 x_{\!\scri}^2\rho_+^2 g)|\in\rho_0^{-n-1}x_{\!\scri}^{-n-1}\rho_+^{-n-1}(\CI+\cA^{(\ell_0,2\ell_{\!\scri},\ell_+)})(M;\Omegaeb M)$, and we have $\CI(M;\Omegaeb M)=x_{\!\scri}^{-n+1}\CI(M;\Omegab M)$. (We may as well allow for a conormal term in~\eqref{EqMEDensity} as well, which makes no difference in the subsequent arguments; one could then simply use $|\dd g|$ as the metric density, which is of class~\eqref{EqMEDensity} plus a lower order conormal error term.)

The semiglobal microlocal propagation results near $\scri^+$ are then the following:

\begin{thm}[Propagation through $\scri^+$: forward direction]
\label{ThmMEFw}
  Let $s\in\R$ and a vector of weights $\alpha=(\alpha_0,2\alpha_{\!\scri},\alpha_+)$; put $\alpha'=\alpha+(2,2,2)$. Suppose $u\in\Heb^{-\infty,\alpha}(M)=\rho_0^{\alpha_0}x_{\!\scri}^{2\alpha_{\!\scri}}\rho_+^{\alpha_+}\Heb^{-\infty}(M)$, and let $f=P u\in\Heb^{-\infty,\alpha'}(M)$. Then $\WFeb^{s,\alpha}(u)\subset\Sigma\cup\WFeb^{s-2,\alpha'}(f)$. Moreover:
  \begin{enumerate}
  \item\label{ItMEFwThru}{\rm (Propagation through $\scri^+$.)} Suppose that\footnote{The condition on $s_0$ can be written as $s_0>-(\alpha_0-\alpha_{\!\scri})-(-\half+\ubar p_1-\alpha_{\!\scri})$, cf.\ the first inequality in~\eqref{EqMEFwThruThr} and the threshold condition \eqref{EqMEFwOutThr}.}
    \begin{equation}
    \label{EqMEFwThruThr}
      \alpha_+<\alpha_{\!\scri}-\frac12<\alpha_0,\qquad
      s>s_0>\frac12-\alpha_0+2\alpha_{\!\scri}-\ubar p_1.
    \end{equation}
    Assuming that\footnote{Carefully note that the assumption on $u$ at $\cR_{\rm c}$ requires only above-threshold regularity $s_0$ on $u$ (while the conclusion of the Theorem gives $s>s_0$ degrees of edge-b-differentiability at $\cR_{\rm c}$).}
    \begin{equation}
    \label{EqMEFwThru}
    \begin{alignedat}{3}
      &\WFeb^{s,\alpha}(u)&&\cap\,\Seb^*_{(I^0)^\circ}M&&=\emptyset, \\
      &\WFeb^{s_0,\alpha}(u)&&\cap\,\cR_{\rm c}&&=\emptyset, \\
      &\WFeb^{s-1,\alpha'}(f)&&\cap\,\Seb^*_{\scri^+}M&&\subset\cR_{\rm out},
    \end{alignedat}
    \end{equation}
    we then have $\WFeb^{s,\alpha}(u)\cap\Seb^*_{\scri^+}M\subset\cR_{\rm out}$.
  \item\label{ItMEFwOut}{\rm (Full control at $\scri^+$.)} Suppose that in addition to~\eqref{EqMEFwThruThr}, we have
    \begin{equation}
    \label{EqMEFwOutThr}
      \alpha_{\!\scri} < -\frac12+\ubar p_1,
    \end{equation}
    and that for an open neighborhood $\cU\subset\Seb^*M$ of $\cR_{\rm out}$, we have
    \begin{equation}
    \label{EqMEFwOut}
    \begin{alignedat}{3}
      &\WFeb^{s,\alpha}(u)&&\cap\,(\cU\setminus\Seb^*_{\scri^+}M)&&=\emptyset, \\
      &\WFeb^{s-1,\alpha'}(f)&&\cap\,\Seb_{\scri^+}^*M&&=\emptyset.
    \end{alignedat}
    \end{equation}
    Then $\WFeb^{s,\alpha}(u)\cap\Seb^*_{\scri^+}M=\emptyset$.
  \end{enumerate}
\end{thm}

All statements can be microlocalized to either of the two components $\Sigma^\pm$ of $\Sigma$. The terminology `forward' refers to the causal nature of the direction in which we are propagating. Relative to the rescaled Hamiltonian vector field, this means propagating regularity of $u$ in the forward direction on $\Sigma^+$ and in the backward direction in $\Sigma^-$.

If $f\in\Heb^{s-1,\alpha'}(M)$, an application of part~\eqref{ItMEFwThru} shows that control of $u$ in $(I^0)^\circ$ can be propagated through $\scri^+$ and gives $\Heb^{s,\alpha}$-control away from the outgoing radial set. In particular, one gets control near the ingoing radial set $\cR_{\rm in,+}$ at the future boundary of $\scri^+$. In order to gain full control of $u$, one needs to understand how the flow continues in $I^+$, which depends on the global geometry of the spacetime of interest; examples, including asymptotically stationary Minkowski spacetimes, are discussed in \cite{HintzNonstat}. Ultimately, if one can propagate control of $u$ (from global propagation through $I^+$ as well as from initial data) to a punctured neighborhood of $\cR_{\rm out}$, part~\eqref{ItMEFwOut} applies and yields full control of $u$ at $\scri^+$ in the differential order sense.

Setting up a solvability theory for $P$ requires estimates for the adjoint $P^*$---defined with respect to a density~\eqref{EqMEDensity} and any choice of positive definite fiber inner product on $E$---in which we propagate regularity in the opposite direction:

\begin{thm}[Propagation through $\scri^+$: backward direction]
\label{ThmMEBw}
  Let $\tilde s\in\R$ and $\tilde\alpha=(\tilde\alpha_0,2\tilde\alpha_{\!\scri},\tilde\alpha_+)$, and put $\tilde\alpha'=\tilde\alpha+(2,2,2)$. Suppose $\tilde u\in\Heb^{-\infty,\tilde\alpha}(M)$, and let $P^*\tilde u=\tilde f\in\Heb^{-\infty,\tilde\alpha'}(M)$. Then $\WFeb^{\tilde s,\tilde\alpha}(\tilde u)\subset\Sigma\cup\WFeb^{\tilde s-2,\tilde\alpha'}(\tilde f)$. Moreover:
  \begin{enumerate}
  \item\label{ItMEBwThru}{\rm (Propagation through $\scri^+$.)} Suppose that
    \begin{equation}
    \label{EqMEBwThru}
      \tilde\alpha_{\!\scri}>-\frac12-\ubar p_1,\qquad
      \WFeb^{\tilde s-1,\tilde\alpha'}(\tilde f)\cap\Seb^*_{\scri^+}M\subset\cW_{\rm c}\cup\cW_{\rm in}.
    \end{equation}
    Then $\WFeb^{\tilde s,\tilde\alpha}(\tilde u)\cap\Seb^*_{\scri^+}M\subset\cW_{\rm c}\cup\cW_{\rm in}$.
  \item\label{ItMEBwOut}{\rm (Full control at $\scri^+$.)} Suppose that in addition to~\eqref{EqMEBwThru}, we have
    \begin{equation}
    \label{EqMEBwOutThr}
      \tilde\alpha_0<\tilde\alpha_{\!\scri}-\frac12<\tilde\alpha_+,\qquad
      \tilde s<\frac12-\tilde\alpha_0+2\tilde\alpha_{\!\scri}+\ubar p_1,
    \end{equation}
    and that for an open neighborhood $\cU\subset\Seb^*M$ of $\cR_{\rm in,+}$, we have
    \begin{equation}
    \label{EqMEBwOut}
    \begin{alignedat}{3}
      &\WFeb^{\tilde s,\tilde\alpha}(\tilde u)&&\cap\,(\cU\cap\Seb^*_{(I^+)^\circ}M)&&=\emptyset, \\
      &\WFeb^{\tilde s-1,\tilde\alpha'}(\tilde f)&&\cap\,\Seb^*_{\scri^+}M&&=\emptyset.
    \end{alignedat}
    \end{equation}
    Then $\WFeb^{\tilde s,\tilde\alpha}(\tilde u)\cap\Seb^*_{\scri^+}M=\emptyset$.
  \end{enumerate}
\end{thm}

\begin{rmk}[Orders and duality]
\label{RmkMEBwDual}
  Given $s,\alpha,\alpha'$ as in Theorem~\ref{ThmMEFw}, the dual space of the space $\Heb^{s-1,\alpha'}$ (in which $f$ is estimated) is $\Heb^{\tilde s,\tilde\alpha}$ where $\tilde s=-s+1$ and $\tilde\alpha=(\tilde\alpha_0,2\tilde\alpha_{\!\scri},\tilde\alpha_+):=-\alpha'=-\alpha-(2,2,2)$ (thus $\tilde\alpha_0=-\alpha_0-2$, $\tilde\alpha_{\!\scri}=-\alpha_{\!\scri}-1$, $\tilde\alpha_+=-\alpha_+-2$). The threshold condition in~\eqref{EqMEBwThru} is then equivalent to~\eqref{EqMEFwOutThr}, and the threshold conditions~\eqref{EqMEBwOutThr} for $\tilde s,\tilde\alpha_0,\tilde\alpha_{\!\scri},\tilde\alpha_+$ are equivalent to~\eqref{EqMEFwThruThr} (except $s_0$ does not enter in~\eqref{EqMEBwOut}).
\end{rmk}

The wave front set statements are proved using positive commutator arguments, and are thus qualitative versions of quantitative statements. The quantitative version of Theorem~\ref{ThmMEFw}\eqref{ItMEFwThru}, restricted to $\Sigma^+$ for definiteness, is the following. Suppose $B,E,E_{\rm c},W\in\Psieb^0(M)$ satisfy the following conditions:
\begin{itemize}
\item $\WFeb'(B)\cap\cR_{\rm out}^+=\emptyset$;
\item all backward null-bicharacteristics starting at $\WFeb'(B)\cap\Sigma^+$ tend to $\cR_{\rm in,+}^+\cup\cR_{\rm in,-}^+\cup\cR_{\rm c}^+$ or enter $\Elleb(E)$ in finite time, all while remaining inside $\Elleb(W)$;
\item $W$ is elliptic also \emph{at} $\cR_{\rm in,+}^+\cup\cR_{\rm in,-}^+\cup\cR_{\rm c}^+$;
\item $E_{\rm c}$ is elliptic at $\cR_{\rm c}^+$.
\end{itemize}
Let $N\in\R$. Then under the assumptions~\eqref{EqMEFwThruThr}, the estimate
\begin{equation}
\label{EqMEFwEst}
  \|B u\|_{\Heb^{s,\alpha}(M)} \leq C\Bigl(\|W P u\|_{\Heb^{s-1,\alpha'}(M)} + \|E u\|_{\Heb^{s,\alpha}(M)} + \|E_{\rm c}u\|_{\Heb^{s_0,\alpha}(M)} + \|u\|_{\Heb^{-N,\alpha}(M)}\Bigr)
\end{equation}
holds in the strong sense that if all quantities on the right are finite, then so is the left hand side, and the estimate holds. While the estimate~\eqref{EqMEFwEst} can be recovered from Theorem~\ref{ThmMEFw}\eqref{ItMEFwThru} via the closed graph theorem \cite[\S4]{VasyMinicourse}, the direct proof via positive commutators provides control on the constant $C$ in terms of norms of the coefficients of $P$ which, in principle, can be made explicit; this is important in applications to quasilinear equations, see e.g.\ \cite{HintzQuasilinearDS}. We leave statements of quantitative versions of the other parts of the above Theorems to the reader.

In order to prove Theorems~\ref{ThmMEFw} and \ref{ThmMEBw}, we work microlocally near each of the radial sets separately. Throughout the remainder of this section, we shall assume
\begin{equation}
\label{EqMEAssm}
\begin{gathered}
  s\in\R,\qquad \alpha=(\alpha_0,2\alpha_{\!\scri},\alpha_+),\qquad \alpha'=(\alpha'_0,2\alpha'_{\!\scri},\alpha'_+):=\alpha+(2,2,2), \\
  \tilde s\in\R,\qquad \tilde\alpha=(\tilde\alpha_0,2\tilde\alpha_{\!\scri},\tilde\alpha_+),\qquad \tilde\alpha'=(\tilde\alpha'_0,2\tilde\alpha'_{\!\scri},\tilde\alpha'_+):=\tilde\alpha+(2,2,2), \\
  u\in\Heb^{-\infty,\alpha}(M),\qquad P u=f\in\Heb^{-\infty,\alpha'}(M), \\
  \tilde u\in\Heb^{-\infty,\tilde\alpha}(M),\qquad P^*\tilde u=\tilde f\in\Heb^{-\infty,\tilde\alpha'}(M).
\end{gathered}
\end{equation}
For conciseness, we only state the qualitative wave front set versions of the radial point estimates near each of the radial sets.

\begin{lemma}[Propagation at $\cR_{\rm c}$]
\label{LemmaMEc}
  Let $\cU\subset\Seb^*M$ be an open neighborhood of $\cR_{\rm c}$. 
  \begin{enumerate}
  \item\label{ItMEcFw}{\rm (Forward propagation.)} Let $s>s_0>\half-\alpha_0+2\alpha_{\!\scri}-\ubar p_1$. Suppose that $\WFeb^{s-1,\alpha'}(f)\cap\cR_{\rm c}=\emptyset$ and $\WFeb^{s_0,\alpha}(u)\cap\cR_{\rm c}=\emptyset$. If $\WFeb^{s,\alpha}(u)\cap(\cU\cap\Seb^*_{(I^0)^\circ}M)=\emptyset$, then we obtain $\WFeb^{s,\alpha}(u)\cap\cR_{\rm c}=\emptyset$.
  \item\label{ItMEcBw}{\rm (Backward propagation.)} Suppose that $\tilde s<\half-\tilde\alpha_0+2\tilde\alpha_{\!\scri}+\ubar p_1$, $\WFeb^{\tilde s-1,\tilde\alpha'}(\tilde f)\cap\cR_{\rm c}=\emptyset$. If $\WFeb^{\tilde s,\tilde\alpha}(\tilde u)\cap(\cU\cap\Seb^*_{\scri^+}M\setminus\cR_{\rm c})=\emptyset$, then $\WFeb^{\tilde s,\tilde\alpha}(\tilde u)\cap\cR_{\rm c}=\emptyset$.
  \end{enumerate}
\end{lemma}

\begin{lemma}[Propagation at $\cR_{\rm in,-}$]
\label{LemmaMEinm}
  Let $\cU\subset\Seb^*M$ be an open neighborhood of $\cR_{\rm in,-}$. 
  \begin{enumerate}
  \item\label{ItMEinmFw}{\rm (Forward propagation.)} Suppose that $\alpha_{\!\scri}<\alpha_0+\half$ and $\WFeb^{s-1,\alpha'}(f)\cap\cR_{\rm in,-}=\emptyset$. If $\WFeb^{s,\alpha}(u)\cap(\cU\cap\Seb^*_{I^0}M\setminus\cR_{\rm in,-})=\emptyset$, then $\WFeb^{s,\alpha}(u)\cap\cR_{\rm in,-}=\emptyset$.
  \item\label{ItMEinmBw}{\rm (Backward propagation.)} Suppose that $\tilde\alpha_0<\tilde\alpha_{\!\scri}-\half$ and $\WFeb^{\tilde s-1,\tilde\alpha'}(\tilde f)\cap\cR_{\rm in,-}=\emptyset$. If $\WFeb^{\tilde s,\tilde\alpha}(\tilde u)\cap(\cU\cap\Seb^*_{(\scri^+)^\circ}M)=\emptyset$, then $\WFeb^{\tilde s,\tilde\alpha}(\tilde u)\cap\cR_{\rm in,-}=\emptyset$.
  \end{enumerate}
\end{lemma}

\begin{lemma}[Propagation at $\cR_{\rm in,+}$]
\label{LemmaMEinp}
  Let $\cU\subset\Seb^*M$ be an open neighborhood of $\cR_{\rm in,+}$. 
  \begin{enumerate}
  \item\label{ItMEinpFw}{\rm (Forward propagation.)} Suppose that $\alpha_+<\alpha_{\!\scri}-\half$ and $\WFeb^{s-1,\alpha'}(f)\cap\cR_{\rm in,+}=\emptyset$. If $\WFeb^{s,\alpha}(u)\cap(\cU\cap\Seb^*_{(\scri^+)^\circ}M)=\emptyset$, then $\WFeb^{s,\alpha}(u)\cap\cR_{\rm in,+}=\emptyset$.
  \item\label{ItMEinpBw}{\rm (Backward propagation.)} Suppose that $\tilde\alpha_{\!\scri}<\tilde\alpha_++\half$ and $\WFeb^{\tilde s-1,\tilde\alpha'}(\tilde f)\cap\cR_{\rm in,+}=\emptyset$. If $\WFeb^{\tilde s,\tilde\alpha}(\tilde u)\cap(\cU\cap\Seb^*_{I^+}M\setminus\cR_{\rm in,+})=\emptyset$, then $\WFeb^{\tilde s,\tilde\alpha}(\tilde u)\cap\cR_{\rm in,+}=\emptyset$.
  \end{enumerate}
\end{lemma}

\begin{lemma}[Propagation at $\cR_{\rm out}$]
\fakephantomsection
\label{LemmaMEout}
  \begin{enumerate}
  \item\label{ItMEoutFw}{\rm (Forward propagation.)} Suppose that $\alpha_{\!\scri}<-\half+\ubar p_1$ and $\WFeb^{s-1,\alpha'}(f)\cap\cR_{\rm out}=\emptyset$. Let $\cU\subset\Seb^*M$ be an open neighborhood of $\cR_{\rm out}$. If $\WFeb^{s,\alpha}(u)\cap(\cU\setminus\cR_{\rm out})=\emptyset$, then $\WFeb^{s,\alpha}(u)\cap\cR_{\rm out}=\emptyset$.
  \item\label{ItMEoutFwLoc}{\rm (Localized forward propagation near $I^0$.)} Using the coordinates~\eqref{EqOGeoCoordI0} for any fixed $T\in\R$, let $0<\rho_{0,-}<\rho_{0,+}<1$ and suppose $\cU^\flat,\cU\subset\Seb^*M\cap\{\rho_0<\rho_{0,+}\}$ are open sets containing $\cR_{\rm out}\cap\{\rho_0<\rho_{0,-}\}$, and $\overline{\cU^\flat}\subset\cU$. If $\alpha_{\!\scri}<-\half+\ubar p_1$, $\WFeb^{s,\alpha}(u)\cap(\cU\setminus\cR_{\rm out})=\emptyset$, and $\WFeb^{s-1,\alpha'}(f)\cap\Seb^*_{\scri^+}M\subset\cR_{\rm out}\setminus\ol\cU$, then $\WFeb^{s,\alpha}(u)\cap\Seb^*_{\scri^+}M\subset\cR_{\rm out}\setminus\ol{\cU^\flat}$.
  \item\label{ItMEoutBw}{\rm (Backward propagation.)} Suppose that $\tilde\alpha_{\!\scri}>-\half-\ubar p_1$ and $\WFeb^{\tilde s-1,\tilde\alpha'}(\tilde f)\cap\cR_{\rm out}=\emptyset$. Then $\WFeb^{\tilde s,\tilde\alpha}(\tilde u)\cap\cR_{\rm out}=\emptyset$.
  \item\label{ItMEoutBwLoc}{\rm (Localized backward propagation near $I^+$.)} Using the coordinates~\eqref{EqOGeoCoordIp} for any fixed $T\in\R$, let $0<\bar\rho_+<1$ and let $\cU\subset\Seb^*M$ be an open set with $\cU\cap\cR_{\rm out}=\{\rho_+<\bar\rho_+\}\cap\cR_{\rm out}$. If $\tilde\alpha_{\!\scri}>-\half-\ubar p_1$ and $\WFeb^{\tilde s-1,\tilde\alpha'}(\tilde f)\cap\cU=\emptyset$, then $\WFeb^{\tilde s,\tilde\alpha}(\tilde u)\cap(\cU\cap\cR_{\rm out})=\emptyset$.
  \end{enumerate}
\end{lemma}

\begin{rmk}[Threshold conditions and $\ubar p_1$]
\label{RmkMEThr}
  The threshold conditions for propagation at $\cR_{\rm in,-}$ (Lemma~\ref{LemmaMEinm}), resp.\ $\cR_{\rm in,+}$ (Lemma~\ref{LemmaMEinp}) only involve the (relative) weights at $I^0$ and $\scri^+$, resp.\ $\scri^+$ and $I^+$. This is consistent with the fact that these radial sets control the uniform behavior (i.e.\ the amplitude) of waves passing \emph{through} (a neighborhood of) $\scri^+$ (i.e.\ they are essentially transported across $\scri^+$ without any change to their amplitude). On the other hand, the radial point estimate at $\cR_{\rm out}$ (Lemma~\ref{LemmaMEout}) naturally \emph{does} involve $\ubar p_1$, as this quantity determines the decay rate at $\scri^+$ out waves tending \emph{towards} $\scri^+$ (cf.\ the first factor in each line of~\eqref{EqOpNorm}). The radial set $\cR_{\rm c}$ (Lemma~\ref{LemmaMEc}) controls waves featuring both aspects (cf.\ its unstable manifold in Figure~\ref{FigMF}).
\end{rmk}

\begin{proof}[Proof of Theorems~\usref{ThmMEFw} and \usref{ThmMEBw}, given Lemmas~\usref{LemmaMEc}--\usref{LemmaMEout}]
  Elliptic regularity in the edge-b-cal\-cu\-lus immediately gives $\WFeb^{s,\alpha}(u)\subset\Sigma\cap\WFeb^{s-2,\alpha'}(f)$. For part~\eqref{ItMEFwThru}, we first apply Lemma~\ref{LemmaMEc}\eqref{ItMEcFw} to get control of $u$ microlocally at $\cR_{\rm c}$. We can then propagate this control using standard real principal propagation along the unstable manifolds of $\cR_{\rm c}$, thus controlling $u$ microlocally in $\Seb^*_{\scri^+\setminus I^+}M\setminus(\cR_{\rm out}\cup\cW_{\rm in})$. In particular, the assumptions of Lemma~\ref{LemmaMEinm}\eqref{ItMEinmFw} are satisfied for a sufficiently small neighborhood of $\cR_{\rm in,-}$. Hence, we obtain control at $\cR_{\rm in,-}$, which can be propagated along $\cW_{\rm in}$. We can then apply Lemma~\ref{LemmaMEinp}\eqref{ItMEinpFw} to get control at $\cR_{\rm in,+}$; this can then be further propagated over $\scri^+\cap I^+$ to any punctured neighborhood of $\cR_{\rm out}$. Altogether, this proves the absence of $\Heb^{s,\alpha}$-wave front set of $u$ over $\scri^+$ except at $\cR_{\rm out}$. Part~\eqref{ItMEFwOut} of Theorem~\ref{ThmMEFw} is now an immediate application of Lemma~\ref{LemmaMEout}\eqref{ItMEoutFw}.

  Theorem~\ref{ThmMEBw}\eqref{ItMEBwThru} follows from Lemma~\ref{LemmaMEout}\eqref{ItMEoutBw} and real principal type propagation in $\Seb^*_{\scri^+}M\setminus(\cW_{\rm in}\cup\cW_{\rm c})$. For the proof of part~\eqref{ItMEBwOut}, one in addition applies Lemma~\ref{LemmaMEinp}\eqref{ItMEinpBw} to get control at $\cR_{\rm in,+}$, which propagates to a punctured neighborhood of $\cR_{\rm in,-}$ over $\scri^+\setminus I^0$. Propagation into $\cR_{\rm in,-}$ is accomplished by Lemma~\ref{LemmaMEinm}\eqref{ItMEinmBw}, from where one can then propagate regularity of $u$ to a punctured neighborhood of $\cR_{\rm c}$. An application of Lemma~\ref{LemmaMEc}\eqref{ItMEcBw} completes the argument.
\end{proof}

We now turn to the proofs of Lemmas~\ref{LemmaMEc}--\ref{LemmaMEout}.

\begin{proof}[Proof of Lemma~\usref{LemmaMEc}]
  We work near $\cR_{\rm c}^+$; the arguments near $\cR_{\rm c}^-$ are completely analogous. We omit the weight at $I^+$ from the notation. Let $\check\alpha=(\check\alpha_0,2\check\alpha_{\!\scri}):=\alpha+(1,1)=(\alpha_0+1,2\alpha_{\!\scri}+1)$. We consider a commutant
  \begin{equation}
  \label{EqMEcCommOp}
    \check A=\Op_\ebop(\check a)\in\Psieb^{s-\frac12,\check\alpha},\qquad A = \check A^*\check A,
  \end{equation}
  with $\check a$ defined momentarily, and shall estimate the $L^2$-pairing\footnote{Regarding the weights of $\sC$, note that $(-2,-2)+2\check\alpha=2\alpha$ indeed.}
  \begin{equation}
  \label{EqMEcPair}
    \Im\la P u,A u\ra = \la\sC u,u\ra,\qquad \sC=\frac{i}{2}[P,A] + \frac{P-P^*}{2 i}A\in\Psieb^{2 s,2\alpha}(M),
  \end{equation}
  in two different ways. Here, we define the $L^2$ inner product (and the adjoint $P^*$) with respect to the volume density~\eqref{EqMEDensity} and a fiber inner product $h_E$ on $E$ which is almost optimal, that is,
  \begin{equation}
  \label{EqMEcInner}
    \ubar p_1(h_E)\in(\ubar p_1-\eps_0,\ubar p_1]
  \end{equation}
  for any fixed $\eps_0>0$. Using the Cauchy--Schwarz inequality and elliptic estimates, the left hand side of~\eqref{EqMEcPair} can be bounded, for any $\delta>0$, by
  \begin{equation}
  \label{EqMEcPairLHS}
    \Im\la P u,A u\ra \geq -|\la\check A^*P u,\check A u\ra| \geq -C\delta^{-1}\|W P u\|_{\Heb^{s-1,\alpha'}}^2 - \delta\|\check A u\|_{\Heb^{\frac12,(-1,-1)}}^2 - C\|u\|_{\Heb^{-N,\alpha}}^2
  \end{equation}
  for a constant $C$ (independent of $u,f,\delta$), where $W\in\Psieb^0$ is any fixed operator that is elliptic on $\WFeb'(A)$.

  We next compute the principal symbol of the `commutator' $\sC$. We employ the coordinates $\xi,\eta_j,\zeta$ from~\eqref{EqMFCoordI0} in the base and the coordinates $\rho_\infty,\hat\eta,\hat\zeta$ from~\eqref{EqMFRcCoord} in the compactified fibers. Recall $G_\ebop^0$ from~\eqref{EqMFSymbI0}, and define, for $\delta_{\!\scri}\in(0,1)$ specified below,
  \[
    \check a=\rho_\infty^{-s+\frac12}\rho_0^{-\check\alpha_0}x_{\!\scri}^{-2\check\alpha_{\!\scri}}\chi(\rho_0)\chi\Bigl(\frac{x_{\!\scri}}{\delta_{\!\scri}}\Bigr)\chi(|\hat\zeta|^2)\chi(\rho_\infty^2 G_\ebop^0);
  \]
  here, for $\eps>0$ fixed below, the cutoff function $\chi\in\CIc((-\eps,\eps))$ is identically $1$ near $0$, and on $[0,\eps)$ satisfies $\chi'\leq 0$, $\sqrt{-\chi'\chi}\in\CI$. Thus, $\check a$ is supported in any given conic neighborhood of $\cR_{\rm c}^+$ upon choosing $\eps>0$ small enough. In view of~\eqref{EqMFLinC}, we have
  \begin{equation}
  \label{EqMEcCommSymb}
  \begin{split}
    \check a H_{G_\ebop^0}\check a &= \Bigl(-s+\frac12-\check\alpha_0+2\check\alpha_{\!\scri}+\cO(\eps)+\cO(\eps^{\ell_0}\eps^{2\ell_{\!\scri}})\Bigr)\rho_\infty^{-1}\check a^2 \\
      &\qquad + \rho_\infty^{-2 s}\rho_0^{-2\check\alpha_0}x_{\!\scri}^{-4\check\alpha_{\!\scri}}\bigl(-b_0^2-b_{\rm c}^2+b_{\!\scri}^2 + q\rho_\infty^2 G_\ebop^0\bigr).
  \end{split}
  \end{equation}
  Here $b_0,b_{\rm c},b_{\!\scri},q\in S^0+\cA^{(\ell_0,2\ell_{\!\scri})}S^0$ are symbols on $\Teb^*M$; the symbol $b_0$ arises from differentiation of $\chi(\rho_0)$ (which by~\eqref{EqMFLinC} gives a negative square when $\eps>0$, hence the support of $\check a$, is small enough); the symbol $b_{\rm c}$ from differentiation of $\chi(|\hat\zeta|^2)$ (which for sufficiently small $\eps>0$ gives a negative square at $x_{\!\scri}=0$, and thus on $\supp\check a$ when $\delta_{\!\scri}$ is sufficiently small); moreover, $b_{\!\scri}$ arises from differentiation of $\chi(\frac{x_{\!\scri}}{\delta_{\!\scri}})$, and $q$ from differentiation of the localizer $\chi(\rho_\infty^2 G_\ebop^0)$ to the characteristic set. In terms of the principal symbol $G=\rho_0^2 x_{\!\scri}^2 G_\ebop^0$ of $P$, this means
  \begin{equation}
  \label{EqMEcCommSymb2}
    \sigmaeb^{2 s,2\alpha}\Bigl(\frac{i}{2}[P,A]\Bigr) = \check a H_G\check a = \rho_0^2 x_{\!\scri}^2\check a H_{G_\ebop^0}\check a + \rho_\infty^{-2 s}\rho_0^{-2\alpha_0}x_{\!\scri}^{-4\alpha_{\!\scri}}\tilde q\cdot\rho_\infty^2\rho_0^{-2}x_{\!\scri}^{-2}G
  \end{equation}
  with $\tilde q\in S^0+\cA^{(\ell_0,2\ell_{\!\scri})}S^0$. By definition of $\check\alpha$, the expression~\eqref{EqMEcCommSymb2} is a symbol of class $S^{2 s,2\alpha}+\cA^{(\ell_0,2\ell_{\!\scri})}S^{2 s,2\alpha}$.

  Turning to the second term of $\sC$ in~\eqref{EqMEcPair}, we need to evaluate $\sigmaeb^{1,(-2,-2,-2)}(\frac{P-P^*}{2 i})$ near $\cR_{\rm c}^+$. In the notation of Definition~\ref{DefOp}, we may replace $P$ here by $P_0$, as the operator $\tilde P$ only contributes an element of $\cA^{(2+\ell_0,2+2\ell_{\!\scri})}S^1$. We have $P_0=\rho_0^2 x_{\!\scri}^2 P_{0,\ebop}$ where the normal operator of $2 P_{0,\ebop}$ is given by the first line of~\eqref{EqOpNorm}. Recall from~\eqref{EqMEDensity} that we are working with an integration density $\mu\rho_0^{-n-1}x_{\!\scri}^{-2 n}|\frac{\dd\rho_0}{\rho_0}\frac{\dd x_{\!\scri}}{x_{\!\scri}}\dd y|$, where $\mu>0$ is smooth. The normal operator of $\rho_0^{-2}x_{\!\scri}^{-2}(P_0-P_0^*)$ is then equal to that of $P_{0,\ebop}-\rho_0^{-2}x_{\!\scri}^{-2}P_{0,\ebop}^*\rho_0^2 x_{\!\scri}^2$ and thus, after a short calculation (see also Example~\ref{ExOpMink} and note that $\Box_{g_0}$ is symmetric with respect to $|\dd g_0|$), given by
  \begin{align*}
    {}^\eop N_{\scri^+,y_0}\Bigl(\rho_0^{-2}x_{\!\scri}^{-2}\frac{P-P^*}{2 i}\Bigr) &= -\frac12(p_1+p_1^*)(x_{\!\scri} D_{x_{\!\scri}}-2\rho_0 D_{\rho_0}) + \tilde N, \\
      &\hspace{5em} \tilde N\in(\CI+\cA^{\ell_0})(\scri^+_{y_0}\setminus I^+;\End(E_{y_0})).
  \end{align*}
  Therefore, over $\scri^+$,
  \begin{equation}
  \label{EqMEcImag}
    \sigmaeb^1\Bigl(\rho_0^{-2}x_{\!\scri}^{-2}\frac{P-P^*}{2 i}\Bigr) = -\frac12(p_1+p_1^*)(\xi-2\zeta) = -\frac12(p_1+p_1^*)\rho_\infty^{-1}(1-2\hat\zeta).
  \end{equation}
  Recall that at $\cR_{\rm c}^+$, we have $\hat\zeta=0$; in combination with~\eqref{EqMEcCommSymb}, we therefore obtain
  \begin{equation}
  \label{EqMEcCSymb}
  \begin{split}
    \sigmaeb^{2 s,2\alpha}(\sC) &= \Bigl(-s+\frac12-\check\alpha_0+2\check\alpha_{\!\scri}-\frac12(p_1+p_1^*)+\cO(\eps)+\cO(\eps^{\ell_0}\eps^{2\ell_{\!\scri}})\Bigr)\rho_0^2 x_{\!\scri}^2\rho_\infty^{-1}\check a^2 \\
      &\qquad + \rho_\infty^{-2 s}\rho_0^{-2\alpha_0}x_{\!\scri}^{-4\alpha_{\!\scri}}\bigl(-b_0^2-b_{\rm c}^2+b_{\!\scri}^2+(q+\tilde q)\rho_\infty^2 G_\ebop^0\bigr).
  \end{split}
  \end{equation}
  The first line is negative (and thus has the same sign as the terms $-b_0^2$ and $-b_{\rm c}^2$ in the second line) under the condition
  \begin{equation}
  \label{EqMEcThr}
    s+\alpha_0-2\alpha_{\!\scri}>\frac12-\ubar p_1\qquad\text{on}\ \supp\check a
  \end{equation}
  in the sense of quadratic forms on $E$. By assumption, this condition holds at $\cR_{\rm c}$ provided we work with small enough $\eps_0>0$ (see \eqref{EqMEcInner}), and therefore also in a sufficiently small neighborhood of $\cR_{\rm c}$ and thus on $\supp\check a$ if we choose $\eps>0$ small enough.

  This symbolic calculation can be turned into an estimate in the standard manner; we sketch this briefly. We give ourselves some extra room $\delta\in(0,s+\alpha_0-2\alpha_{\!\scri}-\half+\ubar p_1)$, and denote by $B_0,B_{\rm c},B_{\!\scri},Q\in\cA^{(-\alpha_0,-2\alpha_{\!\scri})}\Psieb^s$ quantizations of $\rho_\infty^{-s}\rho_0^{-\alpha_0}x_{\!\scri}^{-2\alpha_{\!\scri}}$ times $b_0,b_{\rm c},b_{\!\scri},q+\tilde q$; we further let $\Lambda\in\Psieb^{\frac12,(-1,-1)}$ with principal symbol $\rho_0 x_{\!\scri}\rho_\infty^{-1/2}$ near $\supp\check a$. Then we can write
  \[
    \sC = -\delta(\Lambda\check A)^*(\Lambda\check A) - B^*B - B_0^*B_0 - B_{\rm c}^*B_{\rm c} + B_{\!\scri}^*B_{\!\scri} + R,
  \]
  where $B\in\cA^{(-\alpha_0,-2\alpha_{\!\scri})}\Psieb^s$ is elliptic at $\cR_{\rm c}$ and such that the sum of the first two terms on the right here is a quantization of the first line in~\eqref{EqMEcCSymb}, and $R\in\cA^{(-\alpha_0,-2\alpha_{\!\scri})}\Psieb^{2 s-1}$ (with $\WFeb'(R)\subset\WFeb'(\check A)$) is the error term not captured by our principal symbol calculations. Plugging this into~\eqref{EqMEcPair} and using~\eqref{EqMEcPairLHS}, we obtain
  \begin{align*}
    &\delta\|\check A u\|_{\Heb^{\frac12,(-1,-1)}}^2 + \|B u\|_{L^2}^2 + \|B_0 u\|_{L^2}^2 + \|B_{\rm c}u\|_{L^2}^2 \\
    &\qquad\leq C\delta^{-1}\|W P u\|_{\Heb^{s-1,\alpha'}}^2 + \|B_{\!\scri} u\|_{L^2}^2 + \delta\|\check A u\|_{\Heb^{\frac12,(-1,-1)}}^2 + C\|u\|_{\Heb^{-N,\alpha}}^2 + |\la R u,u\ra|.
  \end{align*}
  Upon cancelling the terms involving $\check A$, dropping the third and fourth terms on the left, and estimating
  \[
    |\la R u,u\ra|\leq C\|W u\|_{\Heb^{s-\frac12,\alpha}}^2+C\|u\|_{\Heb^{-N,\alpha}}^2
  \]
  using elliptic regularity, this gives the estimate
  \begin{equation}
  \label{EqMEcFinal}
    \|B u\|_{L^2} \lesssim \|W P u\|_{\Heb^{s-1,\alpha'}} + \|B_{\!\scri} u\|_{L^2} + \|W u\|_{\Heb^{s-\frac12,\alpha}} + \|u\|_{\Heb^{-N,\alpha}}.
  \end{equation}
  This gives quantitative $\Heb^{s,\alpha}$-control of $u$ near $\cR_{\rm c}$ (where $B$ is elliptic), assuming $\Heb^{s-\frac12,\alpha}$-control of $u$ on the elliptic set of $W$. Note that we can take $\WFeb'(W)$ to be contained in any specified neighborhood of $\cR_{\rm c}^+$ by choosing $\eps>0$ small enough.
  
  In order to make sense of the integrations by parts in this argument, one needs to use a standard regularization argument \cite[\S4]{VasyMinicourse}; we only give a very brief sketch. One replaces $\check a$ by $\check a_r=(1+r\rho_\infty^{-1})^{-\gamma}\check a$, $r\in(0,1]$, where $\gamma>0$. For fixed $r>0$, this effectively replaces $s$ by $s-\gamma$ in the above symbolic calculations, hence the amount $\gamma$ of regularization one can do is limited in view of the requirement~\eqref{EqMEcThr} for $s_0=s-\gamma$ in place of $s$, and one needs to assume that $\WFeb^{s_0,\alpha}(u)\cap\cR_{\rm c}=\emptyset$ in order to justify the integrations by parts. For such $\gamma$ then, one obtains an $r$-dependent version of the estimate~\eqref{EqMEcFinal}, with an implicit constant that is uniform as $r\searrow 0$; upon letting $r\to 0$, a standard functional analysis argument using the weak compactness of the unit ball in $L^2$ implies that $B u\in L^2$ (i.e.\ one gets to \emph{conclude} regularity of $u$ on $\Elleb(B)$), and the estimate~\eqref{EqMEcFinal} holds. This finishes the proof of part~\eqref{ItMEcFw} of the Lemma.

  Turning to part~\eqref{ItMEcBw}, the roles of $P$ and $P^*$ are now reversed. We remark that the choice of fiber inner product on $E$ used to define the adjoint $P^*$ in~\eqref{EqMEAssm} is immaterial, as the adjoints with respect to any two choices are related via conjugation by a smooth bundle isomorphism on $E$; thus, we may as well use the near-optimal inner product already used in the first part of the proof. Thus, in the calculation~\eqref{EqMEcCSymb} (but now using $\tilde s,\tilde\alpha_0,\tilde\alpha_{\!\scri}$ instead of $s,\alpha_0,\alpha_{\!\scri}$, and writing $(\check{\tilde\alpha}_0,2\check{\tilde a}_{\!\scri})=\tilde\alpha+(1,1)$), the terms involving $p_1$ come with the opposite sign; and propagating backwards, we now want $-\tilde s+\frac12-\check{\tilde\alpha}_0+2\check{\tilde\alpha}_{\!\scri}+\frac12(p_1+p_1^*)$ to be \emph{positive} (thus to have the same sign as the term $b_{\!\scri}^2$, but the opposite sign as $-b_0^2,-b_{\rm c}^2$, which are the places where we need to assume a priori control). This leads to the condition
  \[
    \tilde s + \tilde\alpha_0-2\tilde\alpha_{\!\scri} < \frac12+\ubar p_1.
  \]
  Note that since this condition remains valid upon decreasing $\tilde s$, the allowed amount of regularization is now unlimited; this is the reason for the absence of an a priori regularity assumption on the solution $\tilde u$ of $P^*\tilde u=\tilde f$ at $\cR_{\rm c}$.
\end{proof}

The proofs of Lemmas~\ref{LemmaMEinm}--\ref{LemmaMEout} are similar, hence we shall only explain the symbolic aspects.

\begin{proof}[Proof of Lemma~\usref{LemmaMEinm}]
  We shall work near $\cR_{\rm in,-}^+$ and use the coordinates~\eqref{EqMFCoordI0}, \eqref{EqMFRinmCoord} (with $\rho_\infty=\zeta^{-1}$ near, and $\hat\xi=2$ at, $\cR_{\rm in,-}^+$). We shall define $A$ as in~\eqref{EqMEcCommOp}, where now
  \[
    \check a=\rho_\infty^{-s+\frac12}\rho_0^{-\check\alpha_0}x_{\!\scri}^{-2\check\alpha_{\!\scri}}\chi(\rho_0)\chi\Bigl(\frac{x_{\!\scri}}{\delta_{\!\scri}}\Bigr)\chi((2-\hat\xi)^2)\chi(\rho_\infty^2 G_\ebop^0),
  \]
  where $\chi$ is a cutoff function (with support near $0$) and $0<\delta_{\!\scri}\ll 1$ as in the proof of Lemma~\ref{LemmaMEc}. Recalling~\eqref{EqMFLinInm}, we find
  \begin{align*}
    \check a H_{G_\ebop^0}\check a &= \bigl(-2\check\alpha_0+2\check\alpha_{\!\scri} + \cO(\eps) + \cO(\eps^{\ell_0}\eps^{2\ell_{\!\scri}})\bigr)\rho_\infty^{-1}\check a^2 \\
      &\qquad + \rho_\infty^{-2 s}\rho_0^{-2\check\alpha_0}x_{\!\scri}^{-4\check\alpha_{\!\scri}}\bigl( -b_0^2 + b_{\!\scri}^2 + b_{\rm in}^2 + q\rho_\infty G_\ebop^0 \bigr),
  \end{align*}
  where $b_0$, $b_{\!\scri}$, $b_{\rm in}$ arise by differentiation of $\chi(\rho_0)$, $\chi(\frac{x_{\!\scri}}{\delta_{\!\scri}})$, $\chi((2-\hat\xi)^2)$ respectively; for the latter, note that $2-\hat\xi\geq 0$ is monotonically increasing along $H_{G_\ebop^0}$ near $\cR_{\rm in,-}^+$. Moreover, by~\eqref{EqMEcImag}, we have, over $\scri^+$,
  \[
    \sigmaeb^1\Bigl(\rho_0^{-2}x_{\!\scri}^{-2}\frac{P-P^*}{2 i}\Bigr) = -\frac12(p_1+p_1^*)(\xi-2\zeta) = -\frac12(p_1+p_1^*)\rho_\infty^{-1}(\hat\xi-2).
  \]
  At $\cR_{\rm in,-}$, this vanishes. Adding these two contributions to $\sigmaeb^{2 s,2\alpha}(\sC)$ in the notation of~\eqref{EqMEcPair} gives, similarly to~\eqref{EqMEcCSymb},
  \begin{equation}
  \label{EqMEinmSymb}
  \begin{split}
    \sigmaeb^{2 s,2\alpha}(\sC) &= \bigl(-2\check\alpha_0+2\check\alpha_{\!\scri}+\cO(\eps)+\cO(\eps^{\ell_0}\eps^{2\ell_{\!\scri}})\bigr)\rho_0^2 x_{\!\scri}^2\rho_\infty^{-1}\check a^2 \\
      &\qquad + \rho_\infty^{-2 s}\rho_0^{-2\alpha_0}x_{\!\scri}^{-4\alpha_{\!\scri}}\bigl(-b_0^2+b_{\!\scri}^2+b_{\rm in}^2 + (q+\tilde q)\rho_\infty^2 G_\ebop^0 \bigr).
  \end{split}
  \end{equation}
  When propagating in the forward direction, we assume a priori control on $\supp b_{\!\scri}\cup\supp b_{\rm in}$, and obtain control at $\cR_{\rm in,-}^+$ (as well as on the elliptic set of $b_0$), provided that $-\check\alpha_0+\check\alpha_{\!\scri}<0$, or equivalently $\alpha_{\!\scri}<\alpha_0+\half$. The proof is now finished by quantizing the lower case symbols as before. As for the regularization argument, note that the condition on $\alpha_0,\alpha_{\!\scri}$ does not involve the differential order $s$, and hence we can do an arbitrary amount of regularization; thus, there is no a priori regularity requirement at $\cR_{\rm in,-}$ here. This proves the first part of the Lemma.

  For the second part concerning propagation in the backward direction, we now need $-\check{\tilde\alpha}_0+\check{\tilde\alpha}_{\!\scri}>0$ (in the same notation as in the previous proof) to make the main, i.e.\ first, term of~\eqref{EqMEinmSymb} (with tildes added) positive, matching the sign of $b_{\!\scri},b_{\rm in}$, whereas now the term $-b_0^2$ has the opposite sign. Thus, we need microlocal $\Heb^{s,\alpha}$ regularity of $u$ on $\supp b_0$ in order to conclude regularity at $\cR_{\rm in,-}$.
\end{proof}

\begin{proof}[Proof of Lemma~\usref{LemmaMEinp}]
  We work near $\cR_{\rm in,+}^+$ and use the coordinates~\eqref{EqMFCoordIp}, \eqref{EqMFRinpCoord} (with $\rho_\infty=\zeta^{-1}$). Using the same notation as in the proof of Lemma~\ref{LemmaMEc}, We use the commutant
  \[
    \check a = \rho_\infty^{-s+\frac12}x_{\!\scri}^{-2\check\alpha_{\!\scri}}\rho_+^{-\check\alpha_+}\chi\Bigl(\frac{x_{\!\scri}}{\delta_{\!\scri}}\Bigr)\chi(\rho_+)\chi((2-\hat\xi)^2)\chi(\rho_\infty^2 G_\ebop^+).
  \]
  Using~\eqref{EqOpNorm} (and omitting the weight at $I^0$), we compute the normal operator of the imaginary part of $P$ to be
  \begin{align*}
    {}^\eop N_{\scri^+,y_0}\Bigl(x_{\!\scri}^{-2}\rho_+^{-2}\frac{P-P^*}{2 i}\Bigr) &= \frac12(p_1+p_1^*)(x_{\!\scri} D_{x_{\!\scri}}-2\rho_+ D_{\rho_+}) + \tilde N, \\
      &\hspace{5em} \tilde N\in(\CI+\cA^{\ell_+})(\scri^+_{y_0}\setminus I^0;\End(E_{y_0})).
  \end{align*}
  Its principal symbol is $\frac12(p_1+p_1^*)(\xi-2\zeta)=\frac12(p_1+p_1^*)\rho_\infty^{-1}(\hat\xi-2)$; this vanishes at $\cR_{\rm in,+}^+$. Using~\eqref{EqMFLinInp}, we then find (in the notation~\eqref{EqMEcPair})
  \begin{align*}
    \sigmaeb^{2 s,2\alpha}(\sC) &= \bigl(-2\check\alpha_{\!\scri}+2\check\alpha_+ + \cO(\eps)+\cO(\eps^{2\ell_{\!\scri}}\eps^{\ell_+}) \bigr)x_{\!\scri}^2\rho_+^2\rho_\infty^{-1}\check a^2 \\
      &\qquad \rho_\infty^{-2 s}x_{\!\scri}^{-4\alpha_{\!\scri}}\rho_+^{-2\alpha_+}\bigl(-b_{\!\scri}^2 + b_+^2 - b_{\rm in}^2 + (q+\tilde q)\rho_\infty^2 G_\ebop^+\bigr).
  \end{align*}
  Here $b_{\!\scri}$, $b_+$, $b_{\rm in}$ arises from differentiation of $\chi(\frac{x_{\!\scri}}{\delta_{\!\scri}})$, $\chi(\rho_+)$, $\chi((2-\hat\xi)^2)$, respectively. For forward propagation, and under the threshold condition $-\check\alpha_{\!\scri}+\check\alpha_+<0$, i.e.\ $\alpha_+<\alpha_{\!\scri}-\half$, we need to assume a priori control on $\supp b_+$. Furthermore, the amount of regularization is again unlimited since the edge-b-regularity $s$ does not enter in the threshold condition. Following the arguments of the proof of Lemma~\ref{LemmaMEc} yields the desired result. The proof of the backward propagation result is similar.
\end{proof}

\begin{proof}[Proof of Lemma~\usref{LemmaMEout}]
  We only give details for the localized forward propagation statement, i.e.\ part~\eqref{ItMEoutFwLoc}. We again work in the coordinates~\eqref{EqMFCoordI0} on the edge-b-cotangent bundle over $\cU_0(T)$, and we use the coordinates~\eqref{EqMFRinmCoord} on the compactified fibers near $\cR_{\rm out}^+$ (where now $\rho_\infty=-\zeta^{-1}\geq 0$). Fix $\rho_{0,0}\in(\rho_{0,-},\rho_{0,+})$. Recall from Lemma~\ref{LemmaOGeoHyp} that upon fixing $C$ large and $\delta_0>0$ small, the function
  \[
    \tilde\rho_0 := \rho_0-(\rho_{0,0}+C x_{\!\scri}^{2\ell_{\!\scri}})
  \]
  has past timelike differential near its zero set in $x_{\!\scri}<2\delta_0$. For sufficiently small $\delta\in(0,\delta_0]$ and $c_0>0$, the set $\cU_0(T)\cap\{\tilde\rho_0\leq\pm c_0,\ x_{\!\scri}<2\delta\}$ contains $\ol{\cU^\flat}\cap\Seb^*_{\scri^+}M$ (for the `$-$' sign) and is contained in $\cU$ (for the `$+$' sign).

  We again use a cutoff function $\chi\in\CIc((-\eps,\eps))$ for small $0<\eps<2\delta_0$ which satisfies $\chi\equiv 1$ near $0$, and which on $[0,\eps)$ satisfies $\chi'\leq 0$ and $\sqrt{-\chi'\chi}\in\CI$. Let further $\tilde\chi\in\CI(\R)$ be identically $1$ on $(-\infty,-c_0]$, identically $0$ on $[c_0,\infty)$, and so that $\tilde\chi'\leq 0$, $\sqrt{-\tilde\chi'\tilde\chi}\in\CI$. We then use a commutant $A=\check A^*\check A$, $\check A=\Op_\ebop(\check a)$, where we set
  \[
    \check a=\rho_\infty^{-s+\frac12}\rho_0^{-\check\alpha_0}x_{\!\scri}^{-2\check\alpha_{\!\scri}}\tilde\chi(\tilde\rho_0)\chi\Bigl(\frac{x_{\!\scri}}{\delta_{\!\scri}}\Bigr)\chi(\hat\xi^2)\chi(\rho_\infty^2 G_\ebop^0).
  \]
  In light of~\eqref{EqMEcImag}, we have
  \[
    \sigmaeb^1\Bigl(\rho_0^{-2}x_{\!\scri}^{-2}\frac{P-P^*}{2 i}\Bigr)=-\frac12\rho_\infty^{-1}(p_1+p_1^*)(2-\hat\xi),
  \]
  which equals $-\rho_\infty^{-1}(p_1+p_1^*)$ at $\cR_{\rm out}^+$. Using~\eqref{EqMFLinOut1} and the past timelike nature of $\tilde\rho_0$, we can write (in the notation~\eqref{EqMEcPair})
  \begin{align*}
    \sigmaeb^{2 s,2\alpha}(\sC) &= \bigl(2\check\alpha_{\!\scri} - (p_1+p_1)^* + \cO(\eps) + \cO(\eps^{\ell_0}\eps^{2\ell_{\!\scri}})\bigr)\rho_\infty^{-1}\check a^2 \\
      &\qquad + \rho_\infty^{-2 s}\rho_0^{-2\alpha_0}x_{\!\scri}^{-4\alpha_{\!\scri}}\bigl(b_{\!\scri}^2 - \tilde b^2 + b_{\rm out}^2 + q\rho_\infty^2 G_\ebop^0\bigr),
  \end{align*}
  where $b_{\!\scri}$, $\tilde b$, $b_{\rm out}$, and $q\in\cA^{(0,0)}S^0$ (the conormal coefficients necessitated by the merely conormal regularity of $\tilde\rho_0$) arise from differentiation of $\chi(\frac{x_{\!\scri}}{\delta_{\!\scri}})$, $\tilde\chi(\tilde\rho_0)$, $\chi(\hat\xi^2)$, and $\chi(\rho_\infty^2 G_\ebop^0)$, respectively. Since we are considering forward propagation, the terms $b_{\!\scri}^2$ and $b_{\rm out}^2$ are supported where we make a priori assumptions, and the term $\tilde b^2$ has the same sign as the main term (hence in the $L^2$ estimate arising upon quantization can be dropped) provided $2\check\alpha_{\!\scri}-2\ubar p_1<0$, i.e.\ $\alpha_{\!\scri}<-\half+\ubar p_1$. Quantization and regularization proves part~\eqref{ItMEoutFwLoc}.
\end{proof}

\section{Higher b-regularity}
\label{Sb}

As already discussed in~\S\ref{SsIeb}, we do not develop any tools in this paper for the microlocal analysis of b-regularity near $\scri^+\subset M$ (using the notation introduced at the beginning of~\S\ref{SsOp}). Recall that in the coordinates $\rho_0,x_{\!\scri},y$ from~\eqref{EqOGeoCoordI0}, b-regularity on $M$ is iterated regularity under application of the b-vector fields $\rho_0\pa_{\rho_0}$, $x_{\!\scri}\pa_{x_{\!\scri}}$, and $\pa_{y^j}$, whereas edge-b-regularity captures regularity under $\rho_0\pa_{\rho_0}$, $x_{\!\scri}\pa_{x_{\!\scri}}$, and $x_{\!\scri}\pa_{y^j}$. However, b-regularity on $M$ (or equivalently on $\tilde M$) is both natural (e.g.\ on Minkowski space it exactly amounts to regularity under generators of the full Poincar\'e group, in particular under spatial rotations) and useful in applications; for instance, it is used explicitly as a crucial piece of information in the recovery of sharp asymptotics at $\scri^+$ in the stability results \cite[\S3.6]{HintzMink4Gauge}, \cite[\S5.1]{HintzVasyMink4}, and plays a central role (albeit in different terminology) in almost any analysis near null infinity; see \cite{LindbladAsymptotics} for a sharp example.

In this section, we thus demonstrate how to work microlocally with elements of edge-b-spaces which have additional integer amounts of b-regularity; the main tool will be simple commutator arguments in which we commute appropriate b-vector fields through the edge-b-operators of interest. Since in applications, higher b-regularity is mainly of interest for the forward solutions of wave equations, we only discuss regularity results for propagation in the forward (i.e.\ future causal) direction here.

In~\S\S\ref{SsbC}--\ref{SsbH}, we work in a general edge-b-setting and discuss the appropriate class of b-vector fields used for commutation, the class of edge-b-Sobolev spaces with extra b-regularity, and basic microlocal estimates (elliptic regularity, real principal type propagation) on such spaces. In~\S\ref{SsbI}, we return to the analysis near $\scri^+$ and prove microlocal propagation results at radial points with extra b-regularity.

\subsection{Commutator b-vector fields}
\label{SsbC}

As in~\S\ref{SsEBO}, we let $M$ denote a compact manifold with corners and embedded hypersurfaces $H_1,\ldots,H_N$, $N\geq 3$, with $H_2$ being the total space of a fibration $\phi\colon H_2\to Y$, where the base $Y$ is a smooth manifold without boundary, and the typical fiber is a manifold $Z$, possibly with (disconnected) boundary (e.g.\ a closed interval).

\begin{definition}[Commutator b-vector fields]
\label{DefbC}
  The space of \emph{commutator b-vector fields} $\cV_{[\bop]}(M)$ consists of all $V\in\Vb(M)$ so that $[V,W]\in\Veb(M)$ for all $W\in\Veb(M)$.
\end{definition}

Clearly, we have $\Veb(M)\subset\cV_{[\bop]}(M)$. To study this space further, denote by $\CI_\phi(H_2)=\phi^*\CI(Y)$ the space of fiber-constant functions on $H_2$; thus, $f\in\CI(H_2)$ lies in $\CI_\phi(H_2)$ if and only if $V|_{H_2}f=0$ for all $V\in\Veb(M)$. We also recall that an element $V\in\Vb(M)$ lies in $\Veb(M)$ if and only if for all $f\in\CI_\phi(H_2)$ we have $V|_{H_2}f=0$.

\begin{lemma}[A simple criterion]
\label{LemmabCCrit}
  Let $V\in\Vb(M)$. Then $V\in\cV_{[\bop]}(M)$ if and only if $V|_{H_2}\colon\CI_\phi(H_2)\to\CI_\phi(H_2)$.
\end{lemma}
\begin{proof}
  Let $V\in\Vb(M)$. Given $W\in\Veb(M)$, the condition $[V,W]\in\Veb(M)$ is equivalent to the vanishing, for all $f\in\CI_\phi(H_2)$, of
  \[
    [V,W]|_{H_2}f = V|_{H_2} W|_{H_2} f - W|_{H_2} V|_{H_2} f = -W|_{H_2}\bigl( V|_{H_2}f \bigr).
  \]
  The condition that this vanish for all $W\in\Veb(M)$ is in turn equivalent to $V|_{H_2}f\in\CI_\phi(H_2)$ for all $f\in\CI_\phi(H_2)$. This proves the Lemma.
\end{proof}

If the base $Y$ is a point, then $\Veb(M)=\Vb(M)=\cV_{[\bop]}(M)$; if on the other hand the typical fiber $Z$ is a point, then $\CI_\phi(H_2)=\CI(H_2)$ and thus again $\cV_{[\bop]}(M)=\Vb(M)$. If both $Y$ and $Z$ have positive dimension, as is the case in our application, one always has $\cV_{[\bop]}(M)\subsetneq\Vb(M)$. However:

\begin{prop}[Many commutator b-vector fields]
\label{PropbCMany}
  The space $\cV_{[\bop]}(M)$ spans $\Vb(M)$ over $\CI(M)$.
\end{prop}
\begin{proof}
  Define the $\CI_\phi(H_2)$-module
  \begin{equation}
  \label{EqbCH2Char}
    \cV_{[\bop]}(H_2) = \cV_{[\bop]}(M)|_{H_2} = \{ V\in\Vb(H_2) \colon V(\CI_\phi(H_2))\subset\CI_\phi(H_2) \}.
  \end{equation}
  It suffices to show that
  \begin{equation}
  \label{EqbCMany}
    \cV_{[\bop]}(H_2)\ \ \text{spans}\ \ \Vb(H_2)\ \ \text{over}\ \ \CI(H_2).
  \end{equation}
  Indeed, let $V\in\Vb(M)$; then, assuming~\eqref{EqbCMany}, we can write $V|_{H_2}=\sum f_i V_i$ where $f_i\in\CI(H_2)$ and $V_i\in\cV_{[\bop]}(H_2)$. Choosing extensions $\tilde f_i\in\CI(M)$ of $f_i$ and $\tilde V_i\in\cV_{[\bop]}(M)$ of $V_i$, the b-vector field $V-\sum\tilde f_i\tilde V_i$ vanishes at $H_2$ as a b-vector field, and hence a fortiori lies in $\Veb(M)\subset\cV_{[\bop]}(M)$.

  Now, the characterization~\eqref{EqbCH2Char} for $V\in\Vb(H_2)$ to lie in $\cV_{[\bop]}(H_2)$ is equivalent to the condition that $V$ pushes forward along $\phi$ to a well-defined vector field on $Y$. But since $\phi\colon H_2\to Y$ is a fibration, every vector field on $Y$ has a lift to an element $V\in\Vb(H_2)$ (which thus satisfies this condition). It remains to observe that $\Vb(H_2)$ is generated over $\CI(H_2)$ by $\Veb(M)|_{H_2}\subset\cV_{[\bop]}(H_2)$ and the lift of $\cV(Y)$.
\end{proof}

In local coordinates as in~\eqref{EqEBSe1} (or \eqref{EqEBSe2}) in the special case that $Z$ is a closed interval, the space $\cV_{[\bop]}(M)$ is spanned by the vector fields $a(x,y,z)x\pa_x$, $b(x,y,z)\pa_z$ (or $b(x,y,z)z\pa_z$), and $(c(x,y)+x c'(x,y,z))\pa_{y^j}$ for smooth $a,b,c,c'$. This can be used to give an alternative, local coordinate, proof of Proposition~\ref{PropbCMany}.

Using induction, one can show that $[\cV_{[\bop]}(M),\Diffeb^k(M)]\subset\Diffeb^k(M)$. We generalize this to operators on vector bundles:

\begin{definition}[Commutator b-operators]
\label{DefbCOp}
  Let $E\to M$ be a vector bundle. Then the space $\Diff_{[\bop]}^1(M;E)$ of \emph{commutator b-operators} denotes the space of all $X\in\Diffb^1(M;E)$ whose principal symbol is scalar and equal to that of an element $V\in\cV_{[\bop]}(M)$; that is, for all $f\in\CI(M)$ and $\sigma\in\CI(M;E)$, the Leibniz rule $X(f\sigma)=f X\sigma+(V f)\sigma$ holds.
\end{definition}

Given any commutator b-vector field $V$, there exists a commutator b-operator $X$ whose principal symbol equals that of $V$; for example, one can take $X=\nabla_V$, where $\nabla$ is a connection on $E$.

\begin{lemma}[Commutators: differential operators]
\label{LemmabCComm}
  Let $A\in\Diffeb^k(M;E)$ (not necessarily with scalar principal symbol) and $X\in\Diff_{[\bop]}^1(M;E)$. Then $[X,A]\in\Diffeb^k(M;E)$. Similarly, if $\alpha\in\R^N$ is a vector of weights and $A\in\cA^{-\alpha}\Diffeb^k(M;E)$, then $[X,A]\in\cA^{-\alpha}\Diffeb^k(M;E)$.
\end{lemma}
\begin{proof}
  In a local trivialization of $E$, we have $X=V\otimes 1+e$ where $V\in\cV_{[\bop]}(M)$ and $e\in\CI(M;\C^{d\times d})$; here $d$ is the rank of $E$, and $1$ is the identity operator on $\C^d$. Moreover, we have $A=(A_{i j})_{1\leq i,j\leq d}$, where the $A_{i j}\in\Diffeb^k(M)$ are scalar operators. Then $[V\otimes 1,A]=([V,A_{i j}])_{1\leq i,j\leq d}$ is a matrix of elements of $\Diffeb^k(M)$, as is $[e,A]$. (When $A$ is principally scalar, then $[e,A]\in\Diffeb^{k-1}(M)$, but this does not hold in general.) This implies the first claim.

  For the second claim, it suffices to consider $A=w A_0$ where $w\in\cA^{-\alpha}(M)$ and $A_0\in\Diffeb^k(M)$; then $[X,A]=w[X,A_0]+(V w)A_0$. Since the b-vector field $V$ maps $\cA^{-\alpha}(M)$ into itself, the claim follows. 
\end{proof}

More generally:

\begin{lemma}[Commutators: ps.d.o.s]
\label{LemmabCCommPsdo}
  Let $s\in\R$ and $X\in\Diff^1_{[\bop]}(M;E)$. Then for $A\in\Psieb^s(M;E)$ (resp.\ $\cA^{-\alpha}\Psieb^s(M;E)$), we have $[X,A]\in\Psieb^s(M;E)$ (resp.\ $\cA^{-\alpha}\Psieb^s(M;E)$).
\end{lemma}
\begin{proof}
  Via local trivializations, one can reduce the proof to the case that $E$ is trivial, and then to the case that $E$ is the trivial complex rank $1$ vector bundle. We work on the level of Schwartz kernels on the edge-b-double space $M^2_\ebop$, see~\eqref{EqEBPDouble}. Denote the lifts of $X$ to the left and right factor of $M^2_\ebop$ by $X_L$ and $X_R$. If $K$ denotes the Schwartz kernel of $A$, then the Schwartz kernel of $[X,A]$ is given by $D_X K$ where
  \[
    D_X := X_L - X_R^*,
  \]
  with the adjoint acting on edge-b-densities; thus $D_X\in\Diffb^1(M^2_\ebop;\pi_{\ebop,R}^*\Omegaeb M)$. Choosing a trivialization of $\Omegaeb M$, the operator $D_X$ is, modulo an element of $\CI(M^2_\ebop)$, given by the differentiation along the vector field
  \[
    \tilde X = X_L + X_R.
  \]
  When $X\in\Veb(M)$, then $\tilde X\in\Vb(M^2_\ebop)$ since $M^2_\ebop$ was constructed precisely so that edge-b-vector fields lift to be smooth (and transversal across $\diag_\ebop$), and moreover $\tilde X$ is tangent to $\diag_\ebop$. We claim that $\tilde X\in\Vb(M^2_\ebop)$ and $\tilde X$ is tangent to $\diag_\ebop$ also for $X\in\cV_{[\bop]}(M)$; it suffices to check this for a set of generators $X$ of $\cV_{[\bop]}(M)/\Veb(M)$ over $\CI(M)$. We do this in local coordinates on $M^2_\ebop$.

  We shall give details only in one region, namely near the preimage of the interior of $H_2\cap H_1$ under the diagonal embedding $\diag_\ebop\hra M^2_\ebop$, where (upon relabeling if necessary) $H_1$ is a boundary hypersurface with $H_2\cap H_1\neq\emptyset$. (The computations in the case that there does not exist such a boundary hypersurface are simpler still than the ones presented here.) Thus, fix local coordinates $x\geq 0$, $y\in\R^{d_Y}$, $z_0\geq 0$, $z_1\in\R^{d_Z-1}$ near $(H_2\cap H_1)^\circ$, where $d_Y=\dim Y$ and $d_Z=\dim Z$, and $x$, resp.\ $z_0$ is a defining function of $H_2$, resp.\ $H_1$. Denote the corresponding set of product coordinates on $M^2$ by $(x,y,z_0,z_1,x',y',z_0',z_1')$; then $(H_2)_\phi^2=\{(0,y,z_0,z_1,0,y,z_0',z_1')\}$ and $H_1^2=\{(x,y,0,z_1,x',y',0,z'_1)\}$. It suffices to compute $\tilde X$ in the case $X=\pa_{y^j}$, so $X_L=\pa_{y^j}$ and $X_R=\pa_{(y')^j}$.

  To start, upon blowing up $(H_2)_\phi^2$, we replace the coordinates $(x,x',y,y')$ by $(x,s,y,Y)$ where $s=\frac{x'}{x}-1$ and $Y=\frac{y'-y}{x}$, with the lift of the diagonal $\diag_M$ given by $s=Y=0$, $z_0=z'_0$, $z_1=z'_1$. Since $H_1^2$ lifts to $\{(x,s,y,Y,z_0,z_1,z'_0,z'_1)=(x,s,y,Y,0,z_1,0,z'_1)\}$, blowing it up amounts, near the lift of $\diag_M$, to replacing the coordinates $(z_0,z'_0)$ by $(z_0,\tau)$ where $\tau=\frac{z_0'}{z_0}-1$, with $\diag_\ebop$ now given by $s=Y=\tau=0$, $z_1=z'_1$ in the coordinates $(x,s,y,Y,z_0,\tau,z_1,z'_1)$. But then the vector field $X_L$, resp.\ $X_R$ reads $\pa_{y^j}-x^{-1}\pa_{Y^j}$, resp.\ $x^{-1}\pa_{Y^j}$, and therefore $\tilde X=\pa_{y^j}$ indeed lies in $\Vb(M^2_\ebop)$ and is tangent to $\diag_\ebop$.

  Directly from the definition of conormal distributions then, the operator $D_X$ for $X\in\cV_{[\bop]}(M)$ maps $\Psieb^s(M)\to\Psieb^s(M)$ and $\cA^{-\alpha}\Psieb^s(M)\to\cA^{-\alpha}\Psieb^s(M)$, finishing the proof of the Lemma.
\end{proof}

\subsection{Edge,b;b-operators and Sobolev spaces}
\label{SsbH}

In the general setting of~\S\ref{SsbC}, we proceed to describe spaces of operators and distributions with (integer order) b-regularity on top of (microlocal) edge-b-regularity. This uses a mixed algebra of b-differential edge-b-pseudodifferential operators. Such mixed algebras have been used in a variety of settings, e.g.\ implicitly in \cite{MelroseEuclideanSpectralTheory} and explicitly in \cite{VasyPropagationCorners,VasyWaveOndS,MelroseVasyWunschDiffraction}. We shall only consider scalar operators here, and leave the (purely notational) addition of vector bundles to the reader.

\begin{definition}[Mixed algebras and Sobolev spaces]
\label{DefbH}
  Let $s\in\R$, $\alpha\in\R^N$, and $k\in\N_0$. Then the space $\Diffb^k\Psieb^{s,\alpha}(M)$ consists of all operators which are finite sums $\sum_i Q_i P_i$ where $Q_i\in\Diffb^k(M)$ and $P_i\in\Psieb^{s,\alpha}(M)$. The corresponding spaces of \emph{weighted edge,b;b-Sobolev spaces} are defined by
  \[
    H_{\ebop;\bop}^{(s;k),\alpha}(M) := \{ u\in\Heb^{s,\alpha}(M) \colon A u\in\Heb^{s,\alpha}(M)\ \forall\,A\in\Diffb^k(M) \}.
  \]
  When $\alpha=0$, we denote these spaces simply by $\Diffb^k\Psieb^s(M)$ and $H_{\ebop;\bop}^{(s;k)}(M)$, respectively. Operators acting on sections of bundles $E,F\to M$, and Sobolev spaces of $E$-valued distributions, are defined analogously.
\end{definition}

\begin{lemma}[Commuting the differential and pseudodifferential factors]
\label{LemmabHMix}
  Let $s\in\R$, $\alpha\in\R^N$, $k\in\N_0$, and let $Q\in\Diffb^k(M)$ and $P\in\Psieb^{s,\alpha}(M)$. Then there exists a finite collection of operators $Q_j^\flat,Q_j^\sharp\in\Diffb^k(M)$ and $P_j^\flat,P_j^\sharp\in\Psieb^{s,\alpha}(M)$ so that
  \begin{equation}
  \label{EqbHMix}
    Q P = \sum_j P_j^\flat Q_j^\flat,\qquad
    P Q = \sum_j Q_j^\sharp P_j^\sharp.
  \end{equation}
  This remains true, mutatis mutandis, when $P,Q$ are operators on sections of a vector bundle $E\to M$, and when $P\in\cA^{-\alpha}\Psieb^s$.
\end{lemma}
\begin{proof}
  We only consider the case $\alpha=0$. The case $k=0$ (so $Q\in\CI(M)$) is trivial since $\Psieb^s(M)$ is a right and left module over $\CI(M)$.

  Consider next $k=1$, and let $Q\in\Vb(M)$. If $Q\in\cV_{[\bop]}(M)$, then $Q P=P Q+[P,Q]$, and thus Lemma~\ref{LemmabCCommPsdo} implies~\eqref{EqbHMix}. In general, by Proposition~\ref{PropbCMany}, the vector field $Q$ is a finite sum of terms $f V$ where $f\in\CI(M)$ and $V\in\cV_{[\bop]}(M)$. Suppose thus that $Q=f V$. The first claim in~\eqref{EqbHMix} then follows from $Q P=P Q-f[P,V]-[P,f]V$, since $f[P,V]\in\Psieb^s(M)$ and $[P,f]V\in\Psieb^{s-1}\Diffb^1$. The second claim follows similarly from $P Q=Q P+[P,f]V+f[P,V]=Q P+V[P,f]+[[P,f],V]+f[P,V]$ where now $V[P,f]\in\Diffb^1\Psieb^{s-1}$, $[[P,f],V]\in\Psieb^{s-1}$, and $f[P,V]\in\Psieb^s(M)$.

  For $k\geq 2$, we argue inductively. Write $Q$ as a finite sum of terms $Q_1 Q_2$ with $Q_1\in\Diffb^1(M)$ and $Q_2\in\Diffb^{k-1}(M)$. Then $Q_2 P=\sum_j P_{2 j}^\flat Q_{2 j}^\flat$ with $Q_{2 j}^\flat\in\Diffb^{k-1}(M)$ and $P_{2 j}^\flat\in\Psieb^s(M)$ by the inductive hypothesis, and then by the case $k=1$, we can further write $Q_1 P_{2 j}^\flat=\sum_k P_{2 j k}^\flat Q_{2 j k}^\flat$ with $Q_{2 j k}^\flat\in\Diffb^1(M)$ and $P_{2 j k}^\flat\in\Psieb^s(M)$. This proves the first part of~\eqref{EqbHMix}. The proof of the second part is analogous.
\end{proof}

\begin{lemma}[Further basic properties]
\label{LemmabHComm}
  Let $k\in\N_0$, $s\in\R$. Then $\Diffb^k\Psieb^s(M)\subset\Diffb^{k+1}\Psieb^{s-1}(M)$ and $[\Diffb^k(M),\Psieb^s(M)]\subset\Diffb^k(M)\Psieb^{s-1}(M)$, similarly for spaces of weighted operators.
\end{lemma}
\begin{proof}
  For the first part, we fix a finite collection $V_1,\ldots,V_K\in\Veb(M)$ which spans $\Veb(M)$ over $\CI(M)$. We can then write any $A\in\Psieb^s(M)$ in the form $A=\sum_{j=1}^K V_j A_j+R$ where $A_j,R\in\Psieb^{s-1}(M)$; this follows from an analogous decomposition on the level of principal symbols. Since a fortiori $V_j\in\Vb(M)$, we are done.

  In the second part, the case $k=0$ is trivial. The case $k=1$ follows by writing an element of $\Vb(M)$ as a finite sum of terms $f V$ where $f\in\CI(M)$, $V\in\cV_{[\bop]}(M)$, and using that $[\cV_{[\bop]}(M),\Psieb^s(M)]\subset\Psieb^s(M)\subset\Diffb^1\Psieb^{s-1}(M)$ by the first part. The case $k\geq 2$ follows by induction as in the proof of Lemma~\ref{LemmabHMix}.
\end{proof}

\begin{cor}[Algebra properties of the mixed algebra]
\label{CorbHBi}
  If $A_j\in\Diffb^{k_j}\Psieb^{s_j,\alpha_j}(M)$ for $j=1,2$, then $A_1 A_2\in\Diffb^{k_1+k_2}\Psieb^{s_1+s_2,\alpha_1+\alpha_2}(M)$ and $[A_1,A_2]\in\Diffb^{k_1+k_2}\Psieb^{s_1+s_2-1,\alpha_1+\alpha_2}(M)$.
\end{cor}
\begin{proof}
  The first statement is a consequence of Lemma~\ref{LemmabHMix}, which allows us to commute b-differential through edge-b-pseudodifferential operators at will. The second statement, for $\alpha_1=\alpha_2=0$ for notational simplicity, follows for operators $A_j=Q_j P_j$ with $Q_j\in\Diffb^{k_j}(M)$ and $P_j\in\Psieb^{s_j}(M)$ from the calculation
  \[
    [A_1,A_2] = Q_1 Q_2[P_1,P_2] + Q_1[P_1,Q_2]P_2 + Q_2[Q_1,P_2]P_1 + [Q_1,Q_2]P_2 P_1;
  \]
  In the first term, we use $[P_1,P_2]\in\Psieb^{s_1+s_2-1}(M)$; for the remaining three terms, the desired conclusion follows from Lemma~\ref{LemmabHComm}.
\end{proof}

\begin{cor}[Boundedness on Sobolev spaces]
\label{CorbHMap}
  Let $s,m\in\R$, $k,l\in\N_0$, $\alpha,\beta\in\R^N$, with $l\geq k$. Then every $A\in\Diffb^k\Psieb^{m,\alpha}(M)$ defines a bounded linear operator $A\colon H_{\ebop;\bop}^{(s;l),\beta}(M)\to H_{\ebop;\bop}^{(s-m;l-k),\beta-\alpha}(M)$.
\end{cor}
\begin{proof}
  This follows from the definitions upon writing $A=\sum_j P_j^\flat Q_j^\flat$ with $P_j^\flat\in\Psieb^{s,\alpha}(M)$ and $Q_j^\flat\in\Diffb^k(M)$ (cf.\ Lemma~\ref{LemmabHMix}).
\end{proof}

Next, edge,b;b-regularity can be microlocalized:

\begin{definition}[Edge,b;b-wave front set]
\label{DefbHWF}
  Let $s\in\R$, $k\in\N_0$, $\alpha\in\R^N$. Suppose $u\in H_{\ebop;\bop}^{(-\infty;k),\alpha}(M)=\bigcup_{s_0\in\R} H_{\ebop;\bop}^{(s_0;k),\alpha}(M)$. Then $\WF_{\ebop;\bop}^{(s;k),\alpha}(u)\subset\Seb^*M$ is the complement of the set of all $\varpi\in\Seb^*M$ for which there exists $A\in\Psieb^0(M)$, elliptic at $\varpi$, such that $A u\in H_{\ebop;\bop}^{(s;k),\alpha}(M)$.
\end{definition}

For any other $A'\in\Psieb^0(M)$ with $\WFeb'(A')\subset\Elleb(A)$, one then also has $A' u\in H_{\ebop;\bop}^{(s;k),\alpha}(M)$ by the symbolic parametrix construction in the edge-b-algebra. (The a priori membership of $u$ in a space with b-regularity $k$ and weight $\alpha$ is crucial here.)

\begin{lemma}[Microlocality of edge-b-ps.d.o.s on edge,b;b-Sobolev spaces]
\label{LemmabHMicro}
  Let $s,m\in\R$, $k\in\N_0$, and $\alpha,\beta\in\R^N$. Let $u\in H_{\ebop;\bop}^{(-\infty;k),\beta}(M)$ and $A\in\Psieb^{m,\alpha}(M)$. Then
  \[
    \WF_{\ebop;\bop}^{(s-m;k),\beta-\alpha}(A u)\subset\WFeb'(A)\cap\WF_{\ebop;\bop}^{(s;k),\beta}(u).
  \]
\end{lemma}
\begin{proof}
  If $\varpi\notin\WFeb'(A)\cap\WF_{\ebop;\bop}^{(s;k),\beta}(u)$, then there exists $B\in\Psieb^0(M)$, elliptic at $\varpi$, such that either $B A\in\Psieb^{-\infty}(M)$ and hence $B A u\in H_{\ebop;\bop}^{(\infty;k),\beta}(M)$, or $B u\in H_{\ebop;\bop}^{(s;k),\beta}(M)$, in which case we write $I=Q B+R$ where $Q,R\in\Psieb^0(M)$ and $\varpi\notin\WFeb'(R)$, and then for $A'\in\Psieb^0(M)$ which is elliptic at $\varpi$ but with $\WFeb'(A')\cap\WFeb'(R)=\emptyset$, we find $A'(A u)=A'A Q B u+A' A R u$. The first summand on the right lies in $H_{\ebop;\bop}^{(s-m;k),\beta-\alpha}(M)$, and the second summand lies in $H_{\ebop;\bop}^{(\infty;k),\beta-\alpha}(M)$.
\end{proof}

\begin{prop}[Elliptic regularity]
\label{PropbHEll}
  Let $P\in\Psieb^m(M)$. Then we have $\WF_{\ebop;\bop}^{(s;k),\beta}(u)\subset\WF_{\ebop;\bop}^{(s-m;k),\beta}(M)\cup\Char_\ebop(P)$.
\end{prop}
\begin{proof}
  This follows from the symbolic elliptic parametrix construction in $\Psieb(M)$ together with Corollary~\ref{CorbHMap}.
\end{proof}

\begin{prop}[Real principal type propagation]
\label{PropbHProp}
  Suppose $P\in\Psieb^m(M)$ has a real homogeneous principal symbol. Suppose $B,E,G\in\Psieb^0(M)$ are such that $\WFeb'(B)\subset\Elleb(G)$, and so that all backward null-bicharacteristics from $\WFeb'(B)\cap\Char_\ebop(P)$ reach $\Elleb(E)$ in finite time while remaining in $\Elleb(G)$. Then for any fixed $N\in\R$, we have
  \begin{equation}
  \label{EqbHProp}
    \|B u\|_{H_{\ebop;\bop}^{(s;k),\alpha}(M)} \leq C\Bigl( \|G P u\|_{H_{\ebop;\bop}^{(s-m+1;k),\alpha}(M)} + \|E u\|_{H_{\ebop;\bop}^{(s;k),\alpha}(M)} + \|u\|_{H_{\ebop;\bop}^{(-N;k),\alpha}(M)} \Bigr).
  \end{equation}
  This holds in the strong sense that if all terms on the right hand side are finite, then the left hand side is finite and the estimate holds.
\end{prop}
\begin{proof}
  The case $k=0$ (which was already tacitly used in~\S\ref{SsME}) follows from the usual symbolic positive commutator argument. Arguing by induction, we now apply the estimate~\eqref{EqbHProp} to $V u$ in place of $u$, where $V\in\cV_{[\bop]}(M)$. We then use elliptic estimates for the commutators with $V$, such as
  \begin{equation}
  \label{EqbHPropPf}
    \|G[P,V]u\|_{H_{\ebop;\bop}^{(s-m+1;k),\alpha}(M)} \lesssim \|G' u\|_{H_{\ebop;\bop}^{(s+1;k),\alpha}(M)} + \|u\|_{H_{\ebop;\bop}^{(-N;k),\alpha}(M)},
  \end{equation}
  where $G'\in\Psieb^0(M)$, $\WFeb'(G)\subset\Elleb(G')$; this uses the membership $[P,V]\in\Psieb^m(M)$ from Lemma~\ref{LemmabCCommPsdo}. The first term on the right can then be estimated, using the inductive hypothesis, by means of~\eqref{EqbHProp} with $B,s$ replaced by $G',s+1$ (and correspondingly $G$ replaced by an operator with slightly enlarged elliptic set). The inductive step is then completed by noting that $H_{\ebop;\bop}^{(s+1;k),\alpha}(M)\subset H_{\ebop;\bop}^{(s;k+1),\alpha}(M)$ (which follows from Lemma~\ref{LemmabHComm}).
\end{proof}

\subsection{Propagation estimates near \texorpdfstring{$\scri^+$}{null infinity}}
\label{SsbI}

The same proof as for Proposition~\ref{PropbHProp} gives the higher b-regularity analogues of certain radial point estimates from~\S\ref{SsME}. We need to use two facts: first, these estimates only lose one edge-b-derivative relative to elliptic estimates, as otherwise the first term on the right in~\eqref{EqbHPropPf} would require using a successively larger edge-b-regularity order on the $P u$ and $E u$ terms in~\eqref{EqbHProp} as one increases $k$. Second, the estimate for $k=0$ needs to be applicable without an upper bound for the edge-b-regularity $s$ (given fixed weights); note indeed that in the proof of Proposition~\ref{PropbHProp}, the proof of the inductive step uses the inductive hypothesis with an increased value for $s$. (Alternatively, if there is a requirement $s<s_0$ for the case $k=0$, then the stronger bound $s+k<s_0$ guarantees an estimate on $H_{\ebop;\bop}^{(s;k),\alpha}$.) Not a single forward propagation estimates imposes an upper bound on $s$. We thus obtain:

\begin{prop}[Edge,b;b-versions of radial point estimates]
\label{PropbI}
  Theorem~\usref{ThmMEFw} and Lemmas~\usref{LemmaMEc}\eqref{ItMEcFw}, \usref{LemmaMEinm}, \usref{LemmaMEinp}, and \usref{LemmaMEout} remain valid upon replacing $\WFeb^{*,*}$ by $\WF_{\ebop;\bop}^{(*;k),*}$ throughout. Lemma~\usref{LemmaMEc}\eqref{ItMEcBw} remains valid under the same replacement if one requires $\tilde s+k<\half-\tilde\alpha_0+2\tilde\alpha_{\!\scri}+\ubar p_1$.
\end{prop}

\begin{rmk}[Choice of commutator vector fields]
\label{RmkbIVFs}
  In \cite[Proof of Proposition~4.7]{HintzVasyMink4}, we needed to exercise some care when choosing the b-vector fields $V$ (called $G$ there) for commutations: the commutator $[P,V]$ needed to be an operator whose coefficients decayed in a suitable manner. As far as spherical vector fields are concerned, this required $V$ to be (asymptotic to) a Killing vector field, i.e.\ a rotation. Here on the other hand, we can use arbitrary elements of $\cV_{[\bop]}(M)$. The reason is that we already have a sharp edge-b-regularity theory, and the arguments in the present section show how one can inductively upgrade edge-b- to b-regularity. In~\cite{HintzVasyMink4} on the other hand, b-regularity was proved directly, without first establishing edge-b-regularity.
\end{rmk}

\section{Energy estimates and solvability \texorpdfstring{near $I^0$ and $\scri^+\setminus I^+$}{away from future timelike infinity}}
\label{SE}

We continue to denote by $g$ an $(\ell_0,2\ell_{\!\scri},\ell_+)$-admissible metric and $P$ a $g$-admissible operator on the $(n+1)$-dimensional manifold $M$, $n\geq 1$; see~\S\ref{SsOp}. We work in a product decomposition $[0,\eps)_\varrho\times(-\eps,\eps)_v\times Y$ near the light cone at infinity $Y$ on the blown-down manifold $M_0$, in which the underlying metric takes the form~\eqref{EqOGeoAdmDef} (or equivalently~\eqref{EqOGeoAdmDual}--\eqref{EqOGeoMetIp}), and write $\varrho=\frac{1}{t+r}$, $v=\frac{t-r}{t+r}$ (which defines coordinates $t,r$ on $M^\circ$). For every compact subset $K\subset\scri^+\setminus I^+$, there then exists $T\in\R$ so that the domain $\cU_0(T)$ from Definition~\ref{DefOGeoU0p} contains $K$. Define the coordinates $\rho_0=\frac{1}{T-t_*}$, $x_{\!\scri}=\sqrt{\frac{T-t_*}{r}}$ on $\cU_0(T)$ as in~\eqref{EqOGeoCoordI0}.

We work away from $I^+$; thus, we shall only record weights at $I^0$ and $\scri^+$. Lemma~\ref{LemmaOGeoHyp} shows that, given $\ubar\rho_0<\bar\rho_0\in(0,1)$, for large $C$ and small $\delta_0$ the domain
\begin{equation}
\label{EqEOmega}
  \Omega^0_{\delta,\rho} := \{ x_{\!\scri}<\delta,\ \rho_0<\rho+C x_{\!\scri}^{2\ell_{\!\scri}} \} \subset \cU_0(T) \subset M
\end{equation}
contains $K$ when $0<\delta\leq\delta_0$ and $\rho\in[\ubar\rho_0,\bar\rho_0]$, and its boundary hypersurfaces (other than those contained in $\pa M$)
\begin{equation}
\label{EqESigma}
\begin{split}
  \Sigma^{0,\rm in}_{\delta,\rho} &:= \{ x_{\!\scri}=\delta,\ \rho_0<\rho+C\delta^{2\ell_{\!\scri}} \}, \\
  \Sigma^{0,\rm out}_{\delta,\rho} &:= \{ x_{\!\scri}<\delta,\ \rho_0=\rho+C x_{\!\scri}^{2\ell_{\!\scri}} \}
\end{split}
\end{equation}
are spacelike. See Figure~\ref{FigESetup}.

\begin{figure}[!ht]
\centering
\includegraphics{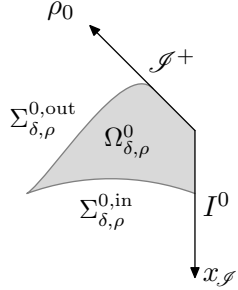}
\caption{The domain~\eqref{EqEOmega} and its boundary hypersurfaces~\eqref{EqESigma} used in the setup of Theorem~\ref{ThmE}.}
\label{FigESetup}
\end{figure}

Since the normal vector field at $\Sigma_{\delta,\rho}^{0,\rm in}$, resp. $\Sigma_{\delta,\rho}^{0,\rm out}$ pointing into, resp.\ out of $\Omega^0_{\delta,\rho}$ is future timelike, the Cauchy problem for $P u=f$ with initial data on $\Sigma_{\delta,\rho}^{0,\rm in}$ is well-posed in $\Omega^0_{\delta,\rho}$. We shall consider here forward, resp.\ backward problems with trivial data at $\Sigma_{\delta,\rho}^{0,\rm in}$, resp.\ $\Sigma_{\delta,\rho}^{0,\rm out}$. We remark here already that the forward problem is well-posed, whereas the backward problem is not since the backward domain of dependence of $\Sigma_{\delta,\rho}^{0,\rm out}$ is a proper subset of $\Omega_{\delta,\rho}^0$.

\begin{definition}[Function spaces]
\label{DefESpaces}
  Let $\delta<\delta_0$ and $\rho\in(\ubar\rho_0,\bar\rho_0)$, and $s\in\R$, $\alpha=(\alpha_0,2\alpha_{\!\scri})\in\R^2$. Fix on $M$ a volume density of the class~\eqref{EqMEDensity}. Then we let
  \begin{align*}
    \Heb^{s,\alpha}(\Omega^0_{\delta,\rho})^{\bullet,-} &:= \left\{ u|_{\Omega^0_{\delta,\rho}} \colon u\in\bar H_\ebop^{s,\alpha}(\Omega^0_{\delta_0,\bar\rho_0}),\ \supp u\subset \ol{\Omega^0_{\delta,\bar\rho_0}} \,\right\}, \\
    \Heb^{s,\alpha}(\Omega^0_{\delta,\rho})^{-,\bullet} &:= \left\{ u|_{\Omega^0_{\delta,\rho}} \colon u\in\bar H_\ebop^{s,\alpha}(\Omega^0_{\delta_0,\bar\rho_0}),\ \supp u\subset \ol{\Omega^0_{\delta_0,\rho}} \,\right\}.
  \end{align*}
\end{definition}

Thus, elements in the space $\Heb^{s,\alpha}(\Omega^0_{\delta,\rho})^{\bullet,-}$ have supported character at $\Sigma_{\delta,\rho}^{0,\rm in}$ and are extendible at $\Sigma_{\delta,\rho}^{0,\rm out}$, with the situation reversed for elements of $\Heb^{s,\alpha}(\Omega^0_{\delta,\rho})^{-,\bullet}$. Moreover, the $L^2$-dual of $\Heb^{s,\alpha}(\Omega^0_{\delta,\rho})^{\bullet,-}$ is $\Heb^{-s,-\alpha}(\Omega^0_{\delta,\rho})^{-,\bullet}$ \cite[Appendix~B]{HormanderAnalysisPDE3}.

Due to the inflexible nature of energy estimates---they are global in the fibers of $\Teb^*M$ and admit the use of differential operators only---as compared to the microlocal estimates of~\S\ref{SsME}, we impose here an additional condition on the edge normal operator of $P$ in order to get sharp results later on:

\begin{definition}[Special admissible operators]
\label{DefESpecial}
  Let $P$ be an admissible operator, and define $p_0\in\rho_0\rho_+(\CI+\cA^{(\ell_0,\ell_+)})(\scri^+;\upbeta^*\End(E))$ and $p_1\in(\CI+\cA^{(\ell_0,\ell_+)})(\scri^+;\upbeta^*\End(E))$ as in Definition~\ref{DefOp}\eqref{ItOpNorm}. We then say that $P$ is a \emph{special admissible operator} if there exists a bundle splitting $E|_Y=\bigoplus_{j=1}^J E_j$ so that $p_1$ is lower triangular, with diagonal entries $p_{1,j j}\in\End(E_j)$ having real spectrum, and so that $p_0$ is strictly lower triangular.
\end{definition}

\begin{rmk}[Discussion]
\label{RmkESpecialDisc}
  Special admissibility is loosely related to weak null structures \cite{LindbladRodnianskiWeakNull}, see also \cite[\S1.1.2]{HintzVasyMink4}. For scalar operators $P$, special admissibility is simply the requirement that $p_1$ be real and $p_0=0$. While special admissibility is a rather strong requirement on $p_0,p_1$, the linearized gauge-fixed Einstein operator considered in \cite[Proposition~3.29 and \S3.6]{HintzMink4Gauge} does satisfy it (see Example~\ref{ExOpMinkEin} on how to relate the notation in the reference to the one used presently); see also \cite[Appendix~A]{HintzMink4Gauge}, which implies the special admissibility of an appropriate gauge-fixed Maxwell equation on Minkowski space.
\end{rmk}

If $P$ is special admissible, then for all $\eps>0$ there exists a positive definite fiber inner product $h_E$ on $E$ so that, in the notation of Definition~\usref{DefMEMin} and Lemma~\usref{LemmaMEMin}, we have
\begin{equation}
\label{EqESpecial}
  \ubar p_1-\eps < \ubar p_1(h_E),\qquad
  -\eps<\frac{p_1-p_1^*}{2 i}<\eps,\qquad
  \sup_{\scri^+}\|p_0\|<\eps,
\end{equation}
where the second inequality is an inequality for quadratic forms on $(\upbeta^*E)|_{\scri^+}$ which we require to hold pointwise on $\scri^+$, and in the third inequality $\|\cdot\|$ denotes the operator norm on the space of linear maps $(E,h_E)\to(E,h_E)$. Indeed, this holds for $h_E$ equal to the rescaling of any fixed positive definite diagonal inner product on $E=\bigoplus_{j=1}^J E_j$ by $\diag(1,\eta,\ldots,\eta^{J-1})$ for sufficiently small $\eta>0$.\footnote{Similar choices of inner products also feature in \cite[\S3.4]{HintzPsdoInner} and towards the end of \cite[\S9.1]{HintzVasyKdSStability}.}

\begin{thm}[Solvability, uniqueness, sharp regularity near $\scri^+\setminus I^+$]
\label{ThmE}
  Let $P$ be a special admissible operator, and define $\ubar p_1$ as in Definition~\usref{DefMEMin}.
  \begin{enumerate}
  \item\label{ItEFw}{\rm (Forward solution.)} Let $s,\alpha_0,\alpha_{\!\scri}\in\R$, and put $\alpha=(\alpha_0,2\alpha_{\!\scri})$ and $\alpha'=\alpha+(2,2)$. Suppose that
    \begin{equation}
    \label{EqEFwAssm}
      s>\frac12-\alpha_0+2\alpha_{\!\scri}-\ubar p_1,\qquad
      \alpha_{\!\scri}<\alpha_0+\frac12,\qquad
      \alpha_{\!\scri}<-\frac12+\ubar p_1.
    \end{equation}
    Let $f\in\Heb^{s-1,\alpha'}(\Omega^0_{\delta,\rho};\upbeta^*E)^{\bullet,-}$. Then the unique distributional forward solution $u$ of $P u=f$ in $\Omega^0_{\delta,\rho}\cap M^\circ$ satisfies $u\in\Heb^{s,\alpha}(\Omega^0_{\delta,\rho};\upbeta^*E)^{\bullet,-}$, and
    \begin{equation}
    \label{EqEFwApriori}
      \|u\|_{\Heb^{s,\alpha}(\Omega^0_{\delta,\rho};\upbeta^*E)^{\bullet,-}}\leq C\|f\|_{\Heb^{s-1,\alpha'}(\Omega^0_{\delta,\rho};\upbeta^*E)^{\bullet,-}},
    \end{equation}
    where $C$ only depends on $P,s,\alpha_0,\alpha_{\!\scri},\delta,\rho$.
  \item\label{ItEBw}{\rm (Backward solution.)} Let $\tilde s,\tilde\alpha_0,\tilde\alpha_{\!\scri}\in\R$, and put $\tilde\alpha=(\tilde\alpha_0,2\tilde\alpha_{\!\scri})$ and $\tilde\alpha'=\tilde\alpha+(2,2)$. Suppose that
    \begin{equation}
    \label{EqEBwAssm}
      \tilde s<\frac12-\tilde\alpha_0+2\tilde\alpha_{\!\scri}+\ubar p_1,\qquad
      \tilde\alpha_0<\tilde\alpha_{\!\scri}-\frac12,\qquad
      \tilde\alpha_{\!\scri}>-\frac12-\ubar p_1.
    \end{equation}
    Let $\tilde u\in\Heb^{\tilde s,\tilde\alpha}(\Omega^0_{\delta,\rho};\upbeta^*E)^{-,\bullet}$ be such that $\tilde f:=P^*\tilde u\in\Heb^{\tilde s-1,\tilde\alpha'}(\Omega^0_{\delta,\rho};\upbeta^*E)^{-,\bullet}$. Then
    \begin{equation}
    \label{EqEBwApriori}
      \|\tilde u\|_{\Heb^{\tilde s,\tilde\alpha}(\Omega^0_{\delta,\rho};\upbeta^*E)^{-,\bullet}}\leq C\|\tilde f\|_{\Heb^{\tilde s-1,\tilde\alpha'}(\Omega^0_{\delta,\rho};\upbeta^*E)^{-,\bullet}},
    \end{equation}
    where $C$ only depends on $P,s,\tilde\alpha_0,\tilde\alpha_{\!\scri},\delta,\rho$. Conversely, given a forcing term $\tilde f\in\Heb^{\tilde s-1,\tilde\alpha'}(\Omega^0_{\delta,\rho};\upbeta^*E)^{-,\bullet}$, there exists $\tilde u$ which satisfies the equation $P^*\tilde u=\tilde f$ and the estimate~\eqref{EqEBwApriori}; this $\tilde u$ is the unique solution with these weights, in the sense that any other solution $\tilde u'\in\Heb^{-\infty,\tilde\alpha}(\Omega^0_{\delta,\rho};\upbeta^*E)^{-,\bullet}$ of $P^*\tilde u'=\tilde f$ is necessarily equal to $\tilde u$.
  \end{enumerate}
\end{thm}

Note that the backward problem is not well-posed without restricting the space in which one seeks the solution, as we are not imposing any (asymptotic) data at $\scri^+\cap\Omega^0_{\delta,\rho}$; in Theorem~\ref{ThmE}\eqref{ItEBw}, the a priori assumption on the decay of $\tilde u'$ stands in for the missing data at null infinity.

\begin{rmk}[Modifications for non-special operators]
\label{RmkEp1No}
  If one keeps the assumption that $P$ be admissible, but drops the assumptions on $p_0,p_1$ made in Definition~\ref{DefESpecial}, then solvability, uniqueness, and sharp regularity can still be obtained by the arguments used in the proof, except the relationships between $\alpha_0,\alpha_{\!\scri}$ need to be modified appropriately; roughly speaking, the inequalities~\eqref{EqEFwAssm} and \eqref{EqEBwAssm} need to hold with enough room to spare so that various error terms arising in the energy (positive commutator) estimate can be absorbed. We shall not make this quantitative here.
\end{rmk}

\begin{proof}[Proof of Theorem~\usref{ThmE}]
  We shall omit the vector bundle $\upbeta^*E$ from the notation unless dealing with the case that $\upbeta^*E$ is nontrivial requires additional arguments. The proof starts with a simple energy estimate in which we estimate $u$ in a space with one order of edge-b-regularity. Higher regularity follows from the microlocal propagation results proved in~\S\ref{SsME}. Solvability of the adjoint (backward) problem on negative regularity spaces follows by duality; this can then be upgraded to the maximal amount of regularity allowed in~\eqref{EqEBwAssm} via microlocal propagation estimates, which will prove part~\eqref{ItEBw}. Another duality argument will then give part~\eqref{ItEFw} in the full range of $s$.

  \pfstep{Part~\eqref{ItEFw}, $s=1$.} We shall prove the energy estimate
  \begin{equation}
  \label{EqEEnergyEst0}
    \|u\|_{\Heb^{1,\alpha}(\Omega^0_{\delta,\rho})^{\bullet,-}} \leq C\|f\|_{\Heb^{0,\alpha'}(\Omega^0_{\delta,\rho})^{\bullet,-}}
  \end{equation}
   for the forward solution $u$ of the equation $P u=f\in\Heb^{0,\alpha'}(\Omega^0_{\delta,\rho})^{\bullet,-}$ by computing the $L^2$ pairing (using the volume density $|\dd g|$) $2\Im\la P u,Z u\ra$ for $Z=i^{-1}V$ where $V$ is a suitably chosen real vector field on $M^\circ$.
   
   \pfsubstep{\eqref{ItEFw}(i)}{Structural considerations for the energy estimate.} For $u\in\CIc((\Omega^0_{\delta,\rho})^\circ)$, we have
  \begin{equation}
  \label{EqEFwComm}
    2\Im\la P u,Z u\ra = \la \sC u,u\ra,\qquad
    \sC=i(P^*Z-Z^*P)=i[P,Z]-(\dv_g i Z)P+2\frac{P-P^*}{2 i}Z,
  \end{equation}
  where we use that $Z^*=Z+\dv_g Z$; here, writing $V=V^j\pa_j$ in local coordinates, we set $\dv_g(V)=|g|^{-1/2}\pa_j(|g|^{1/2}V^j)$ (so $V^*=-V-\dv_g V$). The principal symbol of $\sC$ in $M^\circ$ is
  \[
    \sigma^2(\sC) = H_G(\sigma^1(Z))-(\dv_g i Z)G+2\sigma^1(P_1)\sigma^1(Z),\qquad P_1:=\frac{P-P^*}{2 i},
  \]
  where $G$ is the dual metric function. A simple calculation gives at a point $\varpi\in T^*_p M^\circ$, $p\in M^\circ$, the expression
  \begin{equation}
  \label{EqEFwCommSymb}
    \sigma^2(\sC)(\varpi) = K_{i Z}(\varpi,\varpi) + 2\sigma^1(P_1)(\varpi)\cdot\sigma^1(Z)(\varpi) =: \tilde K_{i Z}(\varpi,\varpi)
  \end{equation}
  where we define for any vector field $V$ the following symmetric 2-cotensors on $M^\circ$:
  \begin{equation}
  \label{EqEFwDef}
  \begin{split}
    K_V&={}^V\!\pi-\frac12 g^{-1}\tr_g({}^V\!\pi),\qquad {}^V\!\pi:=-\cL_V(g^{-1}), \\
    \tilde K_V&=K_V+2\sigma^1(P_1)\otimes_s\sigma^1(i^{-1}V),
  \end{split}
  \end{equation}
  with $\cL_V$ denoting the Lie derivative along $V$.
   
   Concretely now, let $\check\alpha=(\check\alpha_0,2\check\alpha_{\!\scri}):=(\alpha_0,2\alpha_{\!\scri})+(1,1)$, and define the vector fields
  \begin{equation}
  \label{EqEVFs}
    W:=-x_{\!\scri}\pa_{x_{\!\scri}}+(2-c)\rho_0\pa_{\rho_0},\qquad
    V:=\rho_0^{-2\check\alpha_0}x_{\!\scri}^{-4\check\alpha_{\!\scri}}W,\qquad
    Z=i^{-1}V.
  \end{equation}
  Note that for $0<c<2$, the vector field $W$ is future timelike at (hence near) $\scri^+$ for the rescaled metric $g_\ebop:=\rho_0^2 x_{\!\scri}^2 g\in(\CI+\cA^{(\ell_0,2\ell_{\!\scri})})(M;S^2\,\Teb^*M)$ by~\eqref{EqOGeoTimelike}. Moreover, when the bundle $E$ is trivial, then
  \begin{equation}
  \label{EqEVMem}
  \begin{split}
    \sC &\in \rho_0^{-2\alpha_0}x_{\!\scri}^{-4\alpha_{\!\scri}}\bigl(\CI+\cA^{(\ell_0,(0,0))}+\cA^{(\ell_0,2\ell_{\!\scri})}\bigr)\Diffeb^2(M;\upbeta^*E), \\
    \tilde K_{i Z} &\in\rho_0^{-2\alpha_0}x_{\!\scri}^{-4\alpha_{\!\scri}}\bigl(\CI+\cA^{(\ell_0,(0,0))}+\cA^{(\ell_0,2\ell_{\!\scri})}\bigr)(M;S^2\,\Teb M),
  \end{split}
  \end{equation}
  where the $\cA^{(\ell_0,(0,0))}$ terms are contributions from $P_1$.
  
  When $E$ is not trivial, take $\tilde W\in\Diffeb^1(M;\upbeta^*E)$ to be any operator with scalar principal symbol $\sigmaeb^1(\tilde W)=\sigmaeb^1(W)\otimes 1_E$ and edge normal operator at $\scri^+$ equal to that of $W$. The edge normal operator of $\tilde W$ at a fiber $\scri^+_{y_0}$ acts on sections of the trivial bundle $E_{y_0}$, and differentiation of such sections along $W$ is well-defined; thus, any two choices of $\tilde W$ agree modulo $x_{\!\scri}\CI(M;\End(\upbeta^*E))$. We now drop the tilde and work with $W\in\Diffeb^1(M;\upbeta^*E)$ simply, with $V,Z\in\Diffeb^{1,(2\check\alpha_0,4\check\alpha_{\!\scri})}(M;\upbeta^*E)$ given in terms of $W$ as above; we then have~\eqref{EqEVMem} also when $E$ is not trivial. We shall also fix a connection $\nabla^E\in\Diffeb^1(M;\upbeta^*E,\Teb^*M\otimes\upbeta^*E)$ on $E$ whose normal operator at $\scri^+$ agrees with the canonical flat connection (differentiation of functions valued in a fixed vector space).

  On $E$, we fix a near-optimal positive definite fiber inner product $h_E$, i.e.\ \eqref{EqESpecial} holds for any desired $\eps>0$. Fix moreover a positive definite fiber inner product $g_R$ on $\Teb^*M$. Using $g_R$, we can identify $\tilde K_{i Z}$ with a section of $\End(\Teb^*M)$, and upon tensoring with the identity map on $\upbeta^*E$ we obtain a section of $\End(\Teb^*M\otimes\upbeta^*E)$ which we shall still denote by $\tilde K_{i Z}$. Define the adjoint $(\nabla^E)^*\in\Diffeb^1(M;\Teb^*M\otimes\upbeta^*E,\upbeta^*E)$ with respect to the fiber metric $g_R\otimes h_E$. Then $(\nabla^E)^*\tilde K_{i Z}\nabla^E$ and $\sC$ have the same principal symbol. Set
  \[
    \tilde Q=\rho_0^2 x_{\!\scri}^2(x_{\!\scri} D_{x_{\!\scri}}-2\rho_0 D_{\rho_0}) \equiv \tilde Q^* \bmod \cA^{(\ell_0,2\ell_{\!\scri})}(M),
  \]
  defined as an element of $\rho_0^2 x_{\!\scri}^2\Diffeb^1(M;\upbeta^*E)$ in the same fashion as $W$ before. We claim that one can then write
  \begin{equation}
  \label{EqEsCLot}
  \begin{split}
    &\sC = (\nabla^E)^*\tilde K_{i Z}\nabla^E + \tilde\sC + \tilde\sC', \\
    &\qquad \tilde\sC \in \rho_0^{-2\alpha_0}x_{\!\scri}^{-4\alpha_{\!\scri}}(x_{\!\scri}\CI+\cA^{(\ell_0,2\ell_{\!\scri})})\Diffeb^1(M;\upbeta^*E), \\
    &\qquad \tilde\sC' = \frac12\bigl(\tilde Q^*(p_1-p_1^*)Z-Z^*(p_1-p_1^*)\tilde Q\bigr) + \frac{i}{2}\rho_0^2 x_{\!\scri}^2(p_0^* Z - Z^* p_0) \\
    &\qquad\qquad \in \rho_0^{-2\alpha_0}x_{\!\scri}^{-4\alpha_{\!\scri}}(\CI+\cA^{(\ell_0,(0,0))}+\cA^{(\ell_0,2\ell_{\!\scri})})\Diffeb^1(M;\upbeta^*E).
  \end{split}
  \end{equation}
  Thus, the coefficients of $\tilde\sC$ vanish at $\scri^+$; similarly, we will regard $\tilde\sC'$ as an error term near $\scri^+$ since, by assumption, $p_1-p_1^*$ and $p_0,p_0^*$ can be made arbitrarily small for suitable choices of fiber inner products on $E$.

  Only the vanishing factor $x_{\!\scri}$ in the membership statement for $\tilde\sC$ requires proof. We have to show the vanishing of the edge normal operator of $\rho_0^{2\alpha_0}x_{\!\scri}^{4\alpha_{\!\scri}}\tilde\sC$; for this purpose, we may replace $g^{-1},P$ in the definition of $\sC$ and $\tilde K_{i Z}$ by their normal operators at a fiber $\scri^+_{y_0}$,
  \begin{equation}
  \label{EqENorm}
  \begin{split}
    g_0^{-1} &:= \rho_0^2 x_{\!\scri}^2\Bigl(-\frac12 x_{\!\scri}\pa_{x_{\!\scri}}\otimes_s(x_{\!\scri}\pa_{x_{\!\scri}}-2\rho_0\pa_{\rho_0}) + k^{i j}(y_0)(x_{\!\scri}\pa_{y^i})\otimes_s(x_{\!\scri}\pa_{y^j})\Bigr), \\
    P_0 &:= \frac12\rho_0^2 x_{\!\scri}^2\biggl(-\Bigl(x_{\!\scri} D_{x_{\!\scri}}-2 i^{-1}\Bigl(\frac{n-1}{2}+p_{1,y_0}\Bigr)\Bigr)(x_{\!\scri} D_{x_{\!\scri}}-2\rho_0 D_{\rho_0}) \\
      &\hspace{14em} + 2 k^{i j}(y_0)(x_{\!\scri} D_{y^i})(x_{\!\scri} D_{y^j}) + p^0_{0,y_0}\biggr),
  \end{split}
  \end{equation}
  where we use~\eqref{EqOGeoMetI0}, \eqref{EqOpNorm}, and set
  \[
    p_{0,y_0}^0:=p_0^0|_{\scri^+_{y_0}},\qquad p_{1,y_0}:=p_1|_{\scri^+_{y_0}}\in(\CI+\cA^{\ell_0})\bigl([0,1)_{\rho_0};\End(E_{y_0})\bigr).
  \]
  We may moreover work with the metric density $|\dd g_0|=\rho_0^{-n-1}x_{\!\scri}^{-2 n}|\frac{\dd\rho_0}{\rho_0}\frac{\dd x_{\!\scri}}{x_{\!\scri}}\dd y|$.

  In the case that $p_{0,y_0}^0=0$ and $p_{1,y_0}=0$, we have $P_0=P_0^*$. The operators
  \[
    \sC_0:=i(P_0^*Z-Z^*P_0)
  \]
  and $(\nabla^E)^*\tilde K_{0,i Z}\nabla^E$ (the subscripts `$0$' indicating that we are using $g_0,P_0$, $P_{0,1}=\frac{P_0-P_0^*}{2 i}=0$ in place of $g,P,P_1$, and with $\nabla^E$ denoting the trivial connection on $(\upbeta^*E)|_{\scri^+_{y_0}}=\scri^+_{y_0}\times E_{y_0}$) have the same principal symbol, real coefficients, and are symmetric. Hence their difference must in fact be an operator of order $0$; but since both operators annihilate constant sections of $E_{y_0}$, this difference is the zero operator. This proves~\eqref{EqEsCLot} in this case, with $\tilde\sC'=0$.

  For general $p_{1,y_0}$, but still assuming $p_{0,y_0}^0=0$, the operator $P_0$ gets an extra contribution
  \[
    Q = i^{-1}p_{1,y_0}\tilde Q
  \]
  and $P_{0,1}=\frac{Q-Q^*}{2 i}$ has principal symbol $-\half(p_{1,y_0}+p_{1,y_0}^*)\sigma^1(\tilde Q)$, Hence, compared to the case $p_{1,y_0}=0$, the difference $\sC_0-(\nabla^E)^*\tilde K_{i Z}\nabla^E$ gets an extra contribution
  \begin{equation}
  \label{EqEsCErr}
    i(Q^*Z-Z^*Q) - (\nabla^E)^*\Bigl(-2\sigma^1\bigl(\half(p_{1,y_0}+p_{1,y_0}^*)\tilde Q\bigr)\otimes_s\sigma^1(Z)\Bigr)\nabla^E,
  \end{equation}
  where we regard $\sigma^1(\half(p_{1,y_0}+p_{1,y_0})^*\tilde Q)\in P^{[1]}(\Teb^*M\otimes\End(E_{y_0}))\cong\Teb M\otimes\End(E_{y_0})$ and $\sigma^1(Z)\in P^{[1]}(\Teb^*M)\cong\Teb M$ (giving an element of $\Teb M\otimes\End(E_{y_0})$ upon tensoring with the identity on $E_{y_0}$), and further identify $S^2\,\Teb M\cong\Teb^*M\otimes_s\Teb M$ using the Riemannian edge-b-metric $g_R$ on $M$. Using that for any vector fields $V,W$ one has
  \[
   (\nabla^E)^*(2\sigma^1(V)\otimes_s\sigma^1(W))\nabla^E\equiv(V+\dv_{g_0}V)W+(W+\dv_{g_0}W)V
  \]
  where we write `$\equiv$' for the equality of edge normal operators (this identity depending only on the volume density $|\dd g_0|$, and not on the choice of $g_R$), and noting that $\tilde Q=\rho_0^2 x_{\!\scri}^2(x_{\!\scri} D_{x_{\!\scri}}-2\rho_0 D_{\rho_0})$ is symmetric with respect to $|\dd g_0|$, we find that~\eqref{EqEsCErr} is equal to
  \[
    \bigl(-\tilde Q p_{1,y_0}^*Z - Z^* p_{1,y_0}\tilde Q\bigr) - \Bigl( -\frac12\tilde Q(p_{1,y_0}+p_{1,y_0}^*)Z-\frac12 Z^*(p_{1,y_0}+p_{1,y_0}^*)\tilde Q\Bigr) = \tilde\sC'.
  \]
  Expanding the big parenthesis and collecting terms gives~\eqref{EqEsCLot} in the case $p_0=0$.

  Finally, if we allow also $p_{0,y_0}^0$ to be nonzero, the operator $P_0$ gets an extra contribution $\half\rho_0^2 x_{\!\scri}^2 p_{0,y_0}^0$. Relative to the case $p_{0,y_0}^0=0$ just considered, the map $\tilde K_{i Z}$ is unchanged, whereas $\sC_0$ gets an extra contribution $\half i\rho_0^2 x_{\!\scri}^2((p_{0,y_0}^0)^*Z-Z^* p_{0,y_0}^0)$. Putting this term into $\tilde\sC'$ gives~\eqref{EqEsCLot} in general.

  \pfsubstep{\eqref{ItEFw}(ii)}{Energy estimate for compactly supported waves.} We now proceed to compute the leading order term of $\tilde K_{i Z}$ (see~\eqref{EqEFwDef}) at $\scri^+$ for $Z$ defined in~\eqref{EqEVFs}; for this purpose, we may again work with the model metric and edge normal operator~\eqref{EqENorm}. We compute
  \begin{align*}
    {}^V\!\pi &= \rho_0^{-2\alpha_0}x_{\!\scri}^{-4\alpha_{\!\scri}} \Bigl( -4(2-c)\check\alpha_{\!\scri} (\rho_0\pa_{\rho_0})^2 \\
      &\qquad\quad + \bigl( (c-2)\check\alpha_0 + (6-2 c)\check\alpha_{\!\scri}+c-1 \bigr)\,2 \rho_0\pa_{\rho_0}\otimes_s x_{\!\scri}\pa_{x_{\!\scri}} \\
      &\qquad\quad + (2\check\alpha_0-4\check\alpha_{\!\scri} + 1 - c)(x_{\!\scri}\pa_{x_{\!\scri}})^2 + 2 c k^{i j}(x_{\!\scri}\pa_{y^i})\otimes_s(x_{\!\scri}\pa_{y^j}) \Bigr),
  \end{align*}
  and thus
  \begin{align*}
    \tilde K_{i Z} &\equiv \rho_0^{-2\alpha_0}x_{\!\scri}^{-4\alpha_{\!\scri}} \Bigl( (2-c)\bigl(-4\check\alpha_{\!\scri}+2(p_1+p_1^*)\bigr)(\rho_0\pa_{\rho_0})^2 \\
      &\qquad\qquad + \bigl(4\check\alpha_{\!\scri}-2(p_1+p_1^*)+\tfrac{c}{2}(-n+1-4\check\alpha_{\!\scri}+p_1+p_1^*)\bigr)\,2\rho_0\pa_{\rho_0}\otimes_s x_{\!\scri}\pa_{x_{\!\scri}} \\
      &\qquad\qquad + \bigl(-2\check\alpha_{\!\scri}+(p_1+p_1^*)+\tfrac{c}{2}(n-1+2\check\alpha_0)\bigr)(x_{\!\scri}\pa_{x_{\!\scri}})^2 \\
      &\qquad\qquad + \bigl(2-4\check\alpha_{\!\scri}+4\check\alpha_0+c(-n+1-2\check\alpha_0)\bigr) k^{i j}(x_{\!\scri}\pa_{y^i})\otimes_s(x_{\!\scri}\pa_{y^j}) \Bigr) \\
      &\qquad \bmod \rho_0^{-2\alpha_0}x_{\!\scri}^{-4\alpha_{\!\scri}}\cA^{(\ell_0,2\ell_{\!\scri})}\bigl(M;S^2\,\Teb M\otimes\End(\upbeta^*E)\bigr).
  \end{align*}
  Considering the big parenthesis at a point of $\scri^+$, we may diagonalize $p_1+p_1^*$ and thus consider each eigenspace, corresponding to an eigenvalue $2\lambda$, separately; for any desired $\eps>0$, we can ensure that all such eigenvalues satisfy $2\lambda>2\ubar p_1-2\eps$ by choosing a suitable fiber inner product on $E$. The $2\times 2$ minor of the big parenthesis, with respect to the basis $\frac{\dd\rho_0}{\rho_0},\frac{\dd x_{\!\scri}}{x_{\!\scri}}$, has trace $-10\check\alpha_{\!\scri}+10\lambda+\cO(c)$ as $c\searrow 0$, which is thus positive for small fixed $\eps>0$ and all small $c>0$ when $\check\alpha_{\!\scri}<\ubar p_1$, i.e.\ $\alpha_{\!\scri}<-\half+\ubar p_1$. The determinant of the $2\times 2$ minor is $8 c(\check\alpha_{\!\scri}-\lambda)(\check\alpha_{\!\scri}-\check\alpha_0)+\cO(c^2)$, which is positive for small $c>0$ if, in addition, we also have $\check\alpha_{\!\scri}<\check\alpha_0$, i.e.\ $\alpha_{\!\scri}<\half+\alpha_0$. The $(x_{\!\scri}\pa_y)^2$ term of the big parenthesis finally is $4(\check\alpha_0-\check\alpha_{\!\scri}+\half)+\cO(c)$ times a positive definite symmetric 2-tensor in $x_{\!\scri}\pa_{y^j}$ as $c\searrow 0$, the positivity of which is guaranteed upon reducing $c>0$ further if necessary.

  Given weights $\alpha_0,\alpha_{\!\scri}$ as in the statement of the Theorem, we may thus fix $\eps>0$ so that $\alpha_{\!\scri}<-\half+\ubar p_1-\eps$, then pick a near-optimal fiber inner product on $E$ to make $p_1+p_1^*\geq 2\ubar p_1-2\eps$, and then choose $c>0$ sufficiently small in~\eqref{EqEVFs} so that $\rho_0^{2\alpha_0}x_{\!\scri}^{4\alpha_{\!\scri}}\tilde K_{i Z}$ is positive definite at $\Omega^0_{\delta,\rho}\cap\scri^+$ (with a lower bound on its eigenvalues that remains positive as $\eps\searrow 0$), and hence in $\Omega^0_{\delta,\rho}$ for sufficiently small $\delta>0$. Thus, if $\delta>0$ is sufficiently small, then for $u\in\CIc((\Omega^0_{\delta,\rho})^\circ)$ we have
  \[
    \|\nabla^E_\ebop u\|_{\Heb^{0,\alpha}}^2 \leq C\la\nabla^E u,\tilde K_{i Z}\nabla^E u\ra;
  \]
  here, we write $\nabla^E_\ebop u=(\nabla^E_{V_j}u)_{j=1,\ldots,N}$, where $V_1,\ldots,V_N\in\Veb(M)$ denotes any fixed collection of vector fields which span $\Veb(M)$ over $\CI(M)$. Plugging~\eqref{EqEsCLot} into~\eqref{EqEFwComm}, estimating the left hand side of~\eqref{EqEFwComm} using Cauchy--Schwarz, and likewise for the contributions of $\tilde\sC,\tilde\sC'$ in~\eqref{EqEsCLot}, we thus obtain, for $w=x_{\!\scri}+\rho_0^{\ell_0}x_{\!\scri}^{2\ell_{\!\scri}}$,
  \[
    \|\nabla^E_\ebop u\|_{\Heb^{0,\alpha}}^2 \leq C\Bigl( \|P u\|_{\Heb^{0,\alpha'}}^2 + \|w^{1/2}\nabla^E_\ebop u\|_{\Heb^{0,\alpha}}\|w^{1/2}u\|_{\Heb^{0,\alpha}} + \eps'\|\nabla^E_\ebop u\|_{\Heb^{0,\alpha}}\|u\|_{\Heb^{0,\alpha}} \Bigr)
  \]
  for all $u\in\CIc((\Omega^0_{\delta,\rho})^\circ)$ when $\delta>0$ is sufficiently small; here $\eps'=\|p_1-p_1^*\|+\|p_0\|$ can be made arbitrarily small since $P$ is special admissible, while $C$ is independent of $\eps'$. For $\delta,\eps'>0$ sufficiently small, and noting that $w\lesssim x_{\!\scri}^{2\ell_{\!\scri}}\leq\delta^{2\ell_{\!\scri}}$ on $\Omega_{\delta,\rho}^0$ (recalling that $2\ell_{\!\scri}\leq 1$), we may then apply Cauchy--Schwarz to the second and third terms on the right and absorb $\|\nabla^E_\ebop u\|^2$ into the left hand side while retaining a small constant in front of $\|u\|_{\Heb^{0,\alpha}}^2$. We thus obtain
  \begin{equation}
  \label{EqEAlmost}
    \|\nabla^E_\ebop u\|_{\Heb^{0,\alpha}}^2 \leq C\Bigl(\|P u\|_{\Heb^{0,\alpha'}}^2 + (\delta^{\ell_{\!\scri}}+\eps')\|u\|_{\Heb^{0,\alpha}}^2\Bigr),
  \end{equation}
  with $C$ independent of $\delta,\eps'$.
  
  Finally then, we claim that one can estimate
  \begin{equation}
  \label{EqEPoincare}
    \|u\|_{\Heb^{0,\alpha}}\leq C'\|\nabla^E_\ebop u\|_{\Heb^{0,\alpha}}
  \end{equation}
  with a constant $C'$ independent of $\delta$ by integrating the vector field
  \[
    V_0:=\rho_0^{-2\alpha_0}x_{\!\scri}^{-4\alpha_{\!\scri}}(-x_{\!\scri}\pa_{x_{\!\scri}}+2\rho_0\pa_{\rho_0})
  \]
  starting at $x_{\!\scri}=\delta$ (where $u$ vanishes). Phrased as a commutator estimate, this means computing
  \begin{equation}
  \label{EqEPoincareComm}
    2 \Re\la u,V_0 u\ra = -\la(\dv V_0)u,u\ra,
  \end{equation}
  where with respect to the volume density $|\dd g_0|=\rho_0^{-n-1}x_{\!\scri}^{-2 n}|\frac{\dd\rho_0}{\rho_0}\frac{\dd x_{\!\scri}}{x_{\!\scri}}\dd y|$ of the model metric $g_0$ in~\eqref{EqENorm} we have
  \[
    \dv_{g_0} V_0=-(V_0+V_0^*)=\rho_0^2 x_{\!\scri}^2[x_{\!\scri}\pa_{x_{\!\scri}}-2\rho_0\pa_{\rho_0},\rho_0^{-2\check\alpha_0}x_{\!\scri}^{-4\check\alpha_{\!\scri}}] = (4\check\alpha_0-4\check\alpha_{\!\scri})\rho_0^{-2\alpha_0}x_{\!\scri}^{-4\alpha_{\!\scri}},
  \]
  which is a positive multiple of $\rho_0^{-2\alpha_0}x_{\!\scri}^{-4\alpha_{\!\scri}}$ since $\check\alpha_{\!\scri}<\check\alpha_0$. The Cauchy--Schwarz inequality then implies the estimate~\eqref{EqEPoincare}. Plugging this into~\eqref{EqEAlmost} and taking $\delta,\eps'$ sufficiently small, we obtain $\|\nabla^E_\ebop u\|_{\Heb^{0,\alpha}}\leq C\|P u\|_{\Heb^{0,\alpha'}}$ provided $\delta>0$ is sufficiently small, and thus~\eqref{EqEEnergyEst0} by applying~\eqref{EqEPoincare}, for $u\in\CIc((\Omega^0_{\delta,\rho})^\circ)$.

  \pfsubstep{\eqref{ItEFw}(iii)}{Energy estimate for distributions extendible at $\Sigma^{0,\rm out}_{\delta,\rho}$.} We proceed to remove the assumption that $u$ vanish near $\Sigma^{0,\rm out}_{\delta,\rho}$. Thus, in the integration by parts in~\eqref{EqEFwComm}, one now has to add a boundary term at $\Sigma^{0,\rm out}_{\delta,\rho}$. Equivalently, one can multiply the vector field multiplier $Z$ with a sharp cutoff $1_{\Omega^0_{\delta,\rho}}$ (the characteristic function of $\Omega^0_{\delta,\rho}$), in which case~\eqref{EqEFwComm} is valid still, but the operator $\sC$ now has distributional coefficients. Using the latter formulation, we note that for any distribution $\chi$, we have $\cL_{\chi V}g^{-1}=\chi(\cL_V g^{-1})-2 V\otimes_s\nabla\chi$ and thus
  \[
    \tilde K_{\chi V} = \chi\tilde K_V + 2 T(\nabla\chi,V),\qquad T(X,Y):=X\otimes_s Y-\half g(X,Y)g^{-1}.
  \]
  Here, the abstract stress-energy-momentum tensor $T(X,Y)$ is positive definite when $X,Y$ are both future timelike. Let now $\chi=1_{\Omega^0_{\delta,\rho}}$. The extra contribution to $\la\sC u,u\ra$ in~\eqref{EqEsCLot} arising from differentiation of $\chi$ is $2\la T(\nabla\chi,i Z)\nabla^E u,\nabla^E u\ra$. This is the sum of two terms: one from integration over $\Sigma^{0,\rm in}_{\delta,\rho}$ which vanishes since $u$ has supported character (i.e.\ vanishes) there; and another term, arising from integration over the final Cauchy surface $\Sigma^{0,\rm out}_{\delta,\rho}$, which is nonnegative due to the future timelike nature of $\nabla\chi$ (using the spacelike nature of $\Sigma_{\delta,\rho}^{0,\rm out}$) and $i Z$, and thus of the same sign as the main term (arising from $1_{\Omega^0_{\delta,\rho}}\tilde K_{i Z}$). Hence, it can be dropped.

  The estimate~\eqref{EqEPoincare} follows similarly for $u$ which are extendible at $\Sigma^{0,\rm out}_{\delta,\rho}$: in the commutator calculation~\eqref{EqEPoincareComm}, one replaces $V_0$ by $\chi V_0$, $\chi=1_{\Omega^0_{\delta,\rho}}$, which gives
  \[
    2\Re\la u,\chi V_0 u\ra = -\la\chi(\dv V_0)u,u\ra - \la(V_0\chi)u,u\ra.
  \]
  The second term (including the minus sign) is a sum of two terms; one arising from integration over $\Sigma^{0,\rm in}_{\delta,\rho}$ which vanishes due to the support assumption on $u$, and one from integration over $\Sigma^{0,\rm out}_{\delta,\rho}$ which is nonnegative, hence has the same sign as the main term (from $\dv V_0$) and hence can be dropped. This establishes~\eqref{EqEEnergyEst0} for smooth elements $u$ of $\Heb^{1,\alpha}(\Omega^0_{\delta,\rho})^{\bullet,-}$ whose support inside $M$ does not intersect $\scri^+$.

  \pfsubstep{\eqref{ItEFw}(iv)}{Unconditional estimate near $\scri^+$.} For $u\in\Heb^{2,\alpha}(\Omega^0_{\delta,\rho})^{\bullet,-}$, the estimate~\eqref{EqEEnergyEst0} now follows by a density argument. Next, given $f\in\Heb^{0,\alpha'}(\Omega^0_{\delta,\rho})^{\bullet,-}\cap H^1_\loc(M^\circ)$, the unique forward solution $u$ of $P u=f$ lies in $H^2_\loc((\Omega^0_{\delta,\rho})^\circ)$. In order to show that it satisfies the estimate~\eqref{EqEEnergyEst0}, we cannot directly apply the previous estimates, as a priori the growth of $u$ at $\scri^+$ is not controlled. Instead, we apply the above energy estimates on $\Omega^0_{\delta,\rho}\setminus\Omega^0_{\delta',\rho}$, i.e.\ using the sharp cutoff $\chi:=1_{\Omega^0_{\delta,\rho}}-1_{\Omega^0_{\delta',\rho}}$, for $\delta'\in(0,\delta)$. The boundary hypersurfaces of $\Omega^0_{\delta,\rho}\setminus\Omega^0_{\delta',\rho}$ are $\Sigma^{0,\rm in}_{\delta,\rho}$ (where $u$ has supported character), $\Sigma^{0,\rm out}_{\delta,\rho}\setminus\Sigma^{0,\rm out}_{\delta',\rho}$ (where the boundary term has a good sign, as shown before) and $\Sigma_{\delta',\rho}^{0,\rm in}$, where the boundary term again has a good sign (since $\nabla\chi$ is future timelike there). Thus, upon dropping these advantageous boundary terms, we have
  \[
    \|u\|_{\Heb^{1,\alpha}(\Omega^0_{\delta,\rho}\setminus\Omega^0_{\delta',\rho})^{\bullet,-}}\leq C\|P u\|_{\Heb^{0,\alpha'}(\Omega^0_{\delta,\rho}\setminus\Omega^0_{\delta',\rho})},
  \]
  where the constant $C$ is independent of $\delta'$; hence, upon taking $\delta'\searrow 0$ we conclude that $u\in\Heb^{1,\alpha}(\Omega^0_{\delta,\rho})^{\bullet,-}$, and the estimate~\eqref{EqEEnergyEst0} holds.

  Finally, for general $f\in\Heb^{0,\alpha'}(\Omega_{\delta,\rho}^0)^{\bullet,-}$, we pick a sequence $f_j\in\Heb^{1,\alpha'}(\Omega_{\delta,\rho}^0)^{\bullet,-}$ which converges to $f$ in $\Heb^{0,\alpha'}(\Omega_{\delta,\rho}^0)^{\bullet,-}$. The forward solution $u_j$ of $P u_j=f_j$ satisfies $u_j\in\Heb^{1,\alpha}(\Omega_{\delta,\rho}^0)^{\bullet,-}$. Applying~\eqref{EqEEnergyEst0} to the differences $u_j-u_k$ shows that $u_j$ is a Cauchy sequence, and the limit $u\in\Heb^{1,\alpha}(\Omega_{\delta,\rho}^0)^{\bullet,-}$ satisfies $P u=f$ together with the estimate~\eqref{EqEEnergyEst0}.
 
  \pfsubstep{\eqref{ItEFw}(v)}{Unconditional estimate in the original domain $\Omega^0_{\delta,\rho}$.} Finally, we need to prove~\eqref{EqEEnergyEst0} for the \emph{original} value of $\delta>0$, rather than for the (possibly much smaller) value fixed in the course of the above argument, which we denote by $\delta'\in(0,\delta]$ now for clarity. This, however, is straightforward, as one can estimate
  \begin{equation}
  \label{EqEEnergyEst1}
    \|u\|_{\Heb^{1,\alpha}(\Omega^0_{\delta,\rho}\setminus\Omega^0_{\delta',\rho})^{\bullet,-}}\leq C\|P u\|_{\Heb^{0,\alpha'}(\Omega^0_{\delta,\rho}\setminus\Omega^0_{\delta',\rho})^{\bullet,-}},
  \end{equation}
  with `$\bullet$' now indicating supported character at $\Sigma^{0,\rm in}_{\delta,\rho}$, and `$-$' indicating extendible character at $\Sigma^{0,\rm out}_{\delta,\rho}\setminus\Sigma^{0,\rm out}_{\delta',\rho}$ as well as at $\Sigma^{0,\rm in}_{\delta',\rho}$. Indeed, the domain $\Omega^0_{\delta,\rho}\setminus\Omega^0_{\delta',\rho}$ is disjoint from $\scri^+$, and hence the weight at $\scri^+$ is irrelevant, and the timelike function $x_{\!\scri}$ is restricted to a compact subset of $(0,\infty)$; thus,~\eqref{EqEEnergyEst1} is a standard finite time energy estimate. (Concretely, one can use the vector field multiplier $\rho_0^{-2\check\alpha_0}e^{\digamma x_{\!\scri}}W$, with $W=-x_{\!\scri}\pa_{x_{\!\scri}}+\rho_0\pa_{\rho_0}$, or any other smooth future timelike edge-b-vector field, and $\digamma>1$ sufficiently large---thus, the estimate allows for exponential growth of $u$ as $x_{\!\scri}$ decreases, which is acceptable since we are restricting to a domain where $x_{\!\scri}$ is bounded from below by the fixed positive constant $\delta'$. Cf.\ \cite[\S1.1.1 and \S4]{HintzVasyMink4}.)

  \pfstep{Part~\eqref{ItEFw}, $s>1$.} Fix $\rho'\in(\rho,\bar\rho_0)$. Given $f\in\Heb^{s-1,\alpha'}(\Omega^0_{\delta,\rho})^{\bullet,-}$, we first extend $f$ to $\tilde f\in\Heb^{s-1,\alpha'}(\Omega^0_{\delta,\rho'})^{\bullet,-}$ so that $\|\tilde f\|_{\Heb^{s-1,\alpha'}(\Omega^0_{\delta,\rho'})^{\bullet,-}}\leq C\|f\|_{\Heb^{s-1,\alpha'}(\Omega^0_{\delta,\rho})^{\bullet,-}}$ (which can be arranged for some constant $C$ depending only on $s,\alpha,\delta,\rho,\rho'$, but not on $f$). The forward solution $\tilde u$ of $P\tilde u=\tilde f$ then satisfies
  \begin{equation}
  \label{EqESolvEst}
    \|\tilde u\|_{\Heb^{1,\alpha}(\Omega^0_{\delta,\rho'})^{\bullet,-}} \leq C\|\tilde f\|_{\Heb^{0,\alpha'}(\Omega^0_{\delta,\rho'})^{\bullet,-}}
  \end{equation}
  by the first part of the proof. Using the higher regularity of $\tilde f$, we can then propagate $\Heb^{s,\alpha}$ regularity from $x_{\!\scri}>\delta$, where $\tilde f$ and $\tilde u$ vanish and hence are smooth, using (the quantitative estimate versions of) Lemmas~\ref{LemmaMEc}\eqref{ItMEcFw}, \ref{LemmaMEinm}\eqref{ItMEinmFw}, and \ref{LemmaMEout}\eqref{ItMEoutFwLoc}, thus proving local $\Heb^{s,\alpha}$ regularity in $(\Omega^0_{\delta,\rho'})^\circ$. Note here that with $s_0:=1$, the a priori $\Heb^{s_0,\alpha}$ regularity of $u$ at $\cR_{\rm c}$ is strong enough for an application of Lemma~\ref{LemmaMEc}\eqref{ItMEcFw} since, due to the assumptions~\eqref{EqEFwAssm}, we have
  \[
    \frac12+(\alpha_{\!\scri}-\alpha_0) + (\alpha_{\!\scri}-\ubar p_1) < \frac12+\frac12-\frac12 = \frac12 < s_0.
  \]
  Restricting to $\Omega^0_{\delta,\rho}$ gives, in view of~\eqref{EqESolvEst}, the desired estimate
  \begin{align*}
    \|u\|_{\Heb^{s,\alpha}(\Omega^0_{\delta,\rho})^{\bullet,-}} = \|\tilde u|_{\Omega^0_{\delta,\rho}}\|_{\Heb^{s,\alpha}(\Omega^0_{\delta,\rho})^{\bullet,-}} &\leq C\Bigl(\|\tilde f\|_{\Heb^{s-1,\alpha'}(\Omega^0_{\delta,\rho'})^{\bullet,-}}+\|\tilde u\|_{\Heb^{1,\alpha}(\Omega^0_{\delta,\rho'})^{\bullet,-}}\Bigr) \\
      &\leq C'\|f\|_{\Heb^{s-1,\alpha'}(\Omega^0_{\delta,\rho})^{\bullet,-}}.
  \end{align*}
  (This extension and restriction procedure is necessitated by the fact that the microlocal estimates are not sharply localized, unlike energy estimates. See also \cite[\S2.1.3]{HintzVasySemilinear}.)

  \pfstep{Part~\eqref{ItEBw}, $\tilde s\leq 0$.} Let $s=-\tilde s+1$, $(\alpha_0,2\alpha_{\!\scri})=(-\tilde\alpha'_0,-2\tilde\alpha'_{\!\scri})$. The a priori estimate
  \begin{equation}
  \label{EqEFwEst}
    \|u\|_{\Heb^{s,\alpha}(\Omega^0_{\delta,\rho})^{\bullet,-}} \leq C\|P u\|_{\Heb^{s-1,\alpha'}(\Omega^0_{\delta,\rho})^{\bullet,-}}
  \end{equation}
  (valid for all $u$ for which both norms are finite) implies, by duality, that the equation $P^*\tilde u=\tilde f$ for $\tilde f\in(\Heb^{s,\alpha}(\Omega^0_{\delta,\rho})^{\bullet,-})^*=\Heb^{\tilde s-1,\tilde\alpha'}(\Omega^0_{\delta,\rho})^{-,\bullet}$ has a solution $\tilde u\in(\Heb^{s-1,\alpha'}(\Omega^0_{\delta,\rho})^{\bullet,-})^*=\Heb^{\tilde s,\tilde\alpha}(\Omega^0_{\delta,\rho})^{-,\bullet}$ which moreover satisfies the estimate
  \begin{equation}
  \label{EqEBwEst}
    \|\tilde u\|_{\Heb^{\tilde s,\tilde\alpha}(\Omega^0_{\delta,\rho})^{-,\bullet}}\leq C\|\tilde f\|_{\Heb^{\tilde s-1,\tilde\alpha'}(\Omega^0_{\delta,\rho})^{-,\bullet}}
  \end{equation}
  with the same constant $C$ as in~\eqref{EqEFwEst}.

  As for the uniqueness claim, suppose that $\tilde u\in\Heb^{-\infty,\tilde\alpha}(\Omega^0_{\delta,\rho})^{-,\bullet}$ satisfies $P^*\tilde u=0$. Using the microlocal propagation estimates proved in~\S\ref{SsME}, we find that $\tilde u\in\Heb^{\tilde s,\tilde\alpha}(\Omega^0_{\delta',\rho})^{-,\bullet}$ for any $\delta'<\delta$. Given any $f\in(\Heb^{\tilde s,\tilde\alpha}(\Omega^0_{\delta',\rho})^{-,\bullet})^*$, the solvability of $P u=f$ with $u\in(\Heb^{\tilde s-1,\tilde\alpha'}(\Omega^0_{\delta',\rho})^{-,\bullet})^*$ implies that $\la \tilde u,f\ra=\la \tilde u,P u\ra=\la P^*\tilde u,u\ra=0$. Since $f$ is arbitrary, this implies $\tilde u=0$.

  In particular, given $\tilde u\in\Heb^{\tilde s,\tilde\alpha}(\Omega^0_{\delta,\rho})^{-,\bullet}$ for which $\tilde f:=P^*\tilde u$ satisfies $\tilde f\in\Heb^{\tilde s-1,\tilde\alpha'}(\Omega^0_{\delta,\rho})^{-,\bullet}$, the solution $\tilde u'\in\Heb^{\tilde s,\tilde\alpha}(\Omega^0_{\delta,\rho})^{-,\bullet}$ of $P^*\tilde u'=\tilde f$ constructed above by duality must be equal to $\tilde u$; but this implies that $\tilde u=\tilde u'$ obeys the norm bound~\eqref{EqEBwEst}. Thus, the estimate~\eqref{EqEBwEst} holds for all $\tilde u$ for which both sides are finite.

  \pfstep{Part~\eqref{ItEBw}, full range of $\tilde s$.} Let now $0\leq\tilde s<\half-\tilde\alpha_0+2\tilde\alpha_{\!\scri}+\ubar p_1$. Extending the desired forcing term $\tilde f\in\Heb^{\tilde s-1,\tilde a'}(\Omega_{\delta,\rho}^0)^{-,\bullet}$ to a slightly larger domain $\Omega^0_{\delta',\rho}$, with $\delta'>\delta$, to a distribution $\tilde f'$ with norm bounded by a fixed constant times that of $\tilde f$, then finding the solution $P^*\tilde u'=\tilde f'$ on the larger domain in the space $\Heb^{0,\tilde\alpha}(\Omega^0_{\delta',\rho})^{-,\bullet}$ using the previous step of the proof, subsequently using the propagation results in~\S\ref{SsME}, and finally restricting back to $\Omega^0_{\delta,\rho}$ produces a solution of $P^*\tilde u=\tilde f$ with the desired regularity, and gives a quantitative estimate
  \begin{equation}
  \label{EqEBwFullEst}
    \|\tilde u\|_{\Heb^{\tilde s,\tilde\alpha}(\Omega^0_{\delta,\rho})^{-,\bullet}}\leq C\|P^*\tilde u\|_{\Heb^{\tilde s-1,\tilde\alpha'}(\Omega^0_{\delta,\rho})^{-,\bullet}}.
  \end{equation}
  Propagation of regularity and uniqueness of solutions in $\Heb^{-\infty,\tilde\alpha}(\Omega^0_{\delta,\rho})^{-,\bullet}$ of $P^*\tilde u=\tilde f$ implies, as in the previous step, that~\eqref{EqEBwFullEst} holds for all $\tilde u$ for which both sides are finite.

  \pfstep{Part~\eqref{ItEFw}, full range of $s$.} Letting $\tilde s=-s+1$ and $(\tilde\alpha_0,2\tilde\alpha_{\!\scri})=(-\alpha_0',-2\alpha_{\!\scri}')$, this now follows by duality from the a priori estimate~\eqref{EqEBwFullEst}.
\end{proof}

To state the forward version with higher b-regularity, we shall use the Sobolev spaces $H_{\ebop;\bop}^{(s;k),\alpha}(\Omega_{\delta,\rho}^0;\upbeta^*E)^{\bullet,-}$ which are defined analogously to Definition~\ref{DefESpaces}.

\begin{cor}[Higher b-regularity]
\label{CorEFwb}
  Let $P$ be a special admissible operator, and let $\ubar p_1$ be as in Definition~\usref{DefMEMin}. Let $s,\alpha_0,\alpha_{\!\scri}\in\R$ and $k\in\N_0$, put $\alpha=(\alpha_0,2\alpha_{\!\scri})$ and $\alpha'=\alpha+(2,2)$, and suppose that
  \[
    s>\frac12-\alpha_0+2\alpha_{\!\scri}-\ubar p_1,\qquad
    \alpha_{\!\scri}<\alpha_0+\frac12,\qquad
    \alpha_{\!\scri}<-\frac12+\ubar p_1.
  \]
  Let $f\in H_{\ebop;\bop}^{(s-1;k),\alpha'}(\Omega_{\delta,\rho}^0;\upbeta^*E)^{\bullet,-}$. Then the unique distributional forward solution $u$ of $P u=f$ in $\Omega_{\delta,\rho}^0\cap M^\circ$ satisfies $u\in H_{\ebop;\bop}^{(s;k),\alpha}(\Omega_{\delta,\rho}^0;\upbeta^*E)^{\bullet,-}$, and
    \[
      \|u\|_{H_{\ebop;\bop}^{(s;k),\alpha}(\Omega^0_{\delta,\rho};\upbeta^*E)^{\bullet,-}}\leq C\|f\|_{H_{\ebop;\bop}^{(s-1;k),\alpha'}(\Omega^0_{\delta,\rho};\upbeta^*E)^{\bullet,-}},
    \]
    where $C$ only depends on $P,s,k,\alpha_0,\alpha_{\!\scri},\delta,\rho$.
\end{cor}
\begin{proof}
  For $k=0$, this is Theorem~\ref{ThmE}\eqref{ItEFw}. Suppose, by induction, that we have proved the Corollary for some amount $k\in\N_0$ of b-regularity, and let $f\in H_{\ebop;\bop}^{(s-1;k+1),\alpha'}\subset H_{\ebop;\bop}^{((s+1)-1;k),\alpha'}$. Then the inductive hypothesis gives $u\in H_{\ebop;\bop}^{(s+1;k),\alpha}$. Let $X\in\Diff_{[\bop]}^1(M;E)$ be a commutator b-operator (see Definition~\ref{DefbCOp}). Then
  \[
    P(X u) = X f + [P,X]u \in H_{\ebop;\bop}^{(s-1;k),\alpha'}
  \]
  since $[P,X]\in\cA^{(2,2,2)}\Diffeb^2(M)$ by Lemma~\ref{LemmabCComm}. Applying the inductive hypothesis again, we find $X u\in H_{\ebop;\bop}^{(s;k),\alpha}(M)$. Since $X$ was arbitrary, this implies $u\in H_{\ebop;\bop}^{(s;k+1),\alpha}(M)$ by Proposition~\ref{PropbCMany}.
\end{proof}

\section{Control of edge-b-decay at \texorpdfstring{$\scri^+$}{null infinity}}
\label{SNe}

We use the notation $M,M_0,Y$, the coordinates $\varrho=\frac{1}{t+r}$, $v=\frac{t-r}{t+r}$ as laid out at the beginning of~\S\ref{SE}, and we continue to work with an admissible operator $P$ as in Definition~\ref{DefOp}. In the previous section, we obtained full control of forward solutions of $P u=f$ near $\scri^+\setminus I^+$ on the scale of weighted edge-b-Sobolev spaces. We shall now work near $I^+$; we use $t_*=t-r$ and the coordinates $\rho_+=\frac{1}{t_*-T}$, $x_{\!\scri}=\sqrt{\frac{t_*-T}{r}}$ as in~\eqref{EqOGeoCoordIp}, where $T\in\R$ is arbitrary but fixed. We drop the weight at $I^0$ from the notation, and also the vector bundle $\upbeta^*E\to M$ unless additional arguments are required in its presence.

As explained in~\S\ref{SI}, the behavior of waves---both regularity and decay---uniformly near $\scri^+\cap I^+$, cannot be analyzed locally, but rather depends on global information which is not captured by Definition~\ref{DefOp}. In a situation where regularity is controlled globally near $\scri^+$ however (see~\cite{HintzNonstat} for several classes of examples), growth/decay at $\scri^+$, as measured in edge-b-Sobolev spaces, is controlled by the edge normal operator of $P$. In~\S\ref{SsNe}, we prove a priori estimates and solvability results for the localization of the edge normal operator of $P$ near $\scri^+\cap I^+$. These are used in~\S\ref{SsNeLOC} to prove a priori estimates for $P$ near $\scri^+\cap I^+$ which control a solution $u$ of $P u=f$ in the sense of decay near $\scri^+$ (though with a loss of derivatives); see Theorem~\ref{ThmNeP}.

\subsection{Analysis of the edge normal operator}
\label{SsNe}

Fixing $y_0\in Y$ and writing $p_{0,y_0}^+:=p_0^+|_{\scri^+_{y_0}}\in(\CI+\cA^{\ell_+})([0,1)_{\rho_+};\End(E_{y_0}))$, $p_{1,y_0}:=p_1|_{\scri^+_{y_0}}\in(\CI+\cA^{\ell_+})([0,1)_{\rho_+};\End(E_{y_0}))$, consider the edge normal operator of $P$ at $\scri^+$,\footnote{We write $P_{y_0}:={}^\eop N_{\scri^+,y_0}(P)$ for brevity.}
\[
  P_{y_0} = \frac12 x_{\!\scri}^2\rho_+^2\biggl(\Bigl(x_{\!\scri} D_{x_{\!\scri}}-2 i^{-1}\Bigl(\frac{n-1}{2}+p_{1,y_0}\Bigr)\Bigr)(x_{\!\scri} D_{x_{\!\scri}}-2\rho_+ D_{\rho_+}) + (x_{\!\scri} D_y)^2 + p_{0,y_0}^+\biggr),
\]
on ${}^+N\scri^+_{y_0}$ near $\rho_+=0$ (cf.\ \eqref{EqOpNorm}), i.e.\ on the domain\footnote{We are committing an minor abuse of notation here by writing $x_{\!\scri}$ for a fiber-linear coordinate on ${}^+N\scri^+_{y_0}$ whose differential at the zero section we take to agree with the differential of the coordinate function on $M$ denoted $x_{\!\scri}$ above.}
\[
  \cN := [0,1)_{\rho_+}\times[0,\infty)_{x_{\!\scri}}\times\R^{n-1}_y
\]
and acting on sections of the trivial bundle $E_{y_0}$; thus $p_{j,y_0}\in(\CI+\cA^{\ell_+})([0,1)_{\rho_+};\End(E_{y_0}))$ for $j=0,1$. Here, we made a linear change of the $y$-coordinates so that $\pa_{y^1},\ldots,\pa_{y^{n-1}}$ is an orthonormal basis for the metric $\half k(y_0)\in S^2 T_{y_0}^*Y$. Denote the underlying metric by
\begin{align*}
  g_{y_0} &:= 2 x_{\!\scri}^{-2}\rho_+^{-2}\Bigl(-2\frac{\dd x_{\!\scri}}{x_{\!\scri}}\otimes_s\frac{\dd\rho_+}{\rho_+}-\frac{\dd\rho_+^2}{\rho_+^2} + \frac{\dd y^2}{x_{\!\scri}^2}\Bigr), \\
  g_{y_0}^{-1} &= \frac12 x_{\!\scri}^2\rho_+^2\Bigl(x_{\!\scri}\pa_{x_{\!\scri}}\otimes_s(x_{\!\scri}\pa_{x_{\!\scri}}-2\rho_+\pa_{\rho_+}) + (x_{\!\scri}\pa_y)^2\Bigr).
\end{align*}
On $\cN$, we moreover fix the volume density $|\dd g_{y_0}|=2 x_{\!\scri}^{-2 n}\rho_+^{-n-1}|\frac{\dd x_{\!\scri}}{x_{\!\scri}}\frac{\dd\rho_+}{\rho_+}\dd y|$, cf.\ \eqref{EqMEDensity}.

Since $P_{y_0}\in\Diff_{\ebop,I}^2(\cN)$---i.e.\ it is invariant under dilations in $(x_{\!\scri},y)$ and translations in $y$---its analysis will utilize the invariant edge-b-notions (Sobolev spaces, ps.d.o.s, wave front sets) introduced in~\S\ref{SsEBI}. For $\rho\in(0,1)$, let
\[
  \Omega_\rho := \{ \rho_+<\rho \} \subset \cN.
\]

\begin{prop}[Invertibility of the edge normal operator near $I^+$]
\label{PropNe}
  Assume that $P$ is special admissible in the sense of Definition~\usref{DefESpecial}.
  \begin{enumerate}
  \item\label{ItNeFw}{\rm (Forward problem.)} Let $s,\alpha_{\!\scri},\alpha_+\in\R$, and put $\alpha=(2\alpha_{\!\scri},\alpha_+)$ and $\alpha'=\alpha+(2,2)$. Suppose that
    \begin{equation}
    \label{EqNeFwAssm}
      \alpha_+ + \frac12 < \alpha_{\!\scri} < -\frac12+\ubar p_1.
    \end{equation}
    Let $u\in\dot H_{\ebop,I}^{s,\alpha}(\Omega_\rho)$ be such that $f:=P_{y_0}u\in\dot H_{\ebop,I}^{s-1,\alpha'}(\Omega_\rho)$. Then
    \begin{equation}
    \label{EqNeFwApriori}
      \|u\|_{\dot H_{\ebop,I}^{s,\alpha}(\Omega_\rho)}\leq C\|f\|_{\dot H_{\ebop,I}^{s-1,\alpha'}(\Omega_\rho)},
    \end{equation}
    where $C$ only depends on $P,s,\alpha_{\!\scri},\alpha_+,\rho$. Moreover, for any $f\in\dot H_{\ebop,I}^{s-1,\alpha'}(\Omega_\rho)$, there exists $u$ satisfying the equation $P_{y_0}u=f$ and the estimate~\eqref{EqNeFwApriori}; this $u$ is unique in the sense that any other solution $u'\in\dot H_{\ebop,I}^{-\infty,\alpha}(\Omega_\rho)$ of $P_{y_0}u'=f$ is equal to $u$.
  \item\label{ItNeBw}{\rm (Backward problem.)} Let $\tilde s,\tilde\alpha_{\!\scri},\tilde\alpha_+\in\R$, and put $\tilde\alpha=(2\tilde\alpha_{\!\scri},\tilde\alpha_+)$ and $\tilde\alpha'=\tilde\alpha+(2,2)$. Suppose that
    \begin{equation}
    \label{EqNeBwAssm}
      -\frac12-\ubar p_1 < \tilde\alpha_{\!\scri} < \tilde\alpha_+ + \frac12.
    \end{equation}
    Let $\tilde u\in\bar H_{\ebop,I}^{\tilde s,\tilde\alpha}(\Omega_\rho)$ be such that $\tilde f:=P_{y_0}^*\tilde u\in\bar H_{\ebop,I}^{\tilde s-1,\tilde\alpha'}(\Omega_\rho)$. Then
    \begin{equation}
    \label{EqNeBwApriori}
      \|\tilde u\|_{\bar H_{\ebop,I}^{\tilde s,\tilde\alpha}(\Omega_\rho)} \leq C\|\tilde f\|_{\bar H_{\ebop,I}^{\tilde s-1,\tilde\alpha'}(\Omega_\rho)},
    \end{equation}
    where $C$ only depends on $P,\tilde s,\tilde\alpha_{\!\scri},\tilde\alpha_+$. Moreover, for any $\tilde f\in\bar H_{\ebop,I}^{\tilde s-1,\tilde\alpha'}(\Omega_\rho)$, there exists $\tilde u$ which satisfies the equation $P_{y_0}^*\tilde u=\tilde f$ and the estimate~\eqref{EqNeBwApriori}; this $\tilde u$ is unique in the sense that any solution $\tilde u'\in\bar H_{\ebop,I}^{-\infty,\tilde\alpha}(\Omega_\rho)$ of $P_{y_0}^*\tilde u'=\tilde f$ is equal to $\tilde u$.
  \end{enumerate}
\end{prop}
\begin{proof}
  \pfstep{Restricted a priori estimate for $P_{y_0}$, $s=1$; solvability for $P_{y_0}^*$, $\tilde s=0$.} If we impose slightly stronger conditions on the weights (see~\eqref{EqNeFwAssm2} and~\eqref{EqNeBwAssm2} below), this follows from a variant of the proof of Theorem~\ref{ThmE}. Thus, we first prove part~\eqref{ItNeFw} for $s=1$ using an energy estimate, based on
  \[
    2\Im\la P_{y_0}u,Z u\ra = \la\sC u,u\ra,\qquad \sC:=i(P_{y_0}^*Z-Z^*P_{y_0}).
  \]
  Here, we let $\check\alpha=(2\check\alpha_{\!\scri},\check\alpha_+)=(2\alpha_{\!\scri},\alpha_+)+(1,1)$ and define
  \begin{equation}
  \label{EqNeVF}
    W:=x_{\!\scri}\pa_{x_{\!\scri}}-(2+c)\rho_+\pa_{\rho_+},\qquad
    V:=x_{\!\scri}^{-4\check\alpha_{\!\scri}}\rho_+^{-2\check\alpha_+}W,\qquad
    Z:=i^{-1}V;
  \end{equation}
  note that by~\eqref{EqOGeoTimelike}, $W$ is future timelike for $c>0$. Define
  \begin{equation}
  \label{EqNeFwsC}
  \begin{split}
    \tilde Q&:=x_{\!\scri}^2\rho_+^2(x_{\!\scri} D_{x_{\!\scri}}-2\rho_+ D_{\rho_+})=\tilde Q^*, \\
    \tilde K_V &:= K_V + (p_{1,y_0}+p_{1,y_0}^*)\sigma^1(\tilde Q)\otimes_s\sigma^1(i^{-1}V), \\
    \tilde\sC' &:= -\frac12\bigl(\tilde Q(p_{1,y_0}-p_{1,y_0}^*)Z - Z^*(p_{1,y_0}-p_{1,y_0}^*)\tilde Q\bigr) + \frac{i}{2}\rho_0^2 x_{\!\scri}^2\bigl((p_{0,y_0}^0)^*Z-Z^*p_{0,y_0}^0\bigr) \\
      &\qquad \in x_{\!\scri}^{-4\alpha_{\!\scri}}\rho_+^{-2\alpha_+}\bigl(\Diff_{\ebop,I}^1(\cN;E_{y_0})+\cA^{\ell_+}(Z)\Diff_{\ebop,I}^1(\cN;E_{y_0})\bigr).
  \end{split}
  \end{equation}
  Here, $K_V$ is given by~\eqref{EqEFwDef} for the metric $g=g_{y_0}$. We then have
  \[
    \sC = (\nabla^E)^*\tilde K_{i Z}\nabla^E + \tilde\sC'.
  \]
  Indeed, this is proved like~\eqref{EqEsCLot}, except in the present invariant setting, there are no terms which are of lower order in the sense of decay at $x_{\!\scri}=0$. (The various sign changes compared to~\eqref{EqEsCLot} and \eqref{EqEsCErr} are due to the difference of signs in front of the first terms in~\eqref{EqOpNorm}.) We compute
  \begin{align*}
    \tilde K_{i Z} &= x_{\!\scri}^{-4\alpha_{\!\scri}}\rho_+^{-2\alpha_+} \Bigl( \bigl(-2\check\alpha_{\!\scri}+(p_{1,y_0}+p_{1,y_0}^*) + \tfrac{c}{2}(-n+1-2\check\alpha_+)\bigr) (x_{\!\scri}\pa_{x_{\!\scri}})^2 \\
      &\qquad\quad + \bigl(4\check\alpha_{\!\scri}-2(p_{1,y_0}+p_{1,y_0}^*)+\tfrac{c}{2}(n-1+4\check\alpha_{\!\scri}-(p_{1,y_0}+p_{1,y_0}^*))\bigr)\,2 x_{\!\scri}\pa_{x_{\!\scri}}\otimes_s\rho_+\pa_{\rho_+} \\
      &\qquad\quad + (2+c)\bigl(-4\check\alpha_{\!\scri}+2(p_{1,y_0}+p_{1,y_0}^*)\bigr)(\rho_+\pa_{\rho_+})^2 \\
      &\qquad\quad + \bigl(-2+4\check\alpha_{\!\scri}-4\check\alpha_+ + c(-n+1-2\check\alpha_+)\bigr)(x_{\!\scri}\pa_y)^2 \Bigr).
  \end{align*}
  In the basis $\frac{\dd x_{\!\scri}}{x_{\!\scri}}$, $\frac{\dd\rho_+}{\rho_+}$, $\frac{\dd y}{x_{\!\scri}}$, and restricting to an eigenspace of $p_{1,y_0}+p_{1,y_0}^*$ with eigenvalue $2\lambda$ (satisfying $2\lambda>2\ubar p_1-2\eps$), the trace and determinant of the $2\times 2$ minor of $x_{\!\scri}^{4\alpha_{\!\scri}}\rho_+^{2\alpha_+}\tilde K_{i Z}$ are equal to $10(-\check\alpha_{\!\scri}+\lambda)+\cO(c)$ and $8(-\check\alpha_{\!\scri}+\lambda)(\check\alpha_{\!\scri}-\check\alpha_+)c+\cO(c^2)$, respectively, and therefore positive for $\check\alpha_{\!\scri}<\ubar p_1$ (provided $\eps>0$ is sufficiently small and we choose the fiber inner product on $E_{y_0}$ near-optimally) and $\check\alpha_+<\check\alpha_{\!\scri}$ (these are precisely the conditions~\eqref{EqNeFwAssm}), and then $c>0$ sufficiently small. (The term $\tilde\sC'$ in~\eqref{EqNeFwsC} can at the same time be made arbitrarily small by assumption on $p_{1,y_0}$.) However, the $(x_{\!\scri}\pa_y)^2$ term has a positive coefficient (for $c$ near $0$) only under the stronger condition $\check\alpha_+<-\half+\check\alpha_{\!\scri}$, i.e.
  \begin{equation}
  \label{EqNeFwAssm2}
    \alpha_+ + 1 < \alpha_{\!\scri}.
  \end{equation}

  Let us write $\nabla_{\ebop,I}^E u=(\nabla^E_{V_j}u)_{j=1,\ldots,N}$, where $V_1,\ldots,V_N$ spans $\cV_{\ebop,I}(\cN)$ over $\CI_I(\cN)$. An application of Cauchy--Schwarz and the invariant Poincar\'e inequality
  \[
    \|u\|_{\dot H_{\ebop,I}^{0,\alpha}} \leq C'\|\nabla_{\ebop,I}^E u\|_{\dot H_{\ebop,I}^{0,\alpha}}
  \]
  (proved via~\eqref{EqEPoincareComm} with $V_0:=x_{\!\scri}^{-4\alpha_{\!\scri}}\rho_+^{-2\alpha_+}(x_{\!\scri}\pa_{x_{\!\scri}}-2\rho_+\pa_{\rho_+})$ and using only that $\check\alpha_+<\check\alpha_{\!\scri}$, i.e. $\alpha_++\half<\alpha_{\!\scri}$), as in the proof of Theorem~\ref{ThmE}, then proves the a priori estimate~\eqref{EqNeFwApriori} for $s=1$. By duality, defining $\tilde\alpha=-\alpha'$ and $\tilde s=-(s-1)$, and under the strengthened assumption
  \begin{equation}
  \label{EqNeBwAssm2}
    \tilde\alpha_{\!\scri} < \tilde\alpha_+
  \end{equation}
  (which is equivalent to~\eqref{EqNeFwAssm2}), we obtain the solvability of $P_{y_0}^*\tilde u=\tilde f$ with the estimate~\eqref{EqNeBwApriori} for $\tilde s=0$.

  \pfstep{A priori estimate for $P_{y_0}^*$, $\tilde s=1$; solvability for $P_{y_0}$, $s=0$.} One can also prove an energy estimate for the adjoint problem, with extendible distributions at $\rho_+=\rho$ (which is a null hypersurface): with $W$ as before, $Z=i^{-1}V$, and $V=x_{\!\scri}^{-4\check{\tilde\alpha}_{\!\scri}}\rho_+^{-2\check{\tilde\alpha}_+}W$, where $(2\check{\tilde\alpha}_{\!\scri},\check{\tilde\alpha}_+)=(2\tilde\alpha_{\!\scri},\tilde\alpha_+)+(1,1)$, one now uses $1_{\Omega_\rho}Z$ as a multiplier. The symmetric 2-tensor $\tilde K_{i Z}$ (with $p_{0,y_0}^0,p_{1,y_0},p_{1,y_0}^*$ replaced by $(p_{0,y_0}^0)^*,-p_{1,y_0}^*,-p_{1,y_0}$) is then \emph{negative} definite under the assumptions $\check{\tilde\alpha}_{\!\scri}>-\ubar p_1$ and $\check{\tilde\alpha}_+>\check{\tilde\alpha}_{\!\scri}$; the condition for the negativity of the $(x_{\!\scri}\pa_y)^2$ term of $\tilde K_{i Z}$ now reads $\check{\tilde\alpha}_+>-\half+\check{\tilde\alpha}_{\!\scri}$, which now is automatic. The contribution from differentiation of the cutoff $1_{\Omega_\rho}$ at $\rho_+=\rho$ has the same (negative) sign since $\nabla 1_{\Omega_\rho}$, resp.\ $W$ is past, resp.\ future causal, and hence can be discarded. Thus, we obtain the a priori estimate~\eqref{EqNeBwApriori} under the assumptions~\eqref{EqNeBwAssm}.

  By duality, this implies the solvability of $P_{y_0}u=f$, as claimed in part~\eqref{ItNeFw}, for $s=0$ and the full range of weights, together with the estimate~\eqref{EqNeFwApriori}.

  \pfstep{A priori estimate for $P_{y_0}^*$, $\tilde s\geq 1$; solvability for $P_{y_0}$, $s\leq 0$.} Let $\tilde s\geq 1$. If we are given $\tilde u\in\bar H_{\ebop,I}^{\tilde s,\tilde\alpha}(\Omega_\rho)$ with $\tilde f=P_{y_0}^*\tilde u\in\bar H_{\ebop,I}^{\tilde s-1,\tilde\alpha'}(\Omega_\rho)$, then we have the a priori estimate~\eqref{EqNeBwApriori} for $\tilde s=1$. Let now $0<\rho'<\rho$. We claim that
  \begin{equation}
  \label{EqNeBwAprioriLarge}
    \|\tilde u\|_{\bar H_{\ebop,I}^{\tilde s,\tilde\alpha}(\Omega_{\rho'})} \leq C\bigl(\|\tilde f\|_{\bar H_{\ebop,I}^{\tilde s-1,\tilde\alpha'}(\Omega_\rho)} + \|\tilde u\|_{\bar H_{\ebop,I}^{1,\tilde\alpha}(\Omega_\rho)}\bigr).
  \end{equation}
  Indeed, this follows from variants of the microlocal regularity results in Lemmas~\ref{LemmaMEinp}\eqref{ItMEinpBw} and \ref{LemmaMEout}\eqref{ItMEoutBwLoc} in which wave front sets and function spaces are replaced by the invariant versions $\WF_{\ebop,I}$ and $H_{\ebop,I}$, and all localizations to neighborhoods of $x_{\!\scri}=0$ are omitted;\footnote{For example, the invariant version of Lemma~\ref{LemmaMEinp}\eqref{ItMEinpBw} reads as follows. Suppose that $\tilde\alpha_{\!\scri}<\tilde\alpha_++\half$ and $\WF_{\ebop,I}^{\tilde s-1,\tilde\alpha'}(\tilde f)\cap\cR_{\rm in,+,I}=\emptyset$, where, in the coordinates on $\cN$, we set $\cR_{\rm in,+,I}=\{(\rho_+,x_{\!\scri},y;\zeta,\xi,\eta)\colon\rho_+=0,\ \xi=2\zeta\}$ (similarly to~\eqref{EqMFCritInp} but without localization to $x_{\!\scri}=0$). If $\WF_{\ebop,I}^{\tilde s,\tilde\alpha}(\tilde u)\cap(\Seb^*_{\{\rho_+=0\}}M\setminus\cR_{\rm in,+,I})=\emptyset$, then $\WF_{\ebop,I}^{\tilde s,\tilde\alpha}(\tilde u)\cap\cR_{\rm in,+,I}=\emptyset$.} the proofs of these variants are the same upon removing cutoffs in $x_{\!\scri}$ and working with symbols and operators which are homogeneous with respect to dilations and invariant under translations on $\cN$. Estimating the second term on the right in~\eqref{EqNeBwAprioriLarge} using the already proved a priori estimate for $\tilde s=1$, we obtain 
  \begin{equation}
  \label{EqNeBwLarge1}
    \|\tilde u\|_{\bar H_{\ebop,I}^{\tilde s,\tilde\alpha}(\Omega_{\rho'})} \leq C\|\tilde f\|_{\bar H_{\ebop,I}^{\tilde s-1,\tilde\alpha'}(\Omega_\rho)}.
  \end{equation}

  It remains to obtain a quantitative estimate of $\tilde u\in\Hext_{\ebop,I}^{\tilde s,\tilde\alpha}(\Omega_\rho)$ on the full domain $\Omega_\rho\supsetneq\Omega_{\rho'}$. To this end, fix a cutoff $\phi\in\CI(\R)$ so that $\phi\equiv 0$ on $(-\infty,\rho']$ and $\phi\equiv 1$ on $[\rho,\infty)$. Then we have
  \[
    \tilde f':=P_{y_0}^*(\phi\tilde u)=\phi\tilde f+[P_{y_0}^*,\phi]\tilde u \in \bar H_{\ebop,I}^{\tilde s-1,(2\tilde\alpha_{\!\scri}',N+2)}(\Omega_\rho)
  \]
  for any $N$ since $\rho_+$ has a positive lower bound on $\supp\tilde f'$. Extend $\tilde f'$ to an element $\tilde f''\in\bar H_{\ebop,I}^{\tilde s-1,(2\tilde\alpha'_{\!\scri},N+2)}(\Omega_{\rho''})$ where $\rho''>\rho$, with norm bounded by a fixed constant times the norm of $\tilde f'$. Using the solvability of the adjoint problem stated after~\eqref{EqNeBwAssm2} above, we see that upon taking $N>\tilde\alpha_{\!\scri}$, we can solve $P_{y_0}^*\tilde u''=\tilde f''$ with $\tilde u''\in\bar H_{\ebop,I}^{0,(2\tilde\alpha_{\!\scri},N)}(\Omega_{\rho''})$ obeying a norm bound
  \[
    \|\tilde u''\|_{\bar H_{\ebop,I}^{0,(2\tilde\alpha_{\!\scri},N)}(\Omega_{\rho''})}\leq C\|\tilde f''\|_{\bar H_{\ebop,I}^{-1,(2\tilde\alpha'_{\!\scri},N+2)}(\Omega_{\rho''})}\leq C'\|\tilde f'\|_{\bar H_{\ebop,I}^{\tilde s-1,(2\tilde\alpha'_{\!\scri},N+2)}(\Omega_\rho)}.
  \]
  But then higher regularity on the smaller domain $\Omega_\rho\subsetneq\Omega_{\rho''}$ follows as above, showing that $\tilde u''|_{\Omega_\rho} \in \bar H_{\ebop,I}^{\tilde s,(2\tilde\alpha_{\!\scri},N)}(\Omega_\rho)$, with an estimate
  \begin{equation}
  \label{EqNeBwLarge2}
    \|\tilde u''|_{\Omega_\rho}\|_{\bar H_{\ebop,I}^{\tilde s,(2\tilde\alpha_{\!\scri},N)}(\Omega_\rho)} \leq C''\|\tilde f'\|_{\bar H_{\ebop,I}^{\tilde s-1,(2\tilde\alpha'_{\!\scri},N+2)}(\Omega_\rho)} \leq C'''\|\tilde f\|_{\bar H_{\ebop,I}^{\tilde s-1,\tilde\alpha'}(\Omega_\rho)}.
  \end{equation}
  Finally then, we note that $\tilde u''|_{\Omega_\rho}=\phi\tilde u$, since the difference $w:=\tilde u''|_{\Omega_\rho}-\phi\tilde u\in\bar H_{\ebop,I}^{1,\tilde\alpha}(\Omega_\rho)$ satisfies $P_{y_0}^*w=\tilde f'-\tilde f'=0$ and thus $w\equiv 0$ in view of the already proven a priori estimate~\eqref{EqNeBwApriori} in the case $\tilde s=1$. Thus, combining the estimates~\eqref{EqNeBwLarge1} and \eqref{EqNeBwLarge2} gives the a priori estimate~\eqref{EqNeBwApriori} for all $\tilde s\geq 1$.

  By duality, we get the solvability statement in part~\eqref{ItNeFw} for all $s\leq 0$.

  \pfstep{Solvability for $P_{y_0}$, $s\in\R$; a priori estimate for $P_{y_0}^*$, $\tilde s\in\R$.} It remains to prove solvability for $s>0$. Given $f\in\dot H_{\ebop,I}^{s-1,\alpha'}(\Omega_\rho)$, we have already proved that there exists a solution $u\in\dot H_{\ebop,I}^{0,\alpha}(\Omega_\rho)$. But then the invariant versions of the microlocal regularity results, Lemmas~\ref{LemmaMEinp}\eqref{ItMEinpFw} and \ref{LemmaMEout}\eqref{ItMEoutFw}, imply that in fact $u\in\dot H_{\ebop,I}^{s,\alpha}(\Omega_\rho)$, together with a quantitative estimate for its norm in terms of the norm of $f$.

  The a priori estimate~\eqref{EqNeBwApriori} now follows by duality also for $\tilde s=-(s-1)<1$, and hence is now proved for all $\tilde s$.

  \pfstep{A priori estimate for $P_{y_0}$, $s\in\R$.} Given $u\in\dot H_{\ebop,I}^{s,\alpha}(\Omega_\rho)$ with $f=P_{y_0}u\in\dot H_{\ebop,I}^{s-1,\alpha'}(\Omega_\rho)$, we have already shown that the equation $P_{y_0}u'=f$ can be solved with $u'\in\dot H_{\ebop,I}^{s,\alpha}(\Omega_\rho)$ obeying a bound
  \[
    \|u'\|_{\dot H_{\ebop,I}^{s,\alpha}(\Omega_\rho)}\leq C\|f\|_{\dot H_{\ebop,I}^{s-1,\alpha'}(\Omega_\rho)}.
  \]
  But then $P_{y_0}(u-u')=0$. Propagation of invariant edge-b-regularity shows that $u-u'$ has infinite edge-b-regularity; relaxing the weight at $\rho_+=0$, we can thus a fortiori regard $u-u'\in\dot H_{\ebop,I}^{1,(2\alpha_{\!\scri},N)}(\Omega_\rho)$ for $N<\min(\alpha_+,\alpha_{\!\scri}-1)$. But then the energy estimate proved in the first step of the proof applies (as $N+1<\alpha_{\!\scri}$, cf.\ \eqref{EqNeFwAssm2}) and implies that $u-u'\equiv 0$. Therefore, we have $u=u'$, and the proof is complete.
\end{proof}

\begin{rmk}[Alternative approach]
  The fact that one can directly prove an energy estimate for $P_{y_0}^*$ in the full range~\eqref{EqNeBwAssm} of weights allowed in the microlocal propagation estimates is somewhat fortuitous. A conceptually cleaner approach which does not take advantage of this fact starts with energy estimates with sharp weights at $x_{\!\scri}=0$ but with very weak weights at $\rho_+=0$ (i.e.\ very negative for $P_{y_0}$, very positive for $P_{y_0}^*$), which can then be improved by exploiting the invertibility properties of the normal operator of $P_{y_0}$ at $\rho_+=0$. We leave the implementation of this approach, which can be based on the (semiclassical) 0-analysis in the next section, to the interested reader.
\end{rmk}

\subsection{Leading order control at \texorpdfstring{$\scri^+$}{null infinity}}
\label{SsNeLOC}

Using Proposition~\ref{PropNe}, we now show how to control solutions of $P u=f$ or $P^*\tilde u=\tilde f$ near $\scri^+$ to leading order in the sense of \emph{decay}; this complements the regularity results of~\S\ref{SM}. Since we already have full control near $\scri^+\setminus I^+$ by Theorem~\ref{ThmE}, we focus on a neighborhood of $\scri^+\cap I^+$. Fix a compact subset $K\subset\scri^+\setminus I^0$, and choose $T\in\R$ so that $\cU_+(T)\supset K$. Using the coordinates $\rho_+,x_{\!\scri}$ on $\cU_+(T)$ as in~\eqref{EqOGeoCoordIp}, we can fix $\delta_0>0$ small and $\ubar\rho_+<\bar\rho_+\in(0,1)$ so that
\[
  \Omega^+_{\delta,\rho} := \{ x_{\!\scri}<\delta,\ \rho_+<\rho \} \subset \cU_+(T) \subset M
\]
contains $K$ when $0<\delta\leq\delta_0$ and $\rho\in[\ubar\rho_+,\bar\rho_+]$. We do not arrange for the boundary hypersurfaces of $\Omega^+_{\delta,\rho}$ to have any particular causal structure, as questions of regularity and decay near $I^+$ have a global character (i.e.\ answering them requires information about the operator $P$ far from $\scri^+$); we can therefore only prove a priori estimates here. As for function spaces, we fix
\[
  \delta\in(0,\delta_0),\qquad
  \rho\in(\ubar\rho_+,\bar\rho_+),
\]
and weights $\alpha=(2\alpha_{\!\scri},\alpha_+)$, and work with the space
\[
  \bar H_\ebop^{s,\alpha}(\Omega^+_{\delta,\rho})
\]
of extendible distributions and the space
\[
  H_\ebop^{s,\alpha}(\Omega^+_{\delta,\rho})^{\bullet,-} := \bigl\{ u|_{\Omega^+_{\delta,\rho}} \colon u\in\bar H_\ebop^{s,\alpha}(\Omega^+_{\delta_0,\bar\rho_+}) \colon \supp u\subset\ol{\Omega^+_{\delta_0,\rho}}\,\bigr\}
\]
of distributions with supported character at $\Sigma^{+,\rm in}_{\delta,\rho}=\{x_{\!\scri}<\delta,\ \rho_+=\rho\}$. See Figure~\ref{FigNeP}.

\begin{figure}[!ht]
\centering
\includegraphics{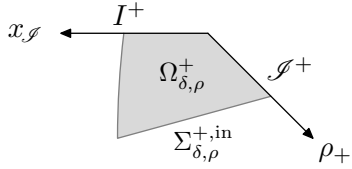}
\caption{The domain $\Omega_{\delta,\rho}^+$ on which Theorem~\ref{ThmNeP} takes place, and one of its boundary hypersurfaces $\Sigma_{\delta,\rho}^{+,\rm in}$.}
\label{FigNeP}
\end{figure}

\begin{thm}[A priori estimate for $P$ near $\scri^+$]
\label{ThmNeP}
  Under the same assumptions on $P$, $s$, $\alpha=(2\alpha_{\!\scri},\alpha_+)$, $\tilde s$, and $\tilde\alpha=(2\tilde\alpha_{\!\scri},\tilde\alpha_+)$ as in Proposition~\usref{PropNe}, and defining $\alpha'=\alpha+(2,2)$, $\tilde\alpha'=\tilde\alpha+(2,2)$, the following estimates hold for some constant $C$.
  \begin{enumerate}
  \item\label{ItNePFw}{\rm (Forward problem.)} For all $\eps>0$, there exists $C_\eps$ so that
    \begin{equation}
    \label{EqNePFw}
    \begin{split}
      \|u\|_{\Heb^{s,\alpha}(\Omega^+_{\delta,\rho})^{\bullet,-}} &\leq C\|P u\|_{\Heb^{s-1,\alpha'}(\Omega^+_{\delta,\rho})^{\bullet,-}} + C_\eps\|u\|_{\Heb^{s+1,(\alpha_{\!\scri}-2\ell_{\!\scri},\alpha_+)}(\Omega^+_{\delta,\rho})^{\bullet,-}} \\
        &\quad + \eps\|u\|_{\Heb^{s+1,\alpha}(\Omega^+_{\delta,\rho})^{\bullet,-}}
    \end{split}
    \end{equation}
    for all $u$ for which all norms are finite.
  \item\label{ItNePBw}{\rm (Backward problem.)} For all $\eps>0$, there exists $C_\eps$ so that
    \begin{equation}
    \label{EqNePBw}
    \begin{split}
      \|\tilde u\|_{\bar H_\ebop^{\tilde s,\tilde\alpha}(\Omega^+_{\delta,\rho})} &\leq C \|P^*\tilde u\|_{\bar H_\ebop^{\tilde s-1,\tilde\alpha'}(\Omega^+_{\delta,\rho})} + C_\eps\|\tilde u\|_{\bar H_\ebop^{\tilde s+1,\tilde\alpha-(2\ell_{\!\scri},0)}(\Omega^+_{\delta,\rho})} \\
        &\quad + \eps\|\tilde u\|_{\bar H_\ebop^{\tilde s+1,\tilde\alpha}(\Omega^+_{\delta,\rho})}
    \end{split}
    \end{equation}
    for all $\tilde u$ for which all norms are finite.
  \end{enumerate}
\end{thm}

Thus, if one ignores the edge-b-differential order, the right hand sides of~\eqref{EqNePFw} and \eqref{EqNePBw} involve norms on $u$ and $\tilde u$ which either feature a weaker weight at $\scri^+$ or have a small prefactor $\eps$. The loss in the edge-b-differential order is acceptable in applications, since regularity is already controlled by the results in~\S\ref{SM} (if supplemented by propagation estimates that are global near future timelike infinity and thus not covered by our $\scri^+$-local theory here); see the discussion following~\eqref{EqSAAdjEst} for an example. This loss arises from the fact that the edge normal operators of $P$ are nonelliptic and lose one order of regularity upon inversion, whereas the difference of $P$ and any of its normal operators is typically a \emph{second} order operator still.

\begin{rmk}[Better error term under stronger assumptions]
\label{RmkNePFwBetter}
 If $s\geq 1$ and $\alpha_++1<\alpha_{\!\scri}$, we may replace $u$ and $s+1$ in the final term on the right hand side of~\eqref{EqNePFw} by $(1-\chi)u$ and $s$, where $\chi$ is identically $1$ near $\scri^+$ and supported in $x_{\!\scri}<\delta'$ for sufficiently small $\delta'\in(0,\delta)$. This follows by writing $u=\chi u+(1-\chi)u$ and using an energy estimate to estimate $\chi u$ in terms of $P(\chi u)=\chi P u+[P,\chi]u\in\dot H_\ebop^{s-1,\alpha'}(\Omega^+_{\delta,\rho})$ using the multiplier vector field $Z$ from~\eqref{EqNeVF}; see also~\eqref{EqNeFwAssm2}.
\end{rmk}

The idea behind the proof of Theorem~\ref{ThmNeP} is to split $u$ into a sum of pieces localized near individual fibers $\scri^+_{y_0}$ of $\scri^+$, use the estimates for $P_{y_0}$ provided by Proposition~\ref{PropNe}, and estimating the error terms arising from $P-P_{y_0}$ and commutators of $P$ with the localizers. In order to implement this, we need to relate norms on edge-b-Sobolev spaces to norms of localizations.\footnote{We expect that one can drop the $\cO(\eps)$ error terms in~\eqref{EqNePFw} and \eqref{EqNePBw} by combining the normal operator inverses controlled by Proposition~\ref{PropNe} into a single operator on the edge-b-double space. This would however require analyzing the regularity of $P_{y_0}^{-1}$ in the parameter $y_0$. We do not pursue this further, as the $\cO(\eps)$ error terms are easily absorbed in applications; see the discussion after equation~\eqref{EqSAAdjEst}, or \cite[Step (i) in the proof of Proposition~5.19]{HintzNonstat}.} We first do this in a collar neighborhood $[0,1)_{x_{\!\scri}}\times\scri^+$ of all of $\scri^+$; we set
\[
  \Omega := [0,\half)_{x_{\!\scri}}\times\scri^+
\]
and shall only consider distributions with support in $\ol\Omega$. Denote the projection $[0,1)_{x_{\!\scri}}\times\scri^+\to\scri^+$ by $\pi$. Cover the (compact) base $Y$ of the fibration $\phi\colon\scri^+\to Y$ by a finite number of coordinate charts $\psi_1,\ldots,\psi_N\colon B(0,2)\to Y$ (where $B(0,R)\subset\R^{n-1}$ is the open $R$-ball) so that $Y=\bigcup_{i=1}^N \psi_i(B(0,1))$; let $\chi\in\CIc(B(0,2))$ be identically $1$ on $\ol{B(0,1)}$, and define $\chi_i\in\CI(M)$ to be equal to $\pi^*\phi^*(\chi\circ\psi_i^{-1})$ on $\pi^{-1}(\phi^{-1}(\psi_i(B(0,2))))$ and $0$ otherwise. Since $\sum_{i=1}^N\chi_i$ has a positive lower bound in the collar neighborhood, we immediately obtain:

\begin{lemma}[Localization to coordinate patches on $Y$]
\label{LemmaNeLocY}
  Let $s\in\R$ and $\alpha\in\R^3$. Then there exists a constant $C>1$ so that for all distributions $u$ with $\supp u\subset\ol\Omega$,
  \[
    C^{-1}\sum_{i=1}^N \|\chi_i u\|_{\Heb^{s,\alpha}(M)} \leq \|u\|_{\Heb^{s,\alpha}(M)} \leq C\sum_{i=1}^N \|\chi_i u\|_{\Heb^{s,\alpha}(M)}.
  \]
\end{lemma}

Passing to local coordinates $y\in B(0,2)\subset\R^{n-1}$ on $\psi_i(B(0,2))\subset Y$, we now work in
\[
  \Omega_i:=[0,\half)_{x_{\!\scri}}\times B(0,2)\times Z \subset [0,1)_{x_{\!\scri}} \times \R^{n-1}_y \times Z,
\]
where $Z\cong[-1,1]$ is the typical fiber of $\scri^+$. Since $\ol{\Omega_i}$ is compact, edge-b-Sobolev spaces of distributions on $[0,1)\times\R^{n-1}\times Z$ with support in $\ol{\Omega_i}$ are well-defined; here, denoting a smooth positive b-density on $Z$ by $\mu_Z$, we fix the volume density
\begin{equation}
\label{EqNeDensity}
  \mu := \rho_0^{-n-1}x_{\!\scri}^{-2 n}\rho_+^{-n-1}\bigl|\tfrac{\dd x_{\!\scri}}{x_{\!\scri}}\dd y\,\mu_Z\bigr|
\end{equation}
on $\Omega_i$, mirroring the earlier choice~\eqref{EqMEDensity}. For $j\in\N$, put
\[
  I_j := \{ k 2^{-j} \colon k\in\Z,\ -3\cdot 2^j\leq k<3\cdot 2^j \},
\]
and denote the $K(j)=(6\cdot 2^j)^{n-1}$ points in $I_j^{n-1}\subset\R^{n-1}$ by
\[
  y_{j,1},\ldots,y_{j,K(j)}.
\]
Fix cutoff functions
\[
  \chi_0\in\CIc([0,1)),\quad \chi_0=1\ \text{on}\ [0,\half], \qquad
  \tilde\chi_1\in\CIc(\R),\quad \tilde\chi_1=1\ \text{on}\ [-1,1],
\]
and define $\chi_1\in\CIc(\R^{n-1})$ by $\chi_1(y)=\tilde\chi_1(y_1)\cdots\tilde\chi_1(y_{n-1})$. For $j\in\N$ and $1\leq k\leq K(j)$, we then set
\[
  \chi_{j,k}(y) := \chi_0(2^j x_{\!\scri})\chi_1\bigl(2^j(y-y_{j,k})\bigr),\qquad
  \chi_j^{\rm tot}(y) := \sum_{k=1}^{K(j)} \chi_{j,k}(y).
\]
Thus, $\chi_{j,k}$ localizes to a size $\sim 2^{-j}$ neighborhood of $\{0\}\times\{y_{j,k}\}\times Z\subset\Omega_i$, and $\chi_j^{\rm tot}$ localizes to a size $\sim 2^{-j}$ neighborhood of $\{0\}\times[-3,3]^{n-1}\times Z$.

\begin{lemma}[Properties of $\chi_{j,k}$ and edge-b-Sobolev norms]
\label{LemmaNeLoc}
  Let $s\in\R$. There exists a constant $C$, only depending on $s$, the choice of the cutoffs $\chi_0,\tilde\chi_1$, and the volume density~\eqref{EqNeDensity}, so that the following statements hold for all $j\in\N$.
  \begin{enumerate}
  \item\label{ItNeLocPOU} For $u\in\dot H_\ebop^s(\Omega_i)$ with $2^j x_{\!\scri}<\frac12$ on $\supp u$, we have
    \begin{equation}
    \label{EqNeLocPOU}
      C^{-1}\sum_{k=1}^{K(j)} \| \chi_{j,k}u \|_{\Heb^s}^2 \leq \|u\|_{\Heb^s}^2 \leq C\sum_{k=1}^{K(j)} \| \chi_{j,k}u \|_{\Heb^s}^2.
    \end{equation}
  \item\label{ItNeLocMultF} Let $f\in\CIc([0,\infty)\times\R^{n-1})$. Then
    \[
      \bigl\|f\bigl(2^j x_{\!\scri},2^j(y-y_{j,k})\bigr)u\bigr\|_{\Heb^s} \leq C\|u\|_{\Heb^s}.
    \]
  \end{enumerate}
  Analogous estimates hold for weighted edge-b-Sobolev spaces, as well as for spaces consisting of distributions with supported or extendible character on $[0,\delta)_{x_{\!\scri}}\times B(0,2)\times[0,\rho)_{\rho_+}\subset\Omega_i$.
\end{lemma}
\begin{proof}
  We note the following three key properties of $\chi_{j,k}$: there exist constants $C$, $L$, and $C_m$, $m\in\N_0$, all independent of $j$, so that
  \begin{enumerate}
  \item for all $m\in\N_0$, we have
    \begin{equation}
    \label{EqNeLocchijkBound}
      \sup_{l+|\gamma|\leq m}|(x_{\!\scri} D_{x_{\!\scri}})^l(x_{\!\scri} D_y)^\gamma \chi_{j,k}|\leq C_m;
    \end{equation}
  \item whenever $k_1,\ldots,k_L\in\{1,\ldots,K(j)\}$ are distinct, we have $\bigcap_{l=1}^L\supp \chi_1(2^j(y-y_{j,k_l}))=\emptyset$;
  \item $C^{-1}\leq\chi_j^{\rm tot}(y)\leq C$ for $y\in B(0,2)\subset\R^{n-1}$.
  \end{enumerate}
  This implies uniform bounds
  \[
    \sum_{l+|\gamma|\leq m} \Bigl|(x_{\!\scri} D_{x_{\!\scri}})^l(x_{\!\scri} D_y)^\gamma\Bigl(\frac{\chi_0(2^j x_{\!\scri})}{\chi_j^{\rm tot}(y)}\Bigr)\Bigr| \leq C'_m,\qquad m\in\N_0.
  \]
  Thus, multiplication by $\chi_0(2^j x_{\!\scri})/\chi_j^{\rm tot}(y)$ is a uniformly bounded map on the space of $u\in\dot H_\ebop^s(\Omega_i)$ with $\supp u\subset\{2^j x_{\!\scri}<\half\}$: for $s\in\N_0$, this follows from the Leibniz rule, and then for general $s\in\R$ by duality and interpolation. The proof of part~\eqref{ItNeLocMultF} of the Lemma is completely analogous.

  Turning to part~\eqref{ItNeLocPOU}, consider first the case $s=0$. Then
  \[
    \|u\|_{L^2}^2 = \|\chi_0(2^j x_{\!\scri})u\|_{L^2}^2 = \biggl\|\frac{\chi_0(2^j x_{\!\scri})}{\chi_j^{\rm tot}(y)}\sum_{k=1}^{K(j)} \chi_{j,k}(y)u(x_{\!\scri},y) \biggr\|_{L^2}^2 \leq C\Re\sum_{k,k'=1}^{K(j)} \la \chi_{j,k}u,\chi_{j,k'}u\ra_{L^2}.
  \]
  But for each $k$, there are at most $L-1$ values of $k'$ besides $k'=k$ for which the inner product is nonzero; thus, we can apply Cauchy--Schwarz and bound the right hand side by $C(1+(L-1))\sum_{k=1}^{K(j)}\|\chi_{j,k}u\|_{L^2}^2$. More generally, for $s\in\N_0$, similar arguments together with the Leibniz rule imply $\|u\|_{\Heb^s}^2\leq C\sum\|\chi_{j,k}u\|_{\Heb^s}^2$, with $C$ depending on $s$ but not on $j$; interpolation then implies this inequality for all real $s\geq 0$.

  The converse estimate for $s=0$ follows from the $j$-independent bound $\sum_{k=1}^{K(j)}\chi_{j,k}^2\leq C'$, which implies
  \[
    \sum_{k=1}^{K(j)} \int \chi_{j,k}^2|u|^2\,\dd\mu \leq C'\int|u|^2\,\dd\mu = C'\|u\|_{L^2}^2.
  \]
  For $s\in\N$ one uses the Leibniz rule, and for real $s\geq 0$ interpolation. Having proved~\eqref{EqNeLocPOU} for real $s\geq 0$, the result for $s<0$ follows by duality.
\end{proof}

\begin{proof}[Proof of Theorem~\usref{ThmNeP}]
  We use the edge normal operator in a spirit somewhat similar to how the b-normal operator is used for example in \cite[\S2.1.2]{HintzVasySemilinear}, namely by means of estimating the localization of $u$ in terms of the edge normal operators of $P$, and bounding the error terms arising from localization and passing between $P$ and its normal operators. Rather than assembling the inverses of all edge normal operators\footnote{This is a key difference to the b-setting, where a single normal operator governs the behavior of a b-operator globally near the whole boundary.} into a single object---which in our non-elliptic setting would be technically rather delicate (unlike in the elliptic setting \cite{MazzeoEdge})---we only use estimates for the edge normal operators at an $\eps$-dense collection of fibers of $\scri^+$, with the estimates applied to a localization of $u$ to $\eps$-neighborhoods of individual fibers. Thus, the final error term in~\eqref{EqNePFw} bounds the difference between $P$ and this finite collection of normal operators. An important point is that the commutators of (the edge normal operators of) $P$ with cutoffs to such $\eps$-neighborhoods gain a power of $x$. The relevant local coordinate calculation is $[x\pa_y,\chi(y/\eps)]=\eps^{-1}x \chi'(y/\eps)$; the large constant $\eps^{-1}$ leads to the second term on the right in~\eqref{EqNePFw}
  
    We now turn to the details. Assume first that $u$ is supported in some $\ol{\Omega_i}$. For $j\in\N$ to be specified below, and putting $\chi_{0,j}(x_{\!\scri}):=\chi_0(2^{j+1}x_{\!\scri})$, we have
  \begin{align*}
    \|u\|_{\Heb^{s,\alpha}(\Omega^+_{\delta,\rho})^{\bullet,-}}^2 &\leq 2\bigl\|\bigl(\chi_0(x_{\!\scri})-\chi_0(2^{j+1} x_{\!\scri})\bigr)u\bigr\|_{\Heb^{s,\alpha}(\Omega^+_{\delta,\rho})^{\bullet,-}}^2  + 2\|\chi_0(2^{j+1}x_{\!\scri})u\|_{\Heb^{s,\alpha}(\Omega^+_{\delta,\rho})^{\bullet,-}}^2 \\
      &\leq C_{j N}\|u\|_{\Heb^{s,(N,\alpha_+)}(\Omega_{\delta,\rho}^+)^{\bullet,-}}^2 + C\sum_{k=1}^{K(j)} \|\chi_{0,j}\chi_{j,k}u\|_{\Heb^{s,\alpha}(\Omega^+_{\delta,\rho})^{\bullet,-}}^2.
  \end{align*}
  for any $N$, with $C$ independent of $j$. We estimate each term in the sum individually using Proposition~\ref{PropNe}, to wit,
  \begin{align}
    &\|\chi_{0,j}\chi_{j,k}u\|^2_{\Heb^{s,\alpha}} \nonumber\\
    &\quad \leq C\bigl\|P_{y_{j,k}}\bigl(\chi_{0,j}\chi_{j,k}u\bigr)\bigr\|^2_{\Heb^{s-1,\alpha'}} \nonumber\\
  \label{EqNePjk}
    &\quad\leq C\Bigl(\bigl\| [ P_{y_{j,k}},\chi_{0,j}\chi_{j,k}]u \bigr\|^2_{\Heb^{s-1,\alpha'}} + \bigl\| \chi_{0,j}\chi_{j,k}(P-P_{y_{j,k}})u \bigr\|^2_{\Heb^{s-1,\alpha'}}  + \bigl\| \chi_{0,j}\chi_{j,k}P u\bigr\|^2_{\Heb^{s-1,\alpha'}}\Bigr).
  \end{align}
  Recall that $\chi_{0,j}(x_{\!\scri})\chi_{j,k}(x_{\!\scri},y)=\chi_{0,j}(x_{\!\scri})\chi_1(2^j(y-y_{j,k}))$. The commutator of $P_{y_{j,k}}$ with $\chi_{0,j}$ is uniformly bounded in $j,k$ as an element of $\cA^{(2,2)}\Diffeb^1$ (since $\sup|(x_{\!\scri}\pa_{x_{\!\scri}})^l\chi_{0,j}|$ has a $j$-independent upper bound for all $l$), with coefficients supported in $x_{\!\scri}\geq 2^{-j-2}>0$. Therefore, for all $N$ there exists $C_{N,j}$ so that
  \[
    \sum_{k=1}^{K(j)} \bigl\|[P_{y_{j,k}},\chi_{0,j}]\chi_1(2^j(y-y_{j,k}))u\bigr\|_{\Heb^{s-1,\alpha'}}^2 \leq C_{N,j}\|u\|_{\Heb^{s,(2 N,\alpha_+)}}^2.
  \]
  On the other hand, if we define $\chi_{j,k}^\sharp=\chi_0(2^{j+1}x_{\!\scri})\chi_1^\sharp(2^j(y-y_{j,k}))$ where $\chi_1^\sharp\in\CIc(\R^{n-1})$ is identically $1$ on $\supp\chi_1$ (thus Lemma~\ref{LemmaNeLoc} applies with $\chi_{j,k}^\sharp$ in place of $\chi_{j,k}$), then $\chi_{0,j}[P_{y_{j,k}},\chi_1(2^j(y-y_{j,k}))]=\chi_{0,j}\chi_{j,k}^\sharp 2^j x_{\!\scri} Q_{j,k}$ where the operators $Q_{j,k}\in\cA^{(2,2)}\Diffeb^1$ are uniformly bounded in $j,k$ (as follows from~\eqref{EqNeLocchijkBound}); the factor $2^j x_{\!\scri}$ here arises via differentiation of $2^j(y-y_{j,k})$ along terms of $P_{y_{j,k}}$ involving $x_{\!\scri}\pa_y$; thus,
  \[
    \sum_{k=1}^{K(j)} \bigl\|\chi_{0,j}[P_{y_{j,k}},\chi_1(2^j(y-y_{j,k}))]u\bigr\|_{\Heb^{s-1,\alpha'}}^2 \leq C 2^{2 j}\|u\|_{\Heb^{s,\alpha-(1,0)}}^2.
  \]

  We now turn to the second term in~\eqref{EqNePjk}. The definition of $P_{y_{j,k}}$ via freezing the coefficients of $P_0=P-\tilde P$ (see Definition~\ref{DefOp}\eqref{ItOpStruct}) at $(x_{\!\scri},y)=(0,y_{j,k})$ as an edge-b-operator implies that we can write
  \[
    \chi_{j,k}(P_0-P_{y_{j,k}}) = \chi_{j,k}\biggl(x_{\!\scri} Q^1_{j,k}+\sum_{l=1}^{n-1}(y-y_{j,k})_l Q^{(l)}_{j,k}\biggr),
  \]
  where $x_{\!\scri}^{-2}\rho_+^{-2}Q_{j,k}^1,x_{\!\scri}^{-2}\rho_+^{-2}Q_{j,k}^{(l)}\in\cA^{(0,0)}\Diffeb^1$ are bounded independently of $j,k$ upon application of any fixed number of edge-b- (or even b-)derivatives. We can thus estimate
  \[
    \sum_{k=1}^{K(j)} \|\chi_{0,j}\chi_{j,k} x_{\!\scri} Q_{j,k}^1 u\|_{\Heb^{s-1,\alpha'}}^2 \leq C\|u\|_{\Heb^{s+1,\alpha-(1,0)}}^2.
  \]
  Moreover, in view of $\chi_{j,k}\cdot 2^j(y-y_{j,k})_l$ obeying uniform (in $j,k$) $L^\infty$-bounds (together with any finite number of edge-b-derivatives), we also have
  \[
    \sum_{k=1}^{K(j)} 2^{-2 j}\|\chi_{0,j}\chi_{j,k} 2^j(y-y_{j,k})_l Q_{j,k}^l u\|_{\Heb^{s-1,\alpha'}}^2 \leq 2^{-2 j}C\|u\|_{\Heb^{s+1,\alpha}}^2
  \]
  The contribution from the lower order (in the sense of decay) contribution $\tilde P$ to $P$ is estimated simply by
  \[
    \sum_k\|\chi_{0,j}\chi_{j,k}\tilde P u\|_{\Heb^{s-1,\alpha'}}^2\leq C\|u\|_{\Heb^{s+1,\alpha-(2\ell_{\!\scri},\ell_+)}}^2.
  \]

  Altogether, we have now proved that
  \begin{align*}
    &\sum_{k=1}^{K(j)} \|\chi_{0,j}\chi_{j,k}u\|_{\Heb^{s,\alpha}(\Omega^+_{\delta,\rho})^{\bullet,-}}^2 \\
    &\qquad \leq (C_{N,j}+C 2^{2 j})\|u\|_{\Heb^{s+1,\alpha-(2\ell_{\!\scri},0)}(\Omega^+_{\delta,\rho})^{\bullet,-}}^2 + C 2^{-2 j}\|u\|_{\Heb^{s+1,\alpha}(\Omega^+_{\delta,\rho})^{\bullet,-}}^2 \\
    &\qquad \qquad + C\|P u\|_{\Heb^{s-1,\alpha'}(\Omega^+_{\delta,\rho})^{\bullet,-}}^2,
  \end{align*}
  with $C$ independent of $j$. Upon choosing $j$ large enough so that $2^{-2 j}C<\eps^2$, this proves the estimate~\eqref{EqNePFw} for $u$ with $\supp u\subset\ol{\Omega_i}$.

  For general $u$, we apply Lemma~\ref{LemmaNeLocY} and note that the arguments given so far apply to $\phi_i u$, $i=1,\ldots,N$; we then only need to observe that $P(\phi_i u)=\phi_i P u+[P,\phi_i]u$, where the commutator $[P,\phi_i]$ gains a power of $x_{\!\scri}$ and loses a differential order (relative to $P$) as an edge-b-differential operator, so $\|[P,\phi_i]u\|_{\Heb^{s-1,\alpha'}}\leq C\|u\|_{\Heb^{s,\alpha-(1,0)}}$.

  The proof of the a priori estimate~\eqref{EqNePBw} is completely analogous.
\end{proof}

\section{Microlocal estimates for the \texorpdfstring{$I^+$-normal operator}{normal operator at future timelike infinity}}
\label{SNp}

We let $P,M,M_0,Y$ be as in the beginning of~\S\ref{SE}, and use the coordinates $\rho_+,x_{\!\scri}$ on $\cU_+(T)$ ($T\in\R$ arbitrary) as in~\eqref{EqOGeoCoordIp}; thus, replacing $I^+$ for simplicity of notation by a collar neighborhood
\[
  I^+ = [0,\eps_{\!\scri})_{x_{\!\scri}} \times Y,\qquad \eps_{\!\scri}>0,
\]
of $I^+\cap\scri^+\subset I^+$, we work in a collar neighborhood
\[
  [0,1)_{\rho_+} \times I^+
\]
of $I^+$ inside $M$, and hence in a trivialization $[0,\infty)_{\rho_+}\times I^+$ of ${}^+N I^+$ (where we abuse notation by writing the fiber-linear coordinate $\dd\rho_+$ as $\rho_+$ simply). We only record weights at $\scri^+$ and $I^+$, and drop the weight at $I^0$ from the notation; we shall also drop the bundle $\upbeta^*E\to M$ from the notation unless it makes a technical difference.

In this section, we demonstrate that the information on $P$ captured by Definition~\ref{DefOp} is sufficient to obtain some control of the Mellin-transformed normal operator of $P$ at $I^+$ (see Definition~\ref{DefEBONormH1M}). Concretely, in the decomposition $P=P_0+\tilde P$ as in~\eqref{EqOpStruct}, only the first term of $P_0$ is of leading order in the sense of decay at $I^+$. We shall analyze the rescaled normal operator of $P_0$:

\begin{definition}[Rescaled normal operator at $I^+$]
\label{DefNpOp}
  Denote by $\pi\colon{}^+N I^+\to I^+$ the projection to the base. We then set
  \begin{align*}
    P_+ &:= N_{I^+}(x_{\!\scri}^{-2}\rho_+^{-2}P) = N_{I^+}(x_{\!\scri}^{-2}\rho_+^{-2}P_0) \\
      & \in \Diff_{\ebop,I}^2({}^+N I^+;\pi^*E) = \Diff_{\ebop,I}^2([0,\infty)_{\rho_+}\times I^+;\pi^*E),
  \end{align*}
  where the subscript `$I$' indicates the dilation-invariance of $P_+$ in the factor $[0,\infty)_{\rho_+}$, and the edge structure is associated with the fibration $[0,\infty)_{\rho_+}\times\pa I^+\cong[0,\infty)_{\rho_+}\times Y\to Y$. The Mellin-transformed normal operator is denoted
  \[
    \wh{P_+}(\lambda) = \rho_+^{-i\lambda}P_+\rho_+^{i\lambda} \in \Diff_0^2(I^+;\pi^*E).
  \]
\end{definition}

We recall that the principal symbol $G$ of $P$ is the dual metric function $G\colon\varpi\mapsto g^{-1}(\varpi,\varpi)$. But by Definition~\ref{DefOGeoAdm} (see also~\eqref{EqOGeoMetIp}), we have
\begin{align*}
  g^{-1} &\equiv g_0^{-1} \bmod \cA^{2+2\ell_{\!\scri},2+\ell_+}(M;S^2\,\Teb M), \\
  g_0^{-1}&\equiv x_{\!\scri}^2\rho_+^2\Bigl(\frac12 x_{\!\scri}\pa_{x_{\!\scri}}\otimes_s(x_{\!\scri}\pa_{x_{\!\scri}}-2\rho_+\pa_{\rho_+}) + k^{-1}(y,x_{\!\scri}\pa_y)\Bigr)\bmod x_{\!\scri}^3\rho_+^2\CI(M;S^2\,\Teb M);
\end{align*}
thus, the principal symbol of $P_+$ is the quadratic form associated with the dilation-invariant extension of $x_{\!\scri}^{-2}\rho_+^{-2}g_0^{-1}|_{I^+}\in\CI(I^+;S^2\,\Teb_{I^+}M)$. Concretely, this means that
\begin{equation}
\label{EqNpSymbol}
  \sigmaeb^2(P_+)(\varpi) =: G_{0,\ebop}(\varpi) := G_+(\varpi) + \tilde G_+(\varpi)
\end{equation}
with $G_+(\varpi)=g_+^{-1}(\varpi,\varpi)\in P^{[2]}(\Teb^*_{I^+}M)$ and $\tilde G_+(\varpi)=\tilde g_+^{-1}(\varpi,\varpi)\in x_{\!\scri} P^{[2]}(\Teb^*_{I^+}M)$, where
\begin{equation}
\label{EqNpGplus}
  g_+^{-1}=\frac12 x_{\!\scri}\pa_{x_{\!\scri}}\otimes_s(x_{\!\scri}\pa_{x_{\!\scri}}-2\rho_+\pa_{\rho_+}) + k^{-1}(y,x_{\!\scri}\pa_y),
\end{equation}
while $\tilde g_+$ is a linear combination of symmetric 2-tensors built out of $x_{\!\scri}\pa_{x_{\!\scri}}$, $\rho_+\pa_{\rho_+}$, and $x_{\!\scri}\pa_{y^j}$, with coefficients lying in $x_{\!\scri}\CI([0,\eps_{\!\scri})_{x_{\!\scri}}\times Y)$.

We shall work with the b-density
\begin{equation}
\label{EqNpBDensity}
  \mu_\bop := \Bigl|\frac{\dd x_{\!\scri}}{x_{\!\scri}}\dd k\Bigr| \in \CI(I^+;\Omegab I^+)
\end{equation}
on $I^+$. On the level of $L^2$-spaces on $[0,\infty)_{\rho_+}\times I^+$, the relationship of $\mu_\bop$ with the density $x_{\!\scri}^{-2 n}\rho_+^{-n-1}|\frac{\dd x_{\!\scri}}{x_{\!\scri}}\frac{\dd\rho_+}{\rho_+}\dd k|$ (see~\eqref{EqMEDensity}) used in~\S\S\ref{SM}--\ref{SNe} is
\begin{equation}
\label{EqNpL2}
\begin{split}
  &x_{\!\scri}^{2\alpha_{\!\scri}}\rho_+^{\alpha_+}L^2\Bigl([0,\infty)\times I^+, x_{\!\scri}^{-2 n}\rho_+^{-n-1}\Bigl|\frac{\dd x_{\!\scri}}{x_{\!\scri}}\frac{\dd\rho_+}{\rho_+}\dd k\Bigr|\Bigr) = x_{\!\scri}^{2\gamma_{\!\scri}}\rho_+^{\gamma_+}L^2\Bigl([0,\infty)\times I^+, \Bigl|\frac{\dd\rho_+}{\rho_+}\Bigr|\mu_\bop\Bigr), \\
  &\qquad \gamma_{\!\scri}=\alpha_{\!\scri}+\frac{n}{2},\qquad \gamma_+=\alpha_++\frac{n+1}{2}.
\end{split}
\end{equation}
Thus, by Plancherel's Theorem, the Mellin transform in $\rho_+$, defined as in~\eqref{EqNpHMellin}, gives an isometric isomorphism
\begin{equation}
\label{EqNpPlancherel}
\begin{split}
  \cM &\colon x_{\!\scri}^{2\alpha_{\!\scri}}\rho_+^{\alpha_+}L^2\Bigl([0,\infty)\times I^+, x_{\!\scri}^{-2 n}\rho_+^{-n-1}\Bigl|\frac{\dd x_{\!\scri}}{x_{\!\scri}}\frac{\dd\rho_+}{\rho_+}\dd y\Bigr|\Bigr) \\
    &\hspace{8em} \xra{\cong} L^2\bigl(\{\Im\lambda=-\gamma_+\}; x_{\!\scri}^{2\gamma_{\!\scri}}L^2(I^+,\mu_\bop)\bigr).
\end{split}
\end{equation}

\subsection{Bounded frequency estimates}
\label{SsNpB}

The principal symbol of $\wh{P_+}(\lambda)$ is an element of $P^{[2]}({}^0 T^*I^+)$. Recall that for $\varpi\in{}^0 T^*I^+$, we have ${}^0\sigma_2(\wh{P_+}(\lambda))(\varpi)=\sigmaeb^2(P_+)(\varpi)$ where we recall the inclusion ${}^0 T^*I^+\subset\Teb^*_{I^+}M$. (In local coordinates, every $\varpi\in{}^0 T^*I^+$ is of the form $\varpi=\xi\frac{\dd x_{\!\scri}}{x_{\!\scri}}+\eta_j\frac{\dd y^j}{x_{\!\scri}}$.) In view of~\eqref{EqNpSymbol}--\eqref{EqNpGplus}, this means
\[
  {}^0\sigma_2(\wh{P_+}(\lambda))(\varpi) = \frac12\xi^2 + k^{i j}(y)\eta_i\eta_j + \tilde G_+(\varpi),\qquad
  \varpi=\xi\frac{\dd x_{\!\scri}}{x_{\!\scri}}+\eta_j\frac{\dd y^j}{x_{\!\scri}},
\]
and hence $\wh{P_+}(\lambda)$ is elliptic near $x_{\!\scri}=0$. Its refined analysis also requires control of its normal operator, which is the conjugation of the edge normal operator of $P_+$ (i.e.\ the normal operator at ${}^+ N_{\pa I^+}I^+$) by the Mellin transform in $\rho_+$; but the edge normal operator (at $I^+\cap\scri_{y_0}^+$) of the b-normal operator (at $I^+$) $P_+$ of $P_0$ is equal to the b-normal operator (at $\scri^+\cap I^+$) of the edge normal operator (at $\scri^+_{y_0}$) of $P_0$, and is thus fixed by Definition~\ref{DefOp}\eqref{ItOpNorm}. To wit, for $y_0\in Y$,
\begin{align*}
  &{}^0 N_{y_0}(\wh{P_+}(\lambda)) = \frac12\bigl(x_{\!\scri} D_{x_{\!\scri}}-2 i^{-1}q_{1,y_0}\bigr)(x_{\!\scri} D_{x_{\!\scri}}-2\lambda) + k^{i j}(y_0)(x_{\!\scri} D_{y^i})(x_{\!\scri} D_{y^j}) + p_{0,y_0}^+, \\
  &\qquad q_{1,y_0}:=\frac{n-1}{2}+p_1|_{\scri^+_{y_0}\cap I^+}\in\End(E_{y_0}),\qquad
          p_{0,y_0}^+:=p_0^+|_{\scri^+_{y_0}\cap I^+}\in\End(E_{y_0}).
\end{align*}
Recalling the definition of $\ubar p_{1,+}$ from Definition~\ref{DefMEMin}, we let
\begin{equation}
\label{EqNpHubarq}
  \ubar q_{1,+} := \frac{n-1}{2}+\ubar p_{1,+}.
\end{equation}

\begin{thm}[Estimates for $\wh{P_+}(\lambda)$ and $\wh{P_+^*}(\lambda)$ for bounded $\lambda$]
\label{ThmNpB}
  If $p_0^+|_{\scri^+\cap I^+}\neq 0$, assume that $P$ is special admissible in the sense of Definition~\usref{DefESpecial}; otherwise assume only that $P$ is admissible. Let $\chi_1,\chi_2\in\CIc(I^+)$, with $\chi_1\equiv 1$ near $\pa I^+$, $\chi_2\equiv 1$ near $\supp\chi_1$, and so that ${}^0 T_p^*I^+$ is spacelike (i.e.\ $G_{0,\ebop}$ in~\eqref{EqNpSymbol} is positive definite on ${}^0 T_p^*I^+$) for $p\in\supp\chi_1$. Use the b-density~\eqref{EqNpBDensity} on $I^+$ to define Sobolev spaces.
  \begin{enumerate}
  \item\label{ItNpBDirect}{\rm (Direct problem.)} Let $s,\gamma_{\!\scri}\in\R$, and suppose that
    \begin{subequations}
    \begin{equation}
    \label{EqNpBAssm}
      {-}\Im\lambda<\gamma_{\!\scri}<\ubar q_{1,+}.
    \end{equation}
    Then for all $N\in\R$, there exists a constant $C$ so that
    \begin{equation}
    \label{EqNpBEst}
      \| \chi_1 u \|_{H_0^{s,2\gamma_{\!\scri}}(I^+)} \leq C\Bigl( \|\chi_2\wh{P_+}(\lambda)u\|_{H_0^{s-2,2\gamma_{\!\scri}}(I^+)} + \|\chi_2 u\|_{H_0^{-N,-N}(I^+)}\Bigr).
    \end{equation}
    \end{subequations}
    Moreover, for all $u$ with $\chi_2 u\in H_0^{-\infty,2\gamma_{\!\scri}}(I^+)$ and $\chi_2\wh{P_+}(\lambda)u=0$, we have $\chi_1 u\in\cA^{2(\ubar q_{1,+}-\eps)}(I^+)$ for all $\eps>0$.
  \item\label{ItNpBAdjoint}{\rm (Adjoint problem.)} Let $\tilde s,\tilde\gamma_{\!\scri}\in\R$, and suppose that
    \begin{subequations}
    \begin{equation}
    \label{EqNpBAssmAdj}
      -\ubar q_{1,+}<\tilde\gamma_{\!\scri}<-\Im\lambda.
    \end{equation}
    Then for all $N\in\R$, there exists a constant $C$ so that
    \begin{equation}
    \label{EqNpBEstAdj}
      \| \chi_1 \tilde u \|_{H_0^{\tilde s,2\tilde\gamma_{\!\scri}}(I^+)} \leq C\Bigl( \|\chi_2\wh{P_+^*}(\lambda)\tilde u\|_{H_0^{\tilde s-2,2\tilde\gamma_{\!\scri}}(I^+)} + \|\chi_2\tilde u\|_{H_0^{-N,-N}(I^+)}\Bigr).
    \end{equation}
    \end{subequations}
  \end{enumerate}
\end{thm}

\begin{rmk}[Weights]
\label{RmkNpBWeights}
  Identifying ${-}\Im\lambda$ with the weight $\gamma_+={-}\Im\lambda$ (cf.\ \eqref{EqNpPlancherel}) and relating $\gamma_{\!\scri},\gamma_+$ to $\alpha_{\!\scri},\alpha_+$ via~\eqref{EqNpL2}, the conditions in~\eqref{EqNpBAssm} are equivalent to $\alpha_+<-\half+\alpha_{\!\scri}$ and $\alpha_{\!\scri}<-\half+\ubar p_1$, except for the use of $\ubar p_{1,+}$ (determined from $p_1$ only at $\scri^+\cap I^+$) instead of $\ubar p_1$ (determined from $p_1$ along all of $\scri^+$). In this sense, Theorem~\ref{ThmNpB}\eqref{ItNpBDirect} applies in the full range of weights allowed in all previous estimates near $\scri^+\cap I^+$ (Lemmas~\ref{LemmaMEinp}, \ref{LemmaMEout}, Proposition~\ref{PropNe}, Theorem~\ref{ThmNeP}). A similar comment applies to Theorem~\ref{ThmNpB}\eqref{ItNpBAdjoint}.
\end{rmk}

\begin{rmk}[Restriction on $\Im\lambda$ and continuation of the resolvent]
\label{RmkNpBCont}
  The condition~\eqref{EqNpBAssm} provides a lower bound $\Im\lambda>-\gamma_{\!\scri}$ of Mellin dual frequencies $\lambda$ for which one can analyze $\wh{P_+}(\lambda)$ as a 0-operator in a straightforward manner. Upon inspection of the proof below, one sees that allowing $\lambda$ to cross the line $\Im\lambda=-\gamma_{\!\scri}$ is closely related to considering the spectral family on an asymptotically hyperbolic space at or across the continuous spectrum, the analysis of which is more delicate; see e.g.\ \cite{MazzeoMelroseHyp,GuillarmouMeromorphic,SaBarretoWangResolvent}.
\end{rmk}

\begin{proof}[Proof of Theorem~\usref{ThmNpB}]
  The proof is based on a standard elliptic parametrix construction in the 0-calculus \cite{MazzeoMelroseHyp,Hintz0Px}, hence we shall be rather brief. By assumption, $\wh{P_+}(\lambda)$ is elliptic near $\supp\chi_1$. 

  If $q_{1,y_0}$ is diagonalizable and $p_{0,y_0}^+=0$, then acting on the individual eigenspaces of $q_{1,y_0}$, the operator ${}^0 N_{y_0}(\wh{P_+}(\lambda))$ is related to the spectral family of $(n-1)$-dimensional hyperbolic space via conjugation by a suitable power of $x_{\!\scri}$ (which also appears in the proof of the invertibility of the reduced normal operator below, cf.\ \eqref{EqNpBConj1}--\eqref{EqNpBConj2}); its inverse is thus explicit \cite[\S6]{MazzeoMelroseHyp}. This can easily be generalized to the case that $q_{1,y_0}$ has nontrivial Jordan blocks. (In this manner, one can prove~\eqref{EqNpBEst} for general $\lambda$ except for an explicitly computable discrete set.)

  We analyze ${}^0 N_{y_0}(\wh{P_+}(\lambda))$ here (for $\Im\lambda>-\gamma_{\!\scri}$) instead in a more conceptual manner, using the general machinery of the 0-calculus. Thus, exploiting the translation invariance of ${}^0 N_{y_0}(\wh{P_+}(\lambda))$ by passing to the Fourier transform in $y$, and then introducing $x:=x_{\!\scri}|\eta|_{k^{-1}}$, $\hat\eta=\frac{\eta}{|\eta|_{k^{-1}}}$, we obtain the \emph{reduced normal operator} in the terminology of \cite{LauterPsdoConfComp}, given by
  \[
    \cP_{y_0} := \frac12(x D_x-2 i^{-1}q_{1,y_0})(x D_x-2\lambda) + x^2 + p_{0,y_0}^+.
  \]
  Considered as an operator on the radial compactification $[0,\infty]_x$ of $[0,\infty)$, it is a b-operator near $x=0$ with normal operator $N_0(\cP_{y_0})=\half(x D_x-2 i^{-1}q_{1,y_0})(x D_x-2\lambda)+p_{0,y_0}^+$, and a weighted scattering operator near $R:=x^{-1}=0$ with principal part $R^{-2}(\half(R^2 D_R)^2+1)$, which is elliptic. The Mellin-transformed normal operator of $\cP_{y_0}$ at $x=0$ is
  \[
    \hat N_0(\cP_{y_0},\zeta)=\half(\zeta-2 i^{-1}q_{1,y_0})(\zeta-2\lambda) + p_{0,y_0}^+,
  \]
  which is thus invertible for all $\zeta\in\C$ except for those lying in the boundary spectrum $\{2\lambda\}\cup \spec(2 i^{-1}q_{1,y_0})$. (For $p_{0,y_0}^+=0$, this is clear; for $p_{0,y_0}^+\neq 0$ on the other hand, one uses the lower triangular, resp.\ strictly lower triangular nature of $q_{1,y_0}$, resp.\ $p_{0,y_0}^+$ coming from the special admissibility assumption on $P$ to obtain the same conclusion.) In particular, the line $\Im\zeta=-2\gamma_{\!\scri}$ does not intersect the boundary spectrum provided $\gamma_{\!\scri}\neq-\Im\lambda$ and $\gamma_{\!\scri}\notin\Re\spec q_{1,y_0}$, which in particular allows for $\gamma_{\!\scri}$ in the range~\eqref{EqNpBAssm}.
  
  Define the b-scattering Sobolev spaces
  \[
    H_{\bop,\scop}^{s,(2\gamma_{\!\scri},\delta)}([0,\infty]_x) := \Bigl\{ u \colon \chi u\in\Hb^{s,2\gamma_{\!\scri}}\Bigl([0,\infty),\Bigl|\frac{\dd x}{x}\Bigr|\Bigr),\ (1-\chi)u \in \la x\ra^{-\delta}H^s(\R,|\dd x|) \Bigr\},
  \]
  where $\chi\in\CIc([0,\infty))$ is identically $1$ near $0$. Then we claim that for $\gamma_{\!\scri}$ in the range~\eqref{EqNpBAssm}, the operator
  \begin{equation}
  \label{EqNpBRedInv}
    \cP_{y_0}\colon H_{\bop,\scop}^{s,(2\gamma_{\!\scri},\delta)}([0,\infty])\to H_{\bop,\scop}^{s-2,(2\gamma_{\!\scri},\delta-2)}([0,\infty])
  \end{equation}
  is not merely Fredholm (which follows at once from elliptic theory in the b-setting near $s=0$ and in the scattering setting near $s^{-1}=0$), but invertible. For the proof, we first assume that $p_{0,y_0}^+=0$. Passing to the Jordan block decomposition of $q_{1,y_0}$, it is sufficient (by exploiting the upper triangular nature of Jordan blocks) to consider the case that $q_{1,y_0}\in\C$ is scalar; and the assumption on $\gamma_{\!\scri}$ reads ${-}\Im\lambda<\gamma_{\!\scri}<\Re q_{1,y_0}$. If $u\in H_{\bop,\scop}^{s,(2\gamma_{\!\scri},\delta)}([0,\infty])$ satisfies $\cP_{y_0}u=0$, then $u\in H_{\bop,\scop}^{\infty,(2\gamma_{\!\scri},\infty)}([0,\infty])$ by elliptic regularity in the scattering calculus, so in particular $u$ is rapidly decaying (together with all $x$-derivatives) as $x\to\infty$. Moreover,
  \begin{subequations}
  \begin{equation}
  \label{EqNpBConj1}
    \tilde u:=x^{-q_{1,y_0}-i\lambda}u\in H_{\bop,\scop}^{\infty,(2\tilde\gamma_{\!\scri},\infty)}([0,\infty]),\qquad
    \tilde\gamma_{\!\scri}:=\frac12\bigl(2\gamma_{\!\scri}+\Im\lambda-\Re q_{1,y_0}\bigr),
  \end{equation}
  satisfies the simpler equation
  \begin{equation}
  \label{EqNpBConj2}
    \Bigl(\frac12(x D_x)^2-\tilde\lambda^2 + x^2\Bigr)\tilde u=0,\qquad
    \tilde\lambda:=\lambda+i q_{1,y_0}.
  \end{equation}
  \end{subequations}
  Note that $\Im\tilde\lambda=\Im\lambda+\Re q_{1,y_0}>0$ and $2\tilde\gamma_{\!\scri}>-\Im\tilde\lambda$; therefore we in fact have $|\tilde u|,|x D_x\tilde u|\leq C x^{\Im\tilde\lambda}$ for $x\leq 1$. We can therefore multiply equation~\eqref{EqNpBConj2} by $\bar{\tilde u}$ and integrate by parts. If $\Re\tilde\lambda\neq 0$, taking imaginary parts then gives $0=2(\Re\tilde\lambda)(\Im\tilde\lambda)\|\tilde u\|^2$, whereas when $\Re\tilde\lambda=0$ one directly obtains $0=\|x\tilde u\|^2+\half\|x D_x\tilde u\|^2+\| |\tilde\lambda|\tilde u\|^2$; in both cases, we conclude that $\tilde u=0$ and hence $u=0$. This proves the injectivity of $\cP_{y_0}$. The surjectivity follows from the injectivity of the adjoint $\cP_{y_0}^*$, which is proved in the same manner. If $p_{0,y_0}^+\neq 0$, then upon passing to the bundle splitting of $E$ in which $q_{1,y_0}$ and $p_{0,y_0}^+$ are (strictly) lower triangular, these arguments apply step by step as one considers the projection of $\cP_{y_0}$ onto the first, second, etc.\ summand of the splitting of $E$.

  As shown in \cite{MazzeoEdge} (see also \cite[\S3.2]{Hintz0Px}, \cite[\S5.5]{AlbinLectureNotes} for detailed expositions, and also \cite{LauterPsdoConfComp}), the invertibility of $\cP_{y_0}$ implies the existence of a parametrix for $\wh{P_+}(\lambda)$ acting on $x_{\!\scri}^{2\gamma_{\!\scri}}L^2(I^+)$ in the large 0-calculus with bounds,
  \[
    Q \in \Psi_0^{-2,(\alpha_\lb,-(n-1),\alpha_\rb)}(I^+),\qquad
    Q\wh{P_+}(\lambda)=I+R_L,\quad R_L\in\Psi_0^{-\infty,(\emptyset,\emptyset,\alpha_\rb)}(I^+).
  \]
  Here, $\Psi_0^{s,(\alpha_\lb,\alpha_\ff,\alpha_\rb)}(I^+)$ is the space of operators whose Schwartz kernels are conormal distributions on the 0-double space $(I^+)^2_0:=[(I^+)^2;\diag_{\pa I^+}]$ where $\diag_{\pa I^+}\subset(\pa I^+)^2$ is the diagonal, valued in (the pullback to $(I^+)^2_0$ of) the right b-density bundle on $(I^+)^2$, which are conormal to the lift of the diagonal $\diag_{I^+}\subset(I^+)^2$, and which are conormal also at the left boundary (the lift of $\pa I^+\times I^+$) with decay rate $\alpha_\lb$ (the power to which a boundary defining function of the left boundary is raised), at the right boundary (the lift of $I^+\times\pa I^+$) with decay rate $\alpha_\rb$, and at the front face with decay rate $\alpha_\ff$; the shift by $(n-1)$ of the weight at the front face is due to the fact that we work with a b-density here, rather than with a 0-density---which is $x_{\!\scri}^{-(n-1)}$ times a b-density on the $n$-dimensional manifold $I^+$. Concretely, the computation of the boundary spectrum of the 0-normal operator above implies that we can take
  \begin{align*}
    \alpha_\lb\geq 2\ubar q_{1,+}-\eps, \qquad
    \alpha_\rb \geq 2\Im\lambda-\eps,
  \end{align*}
  for any $\eps>0$. We conclude that
  \begin{align*}
    \chi_1 u = \chi_1(\chi_2 u) &= \chi_1 Q\wh{P_+}(\lambda)(\chi_2 u) + \chi_1 R_L(\chi_2 u) \\
      &= \chi_1 Q(\chi_2\wh{P_+}(\lambda)u) + \chi_1 R_L(\chi_2 u) - \chi_1 Q[\wh{P_+}(\lambda),\chi_2]u.
  \end{align*}
  Taking the norm in $H_0^{s,2\gamma_{\!\scri}}(I^+)$ gives the estimate~\eqref{EqNpBEst} since the operators $\chi_1 R_L\chi_2$ and $\chi_1 Q[\wh{P_+}(\lambda),\chi_2]$ (noting that $\supp\chi_1\cap\,\supp\dd\chi_2=\emptyset$) are bounded maps $H_0^{-N,-N}(I^+)\to H_0^{s,2\gamma_{\!\scri}}(I^+)$ for all $N$.

  The proof of the adjoint estimate~\eqref{EqNpBEstAdj} is similar. The adjoint $\wh{P_+}(\lambda)^*=\wh{\cP_+^*}(\bar\lambda)$ satisfies~\eqref{EqNpBEstAdj} when $(\tilde s-2,2\tilde\gamma_{\!\scri})=-(s,2\gamma_{\!\scri})$ (the dual orders of the left hand side of~\eqref{EqNpBEst}), which leads to the requirement $-\Im\lambda<-\tilde\gamma_{\!\scri}<\ubar q_{1,+}$. Passing from $\lambda$ to $\bar\lambda$ switches the sign of $\Im\lambda$ and thus leads to the condition~\eqref{EqNpBAssmAdj}.
\end{proof}

\subsection{High frequency estimates}
\label{SsNpH}

We now prove analogues of the estimates of Theorem~\ref{ThmNpB} when $|\Re\lambda|\to\infty$. In the notation of Theorem~\ref{ThmNpB}, let $\gamma_{\!\scri}<\ubar q_{1,+}$, and fix $\gamma_+<\gamma_{\!\scri}$; we shall from now on only consider $\lambda\in\C$ with $\Im\lambda=-\gamma_+$. As $|\Re\lambda|\to\infty$, the operator $\wh{P_+}(\lambda)$ can be regarded as a (uniformly degenerate) differential operator with large parameter $\lambda$ (see \cite[\S9]{ShubinSpectralTheory} for the case of manifolds without boundary). For present purposes, it is sufficient to ignore the symbolic behavior in $\lambda$ and instead consider the semiclassical rescaling
\begin{equation}
\label{EqNpHOp}
  P_{h,\hat\lambda} := h^2\wh{P_+}(h^{-1}\hat\lambda) \in \Diff_{0,\semi}^2(I^+;\pi^*E),\qquad
  h:=\la\lambda\ra^{-1},\quad \hat\lambda:=\frac{\lambda}{\la\lambda\ra}.
\end{equation}
We apply the general considerations of~\S\ref{SsEB0} to the operator $P_+\in\Diffeb^2([0,\infty)_{\rho_+}\times I^+)$ from Definition~\ref{DefNpOp}. With $\lambda$ restricted to the line $\Im\lambda=-\gamma_+$, we have, as $|\Re\lambda|\to\infty$ (and thus $h=|\Re\lambda|^{-1}+\cO(|\Re\lambda|^{-2})\to 0$),
\begin{equation}
\label{EqNpHImLambda}
  \hat\lambda=(\sgn\Re\lambda)-i h\gamma_+ + \cO(h^2).
\end{equation}
Let $\chi_1,\chi_2$ be as in the statement of Theorem~\ref{ThmNpB}, so in particular ${}^0 T^*_p I^+$ is spacelike for $p\in\supp\chi_1$. This implies that $P_{h,\hat\lambda}$ is elliptic (as a \emph{semiclassical} 0-operator) outside a compact subset of ${}^0 T^*_{\supp\chi_1}I^+$, in particular at fiber infinity; in other words, writing $p_{\hat\lambda}=\sigma_{0,\semi}^2(P_{h,\hat\lambda})$, the characteristic set
\[
  {}^0\Sigma_{\hat\lambda} := p_{\hat\lambda}^{-1}(0) \subset {}^0 T^*I^+
\]
has compact intersection with ${}^0 T^*_{\supp\chi_1}I^+$. As noted after Lemma~\ref{LemmaNpHOp}, the Hamiltonian vector field $H_{p_{\pm 1}}$ has a critical point at $\varpi\in{}^0\Sigma_{\pm 1}$ if and only if $H_{G_{0,\ebop}}$ (see~\eqref{EqNpSymbol}) is radial at $\pm\frac{\dd\rho_+}{\rho_+}+\varpi$ (which lies in the characteristic set of $P_+$). In a sufficiently small neighborhood of $x_{\!\scri}=0$, this is equivalent to the membership in one of the two sets
\begin{equation}
\label{EqNpHRadial}
\begin{split}
  {}^0\cR_{\rm in}^\pm &:= \{ (x_{\!\scri},y;\xi,\eta) \colon x_{\!\scri}=0,\ \xi=\pm 2,\ \eta=0 \}, \\
  {}^0\cR_{\rm out}^\pm &:= \{ (x_{\!\scri},y;\xi,\eta) \colon x_{\!\scri}=0,\ \xi=0,\ \eta=0 \}.
\end{split}
\end{equation}
by Lemma~\ref{LemmaMFCrit}. This can of course also be checked directly by noting that, at $\zeta=\pm 1$ (in the coordinates~\eqref{EqMFCoordIp} and using~\eqref{EqMFSymbIp}), we have
\begin{equation}
\label{EqNpHHamVF}
  H_{G^+_\ebop}=2(\xi\mp 1)(x_{\!\scri}\pa_{x_{\!\scri}}+\eta\pa_\eta)-2|\eta|^2\pa_\xi+x_{\!\scri}(4 k^{i j}\eta_i\pa_{y^j}-2(\pa_{y^m}k^{i j})\eta_i\eta_j\pa_{\eta_m}),
\end{equation}
while $H_{\tilde G_+}$ is a lower order contribution, cf.\ \eqref{EqMFHamIp}; that is, the expression~\eqref{EqNpHHamVF} is equal to $H_{p_{\pm 1}}$ up to an error term of class $x_{\!\scri}\cV_0({}^0 T^*I^+)$. Thus, ${}^0\cR_{\rm in}^\pm$ is a source, and ${}^0\cR_{\rm out}^\pm$ is a sink for the flow of $\pm H_{p_{\pm 1}}$ inside ${}^0\Sigma_{\pm 1}$.

\begin{thm}[Semiclassical regularity for $\wh{P_+}(\lambda)$]
\label{ThmNpH}
  Let $s,\gamma_{\!\scri},\gamma_+\in\R$. Consider $\lambda\in\C$ with $\Im\lambda=-\gamma_+$, and let $h=\la\lambda\ra^{-1}$, $\hat\lambda=\frac{\lambda}{\la\lambda\ra}$; thus $\hat\lambda=\pm 1-i h\gamma_++\cO(h^2)$ when $\pm\Re\lambda>0$ (see equation~\eqref{EqNpHImLambda}). Let\footnote{The subscript `$\loc$' means membership in the space $h^{-N}H_{0,h}^{-N,2\gamma_{\!\scri}}$ upon multiplication by any function in $\CIc(I^+)$. This is merely a technical detail; we are only interested in local (but uniform down to $\pa I^+$) semiclassical regularity here, for which the noncompact nature of $I^+$ is irrelevant.} $u,f\in h^{-N}H_{0,h,\loc}^{-N,2\gamma_{\!\scri}}(I^+)$ and $P_{h,\hat\lambda}u=f$ in the notation of~\eqref{EqNpHOp}. Then $\WF_{0,\semi}^{s,2\gamma_{\!\scri}}(u)\subset\WF_{0,\semi}^{s-2,2\gamma_{\!\scri}}(f)\cup{}^0\Sigma_\pm$.
  
  Moreover, recalling the quantity $\ubar q_{1,+}\in\R$ from~\eqref{EqNpHubarq}, we have:
  \begin{enumerate}
  \item\label{ItNpHin}{\rm (Propagation out of ${}^0\cR_{\rm in}^\pm$.)} Suppose that $\gamma_+<\gamma_{\!\scri}$. If $\WF_{0,\semi}^{s-2,2\gamma_{\!\scri}}(h^{-1}f)\cap{}^0\cR_{\rm in}^\pm=\emptyset$, then $\WF_{0,\semi}^{s,2\gamma_{\!\scri}}(u)\cap{}^0\cR_{\rm in}^\pm=\emptyset$.\footnote{Recall that since ${}^0\cR_{\rm in}^\pm\subset{}^0 T^*I^+$ is disjoint from fiber infinity, the wave front set conditions here are in fact independent of the differential orders.}
  \item\label{ItNpHout}{\rm (Propagation into ${}^0\cR_{\rm out}^\pm$.)} Suppose that $\gamma_{\!\scri}<\ubar q_{1,+}$. Let $\cU\subset{}^0 T^*I^+$ denote an open neighborhood of ${}^0\cR_{\rm out}^\pm$. If $\WF_{0,\semi}^{s-2,2\gamma_{\!\scri}}(h^{-1}f)\cap{}^0\cR_{\rm out}^\pm=\emptyset$ and $\WF_{0,\semi}^{s,2\gamma_{\!\scri}}(u)\cap(\cU\setminus\cR_{\rm out})=\emptyset$, then $\WF_{0,\semi}^{s,2\gamma_{\!\scri}}(u)\cap\cR_{\rm out}=\emptyset$.
  \end{enumerate}
\end{thm}

\begin{thm}[Semiclassical regularity for $\wh{P_+^*}(\lambda)$]
\label{ThmNpHAdj}
  Let $\tilde s,\tilde\gamma_{\!\scri},\tilde\gamma_+\in\R$. Consider $\lambda\in\C$ with $\Im\lambda=-\tilde\gamma_+$, and let $h=\la\lambda\ra^{-1}$, $\hat\lambda=\frac{\lambda}{\la\lambda\ra}$. Let $P^\dag_{h,\hat\lambda}:=h^2\wh{P_+^*}(h^{-1}\hat\lambda)$. Let $\tilde u,\tilde f\in h^{-N}H_{0,h,\loc}^{-N,2\tilde\gamma_{\!\scri}}(I^+)$ and $P_{h,\hat\lambda}^\dag\tilde u=\tilde f$. Then $\WF_{0,\semi}^{\tilde s,2\tilde\gamma_{\!\scri}}(\tilde u)\subset\WF_{0,\semi}^{\tilde s-2,2\tilde\gamma_{\!\scri}}(\tilde f)\cup{}^0\Sigma_\pm$.

  Moreover, we have:
  \begin{enumerate}
  \item\label{ItNpHAdjIn}{\rm (Propagation into ${}^0\cR_{\rm in}^\pm$.)} Suppose that $\tilde\gamma_+>\tilde\gamma_{\!\scri}$. Let $\cU\subset{}^0 T^*I^+$ denote an open neighborhood of ${}^0\cR_{\rm in}^\pm$. If $\WF_{0,\semi}^{\tilde s-2,2\tilde\gamma_{\!\scri}}(h^{-1}\tilde f)=\emptyset$ and $\WF_{0,\semi}^{\tilde s,2\tilde\gamma_{\!\scri}}(\tilde u)\cap(\cU\setminus{}^0\cR_{\rm in}^\pm)=\emptyset$, then $\WF_{0,\semi}^{\tilde s,2\tilde\gamma_{\!\scri}}(\tilde u)\cap{}^0\cR_{\rm in}^\pm=\emptyset$.
  \item\label{ItNpHAdjOut}{\rm (Propagation out of ${}^0\cR_{\rm out}^\pm$.)} Suppose that $\tilde\gamma_{\!\scri}>-\ubar q_{1,+}$. If $\WF_{0,\semi}^{\tilde s-2,2\tilde\gamma_{\!\scri}}(h^{-1}\tilde f)\cap{}^0\cR_{\rm out}^\pm=\emptyset$, then $\WF_{0,\semi}^{\tilde s,2\tilde\gamma_{\!\scri}}(\tilde u)\cap{}^0\cR_{\rm out}^\pm=\emptyset$.
  \end{enumerate}
\end{thm}

Mirroring Remark~\ref{RmkNpBWeights}, we recall the relationships $\alpha_{\!\scri}=\gamma_{\!\scri}-\frac{n}{2}$, $\alpha_+=\gamma_+-\frac{n+1}{2}$, $\ubar p_{1,+}=\ubar q_{1,+}-\frac{n-1}{2}$ from~\eqref{EqNpL2} and \eqref{EqNpHubarq}; then the conditions on $\gamma_+,\gamma_{\!\scri},\ubar q_{1,+}$ in Theorem~\ref{ThmNpH}\eqref{ItNpHin}, resp.\ \eqref{ItNpHout} are precisely those on $\alpha_+,\alpha_{\!\scri},\ubar p_1$ in Lemmas~\ref{LemmaMEinp}\eqref{ItMEinpFw}, resp.\ \ref{LemmaMEout}\eqref{ItMEoutFw} upon replacing $\ubar p_1$ by $\ubar p_{1,+}$. A similar comment applies to Theorem~\ref{ThmNpHAdj} and Lemmas~\ref{LemmaMEinp}\eqref{ItMEinpBw} and \ref{LemmaMEout}\eqref{ItMEoutBw}.

\begin{proof}[Proof of Theorems~\usref{ThmNpH} and \usref{ThmNpHAdj}]
  These results follow again from positive commutator arguments. One can in fact simply copy the arguments from the proofs of Lemmas~\ref{LemmaMEc}--\ref{LemmaMEout} but factor out the overall weights of the differential operators under study (cf.\ Definition~\ref{DefNpOp}), and drop localizations in $\rho_+$ and its dual momentum $\zeta$ in the commutator constructions and calculations. We omit the details here, but point out that the subprincipal terms which enter in threshold conditions for radial point estimates for $P_{h,\hat\lambda}$ can be computed from those of $P$ via the formula~\eqref{EqNpHOpSub} where $\frac{\Im\hat\lambda}{h}=-\gamma_++\cO(h)$. The upshot is that the term $\frac{\Im\hat\lambda}{h}\rho_+^{-1}H_p\rho_+\equiv-\gamma_+\rho_+^{-1}H_p\rho_++\cO(h)$ has precisely the same effect as the weight $\rho_+^{-2\gamma_+}$ of $\check a^2$ in the positive commutator proof of, for example, Lemma~\ref{LemmaMEinp} (where we are using $\check\alpha_+$ in place of $\gamma_+$).
\end{proof}

As in~\S\ref{SsME}, the positive commutator proofs give quantitative estimates. We state this for Theorem~\ref{ThmNpH}\eqref{ItNpHout}: suppose $B,W\in\Psi_{0,\semi}^0(I^+)$ have Schwartz kernels supported inside $(\supp\chi_1)^2$ where $\chi_1\in\CIc(I^+)$ is identically $1$ near $\pa I^+$, and their operator wave front sets are compact subsets of ${}^0 T^*I^+$ (thus, differential orders are irrelevant below). Furthermore, suppose that
\begin{itemize}
\item all backward (for the `$+$' sign), resp.\ forward (for the `$-$' sign) null-bicharacteristics starting at $\WF'_{0,\semi}(B)\cap\Sigma_{\pm 1}$ tend to ${}^0\cR_{\rm out}^\pm$ while remaining inside $\Ell_{0,\semi}^0(W)$;
\item $W$ is elliptic on $\WF'_{0,\semi}(B)$ and at ${}^0\cR_{\rm out}^\pm$.
\end{itemize}
Then for all $N\in\R$, and under the assumptions of Theorem~\ref{ThmNpH}\eqref{ItNpHout}, there exists a constant $C$ so that
\[
  \|B u\|_{H_{0,h}^{s,2\gamma_{\!\scri}}(I^+)} \leq C\Bigl( h^{-1}\|W P_{h,\hat\lambda}u\|_{H_{0,h}^{s-2,2\gamma_{\!\scri}}(I^+)} + h^N\|\chi u\|_{H_{0,h}^{-N,2\gamma_{\!\scri}}(I^+)}\Bigr).
\]

We leave the statements of estimates corresponding to the wave front set statements of the remaining parts of the above two Theorems to the reader.

\section{Application to linear waves on asymptotically flat spacetimes}
\label{SA}

We now sketch how to use the black box results proved in~\S\S\ref{SM}--\ref{SNp} to obtain a Fredholm and solvability theory for linear wave operators on a simple class of (future geodesically complete) asymptotically Minkowskian spacetimes. The setting we consider here is closely related to that by Baskin--Vasy--Wunsch \cite{BaskinVasyWunschRadMink,BaskinVasyWunschRadMink2}, in that the metrics and wave operators are asymptotically homogeneous of degree $\pm 2$ with respect to dilations in the forward timelike cone. Compared to the reference, we do allow here for significantly more general behavior near the null infinity, in that the light cone at infinity (in the terminology of \cite{BaskinVasyWunschRadMink}) is blown up here to resolve the metric and operator coefficients. On the flipside, we shall only record basic (functional analytic) results here and do not extract sharp asymptotics of solutions of wave equations.

We work on $\R^{1+n}=\R_t\times\R^n_x$. Denote by $Y=\{\frac{r}{t}=1,\ r^{-1}=0\}\subset\pa\ol{\R^{n+1}}$ the light cone at infinity, and let $M$ denote the square root blow-up of $\tilde M=[\ol{\R^{n+1}};Y]$ at the front face $\tilde\scri^+$ (as in~\S\ref{SssOGeoMink}; see also Figure~\ref{FigOGeo}). Denote by $\rho_0,x_{\!\scri},\rho_+\in\CI(M)$ defining functions of the boundary hypersurfaces $I^0$ (the closure of $\frac{t}{r}<1$, $r^{-1}=0$), $\scri^+$ (the front face of the square root blow-up, i.e.\ the closure of $|t-r|<\infty$, $r^{-1}=0$), $I^+\subset M$ (the closure of $\frac{r}{t}<1$, $t^{-1}=0$).\footnote{Thus, one can, for example, take $\rho_+=t^{-1}$ in $t>1$, $r/t<c_0<1$, further $\rho_0=r^{-1}$ in $r>1$, $0\leq t/r<c_0<1$; and possible choices near $\scri^+$ are given in~\eqref{EqOGeoCoordI0} and \eqref{EqOGeoCoordIp}.} We shall consider forcing problems in the domain
\[
  \Omega := \ol{\{t\geq 0\}}\subset M.
\]
Let then $g$ be a smooth Lorentzian metric on $\R^{n+1}$ (with signature $(-,+,\ldots,+)$) with the following properties:
\begin{enumerate}
\item\label{EqSAMet1} the conformal rescaling $g_\ebop:=\rho_0^2 x_{\!\scri}^2\rho_+^2 g$ is a nondegenerate edge-b-metric on $M$,
  \[
    g_\ebop \in (\CI+\cA^{(\ell_0,2\ell_{\!\scri},\ell_+)})(M;S^2\,\Teb^*M),\qquad
    g_\ebop^{-1} \in (\CI+\cA^{(\ell_0,2\ell_{\!\scri},\ell_+)})(M;S^2\,\Teb M),
  \]
  where $\ell_0\in(0,1]$ and $\ell_{\!\scri}\in(0,\half]$, and moreover $g$ is an $(\ell_0,2\ell_{\!\scri},\ell_+)$-admissible metric (Definition~\ref{DefOGeoAdm});
\item the level sets of $t$ are timelike, with $\dd t$ past timelike;
\item there exists a smooth function $\tau\in\CI(M\setminus(\scri^+\cup I^+))$ which is equal to $t/r$ near $I^0$, and so that $\tau\geq 0$ has past timelike differential in $\Omega^\circ$;
\item\label{EqSAMet3} the b-normal (or \emph{scaling}) vector field\footnote{This is defined in any collar neighborhood of $I^+\subset M$, and its restriction to $I^+$ as a b-vector field is independent of choices.} $\rho_+\pa_{\rho_+}$ is past timelike for $g_\ebop$ at $(I^+)^\circ=I^+\setminus\scri^+$;
\item\label{EqSAMet4} $g$ is non-trapping in the following sense: every past null-bicharacteristic (i.e.\ lift of a past causal null-geodesic) of $g_\ebop$ starting over a point in $\Omega\cap(I^0)^\circ$ (where $(I^0)^\circ=I^0\setminus\scri^+$) reaches the closure of $t^{-1}(0)$; and any past null-bicharacteristic of $g_\ebop$ starting over a point in $\Omega\cap(I^+)^\circ$ tends to the radially compactified edge-b-cotangent bundle over $\scri^+$.
\end{enumerate}

Moreover, we consider a wave type operator $P\in\Diff^2(\R^{n+1})$ with the following properties:
\begin{enumerate}
\setcounter{enumi}{4}
\item\label{EqSAOp1} $P$ is a weighted edge-b-operator, $P\in\rho_0^2 x_{\!\scri}^2\rho_+^2(\CI+\cA^{(\ell_0,2\ell_{\!\scri},\ell_+)})\Diffeb^2(M)$ whose principal symbol is equal to the dual metric function $\zeta\mapsto g^{-1}(\zeta,\zeta)$ of $g$;
\item\label{EqSAOp2} $P$ is $g$-admissible (Definition~\ref{DefOp}). Define $\ubar p_1$ and $\ubar p_{1,+}$ in terms of $p_1$ from Definition~\ref{DefOp}\eqref{ItOpNorm} as in Definition~\ref{DefMEMin};
\item\label{EqSAOp3} the Mellin-transformed normal operator family $\wh{P_+}(\lambda)$ of the b-normal operator $N_{I^+}(x_{\!\scri}^{-2}\rho_+^{-2}P)\in\Diff_{\ebop,I}^2({}^+N I^+)$ (see Definition~\ref{DefNpOp}) has trivial kernel on $\cA^{2\gamma_{\!\scri}}(I^+)$ for all $\lambda\in\C$, $\gamma_{\!\scri}\in\R$ with $-\Im\lambda<\gamma_{\!\scri}<\bar q$ for some fixed $\bar q\leq\ubar p_{1,+}+\frac{n-1}{2}$.
\end{enumerate}

One can also consider special admissible operators $P$ (see Definition~\ref{DefESpecial}) acting on sections of a vector bundle over $M$ which is the pullback of a bundle $E\to\ol{\R^{n+1}}$; we leave the required notational modifications to the reader.

\begin{example}[The model example]
\label{ExSAMink}
  The Minkowski metric~\eqref{EqOMink} satisfies assumptions~\eqref{EqSAMet1}--\eqref{EqSAMet4}. Indeed, it served as the motivating example for admissible metrics near $\scri^+$ in~\S\ref{SsOGeo}. Also, the non-trapping condition is satisfied, since the backward null-bicharacteristics of $g_\ebop$ are limits of appropriate reparameterizations of backwards null-geodesics (lifted to $T^*\R^{n+1}$) at $\pa\ol{\R^{n+1}}$, which indeed have the required property; note here that differently (near $\scri^+$) rescaled maximally extended backwards null-geodesics on $\ol{\R^{n+1}}$ start at the light cone at future infinity and tend to the light cone at past infinity, and on their journey cross the light cone at future infinity, or the closure of $t^{-1}(0)$, or both. Furthermore, the scalar wave operator on Minkowski space satisfies assumptions~\eqref{EqSAOp1}--\eqref{EqSAOp2} with $\ubar p_1=\ubar p_{1,+}=0$ as shown in Example~\ref{ExOpMink}. Furthermore, the operator family $\wh{P_+}(\lambda)$ is a conjugation of the spectral family on hyperbolic $n$-space; see \cite[\S\S7 and 10.1]{BaskinVasyWunschRadMink}. Using this, one can then show that~\eqref{EqSAOp3} is satisfied for $\bar q=\frac{n-1}{2}$. (See also \cite{VasyMinkDSHypRelation,BaskinMarzuolaComp}.)
\end{example}

\begin{thm}[Solving wave-type equations]
\label{ThmSA}
  Let $\Omega\subset M$, $P$, and $g$ be as above. Define edge-b-Sobolev spaces on $M$ with respect to the volume density $|\dd g|$. Let $s,\alpha_0,\alpha_{\!\scri},\alpha_+\in\R$ and $k\in\N_0$, and suppose that
  \[
    \alpha_++\frac12 < \alpha_{\!\scri} < \min\Bigl(\alpha_0+\frac12,-\frac12+\ubar p_1\Bigr),\qquad
    s > -(\alpha_0-\alpha_{\!\scri})-\Bigl(-\frac12+\ubar p_1-\alpha_{\!\scri}\Bigr).
  \]
  Then the operator
  \begin{align*}
    P &\colon \bigl\{ u \in \dot H_{\ebop;\bop}^{(s;k),(\alpha_0,2\alpha_{\!\scri},\alpha_+)}(\Omega) \colon P u \in \dot H_{\ebop;\bop}^{(s-1;k),(\alpha_0+2,2\alpha_{\!\scri}+2,\alpha_++2)}(\Omega) \bigr\} \\
      &\hspace{16em} \to \dot H_{\ebop;\bop}^{(s-1;k),(\alpha_0+2,2\alpha_{\!\scri}+2,\alpha_++2)}(\Omega)
  \end{align*}
  is invertible, where $\dot H_{\ebop;\bop}^{(s;k),(\alpha_0,2\alpha_{\!\scri},\alpha_+)}(\Omega)$ is the subspace of $H_{\ebop;\bop}^{(s;k),(\alpha_0,2\alpha_{\!\scri},\alpha_+)}(M)$ consisting of those elements with support contained in $\Omega$. In other words,\footnote{The equivalence of this statement with the invertibility of $P$ follows from the uniqueness of (distributional) forward solutions of $P u=f$.} for each forcing term $f\in\Hsupp_{\ebop;\bop}^{(s-1;k),(\alpha_0+2,2\alpha_{\!\scri}+2,\alpha_++2)}(\Omega)$, the unique forward solution $u$ of $P u=f$ satisfies $u\in\Hsupp_{\ebop;\bop}^{(s;k),(\alpha_0,2\alpha_{\!\scri},\alpha_+)}(\Omega)$.
\end{thm}
\begin{proof}
  We first consider the case $k=0$.

  \pfstep{(1) Solution up to a cross section of $\scri^+$.} Let $f\in\dot H_\ebop^{s-1,(\alpha_0+2,2\alpha_{\!\scri}+2,\alpha_++2)}(\Omega)$. Using the timelike nature of $t$ and $\tau$, one can solve $P u_0=f$ in the region $t\geq 0$, $\tau\leq\tau_0$ for any fixed $\tau_0<1$, with the solution $u_0$ lying in $\Hb^{s,\alpha_0}$ near $I^0\cap\{t\geq 0,\ \tau\leq\tau_0\}$. (For $s=1$, this follows from an energy estimate with multiplier $e^{-\digamma\tau}\rho_0^{-2(\alpha_0+1)}\pa_\tau$ for large $\digamma>1$; for $s\geq 1$ one uses real principal type propagation of regularity in the b-setting, and to get the full range of $s\in\R$, one uses duality arguments completely analogous to those the proof of Theorem~\ref{ThmE}.)

  Choosing $\tau_0$ close enough to $1$, Theorem~\ref{ThmE} can be used to extend $u_0$ as a solution of $P u_0=f$ to a neighborhood of $I^0\cap\scri^+$, i.e.\ to a domain $\Omega_{\delta,\rho}^0$ for suitable $\delta,\rho$. Indeed, for a cutoff $\chi$ which is $0$ for $\tau\leq\tau_0$ and $1$ for $\tau\geq\half(1+\tau_0)$, Theorem~\ref{ThmE} produces a forward solution $u_1$ of $P u_1=f_1:=P(\chi u_0)=\chi f+[P,\chi]u_0$, which thus has supported character (i.e.\ vanishes) at $\Sigma_{\delta,\rho}^{0,\rm in}$; and necessarily $u_0=(1-\chi)u_0+u_1$, with the right hand side defining the desired extension of $u_0$ of class $\dot H_\ebop^{s,(\alpha_0,2\alpha_{\!\scri})}$ (omitting the weight at $I^+$ for now).

  Fix now a smooth function $\ft$ on $M^\circ$ whose level sets $\ft^{-1}(c)$ for $c\in[-1,1]$ are spacelike and intersect $\scri^+$ in its interior; we further arrange that $\dd\ft$ is past timelike, with $\ft$ an increasing function of $t$ except in a neighborhood of $\scri^+$ where we demand $\ft=\rho_0-(\rho+C x_{\!\scri}^{2\ell_{\!\scri}})$ near $\scri^+$ for some $\rho\in(0,1)$ and large $C$ (cf.\ Lemma~\ref{LemmaOGeoHyp}). Thus, the level sets of $\ft$ interpolate between the level sets of $t$ and the surfaces $\Sigma_{\delta,\rho}^{0,\rm out}$, see~\eqref{EqESigma} and Figure~\ref{FigESetup}. Let now $\chi\in\CI(M)$ denote a cutoff which is $0$ in the causal past of $\ft=0$, and $1$ in the closure of the causal future $\Omega^+\subset M$ of $\ft\geq\half$. Then the global forward solution $u$ of $P u=f$ can be written as $u=(1-\chi)u_0+u'$ where $u'$ is the forward solution of
  \begin{equation}
  \label{EqSAEqn}
    P u'=f':=\chi f+[P,\chi]u_0 \in \dot H_\ebop^{s-1,(\alpha_0+2,2\alpha_{\!\scri}+2,\alpha_++2)}(\Omega).
  \end{equation}
  Since $(1-\chi)u_0\in\dot H_\ebop^{s,(\alpha_0,2\alpha_{\!\scri},\alpha_+)}(\Omega)$ (with the weight at $I^+$ in fact arbitrary), we need to show that the forward solution $u'$ lies in this space as well (now with the weight at $I^0$ arbitrary).

  \pfstep{(2) Solution near $I^+$.} Dropping the weight at $I^0$ from the notation, we shall show that
  \begin{equation}
  \label{EqSAOpRed}
    P \colon \bigl\{ u \in \dot H_\ebop^{s,(2\alpha_{\!\scri},\alpha_+)}(\Omega^+) \colon P u \in \dot H_\ebop^{s-1,(2\alpha_{\!\scri}+2,\alpha_++2)}(\Omega^+) \bigr\} \to \dot H_\ebop^{s-1,(2\alpha_{\!\scri}+2,\alpha_++2)}(\Omega^+)
  \end{equation}
  is invertible, i.e.\ that the equation~\eqref{EqSAEqn} has a solution in $\dot H_\ebop^{s,(\alpha_0+2,2\alpha_{\!\scri}+2,\alpha_++2)}(\Omega)$. We shall in fact show the stronger statement that~\eqref{EqSAOpRed} is invertible for all $s\in\R$. (This is possible since the only microlocal propagation result placing a restriction on $s$ is the radial point estimate at $\cR_{\rm c}$---Lemma~\ref{LemmaMEc}---, which lies outside of $\Omega^+$.)

  We give an argument in the spirit of \cite[\S\S4.2 and 5.3]{HintzVasyMink4}, where we first prove surjectivity for $s=1$ and very negative weights at $I^+$ (i.e.\ allowing for fast polynomial growth) using a global (near $I^+$) energy estimate, which we upgrade to $s\geq 1$ using propagation estimates; to improve the $I^+$-weight to $\alpha_+$, we use the invertibility properties of the $I^+$-normal operator. The full range of $s$ is then obtained using an approximation argument. See Remark~\ref{RmkSAAlt} for an alternative argument which avoids the use of a global energy estimate.

  \pfsubstep{(2.i)}{Solution for $s\geq 1$ and very negative $I^+$-weight.} When $s=1$ and if one replaces $\alpha_+$ by $\alpha'_+\leq\alpha_+$ with $\alpha'_+$ sufficiently negative, a solution can be shown to exist via an energy estimate on $\Omega^+$ in which one exploits the timelike nature of $\rho_+\pa_{\rho_+}$ near $I^+\setminus\scri^+$ by using a vector field multiplier which transitions between $-\rho_+^{-2(\alpha'_++1)}\rho_+\pa_{\rho_+}$ and $x_{\!\scri}\pa_{x_{\!\scri}}-(2+c)\rho_+\pa_{\rho_+}$ near $\scri^+\cap I^+$ (cf.\ \eqref{EqNeVF}); see \cite[Proposition~4.11]{HintzVasyMink4} for such an argument. Higher regularity $s\geq 1$ then follows by propagating edge-b-regularity through $\scri^+$ using Theorem~\ref{ThmMEFw}\eqref{ItMEFwThru} (or indeed just Lemma~\ref{LemmaMEinp}), then along the characteristic set of $P$ in the b-cotangent bundle over $(I^+)^\circ$, and then into $\scri^+$ using Theorem~\ref{ThmMEFw}\eqref{ItMEFwOut} (or indeed just Lemma~\ref{LemmaMEout}\eqref{ItMEoutFwLoc}); the control of $u$ in a punctured neighborhood of the outgoing radial set $\cR_{\rm out}$ required for this final step uses the non-trapping assumption. (If one constructs the solution $u$ via a duality argument, one can cover all $s\geq 0$ in this fashion.)

  \pfsubstep{(2.ii)}{Solution for $s\geq 1$ and the $I^+$-weight $\alpha_+$.} We have so far obtained a solution $u\in\dot H_\ebop^{s,(2\alpha_{\!\scri},\alpha'_+)}(\Omega^+)$, and improve its decay using a standard normal operator, Mellin transform, and contour shifting argument near $I^+$. The key ingredient is the fact that
  \begin{equation}
  \label{EqSAMTInv}
    \wh{P_+}(\lambda)^{-1} \colon H_0^{s-2,2\gamma_{\!\scri}}(I^+) \to H_0^{s,2\gamma_{\!\scri}}(I^+)
  \end{equation}
  (where $\gamma_{\!\scri}=\alpha_{\!\scri}+\frac{n}{2}$ as in~\eqref{EqNpL2}, and the Sobolev spaces are defined with respect to a positive b-density on $I^+$) is an analytic family of bounded operators for $-\Im\lambda<\gamma_{\!\scri}<\bar q$, which moreover satisfies high energy estimates
  \begin{equation}
  \label{EqSAMTInvHi}
    \|v\|_{H_{0,h}^{s,2\gamma_{\!\scri}}(I^+)} \leq C h^{-1}\|\wh{P_+}(\lambda)v\|_{H_{0,h}^{s-2,2\gamma_{\!\scri}}(I^+)},\qquad h=\la\lambda\ra^{-1},
  \end{equation}
  when $-\Im\lambda=\gamma_+$, with $\gamma_+\in(-\infty,\gamma_{\!\scri})$ contained in a compact subinterval. To begin the proof of~\eqref{EqSAMTInv}, one combines Theorem~\ref{ThmNpB}\eqref{ItNpBDirect} and elliptic estimates for $\wh{P_+}(\lambda)$ in $(I^+)^\circ$; the ellipticity of $\wh{P_+}(\lambda)$ (as a differential operator on $(I^+)^\circ$) follows from assumption~\eqref{EqSAMet3} above. This gives the estimate
  \[
    \|v\|_{H_0^{s,2\gamma_{\!\scri}}(I^+)} \leq C\Bigl( \|\wh{P_+}(\lambda)v\|_{H_0^{s-2,2\gamma_{\!\scri}}(I^+)}+\|v\|_{H_0^{-N,-N}(I^+)}\Bigr).
  \]
  Moreover, an application of the parametrix used in the proof of Theorem~\ref{ThmNpB} shows that every $v\in H_0^{s,2\gamma_{\!\scri}}(I^+)$ with $\wh{P_+}(\lambda)v=0$ automatically satisfies $v\in\cA^{2\gamma_{\!\scri}}(I^+)$; since by assumption~\eqref{EqSAOp3} this implies $v=0$, we can drop the compact error term in this estimate upon increasing $C$. (One can choose $C$ uniformly when $\lambda\in\C$ is restricted to a compact set.) In order to get~\eqref{EqSAMTInv}, one also needs the surjectivity of $\wh{P_+}(\lambda)$; this is a consequence of the invertibility of $\wh{P_+}(\lambda)$ for large $|\Re\lambda|$ (when $\Im\lambda$ is contained in a compact subinterval of $(-\infty,\gamma_{\!\scri})$), to which we turn now.

  To wit, the estimate~\eqref{EqSAMTInvHi} follows by applying Theorem~\ref{ThmNpH}\eqref{ItNpHin} to control semiclassical 0-regularity at ${}^0\cR_{\rm in}^\pm$ (the sign corresponding to the sign of $\Re\lambda$), which due to the non-trapping assumption on $P$ can be propagated into a full punctured neighborhood of ${}^0\cR_{\rm out}^\pm$. There, Theorem~\ref{ThmNpH}\eqref{ItNpHout} applies. Thus, one gets~\eqref{EqSAMTInvHi} with an error term $C h^N\|v\|_{H_{0,h}^{-N,2\gamma_{\!\scri}}(I^+)}$ for any fixed $N$, which for sufficiently small $h$ can be absorbed into the left hand side.

  By analogous means, but propagating in the reverse direction using Theorem~\ref{ThmNpHAdj}, one can prove the adjoint estimate
  \[
    \|v\|_{H_{0,h}^{-s+2,-2\gamma_{\!\scri}}(I^+)} \leq C h^{-1}\|\wh{P_+^*}(\lambda)v\|_{H_{0,h}^{-s,-2\gamma_{\!\scri}}(I^+)},\qquad h=\la\lambda\ra^{-1},
  \]
  which in particular implies the triviality of the cokernel of $\wh{P_+}(\lambda)$ for sufficiently large $|\Re\lambda|$. Since the Fredholm index of $\wh{P_+}(\lambda)\colon H_0^{s,2\gamma_{\!\scri}}(I^+)\to H_0^{s-2,2\gamma_{\!\scri}}(I^+)$ is constant in $\lambda$, this completes the proof of~\eqref{EqSAMTInv}.

  Returning to the task of improving the decay of $u\in\dot H_\ebop^{s,(2\alpha_{\!\scri},\alpha'_+)}(\Omega^+)$, one writes
  \[
    N_{I^+}(\rho_+^{-2}x_{\!\scri}^{-2}P)u = \rho_+^{-2}x_{\!\scri}^{-2}f - \bigl(P-N_{I^+}(\rho_+^{-2}x_{\!\scri}^{-2}P)\bigr)u,
  \]
  with the second term on the right lying in $\dot H_\ebop^{s-2,(2\alpha_{\!\scri},\alpha'_++\ell_+)}(\Omega^+)$. Inverting $N_{I^+}(\rho_+^{-2}x_{\!\scri}^{-2}P)$ using the Mellin transform (see Lemma~\ref{LemmaNpHMellin}) and its inverse gives
  \[
    u\in\dot H_\ebop^{s-1,(2\alpha_{\!\scri},\min(\alpha_+,\alpha'_++\ell_+))}(\Omega^+),
  \]
  which improves on the weight of $u$ at $I^+$, at the expense of $1$ edge-b-derivative. The edge-b-regularity can then be improved to $s$ again using the microlocal propagation results as before (note that the propagation results near $\scri^+\cap I^+$ do not require any conditions on $s$). This improves the decay of $u$ at $I^+$ by $\ell_+$ until the decay rate $\alpha_+$ is obtained after a finite number of steps. This proves the Theorem for $s\geq 1$ and $k=0$.

  \pfsubstep{(2.iii)}{Full range of $s$.} Consider now the remaining case that $s<1$. The microlocal propagation estimates imply an a priori estimate for $P$,
  \begin{equation}
  \label{EqSAAdjEst}
    \|u\|_{\dot H_\ebop^{s,(2\alpha_{\!\scri},\alpha_+)}(\Omega^+)} \leq C\Bigl( \|P u\|_{\dot H_\ebop^{s-1,(2\alpha_{\!\scri}+2,\alpha_++2)}(\Omega^+)} + \|u\|_{\dot H_\ebop^{-N,(2\alpha_{\!\scri},\alpha_+)}(\Omega^+)} \Bigr),
  \end{equation}
  where we fix $N$ so that $-N<s-2$. We then apply Theorem~\ref{ThmNeP}\eqref{ItNePFw} to the second term on the right, with $\eps$ in~\eqref{EqNePBw} chosen so small that $C\eps<\half$, and therefore the error term $\eps\|u\|_{\dot H_\ebop^{-N+1,(2\alpha_{\!\scri},\alpha_+)}(\Omega^+)}$ from~\eqref{EqNePFw} can be absorbed into the left hand side of~\eqref{EqSAAdjEst}. This implies that~\eqref{EqSAAdjEst} holds (for a different constant) with the error term $\|u\|_{\dot H_\ebop^{-N+1,(2\alpha_{\!\scri}-2\ell_{\!\scri},\alpha_+)}(\Omega^+)}$ which is improved (i.e.\ weaker) at $\scri^+$. One next estimates $u$, localized to a collar neighborhood of $I^+$, in terms of the $I^+$-normal operator of $\rho_+^{-2}x_{\!\scri}^{-2}P$ applied to $u$ using the (inverse) Mellin transform. Since $\rho_+^{-2}x_{\!\scri}^{-2}P$ differs from its $I^+$-normal operator by an element of $\cA^{(0,\ell_+)}\Diffeb^2$, this improves the error term further to $\|u\|_{\dot H_\ebop^{-N+2,(2\alpha_{\!\scri}-2\ell_{\!\scri},\alpha_+-\ell_+)}(\Omega^+)}$. Since $\dot H_\ebop^{s,(2\alpha_{\!\scri},\alpha_+)}(\Omega^+)$ embeds compactly into this space, this now implies that $P$, as a map~\eqref{EqSAOpRed}, has closed range. But we have already shown that the range includes $\dot H_\ebop^{0,(2\alpha_{\!\scri},\alpha_+)}(\Omega^+)$, which is a dense subspace. Therefore, $P$ in~\eqref{EqSAOpRed} is surjective. The proof is complete in the case $k=0$.

  \pfstep{(3) Higher b-regularity.} The proof of the Theorem for $k\in\N$ follows from the same inductive argument as in the proof of Corollary~\ref{CorEFwb}.
\end{proof}

\begin{rmk}[Alternative proof of solvability]
\label{RmkSAAlt}
  An alternative proof of the invertibility of~\eqref{EqSAOpRed} proceeds as follows. First, one shows that $P$ in~\eqref{EqSAOpRed} is Fredholm (which again uses the $I^+$-normal operator and the Mellin transform as in step (2.iii) of the above proof). The $I^+$-normal operator is surjective on the spaces~\eqref{EqSAOpRed} since solutions can be written down using the (inverse) Mellin transform (using the Paley--Wiener theorem for the support condition). But the localization of the difference of $P$ and its $I^+$-normal operator to a neighborhood $\{\rho_+<\eps\}$ of $I^+$ tends to $0$ in $x_{\!\scri}^2\rho_+^2\cA^{(0,\delta)}\Diffeb^2$ as $\eps\searrow 0$ for any fixed $\delta\in(0,\ell_+)$, and one can then upgrade the Fredholm property of $P$ to invertibility on $\{\rho_+<\eps\}$ for small enough $\eps$. The \emph{finite time} solvability of $P$ on the remaining region $\Omega^+\setminus\{\rho_+<\eps\}$ is clear, and thus one obtains the invertibility of~\eqref{EqSAOpRed}. For a detailed implementation of this approach, see \cite[Proof of Theorem~5.23, and Appendix~A]{HintzNonstat}.
\end{rmk}

\begin{rmk}[Conormal regularity and pointwise decay]
\label{RmkSADecay}
  If we work with $k>\frac{n+1}{2}$ degrees of b-regularity, the pointwise bound on $u$ provided by Theorem~\ref{ThmSA} is $\cO(\rho_{\!\scri}^{\alpha_{\!\scri}+\frac{n}{2}}\rho_+^{\alpha_++\frac{n+1}{2}})$ (cf.\ \eqref{EqNpL2}) by Sobolev embedding (for b-Sobolev spaces). In Example~\ref{ExSAMink}, the strongest possible bound arises by taking $\alpha_{\!\scri}=-\half-\eps$ and $\alpha_+=-1-2\eps$, and hence one gets almost (namely, up to an $r^\eps$ loss for any $\eps>0$) the sharp $\cO(r^{-\frac{n-1}{2}})$ decay towards $\scri^+$, and almost $\cO((t-r)^{-\frac{n-1}{2}})$ decay towards future timelike infinity (which is far from the sharp $(t-r)^{-(n-1)}$ bound on Minkowski spacetimes with odd spacetime dimension, but matches what simple vector field methods give, see e.g.\ \cite{KlainermanUniformDecay}). Improved decay at $I^+$ requires estimates on the meromorphic continuation of $\wh{P_+}(\lambda)^{-1}$ across the line $-\Im\lambda=\ubar p_{1,+}+\frac{n-1}{2}$, which is a delicate problem, cf.\ Remark~\ref{RmkNpBCont}.
\end{rmk}

\bibliographystyle{alphaurl}


\end{document}